\documentclass[10pt,a4paper,titling]{article}
\usepackage[T1]{fontenc}
\usepackage[utf8]{inputenc}
\usepackage[english]{babel}

\usepackage{amsfonts}
\usepackage{amsmath}
\usepackage{amssymb}
\usepackage{amsthm}
\usepackage{mathrsfs}
\usepackage{mathtools}
\usepackage{dsfont}
\usepackage{tikz-cd}
\usetikzlibrary{decorations.pathmorphing}
\usepackage{extpfeil}
\usepackage{enumitem}
\usepackage[perpage,bottom]{footmisc}
\usepackage{caption}
\usepackage{cancel}
\usepackage{textcomp}
\usepackage{xcolor}
\usepackage{minibox}
\usepackage{imakeidx}

\makeindex[columns=2, title=Index, intoc]
\usepackage[linktoc=all]{hyperref}
\usepackage{fancyhdr}
\usepackage{graphicx}
\graphicspath{ {./Figures/} }
\newtheorem{Thm}{Theorem}[section]
\newtheorem{Lem}[Thm]{Lemma}
\newtheorem{Pro}[Thm]{Proposition}
\newtheorem{Cor}[Thm]{Corollary}
\newtheorem{Con}[Thm]{Conjecture}
\theoremstyle{definition}
\newtheorem{Def}[Thm]{Definition}
\newtheorem{Ex}[Thm]{Example}
\newtheorem{Rem}[Thm]{Remark}

\begin{document}
\begin{titlepage}
	\begin{center}
		
		\LARGE \textbf{About Homological Mirror Symmetry}\\[1.62cm] 
		
		\Large Alessandro Imparato\footnote{Department of Mathematics, ETH Zürich, Summer--Autumn 2022. This work was last updated on June 21, 2023.} \\[3.14cm]

		\begin{abstract}
			The B-side of Kontsevich's Homological Mirror Symmetry Conjecture is discussed. We give first a self-contained study of derived categories and their homological algebra, and later restrict to the bounded derived category of schemes and ultimately Calabi--Yau manifolds, with particular emphasis on the basics of the underlying sheaf theory, and the algebraic features therein. Finally, we loosely discuss the lowest dimensional manifestations of homological mirror symmetry, namely for elliptic curves and $K3$ surfaces.
			
			The present work is a sequel to the author's survey \cite{[Imp21]} on the A-side of homological mirror symmetry. 
		\end{abstract}
		
	\end{center}
	
	\vspace*{6cm}
	
	\noindent This article is the author's Master Thesis, supervised by Prof. Dr. Paul Biran and submitted for the degree \textit{Master of Science ETH in Applied Mathematics}. It should be accessible to graduate students.\\
	
	
\end{titlepage}

\newpage
\pagenumbering{roman}
\thispagestyle{plain}

\tableofcontents

\newpage

\setcounter{page}{1}

\pagenumbering{arabic}

\section*{Introduction}\markboth{INTRODUCTION}{INTRODUCTION}
\addcontentsline{toc}{section}{Introduction}
\thispagestyle{plain}

\vspace*{1cm}

In 1994, the mathematician Maxim Kontsevich formulated his famous conjecture (transcripted in \cite{[Kon94]}) of Homological Mirror Symmetry, a theory based on the classical Mirror Symmetry which arose from the calculations of many theoretical physicists investigating dualities between otherwise unrelated spaces. They particularly focused on \textit{Calabi--Yau manifolds}, which are presumed by superstring theory to be the compact portions of our hypothetical 10-dimensional spacetime --- mathematically speaking, they are specific kind of Kähler manifolds, thus possessing both the structure of a symplectic and a complex manifold. Kontsevich's conjecture brought renewed interest to the field, asserting that the duality of a mirror pair of Calabi--Yau manifolds $(X,\overline{X})$ could be rephrased through the abstract language of category theory as a relationship between their symplectic geometric ``A-side'' and their complex algebraic geometric ``B-side''. 

The leading claim was successively refined by other researchers like Kenji Fukaya, resulting in the following statement: there exist equivalences of triangulated categories      
\[
\mathsf{D}^\pi(\widetilde{\mathscr{F}}(X))\cong\mathsf{D^b}(\overline{X})\qquad\text{and}\qquad \mathsf{D}^\pi(\widetilde{\mathscr{F}}(\overline{X}))\cong\mathsf{D^b}(X)\,,
\] 
where $\mathsf{D}^\pi(\widetilde{\mathscr{F}}(X))$ is the \textit{split-closed derived category} of the Fukaya category of $X$, and $\mathsf{D^b}(X)$ is the \textit{bounded derived category} of coherent sheaves on $X$. The A-side of homological mirror symmetry was thoroughly discussed in \cite{[Imp21]}, where we argued that it boils down to understanding $\mathsf{D}^\pi(\widetilde{\mathscr{F}}(X))$ (respectively $\mathsf{D}^\pi(\widetilde{\mathscr{F}}(\overline{X}))$ if one wishes to focus on the mirror partner): the (\textit{enriched}) \textit{Fukaya category} $\widetilde{\mathscr{F}}(X)$ consists of special Lagrangian submanifolds $L\subset X$ each equipped with a local system, in turn made of a complex vector bundle $E\rightarrow L$ endowed with a flat connection $\nabla$ which induces unitary holonomy groups and makes the curvature a constant multiple of the B-field. We called A\textit{-branes} the resulting triplets $\mathcal{L}=(L,E,\nabla)$ as a nod to D-branes --- in terms of which the conjecture above acquires a more apparent physical significance --- and we chained several A-branes together to form bounded complexes by following the recipe for split-closed derived categories. We also found the associated morphism spaces to be computable through Floer cohomology. Most importantly, we saw that $\mathsf{D}^\pi(\widetilde{\mathscr{F}}(X))$ exclusively captures the symplectic geometry of the underlying Calabi--Yau manifold $X$.

On the other hand, in \cite[section 7.3]{[Imp21]} we only briefly touched upon the definition of $\mathsf{D^b}(X)\equiv\mathsf{D^b}(\mathsf{Coh}(X))$: much like its Fukaya counterpart, we learned that it consists of geometric data, now enclosed in the \textit{category of coherent sheaves} $\mathsf{Coh}(X)$ on $X$, together with a purely homological machinery, that of derived categories (whose analogue for $A_\infty$-categories we blindly applied in the split-closed construction, without any pretense at even grasping the classical theory!). If we forget for the moment these base ingredients and try to cheat, obtaining a homological mirror duality by imposing the B-side to look specular to the A-side, we would be led to assume that the fundamental objects ought to be some sort of holomorphically embedded (even-dimensional) submanifolds of $X$ equipped with holomorphic vector bundles $E_i$, and we would argue that the relevant morphisms of any pair labelled by, say, $(E_i,E_j)$ should belong to the $\text{Hom}(E_i,E_j)$-valued Dolbeault cohomology groups of $X$. This approach (for example mocked in \cite{[Asp04]}) is, of course, naive and wrong! And the reason is simple enough: this way, there are too few objects, thus morphisms, compared to those appearing on the A-side. Put in physics words, we fall short of B-branes!     

Our goal here is to correct this false intuition. In doing so, we will find out (as already anticipated above) that there is an equally rich theory regulating the B-side of things. Namely, we will need to embark on a journey through classical derived category theory (which we can parse without having to worry about as many subtleties as those afflicting $A_\infty$-categories), and through the algebraic geometry of sheaves and schemes, which, not unlike Fukaya categories, present their own challenges and thus unavoidable compromises (but really, the reader will probably find the B-side to be the easier one!). As done in the previous installment, we favour again the ``bigger picture'' approach over a detailed review of each tiny aspect involved, for the topic is quite \textit{vast}.\\    

Accordingly, the work consists of two parts: in the first one, we give a self-contained overview of derived categories and their homological algebra, while in the second we discuss the theory of sheaves on various spaces. Specifically, we begin in Chapter 1 by describing abelian categories, the building blocks for derived categories, whose language we introduce in Chapter 2. Under mild conditions, it is also possible to obtain derived functors between derived categories, as we address in Chapter 3. In Chapter 4 we argue that derived categories are particular instances of triangulated categories, thus restoring some desirable properties which are lost in the derivation process. To the second part, in Chapter 5 we recollect the basic category theoretic facts about sheaves, first working on topological spaces and later on varieties and schemes, whose algebraic geometrical language we promptly refresh. In Chapter 6, we apply the knowledge of Part I to the category of coherent sheaves on schemes and we study its cohomology, ultimately restricting our focus to the bounded derived category of Calabi--Yau varieties. Finally, in Chapter 7 we conclude our exploration of homological mirror symmetry by looking at the easiest known proof of its conjecture, that for elliptic curves, with a last-second cameo of $K3$ surfaces. \\

\noindent\textbf{Acknowledgments.} I would like to thank my supervisor, Paul Biran, for taking some time to read this thesis, always approving the very specific vision I had for it, and giving his support in a period of emotional turmoil. And I would like to thank Prof. Will Merry for presenting me two years ago with the wonderful topic that is Homological Mirror Symmetry. I will always be grateful to you.
\vspace*{1cm}

\hfill\textit{This work is dedicated to the memory of William John Merry}

\newpage

\part{Homological Algebra}
\markboth{I\quad HOMOLOGICAL ALGEBRA}{I\quad HOMOLOGICAL ALGEBRA}
\thispagestyle{plain}

\vspace*{2cm}

\noindent Derived categories occupy a place within the broader theory of homological algebra, which debuted in the 1940's and was firmly rooted in the mathematical landscape by the classical \textit{Homological Algebra} by Cartan and Eilenberg of 1956. A little later, in the early 60's, derived categories were first developed by Alexander Grothendieck and his student Jean-Louis Verdier, particularly in the latter's thesis of 1963 (reported in \cite{[Ver96]}), where their axiomatization led to the concept of triangulated category. Reason for the introduction of a derived formalism was to support Grothendieck's research in algebraic geometry, which perhaps also explains why it spreaded quite slowly and acquired only recently renewed impetus. A fruitful symbiosis that advanced the knowledge, among the many subdisciplines, of mirror symmetry itself.

The key idea in the process of derivation of a category, we will see, lies in a generalization of how in commutative algebra rings or modules can be localized, providing them with a concept of ``fraction''. There, one just needs to pick the ``denominators'' to belong to any multiplicative subset, that is, a subset closed under multiplication and possessing a multiplicative unit. We point out that this method was in turn historically inspired by the definition of local rings of algebraic varieties (more about this in Part II). In a similar fashion, we will initially characterize the derived category of \textit{any} given category by formally inverting an arbitrary class of its morphisms and imposing a universal property; only subsequently will we address a more natural approach to their construction.

Of course, wherever categories appear, functors soon follow. This is also the case in the derived setting, where one talks about derived functors. They actually emerged first (in their ``classical'', cohomological vest) in Cartan and Eilenberg's work, before a notion of derived category was even established. Until then, mathematicians were resorting to more involved tools such as spectral sequences to study them; the need for a simpler procedure was in part what stimulated the development of a derived formalism!      

The scope of Part I is as complete an introduction as possible to the algebraic machinery needed to understand derived categories, gradually building towards the item $\mathsf{D^b}(\mathsf{A})$ appearing on the B-side of the Homological Mirror Symmetry Conjecture; in Part II we will be more specific about our preferred ground category $\mathsf{A}$. The precise plan reads as follows: 

\begin{itemize}[leftmargin=0.5cm]
	\item In Chapter 1, we introduce the key features our to-be-derived category $\mathsf{A}$ must have, namely those of an abelian category, which are neatly articulated in four axioms; on any such $\mathsf{A}$ one can then define some very convenient notions, like cohomology, which effectively allows us to do homological algebra. We also characterize functors between abelian categories, addressing their exactness.  
	
	\item In Chapter 2, we begin for good our journey by refreshing the definitions of cochain complex, chain map and how they interact. Then we introduce derived categories via their universal property, and give a qualitative description of their formal structure as result of a localization procedure. Unfortunately, derived categories are \textit{not} abelian... but we can still analyze them through surrogate tools, such as exact and distinguished triangles that suitably simulate the behaviour of short exact sequences. In fact, we can use these to explicitly construct any derived category as the localization of the underlying homotopy category $\mathsf{H(A)}$ by its class of quasi-isomorphisms. We then see how $\mathsf{A}$ embeds into $\mathsf{D(A)}$. Finally, we define Ext-groups and deduce some of their useful properties which will pop up times and again; also, we explain what is the homological dimension of abelian categories and of individual objects therein.      
	
	\item In Chapter 3, we move our focus to the maps relating derived categories, namely right and left derived functors, originating from left respectively right exact standard functors. We show under which conditions they exist --- in which case they are unique --- and how they are constructed, proving that they are also exact, in the sense that they preserve distinguished triangles. Moreover, we see that derived functors descend in cohomology to classical derived functors between the underlying standard categories, still retaining desirable properties like the long exact sequence axiom and the commutativity with respect to composition. We also look at a few ever-present functors and cast some light on their derived counterparts.   
	
	\item In Chapter 4, we expand upon the idea of triangulated categories, only mentioned on the fly in \cite[section 3.1]{[Imp21]}, reviewing Verdier's axioms. We proceed to show that the homotopy category of any abelian category is in fact triangulated, and we extend the proof to the associated derived category; this makes sense of the seemingly clashing choice of terminology regarding distinguished triangles, cones, cohomological functors and friends, then understood in both derived and triangulated sense. This also furnishes an alternative explanation for why derived categories are still so nice despite not being abelian!    
\end{itemize}   

References for this first part are primarily the wonderful book \cite{[GM03]} by Gelfand and Manin, and more occasionally \cite{[HS97]}, \cite{[KS06]}, \cite{[Mac78]} and \cite{[Wei94]}, just five among the classical treatments on derived categories which can be found throughout the mathematical literature. Unless declared otherwise, the reader may safely assume that we are drawing from \cite{[GM03]}.

\newpage
\pagestyle{fancy}

\section{Abelian categories}
\thispagestyle{plain}

\subsection{Additive categories}\label{ch1.1}

As costumary, we begin our story with some fundamentals in category theory. Additive and abelian categories distinguish themselves from ordinary categories in that they fulfill some additional specific axioms, which we proceed to list in the present chapter. The reader is warned to be on the lookout for equivalent choices of axioms and notation adopted throughout literature. Whenever useful, we point out these alterations, but stick close to the narrative of our main references \cite[sections II.5--II.6]{[GM03]} and \cite[section 4.1]{[ABC+09]}.

\begin{Def}\label{addcat}
	A category $\mathsf{C}$ is a \textbf{preadditive category}\index{preadditive category} if the following axiom is satisfied:
	\begin{itemize}[leftmargin=0.5cm]
		\renewcommand{\labelitemi}{\textendash}
		\item (\textbf{A1}) $\text{Hom}_\mathsf{C}:\mathsf{C}^\text{opp}\times\mathsf{C}\rightarrow\mathsf{Ab}$ is a bifunctor, that is, each set $\text{Hom}_\mathsf{C}(X,Y)$ for any $X,\,Y\in\text{obj}(\mathsf{C})$ is an abelian group and the composition law is biadditive, thus giving group homomorphisms $\circ:\text{Hom}_\mathsf{C}(X,Y)\times \text{Hom}_\mathsf{C}(Y,Z)\rightarrow\text{Hom}_\mathsf{C}(X,Z),\,(f,g)\mapsto g\circ f$ (so that $g\circ (f_1+f_2)=g\circ f_1 + g\circ f_2$ and $(g_1+g_2)\circ f= g_1\circ f + g_2\circ f$ for compatible morphisms).
	\end{itemize}
	A preadditive category $\mathsf{C}$ is an \textbf{additive category}\index{additive category} if additionally hold:
	\begin{itemize}[leftmargin=0.5cm]
		\renewcommand{\labelitemi}{\textendash}
		\item (\textbf{A2}) There exists a zero object $0\in\text{obj}(\mathsf{C})$ (then unique up to isomorphism), which by definition makes $\text{Hom}_\mathsf{C}(0,X)$ and $\text{Hom}_\mathsf{C}(X,0)$ one-element groups for any $X\in\text{obj}(\mathsf{C})$;
		
		\item (\textbf{A3}) For any two $X_1,\,X_2\in\text{obj}(\mathsf{C})$ there exists a $Y\in\text{obj}(\mathsf{C})$ and morphisms as in the diagram
		\begin{equation}\label{directprodsum}
			\begin{tikzcd}
				X_1\arrow[r, yshift=1ex, "i_1"] & Y\arrow[l, yshift=-0.5ex, "p_1"]\arrow[r, yshift=-0.5ex, "p_2"'] & X_2\arrow[l, yshift=1ex, "i_2"']
			\end{tikzcd}
		\end{equation}
		such that $p_1\circ i_1=\text{id}_{X_1},\,p_2\circ i_2=\text{id}_{X_2},\,p_2\circ i_1=0=p_1\circ i_2$ and $i_1\circ p_1 + i_2\circ p_2=\text{id}_Y$.
	\end{itemize}
\end{Def}

Additive categories in which Hom-sets are moreover vector spaces over a field $\mathbb{K}$ and composition of morphisms is $\mathbb{K}$-bilinear, are called \textit{$\mathbb{K}$-linear categories}. It is worth noting that axioms (A2) and (A3) are just a concealed way to say that in additive categories we can form direct sums (coproducts) and products of any two objects and still obtain an element in the class $\text{obj}(\mathsf{C})$:

\begin{Lem}
	For $X_1, X_2, Y\in\textup{obj}(\mathsf{C})$ as in axiom \textup{(A3)}, $Y$ is both the direct product and the direct sum of $X_1$ with $X_2$, unique due to their characterizing universal property, so that we may unambiguously write $Y=X_1\times X_2$ respectively $Y=X_1\oplus X_2$ \textup(we favour the latter notation\textup).  
\end{Lem}

\begin{proof}
	We need only check that, given a further $Y'\in\text{obj}(\mathsf{C})$ with associated diagram
	\[
	\begin{tikzcd}
		X_1\arrow[r, yshift=1ex, "i_1'"] & Y'\arrow[l, yshift=-0.5ex, "p_1'"]\arrow[r, yshift=-0.5ex, "p_2'"'] & X_2\arrow[l, yshift=1ex, "i_2'"']
	\end{tikzcd}
	\]
	as in \eqref{directprodsum}, the universal properties for the categorical direct product and direct sum are satified: for any pairs of morphisms $X_1\xleftarrow{f_1}Y'\xrightarrow{f_2}X_2$ respectively $X_1\xrightarrow{g_1}Y'\xleftarrow{g_2}X_2$ there exist dashed morphisms $\varphi:Y'\rightarrow Y$ respectively $\phi:Y\rightarrow Y'$ in $\mathsf{C}$ such that the diagrams
	\begin{equation}
		\begin{tikzcd}
			Y'\arrow[dr, dashed, "\varphi"']\arrow[drr, out=0, in=150, "f_1"]\arrow[ddr, out=270, in=135, "f_2"'] & & \\ & Y\arrow[r, "p_1"']\arrow[d, "p_2"] & X_1 \\
			& X_2 & 
		\end{tikzcd}
		\qquad\text{and}\qquad
		\begin{tikzcd}
			Y' & & \\ & Y\arrow[ul, dashed, "\phi"] & X_1\arrow[l, "i_1"]\arrow[ull, out=150, in=0, "g_1"'] \\
			& X_2\arrow[u, "i_2"']\arrow[uul, out=145, in=270, "g_2"] & 
		\end{tikzcd}
	\end{equation}
	commute. The former is given by $\varphi\coloneqq i_1\circ f_1+ i_2\circ f_2$ (indeed, by axiom (A3), we have $p_1\circ\varphi = \text{id}_{X_1}\circ f_1+0\circ f_2=f_1$ and $p_2\circ\varphi = 0\circ f_1+\text{id}_{X_2}\circ f_2=f_2$), while the latter by $\phi \coloneqq g_1\circ p_1 + g_2\circ p_2$ (so that $\phi\circ i_1=g_1$ and $\phi\circ i_2=g_2$). They are unique by construction. In particular, taking $f_k=p_k'$ and $g_k=i_k'$ for $k=1,2$ we get $\varphi\circ\phi=\text{id}_Y$ and $\phi\circ\varphi=\text{id}_{Y'}$, which confirms that $Y'\cong Y$, hence the uniqueness of \eqref{directprodsum} up to canonical isomorphism.  
\end{proof}

Therefore, axiom (A3) is just the requisite that any two objects admit a direct sum within the additive category, with the morphisms $p_k, i_k$ taking the obvious roles of projections onto and inclusions into the respective factors. Of course, direct sums of any finite\footnote{If the direct sum of any collection of objects, even \textit{infinitely} many, is itself an object, we call the category \textit{cocomplete}. The analogue property for products makes it \textit{complete}.} collection of objects belong to the additive category as well.
Also, for any $Z\in\text{obj}(\mathsf{C})$, the inclusions $i_k:X_k\rightarrow X_1\oplus X_2$ induce the isomorphism of abelian groups
\[
\text{Hom}_\mathsf{C}(Z,X_1)\times\text{Hom}_\mathsf{C}(Z,X_2)\xrightarrow{\sim}\text{Hom}_\mathsf{C}(Z,X_1\oplus X_2),\,(f,g)\mapsto i_1\circ f + i_2\circ g\;.
\]

Moreover, (A1) and (A2) allow us to give a categorical definition of kernel and cokernel of any given morphism. Namely, a morphism $\varphi:X\rightarrow Y$ in a category $\mathsf{C}$ respecting these two axioms induces two functors $\text{ker}\varphi:\mathsf{C}^\text{opp}\rightarrow\mathsf{Ab}$ and $\text{coker}\varphi=(\text{ker}(\varphi^\text{opp}))^\text{opp}$ defined in terms of $X$ and $Y$, and unique whenever existing (here $\varphi^\text{opp}:Y\rightarrow X$ denotes the morphism dual to $\varphi$). Fortunately, we can avoid these tedious constructions --- outlined in \cite[subsections II.5.6--II.5.8]{[GM03]} --- by noting that in almost all the additive categories we are concerned with, the well-known, set-theoretical definitions apply: $\ker(\varphi)$ is the preimage of the zero object and $\text{coker}(\varphi)$ is the target space quotiented by the image of $\varphi$, themselves identified as objects of the category rather than functors on it. However, this is not evident for the category of sheaves of abelian groups: although the kernel sheaf of a morphism of sheaves is itself a sheaf, the analogue for the cokernel is trickier to prove and requires some adjustments; more about this in Section \ref{ch5.1}.

Therefore, we limit ourselves to the following slightly flavoured redefinitions (and invite the reader to check their compatibility when for example $\mathsf{C}=\mathsf{Ab}$).

\begin{Def}\label{monepicoker}
	Let $\mathsf{C}$ be an additive category, $f\in\text{Hom}_\mathsf{C}(X,Y)$. Then $f$ is:
	\begin{itemize}[leftmargin=0.5cm]
		\item injective, or a \textbf{monomorphism}\index{monomorphism},\footnote{When talking about abstract additive/abelian categories, we favour the terminology mono-/epimorphism over injective/surjective morphism, so to draw a distinction with best known categories such as $\mathsf{Ab}$. Also, we will sometimes resort to arrows $\hookrightarrow$ for monomorphisms and $\twoheadrightarrow$ for epimorphisms.} if the group homomorphism given by post-composition $f_*:\text{Hom}_\mathsf{C}(W,X)\rightarrow\text{Hom}_\mathsf{C}(W,Y),\,\varphi\mapsto f\circ\varphi$ is injective for all $W\in\text{obj}(\mathsf{C})$;
		
		\item surjective, or an \textbf{epimorphism}\index{epimorphism}, if the group homomorphism given by pre-composition $f^*:\text{Hom}_\mathsf{C}(Y,Z)\rightarrow\text{Hom}_\mathsf{C}(X,Z),\,\varphi\mapsto\varphi\circ f$ is injective for all $Z\in\text{obj}(\mathsf{C})$.
	\end{itemize}
	Now, with reference to diagram \eqref{kercokerdiag} below, we say that:
	\begin{itemize}[leftmargin=0.5cm]
		\item A morphism $k\in\text{Hom}_\mathsf{C}(K,X)$ is a \textbf{kernel}\index{kernel} of $f$ if the sequence
		\[
		\{0\}\rightarrow\text{Hom}_\mathsf{C}(W,K)\xrightarrow{k_*}\text{Hom}_\mathsf{C}(W,X)\xrightarrow{f_*}\text{Hom}_\mathsf{C}(W,Y)
		\]
		is exact in $\mathsf{Ab}$ for all $W\in\text{obj}(\mathsf{C})$; equivalently, $k$ is a kernel of $f$ if $f\circ k=0$ and any $k'\in\text{Hom}_\mathsf{C}(K',X)$ such that $f\circ k'=0$ can be factored as $k'=k\circ h$ for a unique $h\in\text{Hom}_\mathsf{C}(K',K)$. Then $k$ is necessarily a monomorphism, and we should rather write $\ker(f)\coloneqq k:K\rightarrow X$, though we directly call $\ker(f)\equiv K \in\text{obj}(\mathsf{C})$ the kernel of $f$.
		
		\item A morphism $c\in\text{Hom}_\mathsf{C}(Y,C)$ is a \textbf{cokernel}\index{cokernel} of $f$ if the sequence
		\[
		\{0\}\rightarrow\text{Hom}_\mathsf{C}(C,Z)\xrightarrow{c^*}\text{Hom}_\mathsf{C}(Y,Z)\xrightarrow{f^*}\text{Hom}_\mathsf{C}(X,Z)
		\]
		is exact in $\mathsf{Ab}$ for all $Z\in\text{obj}(\mathsf{C})$; equivalently, $c$ is a cokernel of $f$ if $c\circ f=0$ and any $c'\in\text{Hom}_\mathsf{C}(Y,C')$ such that $c'\circ f= 0$ factors as $c'=h'\circ c$ for a unique $h'\in\text{Hom}_\mathsf{C}(C,C')$. Then $c$ is necessarily an epimorphism, and we should rather write $\text{coker}(f)\coloneqq c:Y\rightarrow C$, though we directly call $\text{coker}(f)\equiv C\in\text{obj}(\mathsf{C})$ the cokernel of $f$.
	\end{itemize}
	\begin{equation}\label{kercokerdiag}
		\begin{tikzcd}[row sep=small]
			\ker(f)\equiv K\arrow[dr, hook, "k"] & & & \text{coker}(f)\equiv C\arrow[dd, dashed, "h'"] \\
			& X\arrow[r, "f"] & Y\arrow[ur, two heads, "c"]\arrow[dr, "c'"'] & \\
			K'\arrow[uu, dashed, "h"]\arrow[ur, "k'"'] & & & C' 
		\end{tikzcd}
	\end{equation}
	By definition, when existing, kernels and cokernels are unique up to isomorphism, leaving the notation unambiguous.
\end{Def}

\subsection{Abelian categories}\label{ch1.2}

Now, for abelian categories only an additional piece of data is required:

\begin{Def}\label{abcat}
	An additive category $\mathsf{A}$ is an \textbf{abelian category}\index{abelian category} if the following axiom is fulfilled:
	\begin{itemize}[leftmargin=0.5cm]
		\renewcommand{\labelitemi}{\textendash}
		\item (\textbf{A4}) Given $X,Y\in\text{obj}(\mathsf{A})$, any morphism $f\in\text{Hom}_\mathsf{A}(X,Y)$ admits a \textbf{canonical decomposition}\index{canonical decomposition} in $\mathsf{A}$ of the form
		\begin{equation}\label{candecompo}
			\begin{tikzcd}
				\ker(f)\arrow[r, hook, "k"] & X\arrow[r, "q"] & I\arrow[r, "j"] & Y\arrow[r, two heads, "c"] & \text{coker}(f)\;,
			\end{tikzcd}
		\end{equation}
		with $j\circ q=f$ and $\text{coker}(k)=I=\ker(c)\in\text{obj}(\mathsf{A})$ (where, by Definition \ref{monepicoker}, we shorten $k=\ker(f)$ and $c=\text{coker}(f)$). In particular, any other decomposition is isomorphic to \eqref{candecompo} through a uniquely determined isomorphism. 
	\end{itemize}
\end{Def}

\begin{Rem}
	\ \vspace*{0.0cm}
	\begin{itemize}[leftmargin=0.5cm]
		\item We can define the \textbf{image}\index{image} of $f$ to be $\text{im}(f)\coloneqq\ker(\text{coker}(f))$ and the \textbf{coimage}\index{coimage} of $f$ to be $\text{coim}(f)\coloneqq\text{coker}(\ker(f))$. Then the full expansion of the canonical decomposition \eqref{candecompo} actually reads
		\[
		\begin{tikzcd}
			\ker(f)\arrow[r, hook, "\ker(f)"] & X\arrow[rrr, "f"]\arrow[dr, two heads, "\text{coim}(f)"'] & & & Y\;\arrow[r, two heads, "\text{coker}(f)"] & \;\text{coker}(f) \\
			& & \text{coim}(f)\arrow[r, "\hat{f}"', "\sim"] & \text{im}(f)\arrow[ur, hook, "\text{im}(f)"'] & & 
		\end{tikzcd}\quad,
		\]
		
		\item Axiom (A4) therefore demands the existence of kernel and cokernel for each morphism in $\mathsf{A}$, and requires $\hat{f}$ to be an isomorphism, so that we may write $I\coloneqq\text{coim}(f)\cong\text{im}(f)$ (any other decomposition by a triple $(I',q',j')$ is isomorphic to \eqref{candecompo} in the sense that $I'\cong I$ canonically\footnote{A proof can be found in \cite[Proposition VIII.3.1]{[Mac78]}.}). In fact, any monomorphism can now be written as kernel of a suitable morphism \big(with reference to \eqref{candecompo}, $f\cong\ker(c)=\ker(\text{coker}(f))$\big) and any epimorphism as a cokernel \big($f\cong\text{coker}(k)=\text{coker}(\ker(f))$\big).
		
		\item Suppose $\ker(f)=0=\text{coker}(f)$. Then $\text{id}_Y:Y\rightarrow Y$ is a valid kernel for $c:Y\rightarrow 0$ and $\text{id}_X:X\rightarrow X$ a valid cokernel for $k:0\rightarrow X$, thus both unique up to isomorphism, which in our notation translates to $\ker(c)\cong Y$ and $\text{coker}(k)\cong X$. Consequently, the axiom implies $X\cong Y$ as objects of $\mathsf{A}$. Rephrased (using Lemma \ref{aboutmonepi} below), in abelian categories any morphism which is both mono and epi is an isomorphism.
		
		\item If $\mathsf{A}$ is additive/abelian, then its dual category $\mathsf{A}^\text{opp}$ obtained by setting $\text{obj}(\mathsf{A}^\text{opp})=\text{obj}(\mathsf{A})$ and $\text{Hom}_{\mathsf{A}^\text{opp}}(X,Y)=\text{Hom}_\mathsf{A}(Y,X)$ (reversing all assignments) is itself additive/abelian. In particular, canonical decompositions are preserved by dualization and mono-/epimorphisms turn into respectively epi-/monomorphisms.
	\end{itemize}
\end{Rem}

\begin{Ex}
	The prototypical example of abelian category is of course $\mathsf{Ab}$, the \textbf{category of abelian groups}\index{category!of abelian groups} and group homomorphisms between them. The known operations clearly fulfill axioms (A1)--(A3), with any trivial group being a valid zero object, mono-/epimorphisms being given by injective/surjective group homomorphisms and co-/kernels coinciding with their algebraical definitions. Existence of a canonical decomposition for $f:X\rightarrow Y$ is guaranteed by the isomorphism theorem for groups: just take $I\coloneqq X/\ker(f)$, $k:\ker(f)\rightarrow X$ the obvious inclusion, $c:Y\rightarrow \text{coker}(f)$ and $q:X\rightarrow X/\ker(f)$ the obvious projections, and finally $j:X/\ker(f)\xrightarrow{\sim}\text{im}(f)$ the asserted group isomorphism; then clearly $j\circ q=f$ and $\text{coker}(k)=X/\text{im}(k)=X/\ker(f)\cong \text{im}(f)=\ker(c)$.
	
	Other examples of abelian categories are the category $\mathsf{Vect}_\mathbb{K}$ of finite-dimen-\break sional vector spaces over a field $\mathbb{K}$ and, more generally, the \textbf{categories ${}_R\mathsf{Mod}$, $\mathsf{Mod}_R$ of left/right $R$-modules}\index{category!of left/right modules} over a ring $R$ (remember that $\mathsf{Mod}_\mathbb{Z}=\mathsf{Ab}$, and in case $R$ is commutative it holds ${}_R\mathsf{Mod}=\mathsf{Mod}_R$). \hfill $\blacklozenge$
\end{Ex}

Working in abelian categories is very advantageous: well-known notions and results from the category $\mathsf{Ab}$ can be implemented and exploited one-to-one. Therefore, we can speak of (co)chain complexes, resolutions, exactness of sequences, (co)homology groups, ... and make use of the Five-Lemma, Snake Lemma and other diagram chase methods. For example:

\begin{Def}\label{abeliansubcat}
	Let $\mathsf{A}$ be an abelian category, $\mathsf{B}\subset\mathsf{A}$ a subcategory. Then $\mathsf{B}$ is an \textbf{abelian subcategory}\index{abelian subcategory} if it is itself abelian and any short exact sequence $0\rightarrow X\xrightarrow{f} Y\xrightarrow{g} Z\rightarrow 0$ in $\mathsf{B}$ (meaning $\ker(f)=0=\text{coker}(g)$ and $\text{im}(f)=\ker(g)$ as subobjects of $\mathsf{B}$ each) is also short exact in $\mathsf{A}$.  
\end{Def} 

A few words to gain further familiarity with mono and epimorphisms in the case of \textit{concrete} abelian categories (whose objects may be regarded as ``sets with extra structure'', thus legitimizing the notation $x\in X$, etc.). 

\begin{Lem}\label{aboutmonepi}
	In a concrete abelian category, a morphism $f:X\rightarrow Y$ is:
	\begin{itemize}[leftmargin=0.5cm]
		\item a monomorphism if and only if it holds $(f(x)=0\in Y$ $\implies$ $x=0\in X)$, or equivalently, if and only if $\ker(f)=0$ \textup(the zero morphism\textup);
		
		\item an epimorphism if and only if it holds \textup($\forall y\in Y$ $\exists x\in X$ s.t. $f(x)=y$\textup), or equivalently, if and only if $\textup{coker}(f)=0$ \textup(the zero morphism\textup). 
	\end{itemize}
	Furthermore, in any abelian category, when given a monomorphism/\!epimorphism $f:X\rightarrow Y$, we can complete it to a short exact sequence of the form 
	\[
	\begin{tikzcd}[column sep=1.1em]
		0\arrow[r] & X\arrow[r, hook, "f"] & Y\arrow[r, two heads, "c"] & \textup{coker}(f)\arrow[r] & 0 & \textup{resp.} & 0\arrow[r] & \ker(f)\arrow[r, hook, "k"] & X\arrow[r, two heads, "f"] & Y\arrow[r] & 0\;.	
	\end{tikzcd}\vspace*{-0.1cm}
	\]
	Consequently, any short exact sequence $0\rightarrow X\xrightarrow{f} Y\xrightarrow{g} Z\rightarrow 0$ is canonically isomorphic to $0\rightarrow \ker(g)\hookrightarrow Y\twoheadrightarrow \textup{coker}(f)\rightarrow 0$ .
\end{Lem}

Here are a couple properties which will provide a shortcut in the construction of derived functors:

\begin{Def}\label{injprojobj}
	Let $\mathsf{A}$ be an abelian category.
	\begin{itemize}[leftmargin=0.5cm]
		\item An $X\in\text{obj}(\mathsf{A})$ is called \textbf{injective object}\index{object!injective} if for every monomorphism $f\in\text{Hom}_\mathsf{A}(Y,Z)$ and any $g\in\text{Hom}_\mathsf{A}(Y,X)$ there exists some $h\in\text{Hom}_\mathsf{A}(Z,X)$ such that $h\circ f= g$ (left diagram below). 
		\newline An $X\in\text{obj}(\mathsf{A})$ is called \textbf{projective object}\index{object!projective} if for every epimorphism $f\in\text{Hom}_\mathsf{A}(Y,Z)$ and any $g\in\text{Hom}_\mathsf{A}(X,Z)$ there exists some $h\in\text{Hom}_\mathsf{A}(X,Y)$ such that $f\circ h= g$ (right diagram).
		\begin{equation}\label{injprojobjdiag}
			\begin{tikzcd}
				Y\arrow[rr, hook, "f"]\arrow[dr, "g"'] & & Z\arrow[dl, dashed, "h"] \\
				& X &
			\end{tikzcd}
			\qquad\qquad
			\begin{tikzcd}
				Y\arrow[rr, two heads, "f"] & & Z \\
				& X\arrow[ul, dashed, "h"]\arrow[ur, "g"'] &
			\end{tikzcd}
		\end{equation}
		
		\item Then $\mathsf{A}$ \textbf{has enough injectives}\index{has enough!injectives} if every $X\in\text{obj}(\mathsf{A})$ admits a \textbf{(right) injective resolution}\index{resolution!(right) injective} by injective objects $\{I^i\}_{i\geq 0}\subset\text{obj}(\mathsf{A})$, that is, there exists a possibly infinite complex
		\[
		0\longrightarrow X\xrightarrow{\varepsilon_X} I^0\longrightarrow I^1\longrightarrow I^2\longrightarrow ...
		\] 
		($X\xrightarrow{\varepsilon_X}I^\bullet$ for short) everywhere exact. 
		\newline Instead, $\mathsf{A}$ \textbf{has enough projectives}\index{has enough!projectives} if every $X\in\text{obj}(\mathsf{A})$ admits a \textbf{(left) projective resolution}\index{resolution!(left) projective} by projective objects $\{P^i\}_{i\leq 0}\subset\text{obj}(\mathsf{A})$, that is, there exists a possibly infinite complex
		\[
		... \longrightarrow P^{-2}\longrightarrow P^{-1}\longrightarrow P^0\xrightarrow{\varepsilon_X} X\longrightarrow 0
		\] 
		($P^\bullet\xrightarrow{\varepsilon_X}X$ for short) everywhere exact.
	\end{itemize}
\end{Def}

For completeness' sake, we also explain how cohomology is defined for generic abelian categories (see \cite[subsections II.6.2--II.6.3]{[GM03]}).

\begin{Def}\label{abeliancohomology}
	Let $\mathsf{A}$ be an abelian category, $X^\bullet\equiv(X^\bullet,d^\bullet)$ a (cochain) complex $(...\xrightarrow{d^{n-1}}X^n\xrightarrow{d^n}X^{n+1}\xrightarrow{d^{n+1}}...)$ in $\mathsf{A}$ (hence $d^n\circ d^{n-1}=0$ for all $n\in\mathbb{Z}$; more in Section \ref{ch2.1}). Then the \textbf{$n$-th cohomology object}\index{object!nth@$n$-th cohomology} of $X^\bullet$ is
	\[
	H^n(X^\bullet)\coloneqq \text{coker}(a^{n-1})\cong\ker(b^n)\in\text{obj}(\mathsf{A})
	\]
	(an object of $\mathsf{A}$ because the latter contains co-/kernels!), where $a^{n-1}, b^n$ are determined by the commutative diagram
	\begin{equation}\label{cohomodiag}
		\begin{tikzcd}
			& \text{coker}(d^{n-1})\arrow[dr, dashed, "b^n"]\arrow[r, two heads] & \text{im}(d^n)\arrow[d, hook] \\
			X^{n-1}\arrow[r, "d^{n-1}"]\arrow[dr, dashed, "a^{n-1}"]\arrow[d, two heads] & X^n\arrow[u, two heads]\arrow[r, "d^n"] & X^{n+1} \\
			\text{im}(d^{n-1})\arrow[r, hook]& \ker(d^n)\arrow[u, hook] & \\
		\end{tikzcd}\quad.\vspace*{-0.5cm}
	\end{equation}
	The complex $X^\bullet$ is called \textit{exact} if it is acyclic at each term, that is, $H^n(X^\bullet)=0$ for all $n\in\mathbb{Z}$.
\end{Def}

It is straightforward to see that in $\mathsf{Ab}$ or ${}_R\mathsf{Mod}$ the definition of cohomology lines up with the usual one: $H^n(X^\bullet)=\text{coker}(a^{n-1})=\ker(d^n)/\text{im}(a^{n-1})=$ $\ker(d^n)/\text{im}(d^{n-1})=\ker\big(X^n\!\rightarrow\!\text{im}(d^n)\big)/\text{im}(d^{n-1})\cong\ker\big(\text{coker}(d^{n-1})\!\twoheadrightarrow\!\text{im}(d^n)\big)$ $=\ker(b^n)$.

\subsection{Functors between abelian categories}\label{ch1.3}

Let us briefly explore the behaviour of functors between abelian categories.

\begin{Def}\label{additiveexactfunc}
	Let $\mathsf{C},\mathsf{C}'$ be additive categories. A functor $\mathsf{F}:\mathsf{C}\rightarrow\mathsf{C}'$ is \textbf{additive}\index{functor!additive} if  the maps $\mathsf{F}_{X,Y}:\text{Hom}_\mathsf{C}(X,Y)\rightarrow\text{Hom}_{\mathsf{C}'}(\mathsf{F}(X),\mathsf{F}(Y))$ are homomorphisms of abelian groups for any $X,Y\in\text{obj}(\mathsf{C})$. If moreover $\mathsf{C}$ and $\mathsf{C}'$ are $\mathbb{K}$-linear, $F$ is \textit{linear} when these maps are actually $\mathbb{K}$-linear maps of vector spaces.
	
	Let $\mathsf{A},\mathsf{A}'$ be abelian categories. An additive functor $\mathsf{F}:\mathsf{A}\rightarrow\mathsf{A}'$ is \textbf{exact}\index{functor!exact (between abelian categories)} if it maps short exact sequences $0\rightarrow X\rightarrow Y\rightarrow Z\rightarrow 0$ in $\mathsf{A}$ to short exact sequences $0\rightarrow \mathsf{F}(X)\rightarrow \mathsf{F}(Y)\rightarrow \mathsf{F}(Z)\rightarrow 0$ in $\mathsf{A}'$. It is just \textbf{left/right exact}\index{functor!left/right exact} if it fails to be exact at $\mathsf{F}(Z)$ respectively $\mathsf{F}(X)$.
\end{Def}

As a standard exercise (carried out in \cite[Propositions II.6.5, II.6.6]{[GM03]}), we have:

\begin{Ex}\label{exactfuncexample}
	Given an abelian category $\mathsf{A}$ and a fixed $X\in\text{obj}(\mathsf{A})$, the (additive) Hom-functors $\text{Hom}_\mathsf{A}(X,\square):\mathsf{A}\rightarrow\mathsf{Ab}$ and $\text{Hom}_\mathsf{A}(\square,X):\mathsf{A}^\text{opp}\rightarrow\mathsf{Ab}$ are both left exact.
	
	Given a ring $R$ and a fixed right/left module $X$ over $R$, the (additive) functors $X\otimes_R \square:{}_R\mathsf{Mod}\rightarrow\mathsf{Ab}$ and $\square\otimes_R X:\mathsf{Mod}_R\rightarrow\mathsf{Ab}$ are both right exact. (If they are also left exact, we call $X$ a flat module over $R$.)	\hfill $\blacklozenge$
\end{Ex}

The Hom-functors are especially useful for identifying injective and projective objects of an abelian category (cf. Definition \ref{injprojobj}).

\begin{Lem}\label{Homforinjprojobj}
	Let $\mathsf{A}$ be an abelian category, $X\in\textup{obj}(\mathsf{A})$ fixed. Then the left exact functors $\textup{Hom}_\mathsf{A}(X,\square)$ and $\textup{Hom}_\mathsf{A}(\square,X)$ are exact if and only if $X$ is a projective respectively an injective object. 
\end{Lem}

\begin{proof}
	Let us start with projectivity. Assume $X$ is projective and consider the short exact sequence $0\rightarrow K\rightarrow Y\rightarrow Z\rightarrow 0$, where the epimorphism $f:Y\twoheadrightarrow Z$ is chosen as in diagram \eqref{injprojobjdiag} right. We must check that $\text{Hom}_\mathsf{A}(X,Y)\rightarrow\text{Hom}_\mathsf{A}(X,Z)\rightarrow 0$ is exact in $\mathsf{Ab}$, that is, that $f_*$ is surjective. Indeed, projectivity of $X$ guarantees that, given any $g\in\text{Hom}_\mathsf{A}(X,Z)$, there exists some $h\in\text{Hom}_\mathsf{A}(X,Y)$ such that $g=f\circ h=f_*(h)$. Conversely, assuming right exactness of $\text{Hom}_\mathsf{A}(X,\square)$, we exploit that any epimorphism $f:Y\twoheadrightarrow Z$ can be completed to a short exact sequence like above by adding its kernel $K\equiv\ker(f)$ on the left. Then, given an arrow $g\in\text{Hom}_\mathsf{A}(X,Z)$, surjectivity of $f_*$ yields an $h\in\text{Hom}_\mathsf{A}(X,Y)$ making \eqref{injprojobjdiag} right commute.
	
	About injectivity. Assume $X$ is injective and consider the short exact sequence $0\rightarrow Y\hookrightarrow Z\rightarrow C\rightarrow 0$, where the monomorphism $f:Y\hookrightarrow Z$ is chosen as in \eqref{injprojobjdiag} left. To prove that $\text{Hom}_\mathsf{A}(Z,X)\rightarrow\text{Hom}_\mathsf{A}(Y,X)\rightarrow 0$ is exact in $\mathsf{Ab}$, that is, that $f^*$ is surjective, we employ injectivity of $X$: for each $g\in\text{Hom}_\mathsf{A}(Y,X)$ there exists some $h\in\text{Hom}_\mathsf{A}(Z,X)$ such that $g=h\circ f=f^*(h)$. The reverse implication works similarly, after having completed the monomorphism $f:Y\hookrightarrow Z$ to a short exact sequence with $C\equiv\text{coker}(f)$ on the right.    
\end{proof}

We mention an important result which justifies the compatibility of abelian categories with ${}_R\mathsf{Mod}$. A proof can be found in \cite[section 1.6]{[Wei94]}.

\begin{Thm}[Freyd--Mitchell Embedding Theorem]
	Let $\mathsf{A}$ be an abelian, small\footnote{A category $\mathsf{A}$ is small if both obj($\mathsf{A}$) and the collection of all its Hom-sets are sets.} category. Then there exists a \textup(unital\textup) ring $R$ and a fully faithful\footnote{A functor $\mathsf{F}:\mathsf{C}\rightarrow\mathsf{D}$ is \textit{full} and \textit{faithful} if $\mathsf{F}_{X,Y}:\textup{Hom}_\mathsf{C}(X,Y)\rightarrow\textup{Hom}_\mathsf{D}(\mathsf{F}(X),\mathsf{F}(Y))$ is surjective respectively injective for all $X, Y \in$ obj$(\mathsf{C})$ (\textit{fully faithful} if both).} exact functor $\mathsf{F}:\mathsf{A}\rightarrow{}_R\mathsf{Mod}$.  
\end{Thm}

Therefore, $\mathsf{A}$ is equivalent to a full subcategory of ${}_R\mathsf{Mod}$ (that is, it embeds into ${}_R\mathsf{Mod}$) and all its structures can be identified with their analogues over modules. In particular, we can think of its objects as being left $R$-modules and its morphisms as being $R$-linear maps. The correspondence however fails for injective and projective objects.

In the proof mentioned above, the following familiar-looking theorem is used:

\begin{Thm}[Yoneda Embedding Theorem]\label{YonedaEmbThm}
	Let $\mathsf{A}$ be an additive category, $\mathsf{Fun}(\mathsf{A},\mathsf{Ab})$, $\mathsf{Fun}(\mathsf{A}^\textup{opp},\mathsf{Ab})$ the \textup(abelian!\textup) categories of all covariant respectively contravariant functors from $\mathsf{A}$ to $\mathsf{Ab}$ and natural transformations between them. Then there exist:
	\begin{itemize}	
		\item A left exact functor $h_*:\mathsf{A}^\textup{opp}\rightarrow\mathsf{Fun}(\mathsf{A},\mathsf{Ab})$ sending $X\mapsto\textup{Hom}_\mathsf{A}(X,\square)$ and $f\mapsto\textup{Hom}_\mathsf{A}(f,\square)=f_*$; if moreover $\mathsf{A}^\textup{opp}$ is abelian, then $h_*$ is fully faithful and thus identifies it with the full abelian subcategory $\{h_X:\mathsf{A}\rightarrow\mathsf{Ab}\mid X\in\textup{obj}(\mathsf{A})\}$.
		
		\item A left exact functor $h^*:\mathsf{A}\rightarrow\mathsf{Fun}(\mathsf{A}^\textup{opp},\mathsf{Ab})$ sending $X\mapsto\textup{Hom}_\mathsf{A}(\square, X)$ and $f\mapsto\textup{Hom}_\mathsf{A}(\square, f)=f^*$; if moreover $\mathsf{A}$ is abelian, then $h^*$ is fully faithful and thus identifies it with the full abelian subcategory $\{h^X:\mathsf{A}^\textup{opp}\rightarrow\mathsf{Ab}\mid X\in\textup{obj}(\mathsf{A})\}$.	
	\end{itemize} 
\end{Thm}

The Yoneda Embedding Theorem also plays a role in the following statement, whose inclusion in the present work might seem a little arbitrary at first, but which will actually prove useful when dealing with exactness of some specific functors between categories of sheaves (see Lemmas \ref{sectfuncexact} and \ref{inversepushforwardexact}).

\begin{Pro}\label{adjointfunc}
	Let $\mathsf{C},\mathsf{C}'$ be additive categories, $\mathsf{F}:\mathsf{C}\rightarrow\mathsf{C}'$, $\mathsf{G}:\mathsf{C}'\rightarrow\mathsf{C}$ a pair of \textup(left respectively right\textup) adjoint functors. That is, suppose there exists a natural isomorphism $\Psi:\textup{Hom}_{\mathsf{C}'}(\mathsf{F}(\square),\blacksquare)\rightarrow\textup{Hom}_\mathsf{C}(\square,\mathsf{G}(\blacksquare))$ of bifunctors $\mathsf{C}^\textup{opp}\times\mathsf{C}'\rightarrow\mathsf{Sets}$, thus yielding bijections $\Psi_{X,Y}:\textup{Hom}_{\mathsf{C}'}(\mathsf{F}(X),Y)\xrightarrow{\sim}\textup{Hom}_\mathsf{C}(X,\mathsf{G}(Y))$ for each $X\in\textup{obj}(\mathsf{C})$, $Y\in\textup{obj}(\mathsf{C}')$. Then $\mathsf{F}$ is right exact and $\mathsf{G}$ is left exact.
\end{Pro}

\begin{proof}
	Let us start from $\mathsf{G}$, considering a short exact sequence $0\rightarrow Y\xrightarrow{f} Y'\xrightarrow{g} Y''\rightarrow 0$ in $\mathsf{C}'$. Applying the left exact $h_{\mathsf{F}(X)}=\text{Hom}_{\mathsf{C}'}(\mathsf{F}(X),\blacksquare):\mathsf{C}'\rightarrow\mathsf{Ab}$ for any $X\in\text{obj}(\mathsf{C})$ (see Example \ref{exactfuncexample}), we obtain the exact top row in
	\[
	\begin{tikzcd}
		0\arrow[r] & \text{Hom}_{\mathsf{C}'}(\mathsf{F}(X),Y)\arrow[r,"f_*"]\arrow[d, "\Psi_{X,Y}"'] & \text{Hom}_{\mathsf{C}'}(\mathsf{F}(X),Y')\arrow[r,"g_*"]\arrow[d, "\Psi_{X,Y'}"] & \text{Hom}_{\mathsf{C}'}(\mathsf{F}(X),Y'')\arrow[d, "\Psi_{X,Y''}"] \\
		0\arrow[r] & \text{Hom}_\mathsf{C}(X,\mathsf{G}(Y))\arrow[r, "\alpha_*"'] & \text{Hom}_\mathsf{C}(X,\mathsf{G}(Y'))\arrow[r, "\beta_*"'] & \text{Hom}_\mathsf{C}(X,\mathsf{G}(Y''))
	\end{tikzcd}\quad.
	\]
	Consequently, as each vertical map is bijective by assumption, the bottom row is exact as well, again for any $X\in\text{obj}(\mathsf{C})$. Since the Yoneda embedding $h^*$ is fully faithful, meaning for example that 
	\[
	\text{Hom}_{\mathsf{Fun}(\mathsf{A}^\textup{opp},\mathsf{Ab})}\big(\text{Hom}_\mathsf{C}(\square,\mathsf{G}(Y)),\text{Hom}_\mathsf{C}(\square,\mathsf{G}(Y'))\big)\cong\text{Hom}_\mathsf{C}(\mathsf{G}(Y),\mathsf{G}(Y'))\,,
	\]
	the bottom row identifies a unique sequence $0\rightarrow\mathsf{G}(Y)\xrightarrow{\alpha}\mathsf{G}(Y')\xrightarrow{\beta}\mathsf{G}(Y'')$ in $\mathsf{C}$ --- then necessarily $\alpha=\mathsf{G}(f)$ and $\beta=\mathsf{G}(g)$ --- which actually inherits exactness: take $X=\mathsf{G}(Y)$, then $\beta\circ\alpha=(\beta_*\circ\alpha_*)(\text{id}_{\mathsf{G}(Y)})=0$; conversely, take $X=\ker(\beta)$, then the inclusion $\iota:X\hookrightarrow\mathsf{G}(Y')$ satisfies $\beta_*(\iota)=0$, so $\iota\in\ker(\beta_*)=\text{im}(\alpha_*)$ and $\ker(\beta)=\text{im}(\iota)\subseteq\text{im}(\alpha)$, collectively implying that $\ker(\beta)=\text{im}(\alpha)$; and similarly $\ker(\alpha)=0$. This shows that $\mathsf{G}$ is left exact.
	
	For $\mathsf{F}$, we just mirror the procedure: any short exact sequence $0\rightarrow X\xrightarrow{f} X'\xrightarrow{g} X''\rightarrow 0$ in $\mathsf{C}$ induces via the left exact $h^{\mathsf{G}(Y)}=\text{Hom}_\mathsf{C}(\square,\mathsf{G}(Y)):\mathsf{C}^\text{opp}\rightarrow\mathsf{Ab}$ and the bijections $\Psi_{\square,Y}^{-1}$ a family of exact sequences $0\rightarrow\text{Hom}_{\mathsf{C}'}(\mathsf{F}(X''),Y)\rightarrow\text{Hom}_{\mathsf{C}'}(\mathsf{F}(X'),Y)\rightarrow\text{Hom}_{\mathsf{C}'}(\mathsf{F}(X),Y)$ for each $Y\in\text{obj}(\mathsf{C}')$, which in turn uniquely identify the desired sequence $\mathsf{F}(X)\xrightarrow{\mathsf{F}(f)}\mathsf{F}(X')\xrightarrow{\mathsf{F}(g)}\mathsf{F}(X'')\rightarrow 0$, thanks to the fully faithful $h_*$ acting as
	\[
	\text{Hom}_{\mathsf{Fun}(\mathsf{A},\mathsf{Ab})}\big(\text{Hom}_{\mathsf{C}'}(\mathsf{F}(X''),\blacksquare),\text{Hom}_{\mathsf{C}'}(\mathsf{F}(X'),\blacksquare)\big)\cong\text{Hom}_{\mathsf{C}'}(\mathsf{F}(X'),\mathsf{F}(X''))\,.
	\]
	The proof of exactness proceeds as above.
\end{proof}

\newpage

\section{Derived categories}
\thispagestyle{plain}

\subsection{The category of complexes}\label{ch2.1}

In this and next chapter we take a look in greater detail at the process of derivation of categories very briefly outlined in \cite[section 3.1]{[Imp21]}. Our main reference remains the very rich \cite[sections III.1--III.4]{[GM03]}. Let us start with some basic terminology.

\begin{Def}
	Let $\mathsf{A}$ be an abelian category. The \textbf{category $\mathsf{Kom(A)}$ of (cochain) complexes}\index{category!of coc@of (cochain) complexes} over $\mathsf{A}$ is defined as usual:
	\begin{itemize}[leftmargin=0.5cm]
		\renewcommand{\labelitemi}{\textendash}
		\item objects\footnote{Observe that we could regard $X^\bullet$ as a $\mathbb{Z}$-graded object $X=\bigoplus_{n\in\mathbb{Z}} X^n$ endowed with a graded morphism $d_X$ of degree 1 squaring to zero.} are cochain complexes $X^\bullet\equiv(X^\bullet, d^\bullet_X)$ of the form
		\[
		...\rightarrow X^{-1}\xrightarrow{d^{-1}_X} X^0\xrightarrow{d^0_X} X^1\xrightarrow{d^1_X}X^2\rightarrow...\;,
		\]
		for $\{X^n\}_{n\in\mathbb{Z}}\subset\textup{obj}(\mathsf{A})$ and morphisms $\{d^n_X\}_{n\in\mathbb{Z}}$ (differentials) in $\mathsf{A}$ such that $d^n_X\circ d^{n-1}_X=0$ for all $n\in\mathbb{Z}$;
		
		\item morphisms are chain maps $f^\bullet:X^\bullet\rightarrow Y^\bullet$, that is, collections of morphisms $f^n:X^n\rightarrow Y^n$ in $\mathsf{A}$ which commute with the differentials, $f^{n+1}\circ d^n_X = d^n_Y\circ f^n:X^n\rightarrow Y^{n+1}$ for all $n\in\mathbb{Z}$. Composition of any compatible pair of morphisms happens componentwise. 
	\end{itemize}
	If we consider as objects only those complexes bounded from below/above, meaning $X^n=0$ for all $n<n_-\in\mathbb{Z}$ respectively $n>n_+\in\mathbb{Z}$, then we rather write $\mathsf{Kom^+(A)}$ respectively $\mathsf{Kom^-(A)}$, and let $\mathsf{Kom^b(A)}\coloneqq\mathsf{Kom^+(A)}\cap\mathsf{Kom^-(A)}$ denote the \textbf{category of bounded complexes}\index{category!of bounded complexes}. All these are full subcategories of $\mathsf{Kom(A)}$.
\end{Def}

For example, projective resolutions belong to $\mathsf{Kom^-(A)}$, while injective ones to $\mathsf{Kom^+(A)}$. As a good excuse to revise the theory of the previous chapter, we prove that:

\begin{Lem}\label{komabelian}
	If $\mathsf{A}$ is an abelian category, so is $\mathsf{Kom(A)}$.
\end{Lem}

\begin{proof}
	Preadditivity of $\mathsf{Kom(A)}$ (axiom (A1) of Definition \ref{addcat}) is easily verified, since we can add and compose compatible morphisms of complexes componentwise, where abelianity of $\mathsf{A}$ ensures the desired relations. The zero object is just the trivial complex of infinitely many zero objects $0\in\text{obj}(\mathsf{A})$ chained together by zero differentials. Moreover, direct products/sums of complexes are again well-defined complexes: for $(X_1^\bullet, d_{X_1}^\bullet), (X_2^\bullet, d_{X_2}^\bullet)\in\text{obj}(\mathsf{Kom(A)})$ define $(X_1^\bullet, d_{X_1}^\bullet)\oplus(X_2^\bullet, d_{X_2}^\bullet)$ to be the cochain complex $((X_1\oplus X_2)^\bullet, d_{X_1\oplus X_2}^\bullet)$ given by $(X_1\oplus X_2)^n\coloneqq X_1^n\oplus X_2^n\in\text{obj}(\mathsf{A})$ and $d_{X_1\oplus X_2}^n=d_{X_1}^n\oplus d_{X_2}^n\in\text{Hom}_\mathsf{A}((X_1\oplus X_2)^n,\break(X_1\oplus X_2)^{n+1})$, so that $d_{X_1\oplus X_2}^n\circ d_{X_1\oplus X_2}^{n-1}=(d_{X_1}^{n}\circ d_{X_1}^{n-1})\oplus(d_{X_2}^{n}\circ d_{X_2}^{n-1})=0$. Defining chain maps $p_k^\bullet:(X_1\oplus X_2)^\bullet\rightarrow X_k^\bullet$ and $i_k^\bullet:X_k^\bullet\rightarrow (X_1\oplus X_2)^\bullet$ for $k=1,2$ to be memberwise the projection respectively inclusion each $(X_1\oplus X_2)^n$ is equipped with (so that they do commute with the differentials), indeed one can conclude the existence and uniqueness of the decomposition \eqref{directprodsum} on the chain level. Therefore, also axioms (A2), (A3) hold and $\mathsf{Kom(A)}$ is an additive category.
	
	What do kernels and cokernels look like in $\mathsf{Kom(A)}$? Consider a chain map $f^\bullet:X^\bullet\rightarrow Y^\bullet$, then each component $f^n\in\text{Hom}_\mathsf{A}(X^n,Y^n)$ admits a canonical decomposition \eqref{candecompo}, so that we may expand the diagram defining $f^\bullet$ to
	\[
	\begin{tikzcd}
		...\arrow[r, dashed] & \text{ker}(f^{n-1})\arrow[r, dashed]\arrow[d, hook, "k^{n-1}"'] & \text{ker}(f^{n})\arrow[r, dashed]\arrow[d, hook, "k^{n}"] & \text{ker}(f^{n+1})\arrow[r, dashed]\arrow[d, hook, "k^{n+1}"] & ... \\
		...\arrow[r] & X^{n-1}\arrow[r, "d_X^{n-1}"]\arrow[d, "q^{n-1}"'] & X^{n}\arrow[r, "d_X^{n}"]\arrow[d, "q^{n}"] & X^{n+1}\arrow[r]\arrow[d, "q^{n+1}"] & ... \\
		...\arrow[r] & I^{n-1}\arrow[r, "d_I^{n-1}"]\arrow[d, "j^{n-1}"'] & I^{n}\arrow[r, "d_I^{n}"]\arrow[d, "j^{n}"] & I^{n+1}\arrow[r]\arrow[d, "j^{n+1}"] & ... \\
		...\arrow[r] & Y^{n-1}\arrow[r, "d_Y^{n-1}"]\arrow[d, two heads, "c^{n-1}"'] & Y^{n}\arrow[r, "d_Y^{n}"]\arrow[d, two heads, "c^{n}"] & Y^{n+1}\arrow[r]\arrow[d, two heads, "c^{n+1}"] & ... \\
		...\arrow[r, dashed] & \text{coker}(f^{n-1})\arrow[r, dashed] & \text{coker}(f^{n})\arrow[r, dashed] & \text{coker}(f^{n+1})\arrow[r, dashed] & ...
	\end{tikzcd}\qquad,
	\]
	where $j^n\circ q^n=f^n$ and $\text{coker}(\text{ker}(f^n))=\text{coker}(k^n)=I^n=\ker(c^n)=$\break $\ker(\text{coker}(f^n))\in\text{obj}(\mathsf{A})$ for all $n\in\mathbb{Z}$. The latter defines a complex $(I^\bullet, d_I^\bullet)\in\text{obj}(\mathsf{Kom(A)})$ canonically, so that the central arrows in the diagram are necessarily the differentials of $I^\bullet$. Now we can connect the top and botton rows with dashed arrows, chosen so to make the overall diagram commutative. These define two further complexes $(\text{ker}(f^\bullet), d_{\text{ker}(f)}^\bullet)$, $(\text{coker}(f^\bullet), d_{\text{coker}(f)}^\bullet)\in\text{obj}(\mathsf{Kom(A)})$ (for example, $d_{\text{ker}(f)}^n\circ d_{\text{ker}(f)}^{n-1}=0$ because it is seen to factor through $d_X^n\circ d_X^{n-1}=0$). In fact, they can be shown to fulfill the universal property of kernel and cokernel (see Definition \ref{monepicoker} and diagram \eqref{kercokerdiag}). We understand that $f^\bullet:X^\bullet\rightarrow Y^\bullet$ canonically decomposes in $\mathsf{Kom(A)}$ as
	\begin{equation*}
		\begin{tikzcd}
			\ker(f^\bullet)\arrow[r, hook, "k^\bullet"] & X^\bullet\arrow[r, "q^\bullet"] & I^\bullet\arrow[r, "j^\bullet"] & Y^\bullet\arrow[r, two heads, "c^\bullet"] & \text{coker}(f^\bullet)\;,
		\end{tikzcd}
	\end{equation*}
	with $j^\bullet\circ q^\bullet = f^\bullet$ and $\text{coker}(k^\bullet)=I^\bullet=\ker(c^\bullet)$. Therefore, axiom (A4) holds and $\mathsf{Kom(A)}$ is consequently abelian.
\end{proof}

A quick glance to Definition \ref{abeliancohomology} reveals that cohomology yields functors from the category of complexes to the underlying abelian category:

\begin{Def}\label{cohomfunc}
	Let $\mathsf{A}$ be an abelian category. Then for any $n\in\mathbb{Z}$ we can define the \textbf{$n$-th cohomology functor}\index{functor!nthc@$n$-th cohomology} $H^n:\mathsf{Kom(A)}\rightarrow\mathsf{A}$ through the assignments \begin{small}\begin{equation}
			\begin{aligned}
				&X^\bullet=(X^\bullet, d^\bullet)\in\text{obj}(\mathsf{Kom(A)})\mapsto H^n(X^\bullet)\in\text{obj}(\mathsf{A})\quad \text{(from \eqref{cohomodiag})} \\
				&f^\bullet\in\text{Hom}_\mathsf{Kom(A)}(X^\bullet,Y^\bullet) \mapsto H^n(f^\bullet)\!=\!\langle f^n(\cdot)\rangle\in \text{Hom}_\mathsf{A}(H^n(X^\bullet),H^n(Y^\bullet))
			\end{aligned}
	\end{equation}\end{small}
	\!\!(latter is the usual morphism induced on cohomology).
\end{Def}

\begin{Def}\label{quasiso}
	Let $\mathsf{A}$ be an abelian category. A morphism $f^\bullet:X^\bullet\rightarrow Y^\bullet$ in $\mathsf{Kom(A)}$ is \textbf{homotopic to $0$}\index{homotopic!to 0} if there exists a chain homotopy $k^\bullet:X^\bullet\rightarrow Y^{\bullet-1}$, that is, morphisms $k^n:X^n\rightarrow Y^{n-1}$ in $\mathsf{A}$ such that
	\[
	f^n=k^{n+1}\circ d^n_X + d^{n-1}_Y\circ k^n: X^n\rightarrow Y^n\,. 
	\]
	Then two morphisms $f^\bullet, g^\bullet\in\text{Hom}_\mathsf{Kom(A)}(X^\bullet,Y^\bullet)$ are \textbf{homotopic}\index{homotopic!pair of morphisms} (written $f^\bullet\sim g^\bullet$) if their difference is homotopic to 0.
	
	We say that a morphism of complexes $f^\bullet: X^\bullet\rightarrow Y^\bullet$ is a \textbf{quasi-isomorphism}\index{quasi-isomorphism} if the induced morphisms $H^n(f^\bullet)$ are isomorphisms in $\mathsf{A}$ for all $n\in\mathbb{Z}$, and thus call $X^\bullet$ and $Y^\bullet$ quasi-isomorphic complexes.
\end{Def}	

It is a standard proof to show that homotopic morphisms of complexes induce the same morphisms in cohomology (this works just as in $\mathsf{Ab}$). We can also invert the functor $H^0$:

\begin{Rem}\label{0-complex}
	There always exists a fully faithful functor $\mathsf{J}:\mathsf{A}\rightarrow\mathsf{Kom(A)}$ which embeds the abelian category $\mathsf{A}$ into its category of complexes: it sends $X\in\text{obj}(\mathsf{A})$ to the \textbf{0-complex}\index{0-complex}
	\[
	X^\bullet\equiv X[0]^\bullet=\big(...\rightarrow 0\rightarrow 0\xrightarrow{d^{-1}} X\xrightarrow{d^0} 0\rightarrow 0\rightarrow...\big)\,
	\]
	acyclic away from 0 and such that $H^0(X^\bullet)\cong X$, and sends $f\in\text{Hom}_\mathsf{A}(X,Y)$ to the morphism of complexes $f^\bullet:X^\bullet\rightarrow Y^\bullet$ given by $f^0=f$ and $f^n=0$ for all $0\neq n\in\mathbb{Z}$. Clearly, $H^0\circ\mathsf{J}$ is isomorphic to the identity functor on $\mathsf{A}$.
	
	It is worth noting that, given a projective resolution $P^\bullet\xrightarrow{\varepsilon_X} X$ as in Definition \ref{injprojobj}, we have a quasi-isomorphism $f^\bullet:P^\bullet\rightarrow X^\bullet$ to the 0-complex of $X$ with $f^0=\varepsilon_X$, $f^i:P^i\rightarrow 0$ for $i<0$ and $f^i=0$ for $i>0$, so that we can effectively replace $X$ with $P^\bullet$ when useful (the precise choice of projective resolution of $X$ is irrelevant since any such two will be homotopic, as shown in \cite[Theorem III.1.3]{[GM03]}). The dual version applies for right resolutions by injective objects.  
\end{Rem}

\subsection{The universal property of derived categories}\label{ch2.2}

We can produce a first tentative definition of derived category of an abelian category as the solution to a universal problem.

\begin{Pro}\label{derivedexists}
	Let $\mathsf{A}$ be an abelian category. Then there exist a category $\mathsf{D(A)}$ with $\textup{obj}(\mathsf{D(A)})=\textup{obj}(\mathsf{Kom(A)})$ and a functor $\mathsf{L}:\mathsf{Kom(A)}\rightarrow\mathsf{D(A)}$ such that:
	\renewcommand{\theenumi}{\roman{enumi}}
	\begin{enumerate}[leftmargin=0.6cm] 
		\item $\mathsf{L}(f^\bullet)$ is an isomorphism for any quasi-isomorphism $f^\bullet$ in $\mathsf{Kom(A)}$;
		
		\item for any functor $\mathsf{F}:\mathsf{Kom(A)}\rightarrow\mathsf{D}$ fulfilling i. there exists a unique functor $\mathsf{G}:\mathsf{D(A)}\rightarrow\mathsf{D}$ such that $\mathsf{F}=\mathsf{G}\circ\mathsf{L}$ as functors.
	\end{enumerate}
	\!We call $\mathsf{D(A)}$ the \textbf{derived category}\index{derived category} of $\mathsf{A}$ \textup(by ii. unique up to equivalence\textup) and $\mathsf{L}$ the \textbf{localization functor}\index{localization functor}.
	Working with $\mathsf{Kom^\#\!(A)}$ for $\#=+,-,\mathsf{b}$ instead produces derived categories $\mathsf{D^\#\!(A)}$, with $\mathsf{D^b(A)}$ taking the name of \textbf{bounded derived category}\index{derived category!bounded} of $\mathsf{A}$.
\end{Pro}

\begin{proof}
	The proof involves the process of localization of a category, which works for an arbitrary category $\mathsf{C}$. We aim to define its derived category as a certain $\mathsf{C}[S^{-1}]$ for $S$ an arbitrary class of morphisms in $\mathsf{C}$ to be turned by $\mathsf{L}:\mathsf{C}\rightarrow\mathsf{C}[S^{-1}]$ into isomorphisms.
	
	Set $\text{obj}(\mathsf{C}[S^{-1}])\coloneqq \text{obj}(\mathsf{C})$ and $\mathsf{L}(X)\coloneqq X$ for all $X\in\text{obj}(\mathsf{C}[S^{-1}])$. The morphisms are constructed in the following ``geometrical'' manner. Introduce formal morphisms $s^{-1}$ for each $s\in S$ and construct an oriented graph $\Gamma$ whose vertices are the objects of $\mathsf{C}$, whose edges are its morphisms, oriented according to the direction of their arrows, and for each $s$ add an edge with reversed orientation, corresponding to $s^{-1}$. Then a morphism $X\rightarrow Y$ in $\mathsf{C}[S^{-1}]$ is the equivalence class including all oriented paths from $X$ to $Y$, where we may translate composition of compatible morphisms to the concatenation of their arrows and simplify paths associated to $s\circ s^{-1}$ and $s^{-1}\circ s$ to identity, ``static'' paths. Then we request $L$ to map any $f:X\rightarrow Y$ in $\mathsf{C}$ to its equivalence class $[f]$, so that in particular $\mathsf{L}(s)$ is an isomorphism of inverse $\mathsf{L}(s^{-1})$ (as their representatives compose to the identities).
	
	Now, given another functor $\mathsf{F}:\mathsf{C}\rightarrow\mathsf{D}$ behaving like $\mathsf{L}$, we define $\mathsf{G}:\mathsf{C}[S^{-1}]\rightarrow\mathsf{D}$ as the functor $\mathsf{G}(X)\coloneqq \mathsf{F}(X)$ for $X\in\text{obj}(\mathsf{C}[S^{-1}])$, $\mathsf{G}([f])\coloneqq\mathsf{F}(f)$ on morphisms of $\mathsf{C}$ and $\mathsf{G}([s^{-1}])\coloneqq\mathsf{F}(s)^{-1}$ for each $s\in S$. Then clearly $\mathsf{F}=\mathsf{G}\circ\mathsf{L}$, and $\mathsf{G}$ is uniquely defined (the choice of representative for each class of morphisms is irrelevant).
	
	In our specific case, just take $\mathsf{C}=\mathsf{Kom(A)}$, $S=\{\text{quasi-isomorphisms in }\break\mathsf{Kom(A)}\}$ (so that, to all effects, each $s^{-1}$ stands for the quasi-inverse) and let $\mathsf{D(A)}\coloneqq\mathsf{Kom(A)}[S^{-1}]$. 
\end{proof}

Observe that we can embed the abelian category $\mathsf{A}$ into its derived category through the fully faithful functor $\mathsf{L}\circ\mathsf{J}:\mathsf{A}\rightarrow\mathsf{D(A)}$, where $\mathsf{J}:\mathsf{A}\rightarrow\mathsf{Kom(A)}$ is as in Remark \ref{0-complex}. This is not obvious at first glance and will be proved later (see Proposition \ref{AinD(A)}), when we have a better understanding of the derived category.

Moreover, since cohomology functors invert quasi-isomorphisms (by definition), they ignore the localization procedure and can thus be regarded also as functors $H^n:\mathsf{D(A)}\rightarrow\mathsf{A}$. Once again, $H^0\circ\mathsf{L}\circ\mathsf{J}$ is isomorphic to $\mathsf{Id_A}$ on $\mathsf{A}$.

\begin{Rem}\label{semisimple}
	As an application of the universal property of derived categories, let us consider the (complete) subcategory $\mathsf{Kom_0(A)}\subset\mathsf{Kom(A)}$ consisting of all cyclic complexes, that is, those $X^\bullet\in\text{obj}(\mathsf{Kom(A)})$ with zero differential, $d^\bullet=0$ --- this makes $\mathsf{Kom_0(A)}$ structurally isomorphic to the product category of infinitely many copies of $\mathsf{A}$. The functor $\mathsf{H}^\bullet:\mathsf{Kom(A)}\rightarrow\mathsf{Kom_0(A)}$ specified by $(X^\bullet,d^\bullet)\mapsto(H^n(X^\bullet),0)_{n\in\mathbb{Z}}$ and $f^\bullet\mapsto (H^n(f^\bullet))_{n\in\mathbb{Z}}$ turns quasi-isomorphisms into (infinite tuples of) isomorphisms, hence by the universal property for the derived category of $\mathsf{A}$ there exists a unique functor $\mathsf{G}:\mathsf{D(A)}\rightarrow\mathsf{Kom_0(A)}$ such that $\mathsf{H}^\bullet=\mathsf{G}\circ\mathsf{L}$.
	
	The interesting part is that if $\mathsf{A}$ is \textbf{semisimple}\index{semisimple category}, meaning that any short exact sequence is isomorphic to one of the form $0\rightarrow X\xrightarrow{(\text{id}_X,0)} X\oplus Y\rightarrow Y\rightarrow 0$ (one says it is \textbf{split}\index{short exact sequence!split}), then $\mathsf{G}$ is actually an equivalence of categories (and vice versa), so that we can infer further properties of $\mathsf{D(A)}$ just by focusing on cyclic complexes! (A proof of this result is presented in \cite[Proposition III.2.4]{[GM03]}.)
\end{Rem}

However, for generic abelian categories the precise structure of their derived category remains so far quite obscure: we need a more flexible approach allowing us to manipulate the class $S$ of selected morphisms. This is achieved by generalizing the concept of multiplicative subset of a ring.

\begin{Def}\label{locclass}
	A class of morphisms $S$ of an arbitrary category $\mathsf{C}$ is called \textbf{localizing}\index{localizing class of morphisms} if:
	\renewcommand{\theenumi}{\alph{enumi}}
	\begin{enumerate}[leftmargin=0.6cm] 
		\item for each $X\in\text{obj}(\mathsf{C})$ holds $\text{id}_X\in S$, and $s, t\in S$ implies $s\circ t\in S$ (whenever defined);
		
		\item for any $f\in\text{Hom}_\mathsf{C}(X,Y)$ and $s:Z\rightarrow Y$ in $S$ there exist $g\in\text{Hom}_\mathsf{C}(W,Z)$ and $t:W\rightarrow X$ in $S$ such that the left square below commutes; dually, for any $f\in\text{Hom}_\mathsf{C}(Y,X)$ and $s:Y\rightarrow Z$ in $S$ there exist $g\in\text{Hom}_\mathsf{C}(Z,W)$ and $t:X\rightarrow W$ in $S$ such that the right square commutes:
		\begin{equation}\label{localsquare}
			\begin{tikzcd}
				W\arrow[r, dashed, "g"]\arrow[d, dashed, "t"'] & Z\arrow[d, "s"] \\
				X\arrow[r, "f"'] & Y
			\end{tikzcd}
			\qquad\qquad
			\begin{tikzcd}
				W & Z\arrow[l, dashed, "g"'] \\
				X\arrow[u, dashed, "t"] & Y\arrow[l, "f"]\arrow[u, "s"']
			\end{tikzcd}\quad;
		\end{equation}
		
		\item for any pair $f_1,f_2\in\text{Hom}_\mathsf{C}(X,Y)$, the existence of an $s:Y\rightarrow Z$ in $S$ such that $s\circ f_1= s\circ f_2$ is equivalent to that of a $t:W\rightarrow X$ in $S$ such that $f_1\circ t= f_2\circ t$.
		\newline \big(Subtracting on both sides, for a given $f\in\text{Hom}_\mathsf{C}(X,Y)$ we may reformulate: $\exists s:Y\rightarrow Z$ in $S$ s.t. $s\circ f=0$ $\Longleftrightarrow$ $\exists t:W\rightarrow X$ in $S$ s.t. $f\circ t= 0$.\big)
	\end{enumerate}
\end{Def}

The advantage of working with a localizing class of morphisms is that in $\mathsf{C}[S^{-1}]$ (where we have inverse morphisms $s^{-1}$ for each $s\in S$; cf. proof of Proposition \ref{derivedexists}) the two squares of \eqref{localsquare} imply also the relations 
\[
s^{-1}\circ f= s^{-1}\circ f\circ t\circ t^{-1}= s^{-1}\circ s\circ g\circ t^{-1} = g\circ t^{-1}:X\rightarrow Z
\]
respectively $f\circ s^{-1}= t^{-1}\circ g:Z\rightarrow X$, both telling us that we can modify the general expression $f_1\circ s_1^{-1}\circ f_2\circ s_2^{-1}\circ...\circ f_k\circ s_k^{-1}$ of a path in $\mathsf{C}[S^{-1}]$ by moving all morphisms of $S$ to the left or right and treat them separately. This trick is crucial for the following description.

\begin{Pro}\label{derivedstructure}
	Let $\mathsf{C}$ be an arbitrary category, with a localizing class of morphisms $S$. Then the derived category of $\mathsf{C}$ has the following explicit structure\textup: $\textup{obj}(\mathsf{C}[S^{-1}])=\textup{obj}(\mathsf{C})$ \textup(cf. Proposition \ref{derivedexists}\textup) and morphisms in $\textup{Hom}_{\mathsf{C}[S^{-1}]}(X,Y)$ are equivalence classes $\varphi=[(s,f)]$ of \textbf{roofs}\index{roof} $(s,f)$ in $\mathsf{C}$, which are diagrams of the form\footnote{The squiggly arrow is a morphism in $\mathsf{C}[S^{-1}]$, so it doesn't quite fit into the roof diagram, whence its distinctive look; we will omit it more often than not though.}
	\begin{equation}\label{roof}
		\begin{tikzcd}
			& X'\arrow[dl, "s"']\arrow[dr, "f"] & \\
			X\arrow[rr, squiggly, "\varphi"'] & & Y
		\end{tikzcd}\vspace*{0.2cm}
	\end{equation}
	for $f\in\textup{Hom}_\mathsf{C}(X',Y)$ and $s\in S$, where the equivalence relation asserts $(s,f)\sim (t,g)$ if and only if there exists a ``topping'' roof $(r, h)$ making the diagram
	\begin{equation}\label{equivroofs}
		\begin{tikzcd}
			& & X'''\arrow[dl, "r"']\arrow[dr, "h"] & & \\
			& X'\arrow[dl, "s"'] & & X''\arrow[dlll, "t"]\arrow[dr, "g"] & \\
			X & & & & Y \arrow[from=ulll, crossing over, "f"']
		\end{tikzcd}\vspace*{0.1cm}
	\end{equation}
	commute. The composition of morphisms $[(t,g)]\circ[(s,f)]$ is possible when the neighbouring arrows share the same target object, and gives $[(s\circ t', g\circ f')]$, the class of
	\begin{equation}\label{roofcompo}
		\begin{tikzcd}
			& & X''\arrow[dl, dashed, "t'"']\arrow[dr, dashed, "f'"] & & \\
			& X'\arrow[dl, "s"']\arrow[dr, "f"'] & & Y'\arrow[dl, "t"]\arrow[dr, "g"] & \\
			X & & Y & & Z 
		\end{tikzcd}\quad,
	\end{equation}
	where the top square is provided by \eqref{localsquare} left. \textup(Observe that, by item a. of Definition \ref{locclass}, indeed $s\circ t'\in S$.\textup) 
	\newline The identity roof therefore reads $X\xleftarrow{\textup{id}_X} X\xrightarrow{\textup{id}_X} X$ \textup(with $\textup{id}_X\in S$\textup). 
\end{Pro}

\begin{proof}
	The verifications of well-definiteness of the equivalence relation $\sim$ and of the composition law are just a matter of carefully drawing the necessary diagrams. We rather focus on the main claim that this category of roofs, provisionally denoted $\mathsf{R}$, indeed coincides with $\mathsf{C}[S^{-1}]$. Thereto, we need only check that it shares the same properties of the derived category described in Proposition \ref{derivedexists}.
	
	First we need a surrogate $\mathsf{L_R}:\mathsf{C}\rightarrow\mathsf{R}$ of the localization functor $\mathsf{L}:\mathsf{C}\rightarrow\mathsf{C}[S^{-1}]$ turning morphism of $S$ into isomorphism in $\mathsf{R}$. Set $\mathsf{L_R}(X)\coloneqq X$ for any $X\in\text{obj}(\mathsf{C})$ and assign to any $f\in\text{Hom}_\mathsf{C}(X,Y)$ the class $\mathsf{L_R}(f)\coloneqq[(\text{id}_X,f)]$ of the roof $X\xleftarrow{\text{id}_X} X\xrightarrow{f} Y$. Then any $s:X\rightarrow Y$ in $S$ has invertible image in $\mathsf{R}$, with inverse $\mathsf{L_R}(s)^{-1}$ being represented by the roof $Y\xleftarrow{s} X\xrightarrow{\text{id}_X} X$ (indeed $[\mathsf{L_R}(s)^{-1}]\circ[\mathsf{L_R}(s)]=[(\text{id}_X,\text{id}_X)]$ by choosing $\text{id}_X$ as dashed arrows in \eqref{roofcompo}, respectively $[\mathsf{L_R}(s)]\circ[\mathsf{L_R}(s)^{-1}]=[(\text{id}_Y,\text{id}_Y)]$ by choosing instead $s^{-1}$).
	
	Now, we must verify the factorization property. Let $\mathsf{F}:\mathsf{C}\rightarrow\mathsf{D}$ be a functor turning any $s\in S$ into an isomorphism in $\mathsf{D}$. Define the functor $\mathsf{G}:\mathsf{R}\rightarrow\mathsf{D}$ expected to satisfy $\mathsf{F}=\mathsf{G}\circ\mathsf{L_R}$ as follows. Firstly, impose $\mathsf{G}(X)\coloneqq \mathsf{F}(X)$ for any $X\in\text{obj}(\mathsf{R})$. Secondly, send any roof class $[(s,f)]\in\text{Hom}_\mathsf{R}(X,Y)$ as in \eqref{roof} to $\mathsf{G}\big([(s,f)]\big)\coloneqq\mathsf{F}(f)\circ\mathsf{F}(s)^{-1}\in\text{Hom}_\mathsf{D}(\mathsf{G}(X),\mathsf{G}(Y))$ (well-definiteness regarding the choice of representative is left as an exercise). The functorial features of $\mathsf{G}$ are evident: $\mathsf{G}\big([(\text{id}_X,\text{id}_X)]\big)=\mathsf{F}(\text{id}_X)\circ\mathsf{F}(\text{id}_X)^{-1}=\text{id}_{\mathsf{F}(X)}=\text{id}_{\mathsf{G}(X)}$ and
	\begin{align*}
		\mathsf{G}\big([(t,g)]\big)\circ\mathsf{G}\big([(s,f)]\big) &=\mathsf{F}(g)\circ\mathsf{F}(t)^{-1}\circ\mathsf{F}(f)\circ\mathsf{F}(s)^{-1} \\
		&=\mathsf{F}(g)\circ\mathsf{F}(f')\circ\mathsf{F}(t')^{-1}\circ\mathsf{F}(s)^{-1} \\
		&=\mathsf{F}(g\circ f')\circ\mathsf{F}(s\circ t')^{-1}=\mathsf{G}\big([(s\circ t', g\circ f')]\big)	
	\end{align*}
	(where in the second equality we detoured through the dashed arrows of diagram \eqref{roofcompo}, while in the third we used the functoriality of $\mathsf{F}$). Then we clearly have $(\mathsf{G}\circ\mathsf{L_R})(X)=\mathsf{G}(X)=\mathsf{F}(X)$ on objects and $(\mathsf{G}\circ\mathsf{L_R})(f)=\mathsf{G}\big([(\text{id}_X,f)]\big)=\mathsf{F}(f)\circ\mathsf{F}(\text{id}_X)^{-1}=\mathsf{F}(f)$ on morphisms. 
	
	It remains to show that such a functor $\mathsf{G}$ is unique. Thereto, consider another candidate $\tilde{\mathsf{G}}$ fulfilling $\mathsf{F}=\tilde{\mathsf{G}}\circ\mathsf{L_R}$. The latter applied to any $X\in\text{obj}(\mathsf{R})$ gives $\tilde{\mathsf{G}}(X)=\mathsf{F}(X)$, whereas $\tilde{\mathsf{G}}$ applied to the equality $[(s,f)]\circ\mathsf{L_R}(s)=[(s,f)]\circ[(\text{id}_{X'},s)]=[(\text{id}_{X'}\circ\text{id}_{X'},f\circ\text{id}_{X'})]=\mathsf{L_R}(f)$ (always referring to \eqref{roof}) yields $\tilde{\mathsf{G}}([(s,f)])\circ\mathsf{F}(s)=\mathsf{F}(f)$, whence $\tilde{\mathsf{G}}([(s,f)])=\mathsf{F}(f)\circ\mathsf{F}(s)^{-1}$, by assumption on $\mathsf{F}$. Then $\tilde{\mathsf{G}}=\mathsf{G}$, as claimed.
\end{proof}

In literature it is possible to encounter a variation of the diagram \eqref{roof} representing $[(s,f)]$, namely the \textit{right} $S$-roof
\[
\begin{tikzcd}
	& X' & \\
	X\arrow[ur, "f"] & & Y\arrow[ul, "s"']
\end{tikzcd}
\]
for $s\in S$ (the usual \eqref{roof} are called \textit{left} $S$-roofs). Composition of any pair of such roofs then appeals to \eqref{localsquare} right. 

As a bonus, we observe that restriction of the localizing class to a full subcategory behaves nicely under localization (shown in \cite[Proposition III.2.10]{[GM03]}):

\begin{Lem}\label{derivedsubcat}
	Let $\mathsf{C}$ be a category with localizing class of morphisms $S$, $\mathsf{B}\subset\mathsf{C}$ a full subcategory such that the restriction $S_\mathsf{B}\subset S$ of $S$ to it is a localizing class for $\mathsf{B}$. Assume also that:
	\renewcommand{\theenumi}{\roman{enumi}}
	\begin{enumerate}[leftmargin=0.6cm]
		\item for any $s:X\rightarrow Y$ in $S$ and $Y\in\textup{obj}(\mathsf{B})$ there exist some $W\in\textup{obj}(\mathsf{B})$ and $f:W\rightarrow X$ such that $s\circ f\in S$, or
		
		\item for any $s:X\rightarrow Y$ in $S$ and $X\in\textup{obj}(\mathsf{B})$ there exist some $Z\in\textup{obj}(\mathsf{B})$ and $f:Y\rightarrow Z$ such that $f\circ s\in S$.
	\end{enumerate} 
	Then $\mathsf{B}[S_\mathsf{B}^{-1}]\subset\mathsf{C}[S^{-1}]$ is a full subcategory, meaning that the inclusion functor $\mathsf{I}:\mathsf{B}[S_\mathsf{B}^{-1}]\rightarrow\mathsf{C}[S^{-1}]$ is fully faithful.
\end{Lem}

\begin{proof}
	Suppose condition (i) holds. We must check that $F_{X,Y}:\text{Hom}_{\mathsf{B}[S_\mathsf{B}^{-1}]}(X,Y)$ $\rightarrow \text{Hom}_{\mathsf{C}[S^{-1}]}(\mathsf{I}(X),\mathsf{I}(Y))$ is bijective for any $X,Y\in\text{obj}(\mathsf{B}[S_\mathsf{B}^{-1}])=\text{obj}(\mathsf{B})$. 
	
	Regarding injectivity, let two $S$-roofs $X\xleftarrow{s} X'\xrightarrow{f} Y$ and $X\xleftarrow{t} X''\xrightarrow{g} Y$ be equivalent in $\mathsf{C}$, by \eqref{equivroofs} meaning there exists a topping roof $X'\xleftarrow{r} X'''\xrightarrow{h} X''$, with $s, t, r\in S$ and morphisms $f, g, h$ in $\mathsf{C}$. By assumption (i) for $s\circ r: X'''\rightarrow X$ in $S$, there exist some $W\in\text{obj}(\mathsf{B})$ and $u:W\rightarrow X'''$ such that $s\circ r\circ u\in S$. But $\mathsf{B}\subset\mathsf{C}$ being a full subcategory means that $s\circ r\circ u$ is moreover a morphism in $\mathsf{B}$, hence in $S_\mathsf{B}$ by definition. Also, $g\circ h\circ u:W\rightarrow Y$ belongs to $\mathsf{B}$ as well, so that the topping roof $X\xleftarrow{s\circ r\circ u} W\xrightarrow{g\circ h\circ u} Y$ implies that $(s,f)$ and $(t,g)$ are equivalent as $S_\mathsf{B}$-roofs in $\mathsf{B}$.
	
	About surjectivity, note that for any $S$-roof $X\xleftarrow{s} X'\xrightarrow{f} Y$ there exist again by (i) and fullness of $\mathsf{B}$ some $W\in\text{obj}(\mathsf{B})$ and $u:W\rightarrow X'$ such that $s\circ u\in S_\mathsf{B}$, and the so obtained $S_\mathsf{B}$-roof $X\xleftarrow{s\circ u} W\xrightarrow{f\circ u} Y$ is clearly equivalent to the original one, onto which it maps through $F_{X,Y}$.
	
	This completes the proof. Just note that under the assumption (ii), we would instead need to use right $S$-roofs.
\end{proof}

Applying last lemma to $\mathsf{B}'\!\coloneqq\!\mathsf{Kom^b(A)}\subset\mathsf{B}\!\coloneqq\!\mathsf{Kom^\#\!(A)}\subset\mathsf{C}\!\coloneqq\!\mathsf{Kom(A)}$ for\break $\#=+,-$, with localizing classes consisting of quasi-isomorphisms of the respective bounded complexes, we obtain chains of full subcategories $\mathsf{D^b(A)}\subset\mathsf{D^\#\!(A)}\subset\mathsf{D(A)}$, so that we can in fact regard $\mathsf{D^b(A)}$ as either the result of a localization process or as the full subcategory of $\mathsf{D(A)}$ whose complexes $X^\bullet$ fulfill $H^n(X^\bullet)\cong 0$ when $n<n_-$ and $n>n_+$, for some $n_-<n_+\in\mathbb{Z}$.

Another useful result in the same style (proven for example in \cite[Theorem 13.2.8]{[KS06]}) is:

\begin{Lem}\label{abeliansubcatderivedequiv}
	Let $\mathsf{B}$ be an abelian category. Let $\mathsf{A}\subset\mathsf{B}$ be a full, abelian subcategory which is \textup{thick} --- that is, for any $0\rightarrow X_1\rightarrow X_2\rightarrow X_3\rightarrow 0$ short exact in $\mathsf{B}$ with $X_1,X_3\in\textup{obj}(\mathsf{A})$ also the ``extension'' $X_2\in\textup{obj}(\mathsf{A})$. Suppose also that any $X\in\textup{obj}(\mathsf{A})$ can be embedded into some $X'\in\textup{obj}(\mathsf{A})$ which is injective as an object of $\mathsf{B}$.
	
	Then, for $\#=\emptyset, +, -, \mathsf{b}$, the natural functor $\mathsf{D^\#\!(A)}\rightarrow \mathsf{D^\#\!(B)}$ induces an equivalence $\mathsf{D^\#\!(A)}\rightarrow \mathsf{D_\mathsf{A}^\#\!(B)}$, where $\mathsf{D_\mathsf{A}^\#\!(B)}\subset \mathsf{D^\#\!(B)}$ is the full additive subcategory whose complexes have cohomology in $\mathsf{A}$ \textup{\big(}that is, $H^n(X^\bullet)\in\textup{obj}(\mathsf{A})$ for all $X^\bullet\in\textup{obj}(\mathsf{D_\mathsf{A}^\#\!(B)})$\textup{\big)}.
\end{Lem}

\subsection{Exact triangles of complexes}\label{ch2.3}

Despite its precise description of the structure of derived categories, we cannot apply Proposition \ref{derivedstructure} to $\mathsf{Kom(A)}$ yet, because $S\coloneqq\{\text{quasi-isomorphisms in}\break\text{Kom}(A)\}$ is \textit{not} a localizing class of morphisms! However, it is when belonging to a quotient of the category of complexes, namely the homotopy category $\mathsf{H(A)}$, as will be discussed in the next section. 

Here we deal with an even more pressing problem: the derived category $\mathsf{D(A)}$ of an abelian category $\mathsf{A}$, as constructed by Proposition \ref{derivedexists}, is not abelian in general! (A justification will be given in Section \ref{ch2.4}, where the failure is shown to occur already at $\mathsf{H(A)}$, closely related to $\mathsf{D(A)}$.) This is indeed worrisome because many notions we introduced in Section \ref{ch1.2} are no longer valid, in particular that of short exact sequences depending on kernels and cokernels. We need due replacements. They come in a familiar aspect (see \cite[chapter 3]{[Imp21]}). 

\begin{Def}\label{translfunc}
	Let $\mathsf{A}$ be an abelian category. The \textbf{$n$-th translation functor}\index{functor!ntht@$n$-th translation} $T^n:\mathsf{Kom(A)}\rightarrow\mathsf{Kom(A)}$ maps $(X^\bullet, d_X^\bullet)\in\text{obj}(\mathsf{Kom(A)})$ to the translated complex $(X[n]^\bullet,d_{X[n]}^\bullet)\in\text{obj}(\mathsf{Kom(A)})$ given by $X[n]^k\coloneqq X^{n+k}$ and $d_{X[n]}^\bullet\coloneqq(-1)^n d_X^{n+\bullet}$, and maps any $f^\bullet\in\text{Hom}_\mathsf{Kom(A)}(X^\bullet,Y^\bullet)$ to the translated chain map $f[n]^\bullet\in\text{Hom}_\mathsf{Kom(A)}(X[n]^\bullet,Y[n]^\bullet)$ given by $f[n]^k\coloneqq f^{n+k}:X^{n+k}\rightarrow Y^{n+k}$.
	
	Clearly, $T^n$ is an equivalence with inverse $T^{-n}$, and does restrict to an autoequivalence of $\mathsf{Kom^\#\!(A)}$ as well as of the associated $\mathsf{D^\#\!(A)}$, for $\#=+,-,\mathsf{b}$.
\end{Def}

Later in Chapter 4 we will see (or actually remember) that the first translation functor $T^1$ is a key ingredient for triangulated categories. Another recurring concept is that of \textit{cones} (for a more thorough discussion, including a topological viewpoint, see \cite[section 1.5]{[Wei94]}):

\begin{Def}\label{coneofcomplexes}
	Let $\mathsf{A}$ be an abelian category, $f^\bullet\in\text{Hom}_\mathsf{Kom(A)}(X^\bullet, Y^\bullet)$.
	\begin{itemize}[leftmargin=0.5cm]
		\item The \textbf{cone (of complexes)}\index{cone!of complexes} of $f^\bullet$ is the cochain complex $(C(f)^\bullet, d_{C(f)}^\bullet)\in\text{obj}(\mathsf{Kom(A)})$ given by $C(f)^n\coloneqq X[1]^n\oplus Y^n$ and
		\begin{align}
			&d_{C(f)}^\bullet \coloneqq \begin{pmatrix} d_{X[1]}^\bullet & 0 \\ f[1]^\bullet & d_Y^\bullet \end{pmatrix}\!:\,X^{\bullet+1}\oplus Y^\bullet\rightarrow X^{\bullet+2}\oplus Y^{\bullet+1}, \nonumber\\[0.2cm] 
			&d_{C(f)}^n(x^{n+1},y^n)=\big(-d_X^{n+1}(x^{n+1}),f^{n+1}(x^{n+1})+d_Y^n(y^n)\big)
		\end{align}
		(indeed fulfilling $d_{C(f)}^n\circ d_{C(f)}^{n-1}=0$, since  $f^{n+1}\circ d_{X[1]}^{n-1} + d_Y^n\circ f^n= -f^{n+1}\circ d_X^n$ $+ f^{n+1}\circ d_X^n=0$ by the chain property of $f^\bullet$). 
		
		\item The \textbf{cylinder (of complexes)}\index{cylinder of complexes} of $f^\bullet$ is the cochain complex $(Z(f)^\bullet, d_{Z(f)}^\bullet)\in\text{obj}(\mathsf{Kom(A)})$ given by $Z(f)^n\coloneqq X^n\oplus C(f)^n= X^n\oplus X[1]^n\oplus Y^n$ and 
		\begingroup
		\allowdisplaybreaks
		\begin{small}\begin{align}
				&\mkern-16mu d_{Z(f)}^\bullet \!\coloneqq\! \begin{pmatrix} d_X^\bullet & -\text{id}_{X[1]}^\bullet & 0 \\ 0 & d_{X[1]}^\bullet & 0 \\ 0 & f[1]^\bullet & d_Y^\bullet \end{pmatrix}\!:\, X^\bullet\oplus X^{\bullet+1}\oplus Y^\bullet\rightarrow X^{\bullet+1}\oplus X^{\bullet+2}\oplus Y^{\bullet+1}, \nonumber\\[0.2cm]
				&\mkern-16mu d_{Z(f)}^n(x^n\!,x^{n+1}\!,y^n)\!=\!\big(d_X^n(x^n)\!-\! x^{n+1}\!,-d_X^{n+1}(x^{n+1}), f^{n+1}(x^{n+1})\!+\! d_Y^n(y^n)\big)
		\end{align}\end{small}
		\endgroup%\vspace*{-0.1cm}
		(a well-defined differential as well). The cylinder projects onto the cone of $f^\bullet$ via the chain map $p_{23}^\bullet:Z(f)^\bullet\rightarrow C(f)^\bullet$ cutting away the first entry.
	\end{itemize}
	Of course, $X^\bullet, Y^\bullet\in\text{obj}(\mathsf{Kom^\#\!(A)})$ implies $C(f)^\bullet, Z(f)^\bullet\in\text{obj}(\mathsf{Kom^\#\!(A)})$ for $\#=+,-,\mathsf{b}$, and similarly in their derived categories. And if the constituents belong to an additive subcategory of complexes, so do the cone and cylinder. 
\end{Def}

We observe that if $X^\bullet$ and $Y^\bullet$ are 0-complexes connected by some $f\in\text{Hom}_\mathsf{A}(X,Y)$, yielding a trivial chain map $f^\bullet=\mathsf{J}(f)\in\text{Hom}_\mathsf{Kom(A)}(X^\bullet,Y^\bullet)$ as per Remark \ref{0-complex}, then $C(f)^\bullet$ has $X$ in degree $-1$ and $Y$ in degree 0, hence $H^{-1}(C(f)^\bullet)\cong\ker(d_{C(f)}^{-1})\cong X\cap\ker(f^0)\cong \ker(f^\bullet)$ respectively $H^0(C(f)^\bullet)\cong$ $\ker(d_Y^0)/\text{im}(f^0)\cong\text{coker}(f^\bullet)$ (use diagram \eqref{cohomodiag} for a more rigorous proof). Consequently, if $f$ is injective, then $C(f)^\bullet$ is quasi-isomorphic to $\mathsf{J}(\text{coker}(f))$, while $f$ surjective implies that $C(f)^\bullet$ is quasi-isomorphic to $T^{-1}\big(\mathsf{J}(\ker(f))\big)$.

Therefore, cones encapsulate to a certain extent the concepts of kernels and cokernels which are lost when passing to the derived category. Now, a key lemma linking cones and cylinders.

\begin{Lem}\label{p3i3quiso}
	Let $\mathsf{A}$ be an abelian category, $f^\bullet\in\textup{Hom}_\mathsf{Kom(A)}(X^\bullet, Y^\bullet)$. Then the diagram
	\begin{equation}\label{distinguidiag}
		\begin{tikzcd}
			& 0\arrow[r] & Y^\bullet\arrow[r, "i_2^\bullet"]\arrow[d, "i_3^\bullet"'] & C(f)^\bullet\arrow[r, "p_1^\bullet"]\arrow[d, "{\textup{id}_{C(f)}^\bullet}"] & X[1]^\bullet\arrow[r] & 0 \\
			0\arrow[r] & X^\bullet\arrow[r, "{i_1^\bullet}"]\arrow[d, "{\textup{id}_X^\bullet}"'] & Z(f)^\bullet\arrow[r, "p_{23}^\bullet"']\arrow[d, "p_3^\bullet"] & C(f)^\bullet\arrow[r] & 0 & \\
			& X^\bullet\arrow[r, "f^\bullet"'] & Y^\bullet & & & \\
		\end{tikzcd}\vspace*{-0.3cm}
	\end{equation}
	commutes and has exact rows in $\mathsf{Kom(A)}$, where $i_k^\bullet, p_k^\bullet$ denote the obvious inclusion respectively projection chain maps,\footnote{Warning: here and henceforth, we denote by the same symbols $i_k^\bullet$, $p_k^\bullet$ the obvious chain maps, even when these are associated to cones or cylinders of different morphisms!} except for 
	\[
	p_3^n:Z(f)^n\rightarrow Y^n,\,(x^n,x^{n+1},y^n)\mapsto f^n(x^n)+y^n\,.
	\]
	In particular, $i_3^\bullet, p_3^\bullet$ are quasi-isomorphisms fulfilling $p_3^\bullet\circ i_3^\bullet=\textup{id}_Y^\bullet$ and $i_3^\bullet\circ p_3^\bullet\sim\textup{id}_{Z(f)}^\bullet$, meaning that $Z(f)^\bullet\cong Y^\bullet\in\textup{obj}(\mathsf{D(A)})$ by Proposition \ref{derivedexists}. 
\end{Lem} 

\begin{proof}
	That the morphisms involved are indeed chain maps is an obvious check, as well as commutativity of the diagram and exactness of its rows. Moreover, $(p_3^n\circ i_3^n)(y^n)=f^n(0)+y^n=\text{id}_Y^n(y^n)$ for all $y^n\in Y^n$ and any $n\in\mathbb{Z}$, so it remains to prove that $i_3^\bullet\circ p_3^\bullet$ and $\text{id}_{Z(f)}^\bullet$ are homotopic chain maps. Let $k^\bullet:Z(f)^\bullet\rightarrow Z(f)^{\bullet-1}$ be given by $k^n(x^n,x^{n+1},y^n)\coloneqq (0,x^n,0)$, then
	\begin{align*}
		(&d_{Z(f)}^{n-1}\circ k^n + k^{n+1}\circ d_{Z(f)}^n)(x^n,x^{n+1},y^n) \\		
		&= d_{Z(f)}^{n-1}(0,x^n,0) + k^{n+1}(d_X^n(x^n)-x^{n+1},-d_X^{n+1}(x^{n+1}),f^{n+1}(x^{n+1})+d_Y^n(y^n)) \\
		&= (-x^n,-d_X^n(x^n),f^n(x^n)) + (0,d_X^n(x^n)-x^{n+1},0)\\
		&= (0,0,f^n(x^n)+y^n) - (x^n,x^{n+1},y^n)=(i_3^n\circ p_3^n-\text{id}_{Z(f)}^n)(x^n,x^{n+1},y^n)\;, 
	\end{align*}
	as desired. According to Definition \ref{quasiso}, $i_3^\bullet$ and $p_3^\bullet$ are then quasi-isomorphisms. 
\end{proof}

We now describe \textit{triangles of complexes}, the replacements of short exact sequences in derived categories. Their choice is a sensible one due to Proposition \ref{shortexactaretriangles} below, and the reason why they are ideal to work with in $\mathsf{D(A)}$ is explained in Theorem \ref{longexactcohomtriangles}.

\begin{Def}\label{exacttriangles}
	Let $\mathsf{A}$ be an abelian category. Then a \textbf{triangle of complexes}\index{triangle of complexes} in $\mathsf{Kom^\#\!(A)}$ or $\mathsf{D^\#\!(A)}$ is a diagram such as $\Delta=\big(X^\bullet\xrightarrow{f^\bullet} Y^\bullet\xrightarrow{g^\bullet} Z^\bullet\xrightarrow{h^\bullet} X[1]^\bullet\big)$ made of complexes and chain maps between them. Given a second one $\tilde{\Delta}$, a \textbf{morphism of triangles of complexes}\index{morphism!of triangles of complexes} is a triple $(i^\bullet,j^\bullet,k^\bullet)$ of chain maps making the diagram
	\begin{equation}\label{triangmorph}
		\begin{tikzcd}
			X^\bullet\arrow[r,"f^\bullet"]\arrow[d,"i^\bullet"'] & Y^\bullet\arrow[r,"g^\bullet"]\arrow[d,"j^\bullet"'] & Z^\bullet\arrow[r,"h^\bullet"]\arrow[d,"k^\bullet"] & {X[1]}^\bullet\arrow[d,"{i[1]}^\bullet"] \\
			\tilde{X}^\bullet\arrow[r,"\tilde{f}^\bullet"'] & \tilde{Y}^\bullet\arrow[r,"\tilde{g}^\bullet"'] & \tilde{Z}^\bullet\arrow[r,"\tilde{h}^\bullet"'] & \tilde{X}[1]^\bullet
		\end{tikzcd}
	\end{equation}
	commute in the chosen category, moreover called an isomorphism if $i^\bullet$, $j^\bullet$, $k^\bullet$ are isomorphisms.
	
	A triangle is specifically a \textbf{distinguished triangle of complexes}\index{distinguished triangle!of complexes} if it is (quasi-)isomorphic to some
	\begin{equation}\label{distinguitriangle}
		X^\bullet\xrightarrow{i_1^\bullet} Z(f)^\bullet\xrightarrow{p_{23}^\bullet} C(f)^\bullet\xrightarrow{\delta[1]^\bullet} X[1]^\bullet
		\quad\equiv\quad
		\begin{tikzcd}[column sep=0.5cm]
			X^\bullet\arrow[rr, "i_1^\bullet"] & & Z(f)^\bullet\arrow[dl, "p_{23}^\bullet"] \\
			& C(f)^\bullet\arrow[ul, "{\delta[1]^\bullet}"]
		\end{tikzcd}
	\end{equation}
	as resulting from the exact middle row of diagram \eqref{distinguidiag}, where on the right we used the triangular depiction we are accustomed to from \cite{[Imp21]} (here the connecting morphism $\delta[1]^\bullet=p_1^\bullet$ is implicitly understood to map to $X[1]^\bullet$).
\end{Def}

\begin{Pro}\label{shortexactaretriangles}
	Let $\mathsf{A}$ be an abelian category. Any short exact sequence $0\rightarrow X^\bullet\xrightarrow{f^\bullet} Y^\bullet\xrightarrow{g^\bullet} Z^\bullet\rightarrow 0$ of complexes in $\mathsf{Kom(A)}$ is quasi-isomorphic to the middle row of diagram \eqref{distinguidiag}, that is, there exists a triple of quasi-isomorphisms as in \eqref{triangmorph}.
	
	Consequently, the associated triangle $X^\bullet\rightarrow Y^\bullet\rightarrow Z^\bullet\rightarrow X[1]^\bullet$ is isomorphic to \eqref{distinguitriangle}, making it a distinguished triangle.  
\end{Pro}

\begin{proof}
	Choose as morphism of triangles $(i^\bullet,j^\bullet,k^\bullet)\coloneqq (\text{id}_X^\bullet,p_3^\bullet,g^\bullet\circ p_2^\bullet)$, fitting into
	\begin{equation*}
		\begin{tikzcd}
			0\arrow[r] & X^\bullet\arrow[r,"i_1^\bullet"]\arrow[d,"\text{id}_X^\bullet"'] & Z(f)^\bullet\arrow[r,"p_{23}^\bullet"]\arrow[d,"p_3^\bullet"] & C(f)^\bullet\arrow[r]\arrow[d,"g^\bullet\circ p_2^\bullet"] & 0 \\
			0\arrow[r] & X^\bullet\arrow[r,"f^\bullet"'] & Y^\bullet\arrow[r,"g^\bullet"'] & Z^\bullet\arrow[r] & 0
		\end{tikzcd}\quad,
	\end{equation*}
	where the cylinder and cone in the first row are built out of our given $f^\bullet$. Indeed $k^\bullet$ is a chain map:
	\begin{align*}
		d_Z^n(k^n(x^{n+1},y^n))&=d_Z^n(g^n(y^n))=0+g^{n+1}(d_Y^n(y^n))\\ 
		&=g^{n+1}(f^{n+1}(x^{n+1}))+g^{n+1}(d_Y^n(y^n))\\
		&=k^{n+1}(-d_X^{n+1}(x^{n+1}),f^{n+1}(x^{n+1})+d_Y^n(y^n))\\
		&=k^{n+1}(d_{C(f)}^n(x^{n+1},y^n))\,,
	\end{align*}
	exploiting exactness of the second row. The so obtained diagram is clearly commutative, again due to $g^n\circ f^n=0$.
	
	Since $\text{id}_X^\bullet$ and $p_3^\bullet$ are quasi-isomorphisms (the latter by Lemma \ref{p3i3quiso}), we need only check that $k^\bullet$ is. Its kernel is the complex $\ker(k)^\bullet = X[1]^\bullet\oplus \ker(g)^\bullet\break= X[1]^\bullet\oplus\text{im}(f)^\bullet\subset C(f)^\bullet$ whose differential is just obtained by restriction, namely $d_{\ker(k)}^n(x^{n+1},x^n)=(-d_X^{n+1}(x^{n+1}),\,x^{n+1}+d_X^n(x^n))$, where injectivity of $f^\bullet$ allows us to identify $X^\bullet\equiv\text{im}(f)^\bullet\subset Y^\bullet$. Now observe that $\text{id}_{\ker(k)}^\bullet\sim 0$ through $h^n:X^{n+1}\oplus X^n\rightarrow X^n\oplus X^{n-1},\, (x^{n+1},x^n)\mapsto (x^n,0)$ \big(indeed it holds $(h^{n+1}\circ d_{\ker(k)}^n+d_{\ker(k)}^{n-1}\circ h^n)(x^{n+1},x^n)=(x^{n+1}+d_X^n(x^n),0)+(-d_X^n(x^n),x^n)=$ $\text{id}_{\ker(k)}^n(x^{n+1},x^n)$\big), so that we obtain $H^n(\text{id}_{\ker(k)}^\bullet)=H^n(0)=0$ and therefore $H^n(\ker(k)^\bullet)=H^n(\text{id}_{\ker(k)}^\bullet)\langle H^n(\ker(k)^\bullet)\rangle=0$.
	
	Consequently, the short exact sequence $0\rightarrow\ker(k)^\bullet\rightarrow C(f)^\bullet\xrightarrow{k^\bullet} Z^\bullet\rightarrow 0$\break --- exact because $\ker(k)^\bullet$ is a subcomplex of the cone of $f^\bullet$ and $k^\bullet$ is surjective (by definition, since $g^\bullet$ is) --- induces\footnote{This is a standard result in homological algebra, which here holds true because $\mathsf{Kom(A)}$ is an abelian category!} the long exact one
	\begin{small}\[
		...\rightarrow 0=H^n(\ker(k)^\bullet)\rightarrow H^n(C(f)^\bullet)\xrightarrow{H^n(k^\bullet)} H^n(Z^\bullet)\rightarrow 0=H^{n+1}(\ker(k)^\bullet)\rightarrow...\,.
		\]\end{small}
	\!\!Then $H^n(k^\bullet)$ is an isomorphisms for all $n\in\mathbb{Z}$, hence $k^\bullet$ is a quasi-isomorphism, as desired. 
\end{proof}

\begin{Rem}\label{alternadistinguitriang}
	Consider the sequence $0\rightarrow Y^\bullet\xrightarrow{i_2^\bullet} C(f)^\bullet\xrightarrow{p_1^\bullet} X[1]^\bullet\rightarrow 0$ arising from $f^\bullet\in\text{Hom}_\mathsf{Kom(A)}(X^\bullet,Y^\bullet)$, clearly short exact. By Proposition \ref{shortexactaretriangles}, it is quasi-isomorphic to the middle row of \eqref{distinguidiag}, hence the associated triangle $X^\bullet\rightarrow Y^\bullet\rightarrow C(f)^\bullet\rightarrow X[1]^\bullet$ (shifted once to the right) is distinguished. 
	
	In fact, many authors call a triangle distinguished when isomorphic to\footnote{Here the first projection in the triangular diagram on the right is formally wrong as written --- $p_1^\bullet$ readily ``climbs up'' to $X[1]^\bullet$ --- but we preserve the notation for visual clarity.} 
	\begin{equation}\label{conetriangle}
		X^\bullet\xrightarrow{f^\bullet} Y^\bullet\xrightarrow{i_2^\bullet} C(f)^\bullet\xrightarrow{p_1^\bullet} X[1]^\bullet
		\quad\equiv\quad
		\begin{tikzcd}[column sep=0.5cm]
			X^\bullet\arrow[rr, "f^\bullet"] & & Y^\bullet\arrow[dl, "i_2^\bullet"] \\
			& C(f)^\bullet\arrow[ul, "{p_1[1]^\bullet}"]
		\end{tikzcd}\quad,
	\end{equation}
	which is compatible with Definition \ref{exacttriangles}, as we have just argued. We will interchange between the two definitions as appropriate.
\end{Rem}

\begin{Thm}\label{longexactcohomtriangles}
	Let $\mathsf{A}$ be an abelian category, $X^\bullet\xrightarrow{f^\bullet} Y^\bullet\xrightarrow{g^\bullet} Z^\bullet\xrightarrow{h^\bullet} X[1]^\bullet$ a distinguished triangle in $\mathsf{D(A)}$. Then this induces a long exact sequence in cohomology
	\begin{equation}
		\!...\rightarrow H^n(X^\bullet)\xrightarrow{H^n(f^\bullet)} H^n(Y^\bullet)\xrightarrow{H^n(g^\bullet)} H^n(Z^\bullet)\xrightarrow{H^n(h^\bullet)} H^{n+1}(X^\bullet)\rightarrow...
	\end{equation}
	in $\mathsf{A}$.
\end{Thm}

\begin{proof}
	By definition of distinguished triangle, we may just focus on \eqref{distinguitriangle}, whose associated short exact sequence in $\mathsf{Kom(A)}$ is the middle row of diagram \eqref{distinguidiag}. This induces the long exact sequence
	\begin{small}\[
		...\!\rightarrow H^n(X^\bullet)\xrightarrow{H^n(i_1^\bullet)} H^n(Z(f)^\bullet)\xrightarrow{H^n(p_{23}^\bullet)} H^n(C(f)^\bullet)\xrightarrow{H^n(\delta[1]^\bullet)} H^{n+1}(X^\bullet)\rightarrow\!...,
		\]\end{small}
	\!where the connecting homomorphism $H^n(\delta[1]^\bullet)$ actually corresponds to $H^n(h^\bullet)\!:$ $H^n(Z^\bullet)\cong H^n(C(f)^\bullet)\rightarrow H^n(X[1]^\bullet)=H^{n+1}(X^\bullet)$, as can be shown by applying the Snake Lemma. Since the notion of distinguished triangles carries over to the derived category of $\mathsf{A}$, indeed the obtained long sequence is exact in $\mathsf{D(A)}$ as well.
\end{proof}

Proposition \ref{shortexactaretriangles} and Theorem \ref{longexactcohomtriangles} together explain why triangles are nice: any short exact sequence in $\mathsf{Kom(A)}$ ultimately induces a long exact one in cohomology, after due manipulations. From the perspective of the functors $H^n$, this just tells us that they are, rather unsurprisingly, \textit{cohomological}.

\begin{Def}\label{cohomfunconabelian}
Let $\mathsf{A}, \mathsf{B}$ be abelian categories. A functor $\mathsf{F}:\mathsf{D^\#\!(A)}\rightarrow\mathsf{B}$ (or on the homotopy category of $\mathsf{A}$ we will introduce below), for $\#=\emptyset, +, -, \mathsf{b}$, is a \textbf{cohomological functor}\index{functor!cohomological (on a derived category)} if any distinguished triangle $X^\bullet\xrightarrow{f^\bullet} Y^\bullet\xrightarrow{g^\bullet} Z^\bullet\xrightarrow{h^\bullet} X[1]^\bullet$ is mapped to a long exact sequence
\[
...\rightarrow\mathsf{F}(T^n(X^\bullet))\rightarrow\mathsf{F}(T^n(Y^\bullet))\rightarrow\mathsf{F}(T^n(Z^\bullet))\rightarrow\mathsf{F}(T^{n+1}(X^\bullet))\rightarrow...\,,
\]
where $T^n$ are the translation autoequivalences on the respective categories of complexes.
\end{Def}

Since $H^n=H^0\circ T^n:\mathsf{D^\#\!(A)}\rightarrow\mathsf{A}$, Theorem \ref{longexactcohomtriangles} indeed shows that all functors $H^n$ are cohomological in name and in fact. Later on we will see that also the $\text{Hom}$-functors are cohomological.

\begin{Rem}\label{C(f)acyclic}
Consider $X^\bullet\xrightarrow{f^\bullet} Y^\bullet\xrightarrow{i_2^\bullet} C(f)^\bullet\xrightarrow{p_1^\bullet} X[1]^\bullet$, the distinguished triangle associated to $f^\bullet\in\text{Hom}_\mathsf{Kom(A)}(X^\bullet,Y^\bullet)$ from Remark \ref{alternadistinguitriang}. By Theorem \ref{longexactcohomtriangles} it induces the long exact sequence in cohomology
\begin{small}\[
	...\rightarrow H^n(X^\bullet)\xrightarrow{H^n(f^\bullet)} H^n(Y^\bullet)\xrightarrow{H^n(i_2^\bullet)} H^n(C(f)^\bullet)\xrightarrow{H^n(p_1^\bullet)} H^{n+1}(X^\bullet)\xrightarrow{H^{n+1}(f^\bullet)}...\,.
	\]\end{small}
\!\!Now, by standard homological algebra we know that if $f^\bullet$ is a quasi-isomorphism, then $C(f)^\bullet$ is acyclic: $\ker(H^n(i_2^\bullet))=\text{im}(H^n(f^\bullet))=H^n(Y^\bullet)$ and $\text{im}(H^n(p_1^\bullet))=\ker(H^{n+1}(f^\bullet))=\{0\}$, hence $H^n(i_2^\bullet)$ and $H^n(p_1^\bullet)$ are both trivial, and we conclude that $H^n(C(f)^\bullet)=\ker(H^n(p_1^\bullet))=\text{im}(H^n(i_2^\bullet))=\{0\}$. The converse is true as well. 
\end{Rem}

\subsection{The homotopy category}\label{ch2.4}

We proceed to introduce all the machinery necessary for Proposition \ref{quasilocalizing}, which will finally validate the explicit description of $\mathsf{D(A)}$ offered by Proposition \ref{derivedstructure}.

\begin{Def}
	Let $\mathsf{A}$ be an abelian category. Then its \textbf{homotopy category}\index{homotopy category} $\mathsf{H(A)}$ is given by $\text{obj}(\mathsf{H(A)})\coloneqq\text{obj}(\mathsf{Kom(A)})$ and by identifying homotopic morphisms $f^\bullet\sim g^\bullet\in\text{Hom}_{\mathsf{Kom(A)}}(X^\bullet, Y^\bullet)$ to the same $f^\bullet=g^\bullet\in\text{Hom}_{\mathsf{H(A)}}(X^\bullet, Y^\bullet)$.
	
	The full subcategories $\mathsf{H^+(A)}, \mathsf{H^-(A)}, \mathsf{H^b(A)}\subset\mathsf{H(A)}$ are similarly constructed from their respective categories of bounded complexes.
\end{Def}

Arguing as in Lemma \ref{komabelian}, it is obvious that $\mathsf{H(A)}$ is an additive category, which actually suffices to restrict the cohomology functors to functors $H^n:\mathsf{H(A)}\rightarrow\mathsf{A}$ as described in Definition \ref{cohomfunc}.\footnote{Note that homotopic morphisms $f^\bullet\sim g^\bullet$ yielding the same cohomological morphisms in $\mathsf{A}$ is all the more consistent with $f^\bullet=g^\bullet$ in $\mathsf{H(A)}$.} Moreover, the functor $\mathsf{Kom(A)}\rightarrow\mathsf{H(A)}$ mapping chain maps to their equivalence class modulo homotopy is an additive functor. And Definition \ref{quasiso} of quasi-isomorphism transposes unscathed to $\mathsf{H(A)}$.

Unfortunately, $\mathsf{H(A)}$ is an abelian category only under a quite specific circumstance, as shown by Verdier in \cite[Proposition II.1.3.6]{[Ver96]}:\footnote{A direct counterexample supporting non-abelianity of $\mathsf{H(A)}$ is provided at \url{https://math.stackexchange.com/questions/1280199/the-homotopy-category-of-complexes/1284823\#1284823}.} 

\begin{Lem}\label{whenD(A)abelian}
	Let $\mathsf{A}$ be an abelian category. Then its homotopy category $\mathsf{H(A)}$ is abelian if and only if $\mathsf{A}$ is semisimple \textup(cf. Remark \ref{semisimple}\textup).
\end{Lem} 

So the desirable notion of exactness of sequences already fails here, and should be replaced by the triangles of Section \ref{ch2.3}. The advantage of passing through the homotopy category is but another:

\begin{Thm}
	Let $\mathsf{A}$ be an abelian category and consider the class $S\coloneqq\{\text{quasi-isomorphisms in }\mathsf{H(A)}\}$. Then the localization $\mathsf{H(A)}[S^{-1}]$ \textup(as constructed in the proof of Proposition \ref{derivedexists}\textup) is canonically isomorphic to $\mathsf{D(A)}$. Similarly, $\mathsf{H^\#\!(A)}[S^{-1}]$ is canonically isomorphic to $\mathsf{D^\#\!(A)}$ for $\#=+,-,\mathsf{b}$.
\end{Thm}

\begin{proof}
	Since the composition $\mathsf{F}:\mathsf{Kom(A)}\rightarrow\mathsf{H(A)}\rightarrow\mathsf{H(A)}[S^{-1}]$ maps by construction quasi-isomorphisms into isomorphisms, by Proposition \ref{derivedexists} there exists a unique functor $\mathsf{G}:\mathsf{D(A)}\rightarrow\mathsf{H(A)}[S^{-1}]$ such that $\mathsf{F}=\mathsf{G}\circ\mathsf{L}$. For any $X\in\text{obj}(\mathsf{Kom(A)})=\text{obj}(\mathsf{D(A)})$ holds $\mathsf{G}(X)=(\mathsf{G}\circ\mathsf{L})(X)=\mathsf{F}(X)=X\in\text{obj}(\mathsf{H(A)}[S^{-1}])$, whence bijectivity on objects of $\mathsf{G}$. Moreover, any morphisms class $[f^\bullet]$ in $\mathsf{H(A)}[S^{-1}]$ admits a representative $f^\bullet$ in $\mathsf{H(A)}$, thus surely one in $\mathsf{Kom(A)}$ by definition of homotopy category, so that $[f^\bullet]=\mathsf{F}(f^\bullet)=(\mathsf{G}\circ\mathsf{L})(f^\bullet)=\mathsf{G}([\tilde{f}^\bullet])$, proving surjectivity of $\mathsf{G}$ on morphisms.
	
	About injectivity on morphisms, observe that $f^\bullet\sim g^\bullet\in\text{Hom}_{\mathsf{Kom(A)}}(X^\bullet, Y^\bullet)$ implies $\mathsf{L}(f^\bullet)=\mathsf{L}(g^\bullet)\in\text{Hom}_{\mathsf{D(A)}}(X^\bullet, Y^\bullet)$. Let us prove this claim explicitly. By assumption, $f^n=g^n+d_Y^{n-1}\circ k^n + k^{n+1}\circ d_X^n$ for a homotopy $k^\bullet:X^\bullet\rightarrow Y^{\bullet-1}$. This gives rise to the morphism $c(k)^\bullet:C(f)^\bullet=X^{\bullet+1}\oplus Y^\bullet\rightarrow C(g)^\bullet=X^{\bullet+1}\oplus Y^\bullet$, $c(k)^n(x^{n+1},y^n)\coloneqq(x^{n+1},y^n+k^{n+1}(x^{n+1}))$, actually a chain map because 
	\begin{align*}
		\big(d_{C(g)}^n\!\circ c(k)^n)&(x^{n+1},y^n\big)\!=\!\big(\!-d_X^{n+1}(x^{n+1}),g^{n+1}(x^{n+1})+d_Y^n(y^n+k^{n+1}(x^{n+1}))\big)\\
		&=\big(-d_X^{n+1}(x^{n+1}), f^{n+1}(x^{n+1})+d_Y^n(y^n)+k^{n+2}(-d_X^{n+1}(x^{n+1}))\big)\\
		&= c(k)^{n+1}\big(-d_X^{n+1}(x^{n+1}),f^{n+1}(x^{n+1})+d_Y^n(y^n)\big)\\
		&=\big(c(k)^{n+1}\circ d_{C(f)}^n\big)(x^{n+1},y^n).
	\end{align*}
	Similarly, $k^\bullet$ defines a chain map $z(k)^\bullet:Z(f)^\bullet\rightarrow Z(g)^\bullet$, $z(k)^n(x^n,x^{n+1},y^n)\coloneqq (x^n,x^{n+1},y^n+k^{n+1}(x^{n+1}))$.
	
	The exact rows of $f^\bullet$ and $g^\bullet$ like the top one in \eqref{distinguidiag} combine to give the diagram
	\begin{equation*}
		\begin{tikzcd}
			0\arrow[r] & Y^\bullet\arrow[r,"i_2^\bullet"]\arrow[d,"\text{id}_Y^\bullet"'] & C(f)^\bullet\arrow[r,"p_1^\bullet"]\arrow[d,"c(k)^\bullet"] & X[1]^\bullet\arrow[r]\arrow[d,"\text{id}_{X[1]}^\bullet"] & 0 \\
			0\arrow[r] & Y^\bullet\arrow[r,"i_2^\bullet"'] & C(g)^\bullet\arrow[r,"p_1^\bullet"'] & X[1]^\bullet\arrow[r] & 0
		\end{tikzcd}\quad,
	\end{equation*}
	which is commutative by definition of $c(k)^\bullet$. Since the rows are short exact in $\mathsf{Kom(A)}$, this induces a diagram in cohomology where all vertical arrows are still identities, except for the $H^n(c(k)^\bullet):H^n(C(f)^\bullet)\rightarrow H^n(C(g)^\bullet)$, however isomorphisms due to the Five-Lemma. Therefore, $c(k)^\bullet$ is a quasi-isomorphism. This plays into showing that $z(k)^\bullet$ is a quasi-isomorphism as well, which can be analogously deduced from the middle rows of Lemma \ref{p3i3quiso}.
	
	Finally, consider the diagram
	\[
	\begin{tikzcd}
		& X^\bullet\arrow[r, "\text{id}_X^\bullet"]\arrow[dl, "f^\bullet"']\arrow[d, "i_1^\bullet"] & X^\bullet\arrow[d, "i_1^\bullet"']\arrow[dr, "g^\bullet"] & \\
		Y^\bullet\arrow[r, "i_3^\bullet"'] & Z(f)^\bullet\arrow[r, "z(k)^\bullet"'] & Z(g)^\bullet\arrow[r, "p_3^\bullet"'] & Y^\bullet
	\end{tikzcd}\quad,
	\]
	which only commutes in the square and right triangle (where $p_3^\bullet$ is the usual ``modified'' projection). The left triangle does not commute in $\mathsf{Kom(A)}$ \big(clearly $i_3^n(f^n(x^n))=(0,0,f^n(x^n))\neq(x^n,0,0)=i_1^n(x^n)$\big), but in $\mathsf{D(A)}$ it \textit{does} commute, because $f^\bullet=p_3^\bullet\circ i_1^\bullet$ (see right triangle) turns under $\mathsf{L}:\mathsf{Kom(A)}\rightarrow\mathsf{D(A)}$ into $\mathsf{L}(f^\bullet)=\mathsf{L}(p_3^\bullet)\circ\mathsf{L}(i_1^\bullet)$, hence $\mathsf{L}(i_1^\bullet)=\mathsf{L}(p_3^\bullet)^{-1}\circ\mathsf{L}(f^\bullet)=\mathsf{L}(i_3^\bullet)\circ\mathsf{L}(f^\bullet)$ (again by Lemma \ref{p3i3quiso}). Thanks to full commutativity of the diagram and $p_3^\bullet\circ z(k)^\bullet\circ i_3^\bullet=\text{id}_Y^\bullet$, we deduce at last that $\mathsf{L}(f^\bullet)=\mathsf{L}(p_3^\bullet\circ z(k)^\bullet\circ i_3^\bullet\circ f^\bullet)=\mathsf{L}(g^\bullet)$. 
	
	Therefore, $\mathsf{G}$ is injective on morphisms: suppose $[f^\bullet]=[g^\bullet]$ in $\mathsf{H(A)}[S^{-1}]$, meaning they share homotopic representatives $f^\bullet\sim g^\bullet$ in $\mathsf{Kom(A)}$ at worst, then we just proved that $\mathsf{L}(f^\bullet)=\mathsf{L}(g^\bullet)$ in $\mathsf{D(A)}$, as desired. We conclude that $\mathsf{G}$ is bijective on both objects and morphisms, thus establishing the equivalence between $\mathsf{H(A)}[S^{-1}]$ and $\mathsf{D(A)}$, respectively between $\mathsf{H^\#\!(A)}[S^{-1}]$ and $\mathsf{D^\#\!(A)}$.  
\end{proof}

Thanks to the statement just shown, we might as well \textit{define} the derived category of an abelian category $\mathsf{A}$ to be $\mathsf{D(A)}\coloneqq\mathsf{H(A)}[S^{-1}]$. Now, Proposition \ref{derivedstructure} can be applied, as justified by the following result.

\begin{Pro}\label{quasilocalizing}
	Let $\mathsf{A}$ be an abelian category. Then the class of morphisms $S\coloneqq\{\text{quasi-isomorphisms in }\mathsf{H(A)}\}$ is localizing in $\mathsf{H(A)}$ \textup(and similarly $S^\#\coloneqq\{\text{quasi-isomorphisms in }\mathsf{H^\#\!(A)}\}$ is localizing in $\mathsf{H^\#\!(A)}$\textup). 
\end{Pro}   

\begin{proof}
	We must check that the three items of Definition \ref{locclass} hold. Point a. is\break obvious: $\text{id}_X^\bullet\in S$ for all $X^\bullet\in\text{obj}(\mathsf{H(A)})$, and $s^\bullet,t^\bullet\in S$ implies $s^\bullet\circ t^\bullet\in S$, by functoriality of each $H^n$.
	
	Regarding b., take any triple $Z^\bullet\xrightarrow{s^\bullet}Y^\bullet\xleftarrow{f^\bullet}X^\bullet$ in $\mathsf{H(A)}$ with $s^\bullet\in S$. The objective is to find a complementing triple making diagram \eqref{localsquare} left commute. Let $i_2^\bullet:Y^\bullet\rightarrow C(s)^\bullet=Z^{\bullet+1}\oplus Y^\bullet$ be the second inclusion, so $i_2^\bullet\circ f^\bullet:X^\bullet\rightarrow C(s)^\bullet$,\break and choose $W^\bullet\coloneqq C(i_2\circ f)[-1]^\bullet= X^\bullet\oplus C(s)^{\bullet-1}=X^\bullet\oplus Z^\bullet\oplus Y^{\bullet-1}$ and $t^\bullet\coloneqq p_1^\bullet:W^\bullet\rightarrow X^\bullet$. These fit into the diagram
	\begin{equation}\label{messydiag}
		\begin{tikzcd}
			{C(i_2\circ f)[-1]^\bullet}\arrow[r, dashed, "{p_1^\bullet}"]\arrow[d, dashed, "g^\bullet"'] & X^\bullet\arrow[r, "{i_2^\bullet\circ f^\bullet}"]\arrow[d, "f^\bullet"'] & C(s)^\bullet\arrow[r, "{i_2^\bullet}"]\arrow[d, "{\text{id}_{C(s)}^\bullet}"] & C(i_2\circ f)^\bullet\arrow[d, "{g[1]}^\bullet"] \\
			Z^\bullet\arrow[r, "s^\bullet"'] & Y^\bullet\arrow[r, "{i_2^\bullet}"'] & C(s)^\bullet\arrow[r, "{p_1^\bullet}"'] & Z[1]^\bullet
		\end{tikzcd}\quad,
	\end{equation}
	where the dashed arrows hint to the desired morphisms. Define $g^\bullet:W^\bullet\rightarrow Z^\bullet$ by $g^n(x^n,z^n,y^{n-1})$ $\coloneqq -z^n$, then the left square commutes modulo homotopy $k^\bullet:W^\bullet\rightarrow Y^{\bullet-1}$ given by $k^n(x^n,z^n,y^{n-1})\coloneqq -y^{n-1}$: indeed we compute $\big(k^{n+1}\circ d_W^n + d_Y^{n-1}\circ k^n\big)(x^n,z^n,y^{n-1})=-(-f^n(x^n)-s^n(z^n)-d_Y^{n-1}(y^{n-1}))+d_Y^{n-1}(-y^{n-1})=f^n(x^n) + s^n(z^n)=\big(f^n\circ p_1^n -s^n\circ g^n\big)(x^n,z^n,y^{n-1})$, where the first equality used the definition of differential of $W^\bullet$,
	\begin{align*}
		d_W^n=-d_{C(i_2\circ f)}^{n-1}=\!-\!\begin{pmatrix} -d_X^n & 0 \\ i_2^n\circ f^n & d_{C(s)}^{n-1} \end{pmatrix} \!=\! \begin{pmatrix}
			d_X^n & 0 & 0 \\ 0 & d_Z^n & 0 \\ -f^n & -s^n & -d_Y^{n-1}
		\end{pmatrix}\!:W^n\rightarrow W^{n+1},
	\end{align*}
	whence $d_W^n(x^n,z^n,y^{n-1})=\big(d_X^n(x^n), d_Z^n(z^n),-f^n(x^n)-s^n(z^n)-d_Y^{n-1}(y^{n-1})\big)$. Therefore, the left-hand square commutes in $\mathsf{H(A)}$. Finally, we argue that $p_1^\bullet$ is a quasi-isomorphism. By Remark \ref{C(f)acyclic} applied to the second row of diagram \eqref{messydiag}, $C(s)^\bullet$ is acyclic. But the first row is itself a distinguished triangle, hence Theorem \ref{longexactcohomtriangles} gives a long exact sequence
	\[
	...\rightarrow H^{n-1}(C(s)^\bullet)\cong 0 \rightarrow H^{n}(W^\bullet)\xrightarrow{H^n(p_1^\bullet)} H^{n}(X^\bullet)\rightarrow H^{n}(C(s)^\bullet)\cong 0\rightarrow...\,,
	\]
	in turn telling us that $H^n(p_1^\bullet)$ must be an isomorphism for all $n\in\mathbb{Z}$, that is, $p_1^\bullet\in S$. Item b. is thus fulfilled (the second condition corresponding to \eqref{localsquare} right is proved analogously).
	
	About condition c., assume $f^\bullet\in\text{Hom}_\mathsf{H(A)}(X^\bullet,Y^\bullet)$ and suppose there exists an $s^\bullet:Y^\bullet\rightarrow Z^\bullet$ in $S$ with $s^\bullet\circ f^\bullet=0$, meaning that $s^\bullet\circ f^\bullet\sim 0$ in $\mathsf{Kom(A)}$ via some $k^\bullet:X^\bullet\rightarrow Z^{\bullet-1}$ (thus fulfilling $d_Z^{n-1}\circ k^n= s^n\circ f^n -k^{n+1}\circ d_X^n$). Then we seek some $t^\bullet:W^\bullet\rightarrow X^\bullet$ in $S$ such that $f^\bullet\circ t^\bullet=0$. Thereto, consider
	\begin{equation*}
		\begin{tikzcd}
			C(s)[-1]^\bullet\arrow[r, "{p_1^\bullet}"] & Y^\bullet\arrow[r, "s^\bullet"] & Z^\bullet\arrow[r, "{i_2^\bullet}"] & {C(s)^\bullet = Y^{\bullet+1}\oplus Z^\bullet} \\
			C(s)[-1]^\bullet\arrow[u, "{\text{id}_{C(s)[-1]}^\bullet}"] & X^\bullet\arrow[l, "g^\bullet"]\arrow[u, "f^\bullet"]\arrow[ur, "0"] & C(g)[-1]^\bullet\arrow[l, dashed, "t^\bullet"] & C(s)[-2]^\bullet=C(s)^{\bullet-2}\arrow[l, "i_2^\bullet"]\\
		\end{tikzcd}\quad,
		\vspace*{-0.3cm}
	\end{equation*}
	where the morphism $g^\bullet$ making the square commute is given by $g^n(x^n)\coloneqq(f^n(x^n),-k^n(x^n))\in Y^n\oplus Z^{n-1}$, a well-defined chain map: 
	\begin{align*}
		-d_{C(s)}^{n-1}&(g^n(x^n))=\big(d_Y^n(f^n(x^n)),-s^n(f^n(x^n))+d_Z^{n-1}(k^n(x^n))\big) \\
		&=\big(f^{n+1}(d_X^n(x^n)),-k^{n+1}(d_X^n(x^n))-d_Z^{n-1}(k^n(x^n))+d_Z^{n-1}(k^n(x^n))\big) \\
		&=g^{n+1}(d_X^n(x^n)).
	\end{align*}
	Choose $W^\bullet\coloneqq C(g)[-1]^\bullet = X^\bullet\oplus C(s)^{\bullet-2}=X^\bullet\oplus Y^{\bullet-1}\oplus Z^{\bullet-2}$ and $t^\bullet\coloneqq p_1^\bullet:W^\bullet\rightarrow X^\bullet$, $t^n(x^n,y^{n-1},z^{n-2})= x^n$. Then $f^\bullet\circ t^\bullet=p_1^\bullet\circ g^\bullet\circ t^\bullet=0$ as desired, since the bottom row is exact. Moreover, $C(s)^\bullet$ is acyclic by Remark \ref{C(f)acyclic}, hence the long exact sequence in cohomology induced by the distinguished triangle in the bottom row implies that $t^\bullet\in S$. The converse direction, namely the existence of a suitable $t^\bullet$ implying that of a suitable $s^\bullet$, can be shown similarly. So item c. holds and we are done. 
\end{proof}

\subsection{The structure of derived categories}\label{ch2.5}

In light of Proposition \ref{quasilocalizing}, let us spell out once again the general structure of a derived category.

\begin{Def}\label{dercatdef}
	Let $\mathsf{A}$ be an abelian category, with associated homotopy category $\mathsf{H(A)}$ and localizing class $S\coloneqq\{\text{quasi-isomorphisms in }\mathsf{H(A)}\}\subset\mathsf{H(A)}$. Then its \textit{derived category} $\mathsf{D(A)}=\mathsf{H(A)}[S^{-1}]$ has the following structure:
	\begin{itemize}[leftmargin=0.5cm]
		\renewcommand{\labelitemi}{\textendash}
		\item $\text{obj}(\mathsf{D(A)})=\text{obj}(\mathsf{Kom(A)})$, hence objects are cochain complexes $(X^\bullet,d_X^\bullet)$ with each $X^n\in\text{obj}(\mathsf{A})$;
		
		\item morphisms $\varphi\in\text{Hom}_\mathsf{D(A)}(X^\bullet,Y^\bullet)$ are equivalence classes $[(s^\bullet,f^\bullet)]$ of roofs
		\begin{equation}
			\begin{tikzcd}
				& {X'}^\bullet\arrow[dl, "s^\bullet"']\arrow[dr, "f^\bullet"] & \\
				X^\bullet & & Y^\bullet
			\end{tikzcd}\quad,
		\end{equation}
		where $f^\bullet\in\text{Hom}_{\mathsf{H(A)}}(X'^\bullet,Y^\bullet)$ and $s^\bullet\in S$. The defining equivalence relation is specified by diagram \eqref{equivroofs}, and composition of morphisms occurs as in \eqref{roofcompo}.
	\end{itemize}
	The derived category is equipped with a \textit{localization functor}\footnote{Pre-composed with the additive functor $\mathsf{Kom(A)}\rightarrow\mathsf{H(A)}$, we get the earlier localization $\mathsf{L}:\mathsf{Kom(A)}\rightarrow\mathsf{D(A)}$ from Proposition \ref{derivedexists}, which then passes on its universal property.} $\mathsf{L}:\mathsf{H(A)}\rightarrow\mathsf{D(A)}$ which is the identity on objects and maps any $f^\bullet\in\text{Hom}_\mathsf{H(A)}(X^\bullet,Y^\bullet)$ to the class $[(\text{id}_X^\bullet,f^\bullet)]\in\text{Hom}_\mathsf{D(A)}(X^\bullet,Y^\bullet)$. Furthermore, it fulfills that $\mathsf{L}(s^\bullet)$ is an isomorphism for any $s^\bullet\in S$, and if any other functor $\mathsf{F}:\mathsf{H(A)}\rightarrow\mathsf{D}$ acts this way, then it factors like $\mathsf{F}=\mathsf{G}\circ\mathsf{L}$ for a unique $\mathsf{G}:\mathsf{D(A)}\rightarrow\mathsf{D}$.
	The variants $\mathsf{D^\#\!(A)}$ have analogous structure.
\end{Def}

By Lemma \ref{whenD(A)abelian} we know that $\mathsf{D(A)}$ fails to be an abelian category in general. But at least:

\begin{Lem}\label{D(A)additive}
	Let $\mathsf{A}$ be an abelian category. Then its derived category $\mathsf{D(A)}$ is additive.
\end{Lem}

\begin{proof}
	We must explain how to sum two morphisms $\varphi_k\in\text{Hom}_\mathsf{D(A)}(X^\bullet,Y^\bullet)$ represented by roofs $X^\bullet\xleftarrow{s_k^\bullet}X_k'^\bullet\xrightarrow{f_k^\bullet}Y^\bullet$, for $k=1,2$. Thereto, we must find a ``common denominator'', that is, make them share the same quasi-isomorphism. This is achieved by the following diagram:
	\begin{equation}
		\begin{tikzcd}
			& & X''^\bullet\arrow[dl, dashed, "r_1^\bullet"']\arrow[dr, dashed, "r_2^\bullet"]\arrow[dd, "t^\bullet"] & & \\
			& X_1'^\bullet\arrow[dl, "f_1^\bullet"']\arrow[dr, "s_1^\bullet"'] & & X_2'^\bullet\arrow[dl, "s_2^\bullet"]\arrow[dr, "f_2^\bullet"] & \\
			Y^\bullet & & X^\bullet & & Y^\bullet 
		\end{tikzcd}\qquad,
	\end{equation}
	where the commutative square is constructed using diagram \eqref{localsquare}; this tells us that $r_1^\bullet\in S$ and hence $r_2^\bullet$ is necessarily a quasi-isomorphism too. Then $t^\bullet\coloneqq s_1^\bullet\circ r_1^\bullet= s_2^\bullet\circ r_2^\bullet\in S$ is our desired common denominator, and setting $g_k^\bullet\coloneqq f_k^\bullet\circ r_k^\bullet$ we see that $\varphi_k$ is also represented by the roof $X^\bullet\xleftarrow{t^\bullet}X''^\bullet\xrightarrow{g_k^\bullet}Y^\bullet$. Then we let $(\varphi_1+\varphi_2)\in\text{Hom}_\mathsf{D(A)}(X^\bullet,Y^\bullet)$ be the class represented by the roof $X^\bullet\xleftarrow{t^\bullet}X''^\bullet\xrightarrow{g_1^\bullet+g_2^\bullet}Y^\bullet$. A careful check reveals that it depends exclusively on $\varphi_1,\varphi_2$ and not on the chosen representatives.
	
	Symmetry of the above construction clearly implies $\varphi_1+\varphi_2=\varphi_2+\varphi_1$, and we recall that the identity is trivially represented by $X^\bullet\xleftarrow{\text{id}_X^\bullet}X^\bullet\xrightarrow{\text{id}_X^\bullet}X^\bullet$. So we recognize the abelian structure of the group $\text{Hom}_\mathsf{D(A)}(X^\bullet,Y^\bullet)$. Biadditivity of the composition is just a matter of properly drawing (a lot of) diagrams. Hence, $\mathsf{D(A)}$ is preadditive.
	
	The zero object of $\mathsf{D(A)}$ is just the trivial complex, and indeed any morphism directed from/to it is necessarily trivial.\footnote{As a sidenote, we observe that $[(s^\bullet,f^\bullet)]=0\in\text{Hom}_\mathsf{D(A)}(X^\bullet,Y^\bullet)$ if and only if there exists some $s^\bullet:Y^\bullet\rightarrow Z^\bullet$ in $S$ (then forcedly non-zero) such that $s^\bullet\circ f^\bullet=0$, or equivalently, some $t^\bullet:W^\bullet\rightarrow X'^\bullet$ in $S$ such that $f^\bullet\circ t^\bullet=0$. A valid representative is $X^\bullet\xleftarrow[\text{id}_X^\bullet]{}X^\bullet\xrightarrow[0]{}Y^\bullet$.} Finally, concerning axiom (A3), given two $X_1^\bullet,X_2^\bullet\in\text{obj}(\mathsf{D(A)})$ and their direct sum $Y^\bullet=X_1^\bullet\oplus X_2^\bullet\in\text{obj}(\mathsf{D(A)})$ (defined in the obvious way), we let $i_k^\bullet\in\text{Hom}_\mathsf{H(A)}(X_k^\bullet,Y^\bullet)$ be represented by the roof $X_k^\bullet\xleftarrow{\text{id}_{X_k}^\bullet} X_k^\bullet\xrightarrow{i_k^\bullet}Y^\bullet$ and $p_k^\bullet\in\text{Hom}_\mathsf{H(A)}(Y^\bullet,X_k^\bullet)$ by $Y^\bullet\xleftarrow{\text{id}_Y^\bullet} Y^\bullet\xrightarrow{p_k^\bullet}X_k^\bullet$. Then the necessary relations described by \eqref{directprodsum} are easily verified; for example
	\[
	\begin{tikzcd}[column sep=0.5cm]
		& & X_1^\bullet\arrow[dl, "\text{id}_{X_1}^\bullet"']\arrow[dr, "i_1^\bullet"] & & \\
		& X_1^\bullet\arrow[dl, "\text{id}_{X_1}^\bullet"']\arrow[dr, "i_1^\bullet"'] & & Y^\bullet\arrow[dl, "\text{id}_Y^\bullet"]\arrow[dr, "p_k^\bullet"] & \\
		X_1^\bullet\arrow[rr, squiggly, "i_1^\bullet"'] & & Y^\bullet\arrow[rr, squiggly, "p_k^\bullet"'] & & X_k^\bullet 
	\end{tikzcd}\;\;\;\;\;\;
	\begin{tikzcd}[column sep=0.5cm]
		& & Y^\bullet\arrow[dl, "\text{id}_Y^\bullet"']\arrow[dr, "\text{id}_Y^\bullet"]\arrow[dd, "\text{id}_Y^\bullet"] & & \\
		& Y^\bullet\arrow[dl, "i_1^\bullet\circ p_1^\bullet"']\arrow[dr, "\text{id}_Y^\bullet"'] & & Y^\bullet\arrow[dl, "\text{id}_Y^\bullet"]\arrow[dr, "i_2^\bullet\circ p_2^\bullet"] & \\
		Y^\bullet & & Y^\bullet\arrow[ll, squiggly, "i_1^\bullet\circ p_1^\bullet"]\arrow[rr, squiggly, "i_2^\bullet\circ p_2^\bullet"'] & & Y^\bullet 
	\end{tikzcd}	
	\]
	tell us that $p_1^\bullet\circ i_1^\bullet=\text{id}_{X_1}^\bullet$, $p_2^\bullet\circ i_1^\bullet=0$ respectively $i_1^\bullet\circ p_1^\bullet + i_2^\bullet\circ p_2^\bullet = \text{id}_Y^\bullet$ hold also at the level of representatives, hence in $\mathsf{D(A)}$. Therefore, the derived category is additive.
\end{proof}

Let us finally explain how to regard an abelian category as a full subcategory of its derived category.

\begin{Def}
	Let $\mathsf{A}$ be an abelian category. For each $k\in\mathbb{Z}$, any $X\in\text{obj}(\mathsf{A})$ defines a complex in $\mathsf{Kom^\#\!(A)}$ (respectively in $\mathsf{H^\#\!(A)}$ or $\mathsf{D^\#\!(A)}$, for $\#=\emptyset,+,-,\mathsf{b}$) of the form\footnote{It should be clear from the context whether $X[n]^\bullet$ refers to an $n$-fold, possibly non-trivial shifted complex or the complex induced by an object.}
	\[
	X[-k]^\bullet\coloneqq\big(...\rightarrow 0\rightarrow 0\rightarrow X\rightarrow 0\rightarrow 0\rightarrow ...\big)\,.
	\]
	where $X$ lies at $k$-th position. We call it the \textbf{$k$-complex}\index{kcomplex@$k$-complex} induced by $X$, acyclic except for $H^k(X[-k]^\bullet)\cong X$. Back in Remark \ref{0-complex}, we met $0$-complexes $X[0]^\bullet$.
	
	In turn, an $X^\bullet\in\mathsf{Kom^\#\!(A)}$ (or belonging to the other categories of complexes) is an \textbf{$H^k$-complex}\index{hk@$H^k$-complex} if it is acyclic away from the $k$-th position, $H^j(X^\bullet)=0$ for all $j\neq k$ (one also says that $X^\bullet$ is \textit{concentrated in degree $k$}). Consequently, $k$-complexes are $H^k$-complexes. 
\end{Def}

\begin{Pro}\label{AinD(A)}
	Let $\mathsf{A}$ be an abelian category. Then the functor $\mathsf{J'}\coloneqq \mathsf{L\circ J}:\mathsf{A}\rightarrow\mathsf{D^\#\!(A)}$ is fully faithful \textup(see again Remark \ref{0-complex}\textup). In particular, $\mathsf{A}$ is equivalent\footnote{That is, $\mathsf{J'}:\mathsf{A}\rightarrow\mathsf{D_0^\#\!(A)}$ is also essentially surjective, meaning that any $H^0$-complex $Y^\bullet\in\text{obj}(\mathsf{D_0^\#\!(A)})$ is isomorphic in $\mathsf{D_0^\#\!(A)}$ to $\mathsf{J'}(X^\bullet)$ for some $X^\bullet\in\text{obj}(\mathsf{A})$.} to the full subcategory $\mathsf{D_0^\#\!(A)}\subset\mathsf{D^\#\!(A)}$ of all $H^0$-complexes.
\end{Pro}

\begin{proof}
	The fully faithful functor $\mathsf{J}:\mathsf{A}\rightarrow\mathsf{Kom^\#\!(A)}$ can be refined to a functor $\mathsf{J}:\mathsf{A}\rightarrow\mathsf{H^\#\!(A)}$, again fully faithful because equality of two morphisms $f^\bullet, g^\bullet$ in $\mathsf{H^\#\!(A)}$ means that $f^\bullet\sim g^\bullet$, which immediately implies $f^0=g^0$ due to triviality of the differentials of 0-complexes (whereas surjectivity on morphisms is obvious). This makes $\mathsf{A}$ into the full subcategory of $\mathsf{H^\#\!(A)}$ of all 0-complexes, so that $\text{Hom}_\mathsf{A}(X,Y)=\text{Hom}_\mathsf{H^\#\!(A)}(X[0]^\bullet,Y[0]^\bullet)$ for all $X,Y\in\text{obj}(\mathsf{A})$.
	
	Our goal is to prove that, given any two 0-complexes $X[0]^\bullet, Y[0]^\bullet$, the group homomorphism
	\[
	F:\text{Hom}_\mathsf{H^\#\!(A)}(X[0]^\bullet,Y[0]^\bullet)\rightarrow\text{Hom}_\mathsf{D^\#\!(A)}(X[0]^\bullet,Y[0]^\bullet),\,f^\bullet\mapsto [(\text{id}_{X[0]}^\bullet,f^\bullet)]
	\]
	is bijective, where $\text{Hom}_\mathsf{D^\#\!(A)}(X[0]^\bullet,Y[0]^\bullet)=\text{Hom}_\mathsf{D^\#\!(A)}(\mathsf{L}(X[0]^\bullet),\mathsf{L}(Y[0]^\bullet))$. We claim that its inverse is given by
	\[
	F^{-1}:\text{Hom}_\mathsf{D^\#\!(A)}(X[0]^\bullet,Y[0]^\bullet)\rightarrow\text{Hom}_\mathsf{H^\#\!(A)}(X[0]^\bullet,Y[0]^\bullet),\; [(s^\bullet,f^\bullet)]\mapsto f^\bullet\;.
	\]
	Obviously, $F^{-1}\circ F$ is the identity on $\text{Hom}_\mathsf{H^\#\!(A)}(X[0]^\bullet,Y[0]^\bullet)$. Conversely, given a $\varphi=[(s^\bullet,f^\bullet)]\in\text{Hom}_\mathsf{D^\#\!(A)}(X[0]^\bullet,Y[0]^\bullet)$, let $\phi\coloneqq (F\circ F^{-1})(\varphi)=[(\text{id}_{X[0]}^\bullet,g^\bullet)]\in\text{Hom}_\mathsf{D^\#\!(A)}(X[0]^\bullet,Y[0]^\bullet)$ for $g^\bullet$ in $\mathsf{H^\#\!(A)}$ consisting of $g^0\coloneqq H^0(f^\bullet)\circ H^0(s^\bullet)^{-1}:X\cong H^0(X[0]^\bullet)\rightarrow H^0(Y[0]^\bullet)\cong Y$ in $H^\#\!(\mathsf{A})$. By the equivalence relation for roofs, $\phi=\varphi$ in $\mathsf{D^\#\!(A)}$ if and only if we can make
	\[
	\begin{tikzcd}
		& & V^\bullet\arrow[dl, "r^\bullet"']\arrow[dr, "h^\bullet"] & & \\
		& Z^\bullet\arrow[dl, "s^\bullet"'] & & X[0]^\bullet\arrow[dlll, "{\text{id}_{X[0]}^\bullet}"]\arrow[dr, "g^\bullet"] & \\
		X[0]^\bullet & & & & Y[0]^\bullet\arrow[from=ulll, crossing over, "f^\bullet"'] 
	\end{tikzcd}
	\]
	commute through a suitable topping roof in $\mathsf{D^\#\!(A)}$. This is given by $V^{-n}\coloneqq Z^{-n}$, $V^0\coloneqq\ker(d_Z^0)$, $V^n\coloneqq 0$ for $n>0$ (with induced differential), and the natural embedding $r^\bullet:V^\bullet\rightarrow Z^\bullet$ (a quasi-isomorphism because equal to the identity in cohomology in non-positive degree), and $h^\bullet$ trivial except for $h^0\coloneqq H^0(s^\bullet)\circ q: \ker(d_Z^0)\rightarrow H^0(Z^\bullet)\rightarrow X$. The commutativity is easily verified: the only non-trivial equality for the left triangle is $h^0(z)=H^0(s^\bullet)\langle z\rangle=s^0(z)=(r^0\circ s^0)(z)$ for all $z\in V^0$; for the right one, $(g^0\circ h^0)(z)=H^0(f^\bullet)\langle z\rangle=f^0(z)=(f^0\circ r^0)(z)$. Therefore, both roofs represent $\varphi=\phi$, meaning that $\mathsf{F}\circ\mathsf{F}^{-1}$ is the identity on $\text{Hom}_\mathsf{D^\#\!(A)}(X[0]^\bullet,Y[0]^\bullet)$ and thus $\mathsf{L}$ is fully faithful. Then $\mathsf{J}'$ is indeed fully faithful (as composition of fully faithful functors).
	
	Finally, we prove that any $H^0$-complex $Z^\bullet$ is isomorphic to some 0-complex $X[0]^\bullet$ within $\mathsf{D_0^\#\!(A)}$. The obvious candidate is its cohomological counterpart $X[0]^\bullet\coloneqq\mathsf{J}'(H^0(Z^\bullet))=(...\rightarrow 0\rightarrow H^0(Z^\bullet)\rightarrow 0\rightarrow...)$ (a sensible choice by definition of $H^0$-complex). The isomorphism is precisely given by the topping roof in the diagram above, where $h^\bullet$ is a quasi-isomorphism because $r^\bullet$, $s^\bullet$ and $\text{id}_{X[0]}^\bullet$ are. Hence, if restricted to its image, $\mathsf{J}'$ yields an equivalence between $\mathsf{A}$ and the full subcategory $\mathsf{D_0^\#\!(A)}$. 
\end{proof}

Therefore, we will sometimes suppress the functor $\mathsf{J}'$ and simply interpret $\mathsf{A}\subset\mathsf{D^\#\!(A)}$.

\subsection{Ext-abelian groups}\label{ch2.6}

\begin{Def}
	Let $\mathsf{A}$ be an abelian category, $X, Y\in\text{obj}(\mathsf{A})$, $k\in\mathbb{Z}$ any. Then the $k$-th \textbf{Ext-abelian group}\index{Ext-abelian group} is given by
	\begin{equation}\label{Extgroup}
		\text{Ext}_\mathsf{A}^k(X,Y)\coloneqq \text{Hom}_{\mathsf{D(A)}}(X[0]^\bullet,Y[k]^\bullet)\,,
	\end{equation}
	$\cong\text{Hom}_\mathsf{D(A)}(X[l]^\bullet,Y[k+l]^\bullet)\in\text{obj}(\mathsf{Ab})$ for any $l\in\mathbb{Z}$, after applying the translation functor $T^l:\mathsf{D(A)}\rightarrow\mathsf{D(A)}$ from Definition \ref{translfunc} --- an equivalence --- thus bijective on Hom-groups. 
	
	The composition law for Hom-groups then induces the bilinear \textbf{Yoneda product}\index{Yoneda product}\footnote{If $X=Y=Z$, then $\diamond$ gives the vector space $\bigoplus_{n\in\mathbb{Z}}\text{Ext}_\mathsf{A}^n(X,X)$ a structure of graded algebra.}
	\begin{equation}
		\diamond:\text{Ext}_\mathsf{A}^k(X,Y)\times\text{Ext}_\mathsf{A}^l(Y,Z)\rightarrow\text{Ext}_\mathsf{A}^{k+l}(X,Z)
	\end{equation}
	(independent of the choice of translation, by above). Since $\mathsf{D(A)}$ is additive, the bifunctor $\text{Hom}_\mathsf{D(A)}$ actually yields bifunctors $\text{Ext}_\mathsf{A}^k:\mathsf{A}^\text{opp}\times\mathsf{A}\rightarrow\mathsf{Ab}$, so that $\text{Ext}_\mathsf{A}^k(X,\square)$ is covariant and $\text{Ext}_\mathsf{A}^k(\square,Y)$ contravariant.
\end{Def}

\begin{Rem}\label{explicitExtmorph}
	In case $k>0$, we can explicitly construct specific morphisms in the $k$-th Ext-group as follows: for the \textit{acyclic}
	\begin{equation}\label{Extcomplex}
		\mkern-12mu K^\bullet\!=\!\big(...\!\rightarrow 0\rightarrow K^{-k}\!\coloneqq\! Y\rightarrow K^{-k+1}\rightarrow...\rightarrow K^0\rightarrow K^1\!\coloneqq\! X\rightarrow 0\rightarrow\!...\big)\,,
	\end{equation}
	we let $c_{K^\bullet}\in\text{Ext}_\mathsf{A}^k(X,Y)$ be the morphism represented by the roof
	\begin{equation}\label{Extroof}
		\begin{tikzcd}
			& \tilde{K}^\bullet\arrow[dl, "s^\bullet"']\arrow[dr, "f^\bullet"] & \\
			X[0]^\bullet & & Y[k]^\bullet
		\end{tikzcd}\quad,
	\end{equation}
	where $\tilde{K}^1\coloneqq 0$ and otherwise $\tilde{K}^n\coloneqq K^n$, $s^0\coloneqq d_K^0:K^0\rightarrow K^1=X$ (only non-trivial in $s^\bullet$) and $f^{-k}\coloneqq\text{id}_Y:K^{-k}=Y\rightarrow Y$ (only non-trivial in $f^\bullet$). Moreover, given another acyclic
	\[
	L^\bullet=\big(...\rightarrow 0\rightarrow L^{-l}\coloneqq Z\rightarrow L^{-l+1}\rightarrow...\rightarrow L^0\rightarrow L^1\coloneqq Y\rightarrow 0\rightarrow...\big)
	\]
	for $l>0$, we can chain $L^\bullet$ and $K^\bullet$ together to get the acyclic complex
	\begin{small}\begin{equation}\label{astcomplex}
			L^\bullet\ast K^\bullet=\big(...\rightarrow 0\rightarrow \mkern-16mu \underbrace{L^{-l}\!=\!Z}_{=(L^\bullet\ast K^\bullet)^{-k-l}}\mkern-16mu\rightarrow \!...\!\rightarrow L^0\!\xrightarrow[d_K^{-k}\circ d_L^0]{}\! K^{-k+1}\!\rightarrow\! ...\!\rightarrow \mkern-12mu\underbrace{K^1\!=\!X}_{=(L^\bullet\ast K^\bullet)^1}\mkern-12mu\rightarrow 0\rightarrow...\big)\,,
	\end{equation}\end{small}
	\!\!whose associated morphism in $\text{Ext}_\mathsf{A}^{k+l}(X,Z)$ is better described in the last item of the following theorem.
\end{Rem}

\begin{Thm}\label{Extprop}
	Let $\mathsf{A}$ be an abelian category, $X,Y\in\textup{obj}(\mathsf{A})$ fixed. The \textup{Ext}-groups have the following properties:
	\renewcommand{\theenumi}{\alph{enumi}}
	\begin{enumerate}[leftmargin=0.5cm]
		\item Short exact sequences $0\rightarrow X\rightarrow X'\rightarrow X''\rightarrow 0$ and $0\rightarrow Y\rightarrow Y'\rightarrow Y''\rightarrow 0$ in $\mathsf{A}$ induce respectively long exact sequences
		\[
		...\rightarrow \textup{Ext}_\mathsf{A}^k(X'',Y)\rightarrow \textup{Ext}_\mathsf{A}^k(X',Y)\rightarrow \textup{Ext}_\mathsf{A}^k(X,Y)\rightarrow \textup{Ext}_\mathsf{A}^{k+1}(X'',Y)\rightarrow...\;,
		\]
		\[
		...\rightarrow \textup{Ext}_\mathsf{A}^k(X,Y)\rightarrow \textup{Ext}_\mathsf{A}^k(X,Y')\rightarrow \textup{Ext}_\mathsf{A}^k(X,Y'')\rightarrow \textup{Ext}_\mathsf{A}^{k+1}(X,Y)\rightarrow...\;.
		\]
		
		\item For $k<0$, $\textup{Ext}_\mathsf{A}^k(X,Y)=\{0\}$.
		
		\item For $k=0$ holds $\textup{Ext}_\mathsf{A}^0(X,Y)\cong \textup{Hom}_\mathsf{A}(X,Y)$.
		
		\item For $k>0$ and any $f\in\textup{Ext}_\mathsf{A}^k(X,Y)$ there exists instead some acyclic complex $K^\bullet$ like in \eqref{Extcomplex} such that $f=c_{K^\bullet}$. Given also $c_{L^\bullet}\in\textup{Ext}_\mathsf{A}^l(Y,Z)$, it holds $c_{L^\bullet\ast K^\bullet}=c_{L^\bullet}\diamond c_{K^\bullet}$.
	\end{enumerate}
\end{Thm}

\begin{proof}
	\ \vspace*{0.0cm}
	\renewcommand{\theenumi}{\alph{enumi}}
	\begin{enumerate}[leftmargin=0.5cm]
		\item Regard the provided short exact sequences as distinguished triangles $X^\bullet\rightarrow X'^\bullet\rightarrow X''^\bullet\rightarrow X[1]^\bullet$ respectively $Y^\bullet\rightarrow Y'^\bullet\rightarrow Y''^\bullet\rightarrow Y[1]^\bullet$ $\mathsf{D^b}(\mathsf{A})$ (cf. Proposition \ref{shortexactaretriangles}). Then we wish for the Hom-analogue of Theorem \ref{longexactcohomtriangles}. The proof requires some knowledge of triangulated categories, of which the derived category $\mathsf{D^\#\!(A)}$ constitutes a notable example (see Corollary \ref{D(A)istriangulated}). We will come back to this in due time...
		
		\item For $k>0$, consider the morphism $\varphi=[(s^\bullet,f^\bullet)]\in\text{Ext}_\mathsf{A}^{-k}(X^\bullet,Y^\bullet)=$ $\text{Hom}_\mathsf{D(A)}(X[0]^\bullet,Y[-k]^\bullet)$ represented by the solid line roof
		\begin{equation}\label{commutingroofs}
			\begin{tikzcd}
				& X'^\bullet\arrow[dl, "s^\bullet"']\arrow[dr, "f^\bullet"] & \\
				X[0]^\bullet & & Y[-k]^\bullet \\
				& \tilde{X}^\bullet\arrow[ul, dashed, "t^\bullet"]\arrow[uu, dashed, "r^\bullet"]\arrow[ur, dashed, "0"'] 
			\end{tikzcd}\quad.
		\end{equation}
		Define a new subcomplex $\tilde{X}^\bullet\xhookrightarrow{r^\bullet}X'^\bullet$ by $\tilde{X}^j\coloneqq X'^j$ if $j<k-1$, $\tilde{X}^{k-1}\coloneqq\ker(d_{X'}^{k-1})$ and $\tilde{X}^j\coloneqq 0$ otherwise, with differential induced by that on $X'^\bullet$. Also set $t^0\coloneqq s^0$ as only non-trivial term of $t^\bullet$. These fit into the dashed part of the diagram, which then clearly commutes (in particular, $s^0\circ r^0=s^0=t^0$ and $f^k\circ r^k=0$ due to $\tilde{X}^k=0$). Now, since by definition $H^n(s^\bullet)$ is always an isomorphism , we have $H^n(X'^\bullet)=0$ except for $H^0(X'^\bullet)\cong X$. This construction forcedly implies that both $r^\bullet$ and $t^\bullet$ are quasi-isomorphisms. We conclude that the lower dashed roof is a valid representative of $\varphi$... but it is the zero roof, hence $\varphi=0$ and, by its arbitrariness, $\text{Ext}_\mathsf{A}^{-k}(X^\bullet,Y^\bullet)=\{0\}$ as desired.
		
		\item Now let $k=0$. In the proof of Proposition \ref{AinD(A)}, we showed that the functor $\mathsf{J}'\!:\!\mathsf{A}\rightarrow\mathsf{D(A)}$ is fully faithful, so that $\text{Hom}_\mathsf{A}(X,Y)\cong\text{Hom}_\mathsf{D(A)}(X[0]^\bullet,Y[0]^\bullet)$ $=\text{Ext}_\mathsf{A}^0(X^\bullet,Y^\bullet)$.
		
		\item The proof concerning the statement about $k>0$ is more involved: it roughly proceeds by successive refinement of a starting roof to versions resembling more and more the desired diagram \eqref{Extroof}, where their compatibility at each step is guaranteed by commutative diagrams of the form \eqref{commutingroofs}. We refer the interested reader to part c) of the proof of \cite[Theorem III.5.5]{[GM03]}. 
		\newline For the behaviour under product, suppose $c_{K^\bullet}=[(s^\bullet,f^\bullet)]\in\text{Ext}_\mathsf{A}^k(X,Y)$, $c_{L^\bullet}=[(r^\bullet,g^\bullet)]\in\text{Ext}_\mathsf{A}^l(Y,Z)$ and $c_{L^\bullet\ast K^\bullet}=[(t^\bullet,h^\bullet)]\in\text{Ext}_\mathsf{A}^{k+l}(X,Z)$, represented by roofs as in \eqref{Extroof}. Consider the (Yoneda) product rule for roofs:
		\[
		\begin{tikzcd}[column sep=0.5cm]
			& & \widetilde{L^\bullet\ast K^\bullet}\arrow[dl, dashed, "\bar{t}^\bullet"]\arrow[dr, dashed, "\bar{h}^\bullet"']\arrow[ddll, out=195, in=68,  "t^\bullet"']\arrow[ddrr, out=-15, in = 120, "h^\bullet"] & & \\
			& \tilde{K}^\bullet\arrow[dl, "s^\bullet"]\arrow[dr, "f^\bullet"'] & & \tilde{L}[k]^\bullet\arrow[dl, "{r[k]^\bullet}"]\arrow[dr, "{g[k]^\bullet}"'] & \\
			X[0]^\bullet & & Y[k]^\bullet & & Z[k+l]^\bullet
		\end{tikzcd}\quad.
		\]
		With reference to the chained complex \eqref{astcomplex}, define the dashed arrows by $\bar{h}^i\coloneqq\text{id}_{L^{i+1}}$ for $-k-l\leq i\leq -k$, respectively $\bar{t}^i\coloneqq\text{id}_{K}^i$ for $-k+1\leq i\leq 0$ and $\bar{t}^0\coloneqq d_L^0$ (otherwise zero). Indeed we have $s^\bullet\circ\bar{t}^\bullet=t^\bullet$, $g[k]^\bullet\circ\bar{h}^\bullet= h^\bullet$ and $f^\bullet\circ\bar{t}^\bullet=r[k]^\bullet\circ\bar{h}^\bullet$, proving that $c_{L^\bullet}\diamond c_{K^\bullet}=c_{L^\bullet\ast K^\bullet}$.   
	\end{enumerate}\vspace*{-0.5cm}
\end{proof}

Here are a couple of definitions capturing the richness of morphisms in a derived category. They are provided as an additional tool for analysis, but won't actually play a role in the rest of the work.

\begin{Def}
	Let $\mathsf{A}$ be an abelian category. Then we define its \textbf{homological dimension}\index{homological dimension} to be\footnote{Observe that in the ``worst'' case, $\text{dim}_h(\mathsf{A})=0$ because by Theorem \ref{Extprop}c. $\text{Ext}_\mathsf{A}^0(X,X)=\text{Hom}_\mathsf{A}(X,X)\ni\text{id}_X$ for any $X\in\text{obj}(\mathsf{A})$. Similarly for $\text{dim}_p(X)$, $\text{dim}_i(X)$.} 
	\begin{equation}\label{homologicaldimension}
		\text{dim}_h(\mathsf{A})\coloneqq\text{max}\big(\{k\in\mathbb{N}_0\mid \exists X,Y\in\text{obj}(\mathsf{A})\text{ s.t. }\text{Ext}_\mathsf{A}^k(X,Y)\neq\{0\}\}\big)
	\end{equation}
	($=\infty$ if no such $k$ exists). Similarly, the \textbf{projective/injective homological dimensions}\index{homological dimension!projective/injective} of a fixed object $X\in\text{obj}(\mathsf{A})$ are 
	\begin{equation}
		\begin{aligned}
			\text{dim}_p(X)&\coloneqq\text{sup}\big(\{k\in\mathbb{N}_0\mid \exists Y\in\text{obj}(\mathsf{A})\text{ s.t. }\text{Ext}_\mathsf{A}^k(X,Y)\neq\{0\}\}\big) \\
			\text{dim}_i(X)&\coloneqq\text{sup}\big(\{k\in\mathbb{N}_0\mid \exists Y\in\text{obj}(\mathsf{A})\text{ s.t. }\text{Ext}_\mathsf{A}^k(Y,X)\neq\{0\}\}\big)
		\end{aligned}
	\end{equation}
	(again possibly $=\infty$).    
\end{Def}

\begin{Pro}\label{injprojExt}
	Let $\mathsf{A}$ be an abelian category. The following conditions about homological dimensions are equivalent:
	\renewcommand{\theenumi}{\roman{enumi}}
	\begin{enumerate}[leftmargin=0.5cm]
		\item $\textup{dim}_h(\mathsf{A})=0 \xLeftrightarrow{(1)} \textup{Ext}_\mathsf{A}^1(X,Y)\!=\!\{0\}$ $\forall X,Y\!\in\!\textup{obj}(\mathsf{A}) \xLeftrightarrow{(2)} \mathsf{A}$ is semisimple.
		
		\item $\textup{dim}_p(X)=0 \xLeftrightarrow{(1)} \textup{Ext}_\mathsf{A}^1(X,Y)=\{0\}$ $\forall Y\!\in\!\textup{obj}(\mathsf{A}) \xLeftrightarrow{(2)} X$ is projective.
		
		\item $\textup{dim}_i(X)=0 \xLeftrightarrow{(1)} \textup{Ext}_\mathsf{A}^1(Y,X)=\{0\}$ $\forall Y\!\in\!\textup{obj}(\mathsf{A}) \xLeftrightarrow{(2)} X$ is injective.
	\end{enumerate}
\end{Pro}

\begin{proof}
	\ \vspace*{0.0cm}
	\renewcommand{\theenumi}{\roman{enumi}}
	\begin{enumerate}[leftmargin=0.5cm]
		\item The implication $\xRightarrow{(1)}$ is obvious. For $\xLeftarrow{(1)}$, consider $k>1$ and any $f=c_{K^\bullet}\in\text{Ext}_\mathsf{A}^k(X,Y)$ as constructed in Remark \ref{explicitExtmorph}. Then the two complexes
		\begin{align*}
			\mkern-12mu K_1^\bullet\!&\coloneqq\!\big(...\!\rightarrow 0\rightarrow K^{-k}\!\coloneqq\! Y\rightarrow...\rightarrow K^{-1}\rightarrow Z\!\coloneqq\!\text{im}(d_K^{-1})\rightarrow 0\rightarrow\!...\big)\,, \\
			K_2^\bullet\!&\coloneqq\!\big(...\!\rightarrow 0\rightarrow \text{im}(d_K^{-1})\rightarrow K^0\rightarrow K^1\coloneqq X\rightarrow 0\rightarrow ...\big)
		\end{align*}
		yield $K_1^\bullet\ast K_2^\bullet=K^\bullet$, hence $f=c_{K_1^\bullet}\diamond c_{K_2^\bullet}$ by Theorem \ref{Extprop}d. But $c_{K_2^\bullet}$ $\in\text{Ext}_\mathsf{A}^1(X,Z)=\{0\}$, hence $f=0$ and $\text{Ext}_\mathsf{A}^k(X,Y)=\{0\}$.
		\newline For $\xRightarrow{(2)}$, consider the short exact sequence $0\rightarrow Y\xrightarrow{f} Z\xrightarrow{g} X\rightarrow 0$, which by Theorem \ref{Extprop}a.\&b. produces the (portion of long) exact sequence 
		\vspace*{-0.2cm}
		\[
		\{0\}\rightarrow\text{Hom}_\mathsf{A}(X,Y)\xrightarrow{f_*}\text{Hom}_\mathsf{A}(X,Z)\xrightarrow{g_*}\text{Hom}_\mathsf{A}(X,X)\rightarrow\text{Ext}_\mathsf{A}^1(X,Y)\!=\!\{0\}.
		\]
		Then there exists some $h:X\rightarrow Z$ such that $g_*(h)=g\circ h=\text{id}_X$, also factorizing through the canonical isomorphism $k:X\rightarrow\text{im}(h)$ and inclusion $\iota:\text{im}(h)\rightarrow Z$. A straightforward application of the Five-Lemma tells us that $0\rightarrow Y\xrightarrow{(\text{id}_Y,0)} Y\oplus\text{im}(h)\xrightarrow{0+k^{-1}} X\cong\text{im}(h)\rightarrow 0$ is isomorphic to the original sequence (with middle isomorphism $f+\iota:Y\oplus\text{im}(h)\rightarrow Z$ making the associated diagram commute). Consequently, any short exact sequence splits, precisely the definition of $\mathsf{A}$ being semisimple (cf. Remark \ref{semisimple}). Conversely, any morphism in $\text{Ext}_\mathsf{A}^1(X,Y)$ is built out of a triple $K^\bullet=(0\rightarrow Y\xrightarrow{f} K^0\xrightarrow{g} X\rightarrow 0)$, splitting by assumption, which can ultimately be shown to yield $c_{K^\bullet}=0$, thus the implication $\xLeftarrow{(2)}$.
		
		\item $\xRightarrow{(1)}$ is clear, and $\xLeftarrow{(1)}$ is shown like in i. through the splitting $\text{Ext}_\mathsf{A}^k(X,Y)=\text{Ext}_\mathsf{A}^{k-1}(\text{im}(d_K^{-1}),Y)\diamond\text{Ext}_\mathsf{A}^1(X,\text{im}(d_K^{-1}))=0$ for all $k\geq 1$.
		\newline Regarding $\xLeftrightarrow{(2)}$, the restriction of $\xLeftrightarrow{(2)}$ in i. implies that given some fixed $X\in\text{obj}(\mathsf{A})$, $\text{Ext}_\mathsf{A}^1(X,Y)=\{0\}$ for all $Y\in\text{obj}(\mathsf{A})$ if and only if any short exact sequence with third object $X$ splits. Recall also the projectivity condition of $X$, diagram \eqref{injprojobjdiag},\vspace*{-0.5cm} 
		\[
		\begin{tikzcd}
			& & & X\arrow[dl, dashed, "h"']\arrow[d, "f"] & \\
			\Big(0\arrow[r] & Y\arrow[r, hook] & \Big)\;Z\arrow[r, two heads, "g"'] & U\arrow[r] & 0
		\end{tikzcd}\vspace*{-0.3cm}
		\]
		(solid arrows are given). For $\xRightarrow{(2)}$, let $Z\times_U X\!\coloneqq\!\{(z,x)\in Z\!\times\! X\mid g(z)\!=\!f(x)\}$. By surjectivity of the projection $p_2:Z\times_U X\rightarrow X$ and assumed splitting of the arising short exact sequence, there exists an inclusion $i_2:X\rightarrow Z\times_U X$. Then $h\coloneqq p_1\circ i_2:X\rightarrow Z$ makes the diagram commute, proving that $X$ is projective. Conversely for $\xLeftarrow{(2)}$, replacing above $U=X$ and $f=\text{id}_X$, there exists by assumption an $h:X\rightarrow Z$ such that $g\circ h=\text{id}_X$, showing that any short exact sequence $0\rightarrow Y\hookrightarrow Z \twoheadrightarrow X\rightarrow 0$ is isomorphic to a splitting one with middle term $Y\oplus\text{im}(h)$. As reasoned above, this is equivalent to $\text{Ext}_\mathsf{A}^1(X,Y)=\{0\}$ for all $Y\in\text{obj}(\mathsf{A})$.
		
		\item This is proved similarly to the previous part. 
	\end{enumerate}\vspace*{-0.5cm}
\end{proof}

Here is another useful result involving projective and injective objects.

\begin{Lem}
	Let $\mathsf{A}$ be an abelian category, $X,X'\in\textup{obj}(\mathsf{A})$.
	\renewcommand{\theenumi}{\roman{enumi}}
	\begin{enumerate}[leftmargin=0.5cm]
		\item Given an acyclic complex
		\[
		P^\bullet\coloneqq\big(...\rightarrow 0\rightarrow X'\rightarrow P^{-k}\rightarrow...\rightarrow P^0\rightarrow X\rightarrow 0\rightarrow...\big)\,,
		\]
		with projective $P^i\in\textup{obj}(\mathsf{A})$, it holds $\textup{dim}_p(X')=\textup{max}(\{\textup{dim}_p(X)-k-1,0\})$.
		
		\item Given an acyclic complex
		\[
		I^\bullet\coloneqq\big(...\rightarrow 0\rightarrow X\rightarrow I^0\rightarrow...\rightarrow I^k\rightarrow X'\rightarrow 0\rightarrow...\big)\,,
		\]
		with injective $I^i\in\textup{obj}(\mathsf{A})$, it holds $\textup{dim}_i(X')=\textup{max}(\{\textup{dim}_i(X)-k-1,0\})$. 
	\end{enumerate}
\end{Lem}

\begin{proof}
	We prove the statement in the projective case i. (the injective scenario can be shown with an analogous procedure). Given any $Y\in\text{obj}(\mathsf{A})$, we construct by induction on $k$ a map $\delta:\text{Ext}_\mathsf{A}^d(X',Y)\rightarrow\text{Ext}_\mathsf{A}^{d+k+1}(X,Y)$ which is surjective for $d=0$ and an isomorphism for $d\geq 1$.
	
	For $k=0$, $P^\bullet=(0\rightarrow X'\rightarrow P^0\rightarrow X\rightarrow 0)$ induces by Theorem \ref{Extprop}a. the portion of long exact sequence
	\[
	\text{Ext}_\mathsf{A}^d(P^0,Y)\rightarrow\text{Ext}_\mathsf{A}^d(X',Y)\xrightarrow{\delta}\mkern-15mu\rightarrow\text{Ext}_\mathsf{A}^{d+1}(X,Y)\rightarrow\text{Ext}_\mathsf{A}^{d+1}(P^0,Y)=\{0\}\,,
	\]
	so that $\delta$ is readily surjective, and moreover an isomorphism when $d>0$ (because $\text{Ext}_\mathsf{A}^d(P^0,Y)=\{0\}$ by Proposition \ref{injprojExt}ii., $P^0$ being projective; and similarly for the third term). For the step $k-1\rightarrow k$, split $P^\bullet=P_1^\bullet\ast P_2^\bullet$ into $P_1^\bullet\coloneqq\big(0\rightarrow X'\rightarrow P^{-k}\rightarrow X''\coloneqq\text{im}(d_P^{-k})\rightarrow 0\big)$ and $P_2^\bullet\coloneqq\big(0\rightarrow X''\rightarrow\break P^{-k+1}\rightarrow...\rightarrow P^0\rightarrow X\rightarrow 0\big)$. The induction hypothesis for $P_2^\bullet$ gives $\delta_2:\text{Ext}_\mathsf{A}^d(X'',Y)\rightarrow\text{Ext}_\mathsf{A}^{d+k}(X,Y)$, while $\delta_1:\text{Ext}_\mathsf{A}^d(X',Y)\rightarrow\text{Ext}_\mathsf{A}^{d+1}(X'',Y)$ is deduced from $P_1^\bullet$ as in the base case. Then $\delta\coloneqq\delta_2[1]\circ\delta_1$ inherits the desired properties.
	
	We conclude that if $d\coloneqq\text{dim}_p(X')\geq 1$, meaning there exists some $Y\in\text{obj}(\mathsf{A})$ such that $\text{Ext}_\mathsf{A}^{\text{dim}_p(X')+k+1}(X,Y)\cong\text{Ext}_\mathsf{A}^{\text{dim}_p(X')}(X',Y)\neq\{0\}$, then $\text{dim}_p(X)=\text{dim}_p(X')+k+1$, whereas $\text{dim}_p(X')= 0$ implies that $\text{Ext}_\mathsf{A}^1(X',Y)=\{0\}$ for all $Y$, thus $\text{Ext}_\mathsf{A}^{k+2}(X,Y)=\{0\}$ and $\text{dim}_p(X)\leq k+1$. All together, $\text{dim}_p(X')=\text{max}(\{\text{dim}_p(X)-k-1,0\})$. 
\end{proof}

\begin{Cor}
	Let $\mathsf{A}$ be an abelian category, $X\in\textup{obj}(\mathsf{A})$.
	\begin{itemize}[leftmargin=0.5cm]
		\item If $\mathsf{A}$ has enough projectives \textup(see Definition \ref{injprojobj}\textup), then $\textup{dim}_p(X)\leq k$ if and only if there exists a projective resolution of $X$ of length $\leq k+1$.
		
		\item If $\mathsf{A}$ has enough injectives, then $\textup{dim}_i(X)\leq k$ if and only if there exists an injective resolution of $X$ of length $k+1$. 
	\end{itemize} 
\end{Cor}

We also mention a lemma which will be crucial for our analysis of elliptic curves in Chapter 7. It refers to \cite[Exercise 13.13]{[KS06]}.

\begin{Lem}\label{whendimhleq1}
	Let $\mathsf{A}$ be an abelian category with $\textup{dim}_h(\mathsf{A})\leq 1$. Then any $X^\bullet\in\textup{obj}(\mathsf{D^b(A)})$ can be written as $X^\bullet\cong\bigoplus_{n\in\mathbb{Z}} H^n(X^\bullet)[-n]^\bullet\in\textup{obj}(\mathsf{D^b(A)})$.
\end{Lem}

\newpage

\section{Derived functors}
\thispagestyle{plain}

\subsection{The construction of derived functors}\label{ch3.1}

In the present chapter we wish to relate derived categories of abelian categories through functors preserving some of their desirable features, such as exactness of sequences, here understood in terms of distinguished triangles:

\begin{Def}\label{exactderivedfunc}
	Let $\mathsf{A}, \mathsf{B}$ be abelian categories. A functor $\mathsf{D(A)}\rightarrow\mathsf{D(B)}$ between their derived categories is \textbf{exact}\index{functor!exact (between derived categories)} if it maps distinguished triangles in $\mathsf{D(A)}$ to distinguished triangles in $\mathsf{D(B)}$, then of the form \eqref{distinguitriangle} or \eqref{conetriangle}.
\end{Def}

We construct these particular ``derived functors'' starting from exact functors of abelian categories, and their extension to complexes.

\begin{Def}
	Let $\mathsf{A}, \mathsf{B}$ be abelian categories. A functor $\mathsf{F}:\mathsf{A}\rightarrow\mathsf{B}$ extends to $\mathsf{Kom^\#\!(F)}:\mathsf{Kom^\#\!(A)}\rightarrow\mathsf{Kom^\#\!(B)}$ (for $\#=\emptyset,+,-,\mathsf{b}$) by acting componentwise: 
	\begin{itemize}[leftmargin=0.5cm]
		\renewcommand{\labelitemi}{\textendash}
		\item for any $X^\bullet\in\text{obj}(\mathsf{Kom^\#\!(A)})$, $\mathsf{F}(X^\bullet)^\bullet\equiv\mathsf{F}(X)^\bullet\in\text{obj}(\mathsf{Kom^\#\!(B)})$ is given by $\mathsf{F}(X)^n\coloneqq\mathsf{F}(X^n)$ with differential $d_{\mathsf{F}(X)}^n\coloneqq \mathsf{F}(d_X^n):\mathsf{F}(X^n)\rightarrow\mathsf{F}(X^{n+1})$, for $n\in\mathbb{Z}$; 
		
		\item for any $f^\bullet:X^\bullet\rightarrow Y^\bullet$ in $\mathsf{Kom^\#\!(A)}$, $\mathsf{F}(f)^\bullet:\mathsf{F}(X)^\bullet\rightarrow\mathsf{F}(Y)^\bullet$ is given by $\mathsf{F}(f)^n\coloneqq \mathsf{F}(f^n)$, $n\in\mathbb{Z}$.
	\end{itemize}
	Since $\mathsf{Kom^\#\!(F)}$ clearly preserves the homotopy relation (one just applies $\mathsf{F}$ to chain homotopies $k^\bullet$; cf. Definition \ref{quasiso}), it restricts to the homotopy functor $\mathsf{H^\#\!(F)}:\mathsf{H^\#\!(A)}\rightarrow\mathsf{H^\#\!(B)}$.
\end{Def}

\begin{Pro}\label{D(F)exact}
	Let $\mathsf{A}, \mathsf{B}$ be abelian categories, $\mathsf{F}:\mathsf{A}\rightarrow\mathsf{B}$ an exact functor. Then $\mathsf{H^\#\!(F)}$ maps quasi-isomorphisms to quasi-isomorphisms, hence induces a functor $\mathsf{D^\#\!(F)}:\mathsf{D^\#\!(A)}\rightarrow\mathsf{D^\#\!(B)}$ on derived categories, which is moreover exact \textup(in the sense of Definition \ref{exactderivedfunc}\textup).
\end{Pro}

\begin{proof}
	Let $X^\bullet\in\text{obj}(\mathsf{Kom^\#\!(A)})$ be acyclic, meaning that $K^n\coloneqq\ker(d_X^n)\cong\text{im}(d_X^{n-1})\subset X^n$ for each $n\in\mathbb{Z}$. Then, by exactness of $\mathsf{F}$ (Definition \ref{additiveexactfunc}), the short exact sequence $0\rightarrow  K^n\xrightarrow{i^n} X^n\xrightarrow{p^n} K^{n+1}\rightarrow 0$ in $\mathsf{A}$ (where $i^\bullet$ is the canonical inclusion and $p^\bullet$ equals $d_X^\bullet$ restricted to its image, so that $d_X^n=i^{n+1}\circ p^n$) is mapped to the short exact sequence $0\rightarrow\mathsf{F}(K^n)\xrightarrow{\mathsf{F}(i^n)} \mathsf{F}(X^n)\xrightarrow{\mathsf{F}(p^n)} \mathsf{F}(K^{n+1})\rightarrow 0$ in $\mathsf{B}$, with $\mathsf{F}(d_X^n)=\mathsf{F}(i^{n+1})\circ\mathsf{F}(p^n)$. Since $\mathsf{F}(i^{n+1})$ is a monomorphism and $\mathsf{F}(p^n)$ an epimorphism, it follows that $\text{im}(\mathsf{F}(d_X^{n-1}))\cong\mathsf{F}(X^{n-1})/\ker(\mathsf{F}(d_X^{n-1}))=\mathsf{F}(X^{n-1})/\ker(\mathsf{F}(p^{n-1}))\cong\text{im}(\mathsf{F}(p^{n-1}))=\mathsf{F}(K^n)$ and $\ker(\mathsf{F}(d_X^n))\cong\mathsf{F}(K^n)$ as well, so that $H^n(\mathsf{F}(X)^\bullet)=\ker(\mathsf{F}(d_X^n))/\text{im}(\mathsf{F}(d_X^{n-1}))\cong\{0\}$. Therefore $\mathsf{F}$ maps acyclic complexes $X^\bullet$ to acyclic complexes $\mathsf{F}(X)^\bullet\in\text{obj}(\mathsf{Kom^\#\!(B)})$.
	
	Now simply recall from Remark \ref{C(f)acyclic} that $f^\bullet:X^\bullet\rightarrow Y^\bullet$ is a quasi-isomorphism in $\mathsf{H^\#\!(A)}$ if and only if its cone $C(f)^\bullet\in\text{obj}(\mathsf{Kom^\#\!(A)})$ is acyclic, which we just saw to imply $\mathsf{F}(C(f)^\bullet)\cong C(\mathsf{F}(f))^\bullet\in\text{obj}(\mathsf{Kom^\#\!(B)})$ acyclic, again equivalent to $\mathsf{F}(f)^\bullet:\mathsf{F}(X)^\bullet\rightarrow\mathsf{F}(Y)^\bullet$ being a quasi-isomorphism in $\mathsf{H^\#\!(B)}$. Consequently, $\mathsf{H^\#\!(F)}$ induces a functor $\mathsf{D^\#\!(F)}$ as claimed. 
	
	Here we just used that $\mathsf{F}$ preserves cones thanks to its additivity: the isomorphism above is $\mathsf{F}(C(f)^\bullet)\cong\mathsf{F}(X)^{\bullet+1}\oplus\mathsf{F}(Y)^\bullet=C(\mathsf{F}(f))^\bullet$. The same is true for cylinders: $\mathsf{F}(Z(f)^\bullet)\cong \mathsf{F}(X)^\bullet\oplus \mathsf{F}(X)^{\bullet+1}\oplus \mathsf{F}(Y)^\bullet=Z(\mathsf{F}(f))^\bullet$. In particular, $\mathsf{D^\#\!(F)}$ maps distinguished triangles in $\mathsf{A}$ --- all those quasi-isomorphic to \eqref{distinguitriangle} obtained from diagram \eqref{distinguidiag} --- to distinguished triangles in $\mathsf{B}$, quasi-isomorphic to the canonical one 
	\[
	\mathsf{F}(X)^\bullet\xrightarrow{\mathsf{F}(i_1)^\bullet} Z(\mathsf{F}(f))^\bullet\xrightarrow{\mathsf{F}(p_{23})^\bullet} C(\mathsf{F}(f))^\bullet\xrightarrow{\mathsf{F}(\delta)[1]^\bullet}\mathsf{F}(X)[1]^\bullet\,.
	\] 
	In other words, $\mathsf{D^\#\!(F)}$ is exact. 
\end{proof}

To construct derived functors we cannot, however, apply exact functors on any kind of complex.

\begin{Def}\label{adaptedclass}
	Let $\mathsf{A}, \mathsf{B}$ be abelian categories. We say that a class of objects $\mathcal{R}\subset\text{obj}(\mathsf{A})$ is \textbf{adapted}\index{adapted class of objects} to a left/right exact functor $\mathsf{F}:\mathsf{A}\rightarrow\mathsf{B}$ if it is stable under finite direct sums, $(R_i\in\mathcal{R}\,\;\forall i\in I\implies\bigoplus_{i\in I}R_i\in\mathcal{R})$, and:
	\begin{itemize}[leftmargin=0.5cm]
		\item in case $\mathsf{F}$ is left exact, if $\mathsf{F}$ maps acyclic complexes in $\mathsf{Kom^+(\mathcal{R})}$ to acyclic complexes in $\mathsf{Kom(B)}$, and for each $X\in\text{obj}(\mathsf{A})$ there exists some $R\supset X$ in~$\mathcal{R}$ (read: any object of $\mathsf{A}$ is a subobject of some object built from $\mathcal{R}$);
		
		\item in case $\mathsf{F}$ is right exact, if $\mathsf{F}$ maps acyclic complexes in $\mathsf{Kom^-(\mathcal{R})}$ to acyclic complexes in $\mathsf{Kom(B)}$, and each $X\in\text{obj}(\mathsf{A})$ is $X=R/R'$ for some $R, R'\in\mathcal{R}$ (read: any object of $\mathsf{A}$ is a quotient object of two objects from $\mathcal{R}$). 
	\end{itemize}
	By the proof of Proposition \ref{D(F)exact}, the condition on acyclic complexes is readily guaranteed as soon as $\mathsf{F}$ is exact (in which case $\mathcal{R}=\mathsf{A}$ is trivially adapted to~it). 
\end{Def}

\begin{Pro}\label{D(R)equivalent}
	Let $\mathsf{A}, \mathsf{B}$ be abelian categories, $\mathcal{R}\subset\textup{obj}(\mathsf{A})$ adapted to a left/right exact functor $\mathsf{F}:\mathsf{A}\rightarrow\mathsf{B}$, and let $S_\mathcal{R}$ be the class of quasi-isomorphisms in $\mathsf{H^+}(\mathcal{R})$ respectively $\mathsf{H^-}(\mathcal{R})$. Then $S_\mathcal{R}$ is localizing and the canonical functors $\mathsf{H^+(\mathcal{R})}[S_\mathcal{R}^{-1}]\rightarrow\mathsf{D^+(A)}$ and $\mathsf{H^-(\mathcal{R})}[S_\mathcal{R}^{-1}]\rightarrow\mathsf{D^-(A)}$ are equivalences.
\end{Pro}

\begin{proof}
	Just focus on the case of $\mathsf{F}$ left exact (the other argument follows analogously). That $S_\mathcal{R}$ is localizing is easy to see: one just needs to retake the steps of Proposition \ref{quasilocalizing}, and notice that any construction made there consisted of successive cones; here, they remain within $\mathcal{R}$ by definition of adapted class of objects.
	
	Now we claim that for any complex $X^\bullet\in\text{obj}(\mathsf{H^+(A)})$ there exist some $Y^\bullet\in\text{obj}(\mathsf{H^+(\mathcal{R})})$ and a quasi-isomorphism $t^\bullet:X^\bullet\rightarrow Y^\bullet$. Assume without loss of generality that $X^n=0$ for $n<0$. Then we determine $Y^n, d_Y^n$ and $t^n$ by induction on $n\in\mathbb{N}$. For $n=0,1$, observe first that by assumption on $\mathcal{R}$ there exists some $R\in\mathcal{R}$ such that $X^0\subset R$; set $Y^0\coloneqq R$ and $t^0$ the associated inclusion. The pushout $Y^0\sqcup_{X^0} X^1\in\text{obj}(\mathsf{A})$ (in $\mathsf{A}$ by abelianity!) fits into the diagram
	\[
	\begin{tikzcd}
		0\arrow[r] & X^0\arrow[r, "d_X^0"]\arrow[d, "t^0"'] & X^1\arrow[d, dashed, "i_2"] & \\
		0\arrow[r] & Y^0\arrow[r, dashed,"i_1"'] & Y^0\sqcup_{X^0} X^1\arrow[r, dashed, "j"'] & Y^1
	\end{tikzcd}\quad,
	\]
	where $Y^0\sqcup_{X^0} X^1$ in turn injects into some $Y^1\in\mathcal{R}$. Define $d_Y^0\coloneqq j\circ i_1$ and $t^1\coloneqq j\circ i_2$. For the induction step $n\rightarrow n+1$, construct the pushout $\text{coker}(d_Y^{n-1})\sqcup_{X^n} X^{n+1}\in\text{obj}(\mathsf{A})$, contained in some $Y^{n+1}\in\mathcal{R}$, and the diagram
	\[
	\begin{tikzcd}
		& X^n\arrow[r, "d_X^n"]\arrow[d, "\bar{t}^n"]\arrow[dl, "t^n"'] & X^{n+1}\arrow[d, dashed, "i_2"] & \\
		Y^n\arrow[r, "p"'] & \text{coker}(d_Y^{n-1})\arrow[r, dashed,"i_1"'] & \text{coker}(d_Y^{n-1})\sqcup_{X^n} X^{n+1}\arrow[r, dashed, "j"'] & Y^{n+1}
	\end{tikzcd}\quad
	\]
	(where $\bar{t}^n\coloneqq p\circ t^n$). Then set $d_Y^{n+1}\coloneqq j\circ i_1\circ p$ and $t^{n+1}\coloneqq j\circ i_2$. This ultimately yields a complex $Y^\bullet\in\text{obj}(\mathsf{H^+(\mathcal{R})})$ and a morphism $t^\bullet:X^\bullet\rightarrow Y^\bullet$. The proof that the latter is actually a quasi-isomorphism is a little more involved; we refer the reader to \cite[subsection III.5.25]{[GM03]}.
	
	The claim now completes the prerequisites for Lemma \ref{derivedsubcat}: $S_\mathcal{R}$ is localizing and condition\footnote{When $\mathsf{F}$ is right exact, one must instead check condition (i), by help of a suitably modified claim.} (ii), namely \big($s^\bullet:\text{obj}(\mathsf{H^+(\mathcal{R})})\ni W^\bullet\rightarrow X^\bullet\in\text{obj}(\mathsf{H^+(A)})$ in $S$ $\implies$ $\exists Y^\bullet\in\text{obj}(\mathsf{H^+(\mathcal{R})})$ $\exists t^\bullet:X^\bullet\rightarrow Y^\bullet$ s.t. $t^\bullet\circ s^\bullet\in S$\big), holds since $t^\bullet\in S$. Consequence is that the functor $\mathsf{H^+(\mathcal{R})}[S_\mathcal{R}^{-1}]\rightarrow\mathsf{H^+(A)}[S^{-1}]=\mathsf{D^+(A)}$ is fully faithful. But the claim above also implies that any complex in $\mathsf{D^+(A)}$ is quasi-isomorphic to one from $\mathsf{H^+(\mathcal{R})}$, so that the functor is actually an equivalence.  
\end{proof}

In practice, it is sometimes convenient to identify classes of objects which are adapted to any left/right exact functor.

\begin{Lem}\label{enoughinjprojobj}
	Assume the abelian category $\mathsf{A}$ contains enough injectives \textup(see Definition \ref{injprojobj}\textup). Then $\mathcal{I}\coloneqq\{X\in\textup{obj}(\mathsf{A})\mid X\text{ is injective}\}\subset\textup{obj}(\mathsf{A})$ is adapted to any left exact functor from $\mathsf{A}$.
	
	Similarly, assume $\mathsf{A}$ contains enough projectives. Then $\mathcal{P}\coloneqq\{X\in\textup{obj}(\mathsf{A})\mid X\text{ is projective}\}\subset\textup{obj}(\mathsf{A})$ is adapted to any right exact functor from $\mathsf{A}$. 
\end{Lem}

\begin{proof}
	We just prove the injective case. Let $\mathsf{B}$ be the target abelian category, $\mathsf{F}:\mathsf{A}\rightarrow\mathsf{B}$ any left exact functor. It is clear from the injectivity diagram \eqref{injprojobjdiag} left that $\mathcal{I}$ is closed under direct sums. Also, the assumption that $\mathsf{A}$ has enough injectives allows us to regard any $X\in\text{obj}(\mathsf{A})$ as a subobject of some $I^0\in\mathcal{I}$ (borrowing the notation of Definition \ref{injprojobj}). 
	
	It remains to check that any acyclic $I^\bullet\in\text{obj}(\mathsf{Kom^+(\mathcal{I})})$ is mapped to an acylic $\mathsf{F}(I^\bullet)\in\text{obj}(\mathsf{Kom^+(B)})$. The zero morphism $0_{I^\bullet}: I^\bullet\rightarrow I^\bullet$ is trivially a quasi-isomorphism in this case. With some additional work (see \cite[subsection III.5.24]{[GM03]}) one can prove this to imply that $0_{I^\bullet}\sim\text{id}_I^\bullet$, and thus, by functoriality of $\mathsf{F}$, that $0_{\mathsf{F}(I^\bullet)}\sim\text{id}_{\mathsf{F}(I)}^\bullet$. But homotopic morphisms induce the same morphisms in cohomology, hence $H^n(\mathsf{F}(I)^\bullet)=H^n(\text{id}_{\mathsf{F}(I)}^\bullet)\langle H^n(\mathsf{F}(I)^\bullet)\rangle = H^n(0_{\mathsf{F}(I)}^\bullet)\langle H^n(\mathsf{F}(I)^\bullet)\rangle \cong 0$ for all $n\in\mathbb{Z}$, meaning that $\mathsf{F}(I)^\bullet$ is acyclic.
\end{proof}

Focusing on the classes of injectives and projectives objects is particularly advantageous because of the following result.

\begin{Pro}\label{injprojderivedarehomotopy}
	Let $\mathsf{A}$ be an abelian category, $\mathcal{I}, \mathcal{P}\subset\textup{obj}(\mathsf{A})$ the classes of injective respectively projective objects defined above. Then the localization functors $\mathsf{L}_\mathcal{I}:\mathsf{H^+(\mathcal{I})}\rightarrow\mathsf{D^+(A)}$ and $\mathsf{L}_\mathcal{P}:\mathsf{H^-(\mathcal{P})}\rightarrow\mathsf{D^-(A)}$ are equivalences onto their images, thus identifying $\mathsf{H^+(\mathcal{I})}$ and $\mathsf{H^-(\mathcal{P})}$ as full subcategories of ${D^+(A)}$ respectively ${D^-(A)}$.
	
	Moreover, if $\mathsf{A}$ has enough injectives/projectives, then the functors are readily equivalences and we can effectively substitute $\mathsf{D^+(A)}$ by $\mathsf{H^+(\mathcal{I})}$ and $\mathsf{D^-(A)}$ by $\mathsf{H^-(\mathcal{P})}$.
\end{Pro} 

\begin{proof}(\textit{Sketch})
	We don't prove this statement in detail (carried in full in \cite[subsections III.5.22--25]{[GM03]}), for some steps are just slight adaptations of the proof of Proposition \ref{D(R)equivalent}, and otherwise just technical. We rather highlight the strategy to adopt in the injective case.
	
	Firstly, one needs to verify that $S_\mathcal{I}\coloneqq\{\text{quasi-isomorphisms in }\mathsf{H^+(\mathcal{I})}\}$ is a localizing class of morphisms. Secondly, one shows that Lemma \ref{derivedsubcat} is applicable by checking the following stronger version of its condition (ii): for $I^\bullet\in\text{obj}(\mathsf{H^+(\mathcal{I})})$, $X^\bullet\in\text{obj}(\mathsf{H^+(A)})$ and a quasi-isomorphism $s^\bullet\in\text{Hom}_\mathsf{H^+(A)}(I^\bullet, X^\bullet)$ there exists some $t^\bullet\in\text{Hom}_\mathsf{H^+(A)}(X^\bullet, I^\bullet)$ such that $t^\bullet\circ s^\bullet\sim\text{id}_I^\bullet$ (a fact we exploited while proving last lemma). This identifies $\mathsf{H^+(\mathcal{I})}[S_\mathcal{I}^{-1}]\subset\mathsf{D^+(A)}$ as a full subcategory. But actually one sees that $\mathsf{H^+(\mathcal{I})}[S_\mathcal{I}^{-1}]=\mathsf{H^+}(\mathcal{I})$ (that is, any morphism in the homotopy category is readily an isomorphism). 
	
	Finally, if $\mathsf{A}$ has enough injectives, one indeed discovers that $\mathsf{H^+(\mathcal{I})}$ is equivalent to the derived category of $\mathsf{A}$: for any $X^\bullet\in\text{obj}(\mathsf{D^+(A)})$ there exist some $I^\bullet\in\text{obj}(\mathsf{H^+(\mathcal{I})})$ and quasi-isomorphism $t^\bullet\in\text{Hom}_\mathsf{D^+(A)}(X^\bullet, I^\bullet)$ (we saw this in Proposition \ref{D(R)equivalent} with $\mathcal{I}$ replaced by the adapted class $\mathcal{R}$).
\end{proof}

Returning to the general framework, with the preparations above at hand, we are finally ready to define derived functors.

\begin{Def}\label{derivedfunc}
	Let $\mathsf{A}, \mathsf{B}$ be abelian categories, $\mathcal{R}\subset\text{obj}(\mathsf{A})$ adapted to a left/right exact functor $\mathsf{F}:\mathsf{A}\rightarrow\mathsf{B}$ and let $S_\mathcal{R}$ be the class of quasi-isomorphisms in $\mathsf{H^+(\mathcal{R})}$ respectively $\mathsf{H^-(\mathcal{R})}$, giving by Proposition \ref{D(R)equivalent} the natural equivalences $\Psi:\mathsf{H^\#\!(\mathcal{R})}[S_\mathcal{R}^{-1}]\rightarrow\mathsf{D^\#\!(A)}$ with quasi-inverses\footnote{Hence by definition, there exist natural isomorphisms $\alpha:\mathsf{Id}_{\mathsf{H^\#\!(\mathcal{R})}[S_\mathcal{R}^{-1}]}\rightarrow\Phi\circ\Psi$ and $\beta:\mathsf{Id}_{\mathsf{D^\#\!(A)}}\rightarrow\Psi\circ\Phi$, thus fulfilling that each $\alpha({X^\bullet})\in\text{Hom}_{\mathsf{H^\#\!(\mathcal{R})}[S_\mathcal{R}^{-1}]}\big(X^\bullet, \Phi(\Psi(X^\bullet))\big)$ for $X^\bullet\in\text{obj}(\mathsf{H^\#\!(\mathcal{R})}[S_\mathcal{R}^{-1}])$ and each $\beta({Y^\bullet})\in\text{Hom}_\mathsf{D^\#\!(A)}\big(Y^\bullet,\Psi(\Phi(Y^\bullet))\big)$ for $Y^\bullet\in\text{obj}(\mathsf{D^\#\!(A)})$ are isomorphisms.} $\Phi:\mathsf{D^\#\!(A)}\rightarrow\mathsf{H^\#\!(\mathcal{R})}[S_\mathcal{R}^{-1}]$. Then:
	\begin{itemize}[leftmargin=0.5cm]
		\item Define $\mathsf{RF}(X^\bullet)^\bullet\equiv\mathsf{RF}(X)^\bullet\in\text{obj}(\mathsf{H^+(B)})$ for any $X^\bullet\in\text{obj}(\mathsf{H^+(\mathcal{R})}[S_\mathcal{R}^{-1}])=\text{obj}(\mathsf{Kom^+(\mathcal{R})})$ by \begin{equation}\label{derivedfunconcomplexes}
			\mathsf{RF}(X)^n\coloneqq\mathsf{F}(X^n)\in\text{obj}(\mathsf{B})\,.
		\end{equation}
		Since $\mathsf{F}$ preserves acyclicity, we get a functor $\bar{\mathsf{F}}:\mathsf{H^+(\mathcal{R})}[S_\mathcal{R}^{-1}]\rightarrow \mathsf{D^+(B)}$ acting like \eqref{derivedfunconcomplexes} on complexes (argue as in the proof of Proposition \ref{D(F)exact}). Together with $\Phi$, this gives the \textbf{right derived functor}\index{derived functor!right} $\mathsf{RF}\coloneqq\bar{\mathsf{F}}\circ\Phi:\mathsf{D^+(A)}\rightarrow\mathsf{D^+(B)}$ of the left exact functor $\mathsf{F}$.
		
		\item Define $\mathsf{LF}(X^\bullet)^\bullet\equiv\mathsf{LF}(X)^\bullet\in\text{obj}(\mathsf{H^-(B)})$ for any $X^\bullet\in\text{obj}(\mathsf{H^-(\mathcal{R})}[S_\mathcal{R}^{-1}])=\text{obj}(\mathsf{Kom^-(\mathcal{R})})$ by
		\begin{equation}
			\mathsf{LF}(X)^n\coloneqq\mathsf{F}(X^n)\in\text{obj}(\mathsf{B})\,.
		\end{equation}
		We get a functor $\bar{\mathsf{F}}:\mathsf{H^-(\mathcal{R})}[S_\mathcal{R}^{-1}]\rightarrow \mathsf{D^-(B)}$, which in turn yields the \textbf{left derived functor}\index{derived functor!left} $\mathsf{LF}\coloneqq\mathsf{F}\circ\Phi:\mathsf{D^-(A)}\rightarrow\mathsf{D^-(B)}$ of the right exact functor~$\mathsf{F}$.
	\end{itemize}
\end{Def}

Since $\Psi$ is canonical, it is quite clear that the derived functors $\mathsf{RF}$ and $\mathsf{LF}$ do not depend on $\Phi$. However, the same is not evident for $\mathcal{R}$. Therefore, we characterize derived functors also through a universal property, much like we did for derived categories.

\begin{Def}\label{universalderivedfunc}
	Let $\mathsf{A}, \mathsf{B}$ be abelian categories. The \textit{right/left derived functor} of an additive left/right exact functor $\mathsf{F}:\mathsf{A}\rightarrow\mathsf{B}$ is a pair $(\mathsf{RF},\varepsilon_\mathsf{F})$ respectively $(\mathsf{LF},\varepsilon_\mathsf{F})$ where $\mathsf{RF}, \mathsf{LF}$ are exact functors and $\varepsilon_\mathsf{F}$ are natural transformations fitting into the diagrams
	\begin{equation}\label{derfuncdiag}
		\mkern-18mu
		\begin{tikzcd}[row sep=1cm, column sep=0.6cm]
			& \mathsf{D^+(A)}\arrow[dr, "\mathsf{RF}"] & \\
			\mathsf{H^+(A)}\arrow[ur, "\mathsf{L_A}"]\arrow[dr, "\mathsf{H^+(F)}"']\arrow[rr, yshift=0.5ex, out=15, in=165, "\mathsf{RF\circ L_A}"]\arrow[rr, yshift=-0.5ex, out=-15, in=195, "\mathsf{L_B\circ H^+(F)}"'] & & \mathsf{D^+(B)} \\
			& \mathsf{H^+(B)}\arrow[ur, "\mathsf{L_B}"']\arrow[uu, shorten=6ex, "\varepsilon_\mathsf{F}"] & 
		\end{tikzcd}\quad,\quad
		\begin{tikzcd}[row sep=1cm, column sep=0.6cm]
			& \mathsf{D^-(A)}\arrow[dr, "\mathsf{LF}"]\arrow[dd, shorten=6ex, "\varepsilon_\mathsf{F}"'] & \\
			\mathsf{H^-(A)}\arrow[ur, "\mathsf{L_A}"]\arrow[dr, "\mathsf{H^-(F)}"']\arrow[rr, yshift=0.5ex, out=15, in=165, "\mathsf{LF\circ L_A}"]\arrow[rr, yshift=-0.5ex, out=-15, in=195, "\mathsf{L_B\circ H^-(F)}"'] & & \mathsf{D^-(B)} \\
			& \mathsf{H^-(B)}\arrow[ur, "\mathsf{L_B}"'] & 
		\end{tikzcd}
	\end{equation}
	and satisfying the following universal property: for any exact functor (as in Definition \ref{exactderivedfunc}) $\mathsf{G}:\mathsf{D^\#\!(A)}$ $\rightarrow\mathsf{D^\#\!(B)}$ and any natural transformation $\varepsilon:\mathsf{L_B\circ H^+(F)}\rightarrow\mathsf{G\circ L_A}$ respectively $\varepsilon:\mathsf{G\circ L_A}\rightarrow\mathsf{L_B\circ H^-(F)}$, there exist unique natural transformations $\eta:\mathsf{RF}\rightarrow\mathsf{G}$ and $\eta:\mathsf{G}\rightarrow\mathsf{LF}$ making the diagrams of functors
	\begin{equation}\label{universalderfuncdiag}
		\mkern-18mu
		\begin{tikzcd}[column sep=0.2cm]
			& \mathsf{L_B\circ H^+(F)}\arrow[dl, "\varepsilon_\mathsf{F}"']\arrow[dr, "\varepsilon"] & \\
			\mathsf{RF\circ L_A}\arrow[rr, dashed, "\eta\circ\mathsf{L_A}"'] & & \mathsf{G\circ L_A}
		\end{tikzcd}\quad,\quad
		\begin{tikzcd}[column sep=0.2cm]
			& \mathsf{L_B\circ H^-(F)} & \\
			\mathsf{G\circ L_A}\arrow[ur, "\varepsilon"]\arrow[rr, dashed, "\eta\circ\mathsf{L_A}"'] & & \mathsf{LF\circ L_A}\arrow[ul, "\varepsilon_\mathsf{F}"']
		\end{tikzcd}
	\end{equation}
	commute. 
\end{Def}

\subsection{Uniqueness and existence}\label{ch3.2}

In this section we prove that the right derived functor $\mathsf{RF}:\mathsf{D^+(A)}\rightarrow\mathsf{D^+(B)}$ as characterized by Definitions \ref{derivedfunc} and \ref{universalderivedfunc} exists and is unique. Analogous arguments confirm the existence and uniqueness of left derived functors.

\begin{Lem}
	The right derived functor $\mathsf{RF}:\mathsf{D^+(A)}\rightarrow\mathsf{D^+(B)}$ of a left exact functor $\mathsf{F}:\mathsf{A}\rightarrow\mathsf{B}$ between abelian categories $\mathsf{A}, \mathsf{B}$ is unique. 
\end{Lem}

\begin{proof}
	This is an immediate consequence of \eqref{universalderfuncdiag} left: suppose there is another suitable pair $(\overline{\mathsf{RF}},\bar{\varepsilon}_\mathsf{F})$, then the universal property of $(\mathsf{RF},\varepsilon_\mathsf{F})$ for $(\mathsf{G},\varepsilon)=(\overline{\mathsf{RF}},\bar{\varepsilon}_\mathsf{F})$ and of $(\overline{\mathsf{RF}},\bar{\varepsilon}_\mathsf{F})$ for $(\mathsf{G},\varepsilon)=(\mathsf{RF},\varepsilon_\mathsf{F})$ gives the combined diagram
	\[
	\begin{tikzcd}[column sep=0.2cm]
		& \mathsf{L_B\circ H^+(F)}\arrow[dl, "\varepsilon_\mathsf{F}"']\arrow[dr, "\bar{\varepsilon}_\mathsf{F}"] & \\
		\mathsf{RF\circ L_A}\arrow[rr, dashed, yshift=1ex, "\eta\circ\mathsf{L_A}"] & & \overline{\mathsf{RF}}\circ\mathsf{L_A}\arrow[ll, dashed, yshift=-1ex, "\bar{\eta}\circ\mathsf{L_A}"]
	\end{tikzcd}\quad.
	\]
	By functoriality, the dashed arrows compose to $(\bar{\eta}\circ\eta)\circ\mathsf{L_A}:\mathsf{RF}\circ\mathsf{L_A}\rightarrow\mathsf{RF}\circ\mathsf{L_A}$ and $(\eta\circ\bar{\eta})\circ\mathsf{L_A}:\overline{\mathsf{RF}}\circ\mathsf{L_A}\rightarrow\overline{\mathsf{RF}}\circ\mathsf{L_A}$, whence we read off the automorphisms $\bar{\eta}\circ\eta$ of $\mathsf{RF}$ and $\eta\circ\bar{\eta}$ of $\overline{\mathsf{RF}}$. A further application of \eqref{universalderfuncdiag} confirms that these are identity natural transformations. Thus $\overline{\mathsf{RF}}\cong\mathsf{RF}$ and $\bar{\eta}$ is the unique natural isomorphism inverting $\eta$. We conclude that the right derived functor of $\mathsf{F}$ is indeed unique.  
\end{proof}

To show the existence of derived functors we need some more work. First, a preparatory lemma.

\begin{Lem}\label{preplemma}
	Let $\mathsf{A}, \mathsf{B}$ be abelian categories, $\mathcal{R}\subset\textup{obj}(\mathsf{A})$ adapted to a left exact functor $\mathsf{F}:\mathsf{A}\rightarrow\mathsf{B}$, and let $S_\mathcal{R}$ be the class of quasi-isomorphisms in $\mathsf{H^+}(\mathcal{R})$. Given any triangle $\Delta=\big(X^\bullet\xrightarrow{f^\bullet}Y^\bullet\rightarrow Z^\bullet\rightarrow X[1]^\bullet\big)$ in $\mathsf{H^+(\mathcal{R})}[S_\mathcal{R}^{-1}]$ that becomes distinguished when sitting in $\mathsf{D^+(A)}$, there exists a distinguished triangle $\Delta_\mathcal{R}$ in $\mathsf{H^+(\mathcal{R})}$ of the form \eqref{conetriangle} isomorphic to $\Delta$ in $\mathsf{D^+(A)}$. 
\end{Lem}

\begin{proof}
	Proposition \ref{D(R)equivalent} furnishes the equivalence $\Phi:\mathsf{D^+(A)}\rightarrow\mathsf{H^+(\mathcal{R})}[S_\mathcal{R}^{-1}]$, which translates $f\in\text{Hom}_\mathsf{D^+(A)}(X^\bullet,Y^\bullet)$ into a morphisms class in $\mathsf{H^+(\mathcal{R})}[S_\mathcal{R}^{-1}]$ represented by the roof $X^\bullet\xleftarrow{s^\bullet}T^\bullet\xrightarrow{g^\bullet} Y^\bullet$, for $g^\bullet$ in $H^+(\mathcal{R})$ and $s^\bullet\in S_\mathcal{R}$. Consider then the distinguished triangle $\Delta_\mathcal{R}\coloneqq\big(T^\bullet\xrightarrow{g^\bullet} Y^\bullet\rightarrow C(g)^\bullet\rightarrow T[1]^\bullet\big)$ in $\mathsf{H^+(\mathcal{R})}$ and the diagram
	\[
	\begin{tikzcd}
		T^\bullet\arrow[r, "g^\bullet"]\arrow[d, "s^\bullet"'] & Y^\bullet \arrow[r]\arrow[d, "\text{id}_Y^\bullet"] & C(g)^\bullet\arrow[r] & T[1]^\bullet\arrow[d, "{s[1]^\bullet}"] \\
		X^\bullet\arrow[r, "f^\bullet"'] & Y^\bullet\arrow[r] & Z^\bullet\arrow[r] & X[1]^\bullet
	\end{tikzcd}
	\]
	of distinguished triangles in $\mathsf{D^+(A)}$ (by assumption) with commutative left square. Imposing overall commutativity, we can complete it through some $t^\bullet:C(g)^\bullet\rightarrow Z^\bullet$ which is, by the Five Lemma, a quasi-isomorphism in $\mathsf{D^+(A)}$ (since $s^\bullet$ and $\text{id}_Y^\bullet$ are). Therefore, by Definition \ref{exacttriangles}, $\Delta$ and $\Delta_\mathcal{R}$ are isomorphic in $\mathsf{D^+(A)}$.  
\end{proof}

\begin{Thm}\label{derivedfuncisexact}
	The right derived functor $\mathsf{RF}:\mathsf{D^+(A)}\rightarrow\mathsf{D^+(B)}$ of a left exact functor $\mathsf{F}:\mathsf{A}\rightarrow\mathsf{B}$ between abelian categories $\mathsf{A}, \mathsf{B}$ exists, is exact and fulfills the universal property of Definition \ref{universalderivedfunc}.
\end{Thm}

\begin{proof}
	As a starting point, consider the action \eqref{derivedfunconcomplexes} of $\mathsf{RF}$ on complexes, the equivalences $\Psi:\mathsf{H^+(\mathcal{R})}[S_\mathcal{R}^{-1}]\rightarrow\mathsf{D^+(A)}$ and $\Phi:\mathsf{D^+(A)}\rightarrow\mathsf{H^+(\mathcal{R})}[S_\mathcal{R}^{-1}]$, as well as the natural isomorphisms $\alpha:\mathsf{Id}_{\mathsf{H^+(\mathcal{R})}[S_\mathcal{R}^{-1}]}\rightarrow\Phi\circ\Psi$ and $\beta:\mathsf{Id}_{\mathsf{D^+(A)}}\rightarrow\Psi\circ\Phi$ (in footnote) of Definition \ref{derivedfunc}. There, we defined the intermediate functor $\bar{\mathsf{F}}:\mathsf{H^+(\mathcal{R})}[S_\mathcal{R}^{-1}]\rightarrow\mathsf{D^+(B)}$, which clearly commutes with the localization, $\bar{\mathsf{F}}\circ\mathsf{L_\mathcal{R}}=\mathsf{L_B}\circ\mathsf{H^+(F)}:\mathsf{H^+(\mathcal{R})}\rightarrow\mathsf{D^+(B)}$, and is consequently exact \big(by Remark \ref{alternadistinguitriang} and because cones are preserved: $\bar{\mathsf{F}}(C(f)^\bullet)=\bar{\mathsf{F}}(\mathsf{L_\mathcal{R}}(C(f)^\bullet))=\mathsf{L_B}(\mathsf{H^+(F)}(C(f)^\bullet))=\mathsf{F}(C(f)^\bullet)\cong C(\mathsf{F}(f)^\bullet)$\big).
	
	Exactness of $\mathsf{RF}=\bar{\mathsf{F}}\circ\Phi$ is easily proved: it remains to check that $\Phi$ is exact. First notice that it maps triangles in $\mathsf{D^+(A)}$ to triangles in $\mathsf{H^+(\mathcal{R})}[S_\mathcal{R}^{-1}]$, because commuting with the translation functor. To show the same for distinguished triangles, just apply Lemma \ref{preplemma}; in its language, any distinguished $\Delta$ in $\mathsf{D^+(A)}$ is mapped by $\Phi$ to some distinguished (localized) $\Delta_\mathcal{R}$ in $\mathsf{H^+(\mathcal{R})}[S_\mathcal{R}^{-1}]$.
	
	Now we construct the natural transformation $\varepsilon_\mathsf{F}$ fitting into diagram \eqref{derfuncdiag} left. Pick $X^\bullet\in\text{obj}(\mathsf{D^+(A)})=\text{obj}(\mathsf{H^+(A)})$ and let $Y^\bullet\coloneqq(\Phi\circ\mathsf{L_A})(X^\bullet)\in\text{obj}(\mathsf{H^+(\mathcal{R})}[S_\mathcal{R}^{-1}])=\text{obj}(\mathsf{H^+(\mathcal{R})})$, so that $\beta({X^\bullet}):X^\bullet\rightarrow\Psi(Y^\bullet)$ is an isomorphism in $\mathsf{D^+(A)}$, without loss of generality given by $X^\bullet\xrightarrow{s^\bullet} Z^\bullet\xleftarrow{t^\bullet} Y^\bullet$ in $\mathsf{H^+(A)}$ for $s,t\in S_\mathsf{A}$ and $Z^\bullet\in\text{obj}(\mathsf{H^+(\mathcal{R})})$ (cf. the claim in the proof of Proposition \ref{D(R)equivalent}; then actually $X^\bullet\in\text{obj}(\mathsf{H^+(\mathcal{R})})$ too). $\mathcal{R}$ being adapted, the usual argument about acyclic cones tells us that $\mathsf{H^+(F)}(s^\bullet)$, $\mathsf{H^+(F)}(t^\bullet)\in S_\mathsf{B}$. Hence, applying $\mathsf{L_B}\circ\mathsf{H^+(F)}$ to the roof and exploiting the commuting relation of $\bar{\mathsf{F}}$ and definition of $\mathsf{RF}$, we obtain a morphism
	\[
	\varepsilon_\mathsf{F}(X^\bullet):\mathsf{L_B}(\mathsf{H^+(F)}(X^\bullet))\rightarrow\mathsf{L_B}(\mathsf{H^+(F)}(Y^\bullet))=\bar{\mathsf{F}}(\mathsf{L_\mathcal{R}}(Y^\bullet))=\mathsf{RF}(\mathsf{L_A}(X^\bullet))
	\]
	in $\mathsf{D^+(B)}$, which actually can be shown not to depend on the specific roof chosen. Let us prove that the induced $\varepsilon_\mathsf{F}:\mathsf{L_B}\circ\mathsf{H^+(F)}\rightarrow\mathsf{RF}\circ\mathsf{L_A}$ is a natural transformation. So consider $X_k^\bullet\in\text{obj}(\mathsf{D^+(A)})=\text{obj}(\mathsf{H^+(A)})$ for $k=1,2$ and $\varphi\in\text{Hom}_\mathsf{D^+(A)}(X_1^\bullet,X_2^\bullet)$, which by naturality of $\beta$ yields diagrams
	\[
	\begin{tikzcd}[column sep=1cm, row sep=1cm]
		X_1^\bullet\arrow[r, "\beta(X_1^\bullet)"]\arrow[d, "\varphi"'] & \Psi(\Phi(X_1^\bullet))\arrow[d, "\Psi(\Phi(\varphi))"] \\
		X_2^\bullet\arrow[r, "\beta(X_2^\bullet)"'] & \Psi(\Phi(X_2^\bullet))
	\end{tikzcd}\quad\implies\quad
	\begin{tikzcd}[column sep=1cm, row sep=1cm]
		\mathsf{L_B}(\mathsf{H^+(F)}(X_1^\bullet))\arrow[r, "\varepsilon_\mathsf{F}(X_1^\bullet)"]\arrow[d, "\mathsf{L_B}(\mathsf{H^+(F)}(\varphi))"'] & \mathsf{RF}(\mathsf{L_A}(X_1^\bullet))\arrow[d, "\mathsf{RF}(\mathsf{L_A}(\varphi))"] \\
		\mathsf{L_B}(\mathsf{H^+(F)}(X_2^\bullet))\arrow[r, "\varepsilon_\mathsf{F}(X_2^\bullet)"'] & \mathsf{RF}(\mathsf{L_A}(X_2^\bullet))
	\end{tikzcd}
	\]
	in $\mathsf{D^+(A)}$ respectively $\mathsf{D^+(B)}$ (the second diagram is obtained from the first by applying $\mathsf{L_B}\circ\mathsf{H^+(F)}$ and reasoning just like before). This proves that $\varepsilon_\mathsf{F}$ is a natural transformation, and in fact unique due to the universal property of the localization functors (see Definition \ref{dercatdef}).
	
	Finally, we must verify the universality of $(\mathsf{RF},\varepsilon_\mathsf{F})$. Consider another exact functor $\mathsf{G}:\mathsf{D^+(A)}\rightarrow\mathsf{D^+(B)}$ and a natural transformation $\varepsilon:\mathsf{L_B\circ H^+(F)}\rightarrow \mathsf{G\circ L_A}$. Our goal is to define a natural transformation $\eta:\mathsf{RF}\rightarrow\mathsf{G}$ making \eqref{universalderfuncdiag} left commute. Let $X^\bullet\in\text{obj}(\mathsf{D^+(A)})=\text{obj}(\mathsf{H^+(A)})$. Then we have $\varepsilon(X^\bullet)\in\text{Hom}_\mathsf{D^+(B)}(\mathsf{F}(X^\bullet),\mathsf{G}(X^\bullet))$ and $\beta(X^\bullet)\in\text{Hom}_\mathsf{D^+(A)}\big(X^\bullet,\Psi(\Phi(X^\bullet))\big)$, the latter an isomorphism represented again by the roof $X^\bullet\xrightarrow{s^\bullet} Z^\bullet\xleftarrow{t^\bullet} Y^\bullet=(\Psi\circ\Phi)(X^\bullet)$, which yields the commutative diagram
	\[
	\begin{tikzcd}[row sep=1cm]
		\mkern-69mu(\mathsf{L_B}\circ\mathsf{H^+(F)})(X^\bullet)=\mathsf{F}(X^\bullet)\arrow[r, "\varepsilon(X^\bullet)"]\arrow[d, shift right=-24mu, "(\mathsf{L_B}\circ\mathsf{F})(s^\bullet)"']\arrow[dd, dashed, shift right=36mu, "\varepsilon_\mathsf{F}(X^\bullet)"']  & \mathsf{G}(X^\bullet)=(\mathsf{G}\circ\mathsf{L_A})(X^\bullet)\arrow[d, shift right=36mu, "(\mathsf{G}\circ\mathsf{L_A})(s^\bullet)"]\arrow[dd, dashed, shift left=24mu, "\mathsf{G}(\beta(X^\bullet))"] \\
		\mkern+90mu\mathsf{F}(Z^\bullet)\arrow[r, shorten=6.4ex, xshift=-5.9ex, "\varepsilon(Z^\bullet)"] & \mkern-128mu\mathsf{G}(Z^\bullet) \\
		\mkern-43mu(\mathsf{RF}\circ\mathsf{L_A})(X^\bullet)=\mathsf{F}(Y^\bullet)\arrow[r, shorten=-1ex, xshift=1.5ex, "\varepsilon(Y^\bullet)"']\arrow[u, shift right=24mu, "(\mathsf{L_B}\circ\mathsf{F})(t^\bullet)"] & \mkern+19mu\mathsf{G}(Y^\bullet)=\mathsf{G}((\Psi\circ\Phi)(X^\bullet))\arrow[u, shift right=-36mu, "(\mathsf{G}\circ\mathsf{L_A})(t^\bullet)"']
	\end{tikzcd}
	\]
	in $\mathsf{D^+(B)}$, with vertical columns made of quasi-isomorphisms. Inverting those eating $t^\bullet$, we get two dashed arrows $\varepsilon_\mathsf{F}(X^\bullet):(\mathsf{L_B\circ H^+(F)})(X^\bullet)\rightarrow(\mathsf{RF\circ L_A})(X^\bullet)$ (by the usual argument from above) and $\mathsf{G}(\beta(X^\bullet)):\mathsf{G}(X^\bullet)\rightarrow\mathsf{G}((\Psi\circ\Phi)(X^\bullet))$, the latter being an isomorphism in $\mathsf{D^+(B)}$ (by functoriality). Then we define
	\[
	\eta(X^\bullet)\coloneqq\mathsf{G}(\beta(X^\bullet))^{-1}\circ\varepsilon((\Psi\circ\Phi)(X^\bullet)):\mathsf{RF}(X^\bullet)\rightarrow\mathsf{G}(X^\bullet)\,,
	\]
	which in turn specifies a natural transformation $\eta:\mathsf{RF}\rightarrow\mathsf{G}$ since both $\beta$ and $\varepsilon$ are. We can read off \eqref{universalderfuncdiag} from the outer square of the last diagram. That $\eta(X^\bullet)$, thus $\eta$, is unique follows from $\mathsf{G}(\beta(X^\bullet))$ being an isomorphism.
\end{proof}

\subsection{Classical derived functors and examples}\label{ch3.3}

\begin{Def}\label{classicalderivedfunc}
	Let $\mathsf{A}, \mathsf{B}$ be abelian categories, $\mathsf{F}:\mathsf{A}\rightarrow\mathsf{B}$ a left exact functor. Then $\mathsf{R}^n\mathsf{F}\coloneqq H^0(T^n\circ\mathsf{RF}) = H^n(\mathsf{RF}):\mathsf{A}\rightarrow\mathsf{B}$ is the \textbf{classical $n$-th right derived functor}\index{derived functor!classical right} of $\mathsf{F}$, where $T^n:\mathsf{D^+(B)}\rightarrow\mathsf{D^+(B)}$ is the translation autoequivalence and $H^n:\mathsf{D^+(B)}\rightarrow\mathsf{B}$ the $n$-th cohomological functor.\footnote{The notation may cause a little confusion: $H^0$ and $H^n$ act on both $\mathsf{D^\#\!(A)}$ and $\mathsf{D^\#\!(B)}$, yielding $\mathsf{A}$ respectively $\mathsf{B}$. Properly said, $\mathsf{R}^n\mathsf{F}=H^n\circ\mathsf{RF}\circ\mathsf{J'}$ where $\mathsf{J'}:\mathsf{A}\rightarrow\mathsf{D^\#\!(A)}$ is the fully faithful embedding from Proposition \ref{AinD(A)}, so that $\mathsf{R}^n\mathsf{F}(X)=H^n(F(X[0])^\bullet)$ for $X\in\text{obj}(\mathsf{A})$. Same remark applies to classical left derived functors.}
	
	If instead $\mathsf{F}$ is right exact, then $\mathsf{L}^n\mathsf{F}\coloneqq H^0(T^n\circ\mathsf{LF}) = H^n(\mathsf{LF}):\mathsf{A}\rightarrow\mathsf{B}$ is the \textbf{classical $n$-th left derived functor}\index{derived functor!classical left} of $\mathsf{F}$ (with corresponding functors $T^n$ and $H^n$ on $\mathsf{D^-(B)}$). 
\end{Def}

\begin{Pro}\label{ExtRHom}
	Let $\mathsf{A}$ be an abelian category with enough injectives, $X\in\textup{obj}(\mathsf{A})$ fixed, yielding the left exact functor $\textup{Hom}_\mathsf{A}(X,\square):\mathsf{A}\rightarrow\mathsf{Ab}$ \textup(see Example \ref{exactfuncexample}\textup), of which we can then consider the right derived functor $\mathsf{R}\textup{Hom}_\mathsf{A}(X,\square):\mathsf{D^+(A)}\rightarrow\mathsf{D^+(Ab)}$. Then there is an isomorphism of functors \begin{equation}\label{ExtRHomeq}
		\textup{Ext}_\mathsf{A}^n(X,\square)\cong\mathsf{R}^n\textup{Hom}_\mathsf{A}(X,\square):\mathsf{A}\rightarrow\mathsf{Ab}
	\end{equation}
	for all $n\in\mathbb{Z}$.
	
	Specularly, if $\mathsf{A}$ has enough projectives and $Y\in\textup{obj}(\mathsf{A})$ is fixed, the contravariant left exact functor $\textup{Hom}_\mathsf{A}(\square,Y):\mathsf{A}^\textup{opp}\rightarrow\mathsf{Ab}$ yields
	\begin{equation}\label{eqExtRHom}
		\textup{Ext}_\mathsf{A}^n(\square,Y)\cong\mathsf{R}^n\textup{Hom}_\mathsf{A}(\square, Y):\mathsf{A^\textup{opp}}\rightarrow\mathsf{Ab}\,.
	\end{equation}  
\end{Pro}

\begin{proof}
	On the one hand, we can replace any $Y\equiv Y[0]^\bullet\in\text{obj}(\mathsf{Kom^+(A)})=\text{obj}(\mathsf{D^+(A)})$ by any of its injective resolutions $Y\xrightarrow{\varepsilon_Y}I^\bullet$ in $\mathsf{Kom^+(\mathcal{I})}$ (all quasi-isomorphic to it by Remark \ref{0-complex}), so that $\text{Ext}_\mathsf{A}^n(X,Y)=\text{Hom}_\mathsf{D^+(A)}(X[0]^\bullet,Y[n]^\bullet)$ $\cong\text{Hom}_\mathsf{D^+(A)}(X[0]^\bullet,I[n]^\bullet)$. Now one can elaborate on the argument in the second paragraph of the proof of Proposition \ref{injprojderivedarehomotopy} to show that, for a generic abelian category $\mathsf{A}$, $\text{Hom}_\mathsf{H(A)}(Z_1^\bullet,Z_2^\bullet)\cong\text{Hom}_\mathsf{D(A)}(Z_1^\bullet,Z_2^\bullet)$ as soon as $Z_1^\bullet\in\text{obj}(\mathsf{Kom^-(\mathcal{P})})$ or $Z_2^\bullet\in\text{obj}(\mathsf{Kom^+(\mathcal{I})})$. In our case, this ultimately yields $\text{Ext}_\mathsf{A}^n(X,Y)\cong\text{Hom}_\mathsf{H^+(A)}(X[0]^\bullet,I[n]^\bullet)$.
	
	About the right-hand side of \eqref{ExtRHomeq}, by definition (cf. footnote) we have $\mathsf{R}^n\text{Hom}_\mathsf{A}(X,Y)\!=\!H^n(\text{Hom}_\mathsf{A}(X,Y)[0]^\bullet)\!\cong\! H^n(\text{Hom}_\mathsf{Kom^+(A)}^\bullet(X[0]^\bullet,I[0]^\bullet))$, where the first equality holds by \eqref{derivedfunconcomplexes} and the second by the aforementioned substitution of $Y$. To understand the last cohomological group, we must explain what is the complex $\text{Hom}_\mathsf{Kom(A)}^\bullet(Z_1^\bullet,Z_2^\bullet)\in\text{obj}(\mathsf{Kom(Ab)})=\text{obj}(\mathsf{D(Ab)})$ given some fixed $Z_1^\bullet, Z_2^\bullet\in\text{obj}(\mathsf{Kom(A)})$. It is defined as 
	\begin{align}\label{innerHomcomp}
		& \text{Hom}_\mathsf{Kom(A)}^n(Z_1^\bullet, Z_2^\bullet)\coloneqq\prod_{i\in\mathbb{Z}}\text{Hom}_\mathsf{A}(Z_1^i,Z_2^{i+n})\in\text{obj}(\mathsf{Ab})\quad\text{and}\nonumber \\
		& d^n:\text{Hom}_\mathsf{Kom(A)}^n(Z_1^\bullet, Z_2^\bullet)\rightarrow \text{Hom}_\mathsf{Kom(A)}^{n+1}(Z_1^\bullet, Z_2^\bullet),\, f\equiv \prod_{i\in\mathbb{Z}}f^i\mapsto d^n(f)\,,\\
		& d^n(f)\coloneqq\prod_{i\in\mathbb{Z}}\big(d_{Z_2}^{i+n}\circ f^i -(-1)^nf^{i+1}\circ d_{Z_1}^i\big)\,. \nonumber
	\end{align}
	One can carefully verify that $d^{n+1}\circ d^n=0$, so that $\text{Hom}_\mathsf{Kom(A)}^\bullet(Z_1^\bullet,Z_2^\bullet)$ is indeed a well-defined complex of (product) abelian groups. By definition, any $n$-th cocycle $f\in Z^n(\text{Hom}_\mathsf{Kom(A)}^\bullet(Z_1^\bullet,Z_2^\bullet))$ must fulfill $d^n(f)=0$, then possible if and only if for each $i\in\mathbb{Z}$ holds $(-1)^n d_{Z_2[n]}^i\circ f^i= d_{Z_2}^{i+n}\circ f^i = (-1)^n f^{i+1}\circ d_{Z_1}^i$, that is, if and only if $f$ is a morphism of complexes in $\text{Hom}_\mathsf{Kom(A)}(Z_1^\bullet,Z_2[n]^\bullet)$. Similarly, rewriting the differential in \eqref{innerHomcomp} shows that any coboundary $f\in B^n(\text{Hom}_\mathsf{Kom(A)}^\bullet(Z_1^\bullet,Z_2^\bullet))$ is of the form $f=k^{n+1}\circ d_{Z_1} +d_{Z_2[n]}\circ k^n$ for some morphism of complexes $k^\bullet:\text{Hom}_\mathsf{Kom(A)}^\bullet(Z_1^\bullet,Z_2^\bullet)\rightarrow \text{Hom}_\mathsf{Kom(A)}^{\bullet-1}(Z_1^\bullet,Z_2^\bullet)$, that is, coboundaries are all those morphisms $Z_1^\bullet\rightarrow Z_2[n]^\bullet$ homotopic to the zero morphism. Hence, modding by them we obtain
	\[
	H^n(\text{Hom}_\mathsf{Kom(A)}^\bullet(Z_1^\bullet,Z_2^\bullet))=\frac{Z^n(\text{Hom}_\mathsf{Kom(A)}^\bullet(Z_1^\bullet,Z_2^\bullet))}{B^n(\text{Hom}_\mathsf{Kom(A)}^\bullet(Z_1^\bullet,Z_2^\bullet))}\cong\text{Hom}_\mathsf{H(A)}(Z_1^\bullet,Z_2[n]^\bullet)\,.
	\]
	Therefore, $\mathsf{R}^n\text{Hom}_\mathsf{A}(X,Y)\cong\text{Hom}_\mathsf{H^+(A)}(X[0]^\bullet,I[n]^\bullet)$, and we conclude thanks to the previous paragraph that $\mathsf{R}^n\text{Hom}_\mathsf{A}(X,Y)\cong\text{Ext}_\mathsf{A}^n(X,Y)$, as desired. The argument for the isomorphism \eqref{eqExtRHom} is similar.
\end{proof}

\begin{Rem}\label{classicalderivedonresolutions}
	If $\mathsf{A}$ is an abelian category with enough injectives, the proof just concluded suggests an explicit way to compute $\mathsf{R}^n\mathsf{F}(X)$ of a left exact $\mathsf{F}:\mathsf{A}\rightarrow\mathsf{B}$ for any $X\in\text{obj}(\mathsf{A})$: by assumption and Remark \ref{0-complex}, $X$ admits a right injective resolution $X\xrightarrow{\varepsilon_X}I^\bullet$ (unique up to quasi-isomorphism) and can be replaced by it.\footnote{In fact, it is enough to consider an $\mathsf{F}$-acyclic right resolution for $X$, as in Definition \ref{Facyclic} below. Right injective resolutions are a special case thereof.} Then
	\begin{equation}\label{classderinjres}
		\mathsf{R}^n\mathsf{F}(X)\cong H^n(\mathsf{F}(I)^\bullet)\in\text{obj}(\mathsf{B})\,.
	\end{equation}
	Specularly, if $\mathsf{A}$ has enough projectives and $\mathsf{F}$ is right exact, substituting $X$ by any left projective resolution $P^\bullet\xrightarrow{\varepsilon_X} X$ yields
	\begin{equation}\label{classderprojres}
		\mathsf{L}^n\mathsf{F}(X)\cong H^n(\mathsf{F}(P)^\bullet)\in\text{obj}(\mathsf{B})\,.
	\end{equation}
\end{Rem}

\begin{Cor}\label{whenclassicalderivedtrivial}
	Let $\mathsf{F}:\mathsf{A}\rightarrow\mathsf{B}$ be a left/right exact functor between abelian categories. Then $\mathsf{R}^0\mathsf{F}\cong\mathsf{F}\cong\mathsf{L}^0\mathsf{F}$ and $\mathsf{R}^{-n}\mathsf{F}=0=\mathsf{L}^n\mathsf{F}$, for $n\in\mathbb{N}$.
	
	Furthermore, any short exact sequence $0\rightarrow X_1\rightarrow X_2\rightarrow X_3\rightarrow 0$ in $\mathsf{A}$ induces long exact sequences in $\mathsf{B}$ of the form
	\begin{align*}
		0&\rightarrow \mathsf{F}(X_1)\rightarrow \mathsf{F}(X_2)\rightarrow \mathsf{F}(X_3)\rightarrow \mathsf{R}^1\mathsf{F}(X_1)\rightarrow... \\
		&\rightarrow \mathsf{R}^n\mathsf{F}(X_1)\rightarrow \mathsf{R}^n\mathsf{F}(X_2)\rightarrow\mathsf{R}^n\mathsf{F}(X_3)\rightarrow \mathsf{R}^{n+1}\mathsf{F}(X_1)\rightarrow ...\,,  
	\end{align*}
	respectively
	\begin{align*}
		...&\rightarrow \mathsf{L}^{n+1}\mathsf{F}(X_3)\rightarrow \mathsf{L}^n\mathsf{F}(X_1)\rightarrow \mathsf{L}^n\mathsf{F}(X_2)\rightarrow\mathsf{L}^n\mathsf{F}(X_3)\rightarrow... \\
		&\rightarrow \mathsf{L}^1\mathsf{F}(X_3)\rightarrow \mathsf{F}(X_1)\rightarrow \mathsf{F}(X_2)\rightarrow \mathsf{F}(X_3)\rightarrow 0\,.  
	\end{align*}
\end{Cor}

\begin{proof}
	This is a straightforward consequence of the definition of classical derived functor, as can be seen by its action on objects after replacing them with suitable right/left resolutions as done in Remark \ref{classicalderivedonresolutions}. Note that in the special case when $\mathsf{F}$ is an Hom-functor, we reach the same conclusion by Proposition \ref{ExtRHom} and Theorem \ref{Extprop}.
	
	The statement about long exact sequences follows from the fact that derived functors are exact: the distinguished triangle $X_1^\bullet\rightarrow X_2^\bullet\rightarrow X_3^\bullet\rightarrow X_1[1]^\bullet$ in $\mathsf{D}^\#\!(A)$ associated to the given short exact sequence is mapped to the distinguished triangle $X_1^\bullet\rightarrow X_2^\bullet\rightarrow X_3^\bullet\rightarrow X_1[1]^\bullet$ in $\mathsf{D}^\#\!(B)$, which in turn induces by Theorem \ref{longexactcohomtriangles} the declared long exact sequences in cohomology. 
\end{proof}

\begin{Ex}\label{Torfunc}
	Inspired again by Example \ref{exactfuncexample}, we wonder what happens when deriving the right exact functor $\square\otimes_R Y:\mathsf{Mod}_R\rightarrow\mathsf{Ab}$, where $R$ is an associative, unital ring and $Y\in\text{obj}({}_R\mathsf{Mod})$ is fixed. Then it turns out that $\mathcal{R}\coloneqq\{X\in\text{obj}(\mathsf{Mod}_R)\mid X\text{ is flat}\}\subset\text{obj}(\mathsf{Mod}_R)$ is adapted to $\square\otimes_R Y$ (proof omitted). We thus obtain the left derived functor $\mathsf{L}(\square\otimes_R Y):\mathsf{D^-(Mod}_R)\rightarrow\mathsf{D^-(Ab)}$ and its classical counterparts 
	\begin{equation}
		\text{Tor}_n^R(\square,Y)\coloneqq\mathsf{L}^{-n}(\square\otimes_R Y):\mathsf{Mod}_R\rightarrow\mathsf{Ab}\,,
	\end{equation}
	the $n$-th \textbf{Tor-functors}\index{Tor-functor} ($n\geq 0$). 
	
	The specular construction for the right exact $X\otimes_R \square:{}_R\mathsf{Mod}\rightarrow\mathsf{Ab}$ and $X\in\text{obj}(\mathsf{Mod}_R)$ gives 
	\begin{equation}
		\text{Tor}^R_n(X,\square)\coloneqq\mathsf{L}^{-n}(X\otimes_R \square):{}_R\mathsf{\mathsf{Mod}}\rightarrow\mathsf{Ab}\,.
	\end{equation}
	\hfill $\blacklozenge$
\end{Ex}

\begin{Def}\label{Facyclic}
	Let $\mathsf{RF},\mathsf{LF}:\mathsf{D^\pm(A)}\rightarrow\mathsf{D^\pm(B)}$ be the right/left derived functor of a left/right exact functor $\mathsf{F}:\mathsf{A}\rightarrow\mathsf{B}$ between abelian categories. Then $X\in\text{obj}(\mathsf{A})$ is \textbf{$\mathsf{F}$-acyclic}\index{facy@$\mathsf{F}$-acyclic} if $\mathsf{R}^n\mathsf{F}(X)=H^n(\mathsf{F}(X[0])^\bullet)=0\in\text{obj}(\mathsf{B})$ for all $n>0$, respectively if $\mathsf{L}^n\mathsf{F}(X)=0$ for all $n<0$.
\end{Def}

In this setting, if the class $\mathcal{Z}$ of all $\mathsf{F}$-acyclic objects in $\mathsf{A}$ is sufficiently large --- meaning that any object of $\mathsf{A}$ is a subobject of an $\mathsf{F}$-acyclic one --- then one can prove that $\mathcal{Z}$ contains any $\mathcal{R}\subset\text{obj}(\mathsf{A})$ adapted to $\mathsf{F}$, particularly $\mathcal{I}$ of all injective objects, and any sufficiently large subclass $\tilde{\mathcal{Z}}\subset\mathcal{Z}$ is adapted to $\mathsf{F}$ (see \cite[Theorem III.6.16]{[GM03]}). So the existence of derived functors (almost) reverts to that of adapted classes.
\vspace*{0.3cm}

\noindent As a final word about general derived functors, we deal with their restrictibility and composability.

\begin{Pro}\label{derivedfuncrestrictstobdd}
	Let $\mathsf{A},\mathsf{B}$ be abelian categories, with $\mathsf{A}$ having enough injectives, and let $\mathsf{F}:\mathsf{A}\rightarrow\mathsf{B}$ be a left exact functor with associated right derived functor $\mathsf{RF}:\mathsf{D^+(A)}\rightarrow\mathsf{D^+(B)}$.
	\renewcommand{\theenumi}{\roman{enumi}}
	\begin{enumerate}[leftmargin=0.5cm]
		\item Suppose $\mathsf{C}\subset\mathsf{B}$ is a thick subcategory \textup(see Lemma \ref{abeliansubcatderivedequiv}\textup) with $\mathsf{R}^i\mathsf{F}(X)\in\textup{obj}(\mathsf{C})$ for all $X\in\textup{obj}(\mathsf{A})$, and assume that there exists some $n\in\mathbb{Z}$ such that $\mathsf{R}^i\mathsf{F}(X)=0$ for all $i<n$ and $X\in\textup{obj}(\mathsf{A})$. Then $\mathsf{RF}$ takes values into the full subcategory of $\mathsf{D^+(B)}$ consisting of complexes with cohomology in $\mathsf{C}$.
		
		\item If $\mathsf{RF}(X[0]^\bullet)\in\textup{obj}(\mathsf{D^b(B)})$ for any $X\in\textup{obj}(\mathsf{A})$, then $\mathsf{RF}$ restricts to an exact functor $\mathsf{RF}:\mathsf{D^b(A)}\rightarrow\mathsf{D^b(B)}$. \textup(This is valid even if $\mathsf{A}$ does not have enough injectives but $\mathsf{RF}$ still exists; see \textup{\cite[Corollary 2.68]{[Huy06]}} and the ensuing remark.\textup) 
	\end{enumerate}
\end{Pro}

\begin{Pro}\label{derivedfunccompo}
	Let $\mathsf{A}, \mathsf{B}, \mathsf{C}$ be abelian categories, $\mathsf{F}:\mathsf{A}\rightarrow\mathsf{B}$ and $\mathsf{G}:\mathsf{B}\rightarrow\mathsf{C}$ additive left exact functors, $\mathcal{R}_\mathsf{A}\subset\textup{obj}(\mathsf{A}), \mathcal{R}_\mathsf{B}\subset\textup{obj}(\mathsf{B})$ classes adapted to $\mathsf{F}$ respectively $\mathsf{G}$ such that $\mathsf{F}(\mathcal{R}_\mathsf{A})\subset\mathcal{R}_\mathsf{B}$. Then the right derived functor $\mathsf{R(G\circ F)}:\mathsf{D^+(A)}\rightarrow\mathsf{D^+(C)}$ exists and the natural transformation $E:\mathsf{R(G\circ F)}\rightarrow \mathsf{RG\circ RF}$ is a natural isomorphism.
\end{Pro}

\begin{proof}
	The proof is straightforward. By assumption, $\mathsf{F}$ maps $\mathsf{Kom^+(\mathcal{R}_\mathsf{A})}$ into $\mathsf{Kom^+(\mathcal{R}_\mathsf{B})}$, hence $\mathcal{R}_\mathsf{A}$ is clearly adapted to the additive left exact functor $\mathsf{G\circ F}:\mathsf{A}\rightarrow\mathsf{C}$ as well. Therefore, $\mathsf{R(G\circ F)}$ exists and is constructed as in Definition \ref{derivedfunc}. Moreover, it is exact (using Lemma \ref{preplemma} and the fact that both $\mathsf{RF}, \mathsf{RG}$ are) and unique. The universal property \eqref{universalderfuncdiag} provides the natural transformation $E:\mathsf{R(G\circ F)}\rightarrow \mathsf{RG\circ RF}$ (what we called $\eta$ in the diagram) towards the exact functor $\mathsf{RG\circ RF}:\mathsf{D^+(A)}\rightarrow\mathsf{D^+(C)}$. Since $E(X^\bullet)$ is manifestly an isomorphism in $\mathsf{D^+(C)}$ for each $X^\bullet\in\text{obj}(\mathsf{Kom^+(\mathcal{R}_\mathsf{A})})$, and any object of $\mathsf{D^+(A)}$ has such form by assumption (up to isomorphism), $E$ is a natural isomorphism. 
\end{proof}

At the level of classical derived functors, we understand that there are isomorphisms $E^n(X):\mathsf{R}^n(\mathsf{G\circ F})(X)\xrightarrow{\sim}\mathsf{R}^n\mathsf{G}(\mathsf{RF}(X))$ in $\mathsf{C}$ for all $X\in\text{obj}(\mathsf{A})$. They suggest that it is possible to reconstruct the classical derived functor of a composition of functors from the classical derived functors of its components --- something non-trivial because we observe that the act of taking cohomology destroys the composition law of functors! Key is a proper description of the isomorphisms $E^n(X)$, along with some additional data, together forming a \textit{spectral sequence}. We refer the interested reader to \cite[section III.7]{[GM03]}.

\newpage

\section{Triangulated categories}
\thispagestyle{plain}

\subsection{The axioms of triangulated categories}\label{ch4.1}

Derived categories have a very nice property: they are triangulated. We have already discussed the basics of triangulated categories back in \cite[section 3.1]{[Imp21]}, where we saw them to obey a set of specific axioms. Since we will now spend a little more time on the subject so to prove the debut statement, it is beneficial to mention again their defining features.

\begin{Def}\label{triangcat}
	An additive category $\mathsf{D}$ is a \textbf{triangulated category}\index{triangulated category} if it is endowed with an additive autoequivalence $T:\mathsf{D}\rightarrow\mathsf{D}$, the \textbf{translation functor}\index{functor!translation} (defined analogously to Definition \ref{translfunc}), and a class of \textbf{distinguished triangles}\index{distinguished triangle}
	\begin{equation}\label{triangtriangle}
		X\xrightarrow{u}Y\xrightarrow{v}Z\xrightarrow{w}T(X)=X[1]
		\qquad\equiv\qquad
		\begin{tikzcd}
			X\arrow[rr, "u"] & & Y\arrow[dl, "v"]\\
			& Z\arrow[ul, "{w[1]}"]
		\end{tikzcd}
	\end{equation}
	in $\mathsf{D}$, written $(u,v,w)$ for short, and \textbf{morphisms of distinguished triangles}\index{morphism!of distinguished triangles} $(f,g,h)$ fitting into\vspace*{-0.2cm}
	\begin{equation}
		\begin{tikzcd}
			X\arrow[r,"u"]\arrow[d,"f"'] & Y\arrow[r,"v"]\arrow[d,"g"'] & Z\arrow[r,"w"]\arrow[d,"h"] & {X[1]}\arrow[d,"{f[1]}"] \\
			X'\arrow[r,"u'"'] & Y'\arrow[r,"v'"'] & Z'\arrow[r,"w'"'] & X'[1]
		\end{tikzcd}
	\end{equation}
	--- also isomorphisms if $f,g,h$ are isomorphisms in $\mathsf{D}$ --- which satisfy the following axioms:
	\begin{itemize}[leftmargin=0.5cm]
		\renewcommand{\labelitemi}{\textendash}
		
		\item (\textbf{T1})\renewcommand{\theenumi}{\alph{enumi}}\vspace*{-0.3cm} 
		\begin{enumerate}[leftmargin=0.5cm] 
			\item $X\xrightarrow{\textup{id}_X}X\rightarrow 0\rightarrow X[1]$ is distinguished for any $X\in\text{obj}(\mathsf{D})$;
			
			\item any triangle (generally of the form \eqref{triangtriangle}) isomorphic to a distinguished triangle is itself distinguished;
			
			\item any morphism $u\in\text{Hom}_\mathsf{D}(X,Y)$ can be completed to a distinguished triangle like \eqref{triangtriangle}.
		\end{enumerate}
		
		\item (\textbf{T2}) A triangle $X\xrightarrow{u}Y\xrightarrow{v}Z\xrightarrow{w}X[1]$ is distinguished if and only if $Y\xrightarrow{v}Z\xrightarrow{w}X[1]\xrightarrow{-u[1]}Y[1]$ is distinguished (or equivalently, if and only if $Z[-1]\xrightarrow{-w[-1]}X\xrightarrow{u}Y\xrightarrow{v}Z$ is).
		
		\item (\textbf{T3}) The braid diagram (see \cite{[May01]}):
		\begin{equation}\label{commbraid}
			\begin{tikzcd}
				X \arrow[rr, orange, bend left = 45, "w" black]\arrow[dr, red, "u"' black] & \circlearrowleft & Z\arrow[rr, green, bend left = 45, "v'" black]\arrow[dr, orange, "w'" black] & & X' \arrow[rr, cyan, bend left = 45, "a''" black]\arrow[dr, green, "v''"' black] & \circlearrowleft & \mkern-6mu Z'[1] \\
				& Y\arrow[ur, green, "v"' black]\arrow[dr, red, "u'"' black] & & Y' \arrow[ur, dashed, cyan, "a'" black]\arrow[dr, orange, "w''" black] & & \mkern-6mu Y[1]\arrow[ur, "{u'[1]}"'] \\
				& & Z'\arrow[ur, dashed, cyan, "a" black]\arrow[rr, red, bend right = 45, "u''"' black] & & \mkern-6mu X[1]\arrow[ur, "{u[1]}"']
			\end{tikzcd}
		\end{equation}
		made of red $(u,u',u'')$, green $(v,v',v'')$ and orange $(w,w',w'')$ distinguished triangles, where $w = v\circ u$ and $a''= u'[1]\circ v''$, can be completed to a commutative diagram in $\mathsf{D}$ via dashed arrows $a$ and $a'$ forming a blue distinguished triangle $(a,a',a'')$. 
	\end{itemize}
\end{Def}  

In literature, the visually less elegant (though triangle-shape preserving) octahedron diagram is more commonly used to describe (T3), in fact also known as \textit{octahedral axiom}:
\begin{equation}\label{octahedraldiag}
	\begin{tikzcd}
		& Y\arrow[dr, green, "v" black] & \\
		X'\arrow[d, cyan, "{a''[1]}"' black]\arrow[ur, green, "{v''[1]}" black] & & Z\arrow[ll, green, "v'"' black]\arrow[ddl, orange, "w'"' black] \\
		Z'\arrow[from=uur, red, crossing over, "\!\!u'" black]\arrow[rr, red, crossing over, "{u''[1]}"' black]\arrow[dr, dashed, cyan, "a"' black] & & X\arrow[u, orange, "w"' black]\arrow[uul, red, crossing over, "\overset{}{u}" black] \\
		& Y'\arrow[ur, orange, "{w''[1]}"' black]\arrow[uul, dashed, cyan, "a'"' black] &
	\end{tikzcd}
\end{equation}
Looking at colours for reference, we can reformulate the condition above as: given the upper pyramid, with distinguished front red/back green facets and commuting lateral facets, it is possible to construct a compatible lower pyramid with distinguished left blue/right orange facets and commuting front/back facets, so that the overall octahedron commutes.

We omitted what is often labeled as the axiom (TR3), according to Verdier's terminology (see \cite[section II.1]{[Ver96]}). This because it can actually be deduced from our (T1)--(T3), as shown in \cite[Lemma 3.3]{[Imp21]}. Still, we state it again, for it plays a crucial role in the proofs to follow:

\begin{Lem}[TR3]\label{tr3}								% (TR3)
	Assume the rows of the diagram
	\begin{equation}\label{TR3diag}
		\begin{tikzcd}
			X\arrow[r,"u"]\arrow[d,"f"'] & Y \arrow[r,"v"]\arrow[d,"g"'] & Z \arrow[r,"w"]\arrow[d, dashed, "h"] & X[1]\arrow[d, "{f[1]}"] \\
			X'\arrow[r, "u'"'] & Y'\arrow[r, "v'"'] & Z'\arrow[r, "w'"'] & X'[1]
		\end{tikzcd}
	\end{equation}
	in $\mathsf{D}$ are two distinguished triangles and suppose the left square commutes. Then there is a dashed arrow $h\in\textup{Hom}_\mathsf{D}(Z,Z')$ making the whole diagram commute, so that $(f,g,h)$ is a morphism of distinguished triangles; moreover, if $f$ and $g$ are isomorphisms, so is $h$.
\end{Lem}

\begin{Rem}\label{consecindistitriang}
	A few small but remarkable first consequences of (TR3) are:
	\begin{itemize}[leftmargin=0.5cm]
		\item The composition of any two consecutive morphisms in a distinguished triangle is zero: just take in \eqref{TR3diag} the top row to be the triangle $X'\xrightarrow{\textup{id}_{X'}}X'\rightarrow 0\rightarrow X'[1]$, distinguished by (T1)a., and $(f,g)=(\text{id}_{X'},u')$; then the only $h:0\dashrightarrow Z'$ making the diagram commute is $h=0$, so that $v'\circ u'=h\circ 0=0$ as claimed.
		
		\item The complementing distinguished triangle in (T1)c. is unique up to isomorphism: for any such two as in \eqref{TR3diag}, just take $u'=u: X'=X\rightarrow Y'=Y$, $f=\text{id}_X$ and $g=\text{id}_Y$; then by lemma there exists an isomorphism $h:Z\dashrightarrow Z'$.
	\end{itemize}  
\end{Rem}

The second bullet point prompts the following quite reminiscent definition.

\begin{Def}\label{conesintriangulated}
	Let $\mathsf{D}$ be a triangulated category. We call the object $Z$ completing $u\in\text{Hom}_\mathsf{D}(X,Y)$ to a distinguished triangle as per (T1)c. the \textbf{cone}\index{cone} of $u$, and write $Z\equiv C(u)$. By last remark it is unique only up to isomorphism.
	
	Conversely, if for a given $Y\in\text{obj}(\mathsf{D})$ there are $X,Z\in\text{obj}(\mathsf{D})$ and suitable morphisms such that $X\rightarrow Y\rightarrow Z\rightarrow X[1]$ is a distinguished triangle in $\mathsf{D}$, then $Y$ is sometimes called the \textbf{extension}\index{extension (of objects)} of $Z$ by $X$. Then a full subcategory $\mathsf{D}'\subset\mathsf{D}$ is \textit{stable under extensions} if for any such distinguished triangle in $\mathsf{D}$ holds $\big(X,Z\in\text{obj}(\mathsf{D}')\implies Y\in\text{obj}(\mathsf{D}')\big)$.
\end{Def}

\begin{Rem}
	One might hope that cones behave functorially under composition of morphisms, namely that $C(v\circ u)\cong C(v)\circ C(u)$ for compatible morphisms in $\mathsf{D}$. Unfortunately, this is not the case. 
	
	However, we can exploit the axiom (T3) to construct $C(v\circ u)$. Looking at the commutative braid \eqref{commbraid}, we can replace without loss of generality $Z'=C(u)$, $X'=C(v)$ and $Y'=C(v\circ u)$. Then the blue distinguished triangle demanded by the axiom is $C(u)\xrightarrow{a} C(v\circ u)\xrightarrow{a'} C(v)\xrightarrow{a''} C(u)[1]$, which allows us to interpret $C(v\circ u)=C(a'')[-1]=C(u'[1]\circ v'')[-1]$.
\end{Rem} 

\begin{Pro}\label{Homareexact}
	Let $\mathsf{D}$ be a triangulated category, $X\xrightarrow{u} Y\xrightarrow{v} Z\xrightarrow{w} X[1]$ a distinguished triangle, $D\in\textup{obj}(\mathsf{D})$ fixed. Then\begin{normalsize}
		\begin{align}\label{covHomexact}
			...\rightarrow \textup{Hom}_\mathsf{D}(D,X[n])\xrightarrow{u_*[n]} \textup{Hom}_\mathsf{D}(D,Y[n])&\xrightarrow{v_*[n]}\textup{Hom}_\mathsf{D}(D,Z[n]) \\
			&\xrightarrow{w_*[n]}\textup{Hom}_\mathsf{D}(D,X[n+1])\rightarrow...\,, \nonumber
		\end{align}
		\begin{align}\label{contraHomexact}
			...\rightarrow \textup{Hom}_\mathsf{D}(X[n+1],D)\xrightarrow{w^*[n]} \textup{Hom}_\mathsf{D}(Z[n],D)&\xrightarrow{v^*[n]}\textup{Hom}_\mathsf{D}(Y[n],D) \\
			&\xrightarrow{u^*[n]}\textup{Hom}_\mathsf{D}(X[n],D)\rightarrow...\,, \nonumber
		\end{align}
	\end{normalsize}are long exact sequences, where for example $u_*[n]=\textup{Hom}_\mathsf{D}(D,u[n])=u[n]\circ\square$ while $u^*[n]=\textup{Hom}_\mathsf{D}(u[n],D)=\square\circ u[n]$.
\end{Pro}

\begin{proof}
	We just prove exactness of the first sequence, the second following the same outline. We can restrict our analysis to e.g. $\text{Hom}_\mathsf{D}(D,Y)$, because exactness at the other Hom-groups follows by translation of the generating triangle, which leaves it distinguished by axiom (T2). Hence, to show is that $\ker(v_*)=\text{im}(u_*)$.
	
	Firstly, suppose $f=u_*(g)=u\circ g\in\text{im}(u_*)$ for some $g\in\text{Hom}_\mathsf{D}(D,X)$. Then Remark \ref{consecindistitriang} immediately implies that $v_*(f)=v\circ u\circ g=0\circ g=0$, that is, $f\in\ker(v_*)$. Conversely, suppose $f\in\ker(v_*)$, and look at
	\[
	\begin{tikzcd}
		D\arrow[r]\arrow[d,"f"'] & 0 \arrow[r]\arrow[d] & D[1] \arrow[rr, "-\text{id}_{D[1]}"]\arrow[d, dashed, "{g[1]}"] & & D[1]\arrow[d, "{f[1]}"] \\
		Y\arrow[r, "v"'] & Z\arrow[r, "w"'] & X[1]\arrow[rr, "{-u[1]}"'] & & Y[1]
	\end{tikzcd}\quad,
	\]
	where both rows are distinguished triangles by (T1) and (T2). Since the left square commutes by assumption on $f$, Lemma \ref{tr3} provides the dashed arrow $g[1]$, whence we recover $g\in\text{Hom}_\mathsf{D}(D,X)$ again thanks to the translation allowed by (T2). This builds a further square to the left of the diagram, with leftmost vertical arrow $g$, whose commutativity guarantees that $f=f\circ\text{id}_D=u\circ g=u_*(g)\in\text{im}(u_*)$.
\end{proof}

Functors on triangulated categories can be cohomological or exact as well:

\begin{Def}\label{cohomfuncontriangulated}
	Let $(\mathsf{D}, T_\mathsf{D})$, $(\mathsf{D}', T_{\mathsf{D}'})$ be triangulated categories, $\mathsf{A}$ an abelian category.
	\begin{itemize}[leftmargin=0.5cm]
		\item A functor $\mathsf{F}:\mathsf{D}\rightarrow\mathsf{A}$ is \textbf{cohomological}\index{functor!cohomological (on a triangulated category)} if it is additive and any distinguished triangle in $\mathsf{D}$ like \eqref{triangtriangle} induces an exact sequence $\mathsf{F}(X)\xrightarrow{\mathsf{F}(u)}\mathsf{F}(Y)\xrightarrow{\mathsf{F}(v)}\mathsf{F}(Z)$ in $\mathsf{A}$, so $\ker({\mathsf{F}(v)})=\text{im}(\mathsf{F}(u))$.
		
		\item A functor $\mathsf{F}:\mathsf{D}\rightarrow\mathsf{D}'$ is \textbf{triangulated}\index{functor!triangulated (or exact)} (or \textit{exact}) if it commutes with the translation functors, $\mathsf{F}\circ T_\mathsf{D}=T_{\mathsf{D}'}\circ\mathsf{F}$, and maps distinguished triangles in $\mathsf{D}$ to distinguished triangles in $\mathsf{D}'$.
	\end{itemize}
	
\end{Def}

We see with little effort that the cohomological condition is just the same as in Definition \ref{cohomfunconabelian}. However, we can bypass the requisite of a \textit{long} exact sequence because this comes for free in triangulated categories: if $\mathsf{F}$ is cohomological, by axiom (T2) also the sequences associated to the distinguished triangles shifted once to the left and to the right are exact as well, hence they chain together along with all other translated copies to form the expected long exact sequence.

By Proposition \ref{Homareexact}, the covariant functor $\textup{Hom}_\mathsf{D}(D,\square):\mathsf{D}\rightarrow\mathsf{Ab}$ and the contravariant functor $\textup{Hom}_\mathsf{D}(\square,D):\mathsf{D}\rightarrow\mathsf{Ab}$ are cohomological.\footnote{From the perspective of distinguished triangles, this means that they are exact. A generic triangle in $\mathsf{D}$ is called \textit{exact} if it induces long exact sequences upon application of $\textup{Hom}_\mathsf{D}(D,\square)$ and $\textup{Hom}_\mathsf{D}(\square,D)$ for every $D\in\text{obj}(\mathsf{D})$.}  

\vspace*{0.5cm}

\noindent As a final note, one may wonder whether the morphism analogue of axiom (T1)c. holds. That is, which data is necessary to construct a morphism between two given distinguished triangles connected by one vertical arrow?

\begin{Cor}
	Let $\mathsf{D}$ be a triangulated category and suppose we are given two distinguished triangles
	\[
	\begin{tikzcd}
		X\arrow[r,"u"]\arrow[d, dashed, "f"'] & Y \arrow[r,"v"]\arrow[d,"g"'] & Z \arrow[r, "w"]\arrow[d, dashed, "h"] & X[1]\arrow[d, dashed, "{f[1]}"] \\
		X'\arrow[r, "u'"'] & Y'\arrow[r, "v'"'] & Z'\arrow[r, "w'"'] & X'[1]
	\end{tikzcd}
	\]
	just connected by some $g:Y\rightarrow Y'$. If $v'\circ g\circ u=0$, then it is possible to complete $g$ to a morphism of triangles $(f,g,h)$ as depicted, unique as soon as $\textup{Hom}_\mathsf{D}(X,Z'[-1])=\{0\}$. 
	
	\textup(The converse statement is trivially true: from the readily complete commutative diagram we obtain $v'\circ g\circ u=v'\circ u'\circ f=0$, by Remark \ref{consecindistitriang}.\textup) 
\end{Cor}

\begin{proof}
	Apply $\text{Hom}_\mathsf{D}(X,\square)$ to the primed distinguished triangle to get (according to Proposition \ref{Homareexact}) a long exact sequence like \eqref{covHomexact} containing
	\begin{small}\[
		...\mkern-6mu\rightarrow \!\textup{Hom}_\mathsf{D}(X,Z'[-1])\!\xrightarrow{w_*'[-1]}\! \textup{Hom}_\mathsf{D}(X,X')\!\xrightarrow{u_*'}\!\textup{Hom}_\mathsf{D}(X,Y')\!\xrightarrow{v_*'}\!\textup{Hom}_\mathsf{D}(X,Z')\!\rightarrow\mkern-6mu...\,.
		\]\end{small}
	Any $f\in\textup{Hom}_\mathsf{D}(X,X')$ such that $u'\circ f= g\circ u\in\text{Hom}_\mathsf{D}(X,Y')$ must be a preimage of the form $f=(u_*')^{-1}(g\circ u)+\tilde{f}$, where for $\tilde{f}\in\text{Hom}_\mathsf{D}(X,X')$ to vanish under $u_*'$ we must have $\tilde{f}=w_*'[-1](l)\in\text{im}(w_*'[-1])$ ($=\ker(u_*')$ by exactness) for some $l\in\text{Hom}_\mathsf{D}(X,Z'[-1])$. Lemma \ref{tr3} then provides a morphism $h\in\text{Hom}_\mathsf{D}(Z,Z')$ making the overall diagram commute.
	
	Such $h$ is determined up to the image under $w^*$ of some $\tilde{l}\in\text{Hom}_\mathsf{D}(X[1],Z')$ (arguing just as above, because then $(h+w^*(\tilde{l}))\circ v= h\circ v + \tilde{l}\circ \cancel{w\circ v}=v'\circ g$ for the second square). Now simply recall that the translation functor $T:\mathsf{D}\rightarrow\mathsf{D}$ is an equivalence, meaning that $T_{X,Z'[-1]}:\text{Hom}_\mathsf{D}(X,Z'[-1])\rightarrow\text{Hom}_\mathsf{D}(X[1],Z')$ is bijective. So $h$ is unique in case $\text{Hom}_\mathsf{D}(X,Z'[-1])=\{0\}$, which also uniquely sets $f$ and the morphism of triangles $(f,g,h)$.
\end{proof}

\subsection{Triangularity of homotopy categories}

The goal of this subsection is to prove the following theorem:

\begin{Thm}\label{homotopycatistriangulated}
	Let $\mathsf{A}$ be an abelian category. Then its homotopy category $\mathsf{H^\#\!(A)}$ \textup(for $\#=\emptyset,+,-,\mathsf{b}$\textup), equipped with the translation functor $T=T[1]:\mathsf{H^\#\!(A)}\rightarrow\mathsf{H^\#\!(A)}$ from Definition \ref{translfunc} and with distinguished triangles of complexes from Definition \ref{exacttriangles} as the distinguished triangles, is a triangulated category.
\end{Thm}

We check each axiom of a triangulated category separately, introducing some more theoretical bits in the process. A first point worth stressing out (again) is:

\begin{Rem}
	Distinguished triangles of complexes in $\mathsf{H^\#\!(A)}$ are all those which are isomorphic to some diagram like \eqref{distinguitriangle}. We previously argued (see Remark \ref{alternadistinguitriang}) that this is equivalent to ask that they are isomorphic to triangles of the form \eqref{conetriangle}, by resorting to Proposition \ref{shortexactaretriangles}. In $\mathsf{H^\#\!(A)}$ (not yet introduced back then) we can directly show that there exists an isomorphism of triangles of complexes
	\[
	\begin{tikzcd}
		X^\bullet\arrow[r, "f^\bullet"]\arrow[d, "\text{id}_X^\bullet"'] & Y^\bullet\arrow[r, "i_2^\bullet"]\arrow[d, "i_3^\bullet"'] & C(f)^\bullet\arrow[r, "p_1^\bullet"]\arrow[d, "\text{id}_{C(f)}^\bullet"] & X[1]^\bullet\arrow[d, "\text{id}_{X[1]}^\bullet"] \\
		X^\bullet\arrow[r, "i_1^\bullet"'] & Z(f)^\bullet\arrow[r, "p_{23}^\bullet"'] & C(f)^\bullet\arrow[r, "p_1^\bullet"'] & X[1]^\bullet
	\end{tikzcd}\quad.
	\]
	The only non-trivial verification is the commutativity of the left square. In fact, $i_3^n\circ f^n - i_1^n= d_{Z(f)}^{n-1}\circ k^n-k^{n+1}\circ d_X^n$ for the homotopy $k^\bullet:X^\bullet\rightarrow Z(f)^{\bullet-1}$ given by $k^n(x^n)\coloneqq(0,x^n,0)$ (so the commutativity breaks in $\mathsf{Kom^\#\!(A)}$!). Moreover, we know from Lemma \ref{p3i3quiso} that $i_3^\bullet$ is a quasi-isomorphism. 
\end{Rem}

Therefore, we rather interpret distinguished triangles in $\mathsf{H^\#\!(A)}$ as those isomorphic to \eqref{conetriangle} --- a wink to cones in general triangulated categories, by Definition \ref{conesintriangulated}.

\begin{proof}[Proof of \textup{(T1)}]
	Item (T1)b. follows from the very definition of distinguished triangle of complexes, whereas part (T1)c. by completing $u^\bullet:X^\bullet\rightarrow Y^\bullet$ with its cone of complexes $C(u)^\bullet=X[1]^\bullet\oplus Y^\bullet\in\text{obj}(\mathsf{H^\#\!(A)})$ and natural maps $i_2^\bullet$ and $p_1^\bullet$ in $\mathsf{H^\#\!(A)}$. 
	
	About (T1)a., for any $X^\bullet\in\text{obj}(\mathsf{H^\#\!(A)})$ the triangle $X^\bullet\xrightarrow{\text{id}_X^\bullet}X^\bullet\rightarrow 0\rightarrow X[1]^\bullet$ almost resembles the distinguished $X^\bullet\xrightarrow{\text{id}_X^\bullet}X^\bullet\xrightarrow{i_2^\bullet} C(\text{id}_X)^\bullet\xrightarrow{p_1^\bullet} X[1]^\bullet$. The obvious triple $(\text{id}_X^\bullet,\text{id}_X^\bullet,0^\bullet)$ is an isomorphism of triangles between them. Indeed, $\text{id}_{C(\text{id}_X)}^n-0^n=k^{n+1}\circ d_{C(\text{id}_X)}^n+d_{C(\text{id}_X)}^{n-1}\circ k^n$ for the homotopy $k^\bullet:C(\text{id}_X)^\bullet\rightarrow C(\text{id}_X)^{\bullet-1}$ given by $k^n(x^{n+1},x^n)\coloneqq(x^n,0)$, so that in $\mathsf{H^\#\!(A)}$ we may substitute $0^\bullet:0\rightarrow C(\text{id}_X)^\bullet$ with the isomorphism $\text{id}_{C(\text{id}_X)}^\bullet$.
\end{proof}

\begin{proof}[Proof of \textup{(T2)}]
	Let $\Delta_1\coloneqq\big(X^\bullet\xrightarrow{u^\bullet} Y^\bullet\xrightarrow{v^\bullet} Z^\bullet\xrightarrow{w^\bullet} X[1]^\bullet\big)$ and $\Delta_2\coloneqq\big(Y^\bullet\xrightarrow{v^\bullet} Z^\bullet\xrightarrow{w^\bullet} X[1]^\bullet\xrightarrow{-u[1]^\bullet}Y[1]^\bullet\big)$ be triangles of complexes in $\mathsf{H^\#\!(A)}$. It suffices to prove that $\Delta_1$ distinguished implies $\Delta_2$ distinguished (the converse direction is obtained by repeating the argument below twice for $\Delta_2$ and shifting). So assume $\Delta_1$ is distinguished and thus (without loss of generality) that $Z^\bullet=C(u)^\bullet$ with natural $v^\bullet=i_2^\bullet$, $w^\bullet=p_1^\bullet$. Consider then the distinguished triangle $\Delta_3\coloneqq\big(Y^\bullet\xrightarrow{v^\bullet} Z^\bullet\xrightarrow{i_2^\bullet} C(v)^\bullet\xrightarrow{p_1^\bullet}Y[1]^\bullet\big)$, where by definition $C(v)^\bullet=Y[1]^\bullet\oplus Z^\bullet=Y[1]^\bullet\oplus X[1]^\bullet\oplus Y^\bullet$ and $d_{C(v)}^n(y^{n+1},x^{n+1},y^n)=(-d_Y^{n+1}(y^{n+1}),-d_X^{n+1}(x^{n+1}), y^{n+1}+u^{n+1}(x^{n+1})+d_Y^n(y^n))$. We claim that
	\[
	\begin{tikzcd}
		Y^\bullet\arrow[r, "v^\bullet"]\arrow[d, "\text{id}_Y^\bullet"'] & Z^\bullet\arrow[r, "p_1^\bullet"]\arrow[d, "\text{id}_Z^\bullet"'] & X[1]^\bullet\arrow[r, "{-u[1]^\bullet}"]\arrow[d, "\theta^\bullet"] & Y[1]^\bullet\arrow[d, "\text{id}_{Y[1]}^\bullet"] \\
		Y^\bullet\arrow[r, "v^\bullet"'] & Z^\bullet\arrow[r, "i_2^\bullet"'] & C(v)^\bullet\arrow[r, "p_1^\bullet"'] & Y[1]^\bullet
	\end{tikzcd}\quad.
	\]
	is a morphism of triangles in $\mathsf{H^\#\!(A)}$, where $\theta^n(x^{n+1})\coloneqq(-u^{n+1}(x^{n+1}),x^{n+1},0)$. The right square commutes since $p_1^n(\theta^n(x^{n+1}))=-u^{n+1}(x^{n+1})=-u[1]^n(x^{n+1})$, while the middle one does so up to homotopy $k^\bullet:Z^\bullet\rightarrow C(v)^{\bullet-1}$ given by $k^n(x^{n+1},y^n)\coloneqq(y^n,0,0)$. Finally, $p_2^\bullet:C(v)^\bullet\rightarrow X[1]^\bullet$ inverts $\theta^\bullet$: indeed $p_2^\bullet\circ\theta^\bullet=\text{id}_X^\bullet$ and $\text{id}_{C(v)}^n - \theta^n\circ p_2^n =k^{n+1}\circ d_{C(v)}^n+d_{C(v)}^{n-1}\circ k^n$ through the homotopy $k^\bullet:C(v)^\bullet\rightarrow C(v)^{\bullet-1},\, k^n(y^{n+1},x^{n+1},y^n)\coloneqq(y^n,0,0)$. Therefore, $(\text{id}_Y^\bullet,\text{id}_Z^\bullet,\theta^\bullet)$ is an isomorphism between $\Delta_2$ and $\Delta_3$, which confirms the former to be distinguished too.
\end{proof}

\begin{Rem}
	Although implied by (T1)--(T3), we will use Lemma \ref{tr3} to prove the remaining axiom (T3), so an independent proof is necessary in order to avoid a circular argument. But this is a quick job: with reference to diagram \eqref{TR3diag}, we can take $Z^\bullet=C(u)^\bullet$, $Z'^\bullet=C(u')^\bullet$ and natural neighbouring morphisms; then the dashed arrow making the whole commute is $h^\bullet\coloneqq f[1]^\bullet \oplus g^\bullet:Z^\bullet\rightarrow Z'^\bullet$.
\end{Rem}

\begin{Lem}\label{semisplit}
	Let $\mathsf{A}$ be an abelian category. Then any short exact sequence $0\rightarrow X^\bullet\xrightarrow{u^\bullet} Y^\bullet\xrightarrow{v^\bullet} Z^\bullet\rightarrow 0$ in $\mathsf{Kom^\#\!(A)}$ is \textbf{semisplit}\index{short exact sequence!semisplit} --- that is, there exists a family $\{w^n:Z^n\rightarrow Y^n\}_{n\in\mathbb{Z}}$ \textup(not necessarily forming a morphism of complexes\textup) such that $v^n\circ w^n=\textup{id}_{Z^n}$ for all $n\in\mathbb{Z}$ --- if and only if it can be completed in $\mathsf{H^\#\!(A)}$ to a distinguished triangle $X^\bullet\xrightarrow{u^\bullet} Y^\bullet\xrightarrow{v^\bullet} Z^\bullet\rightarrow X[1]^\bullet$.
\end{Lem}

\begin{proof}
	Assume first that the given short exact sequence can be completed to a distinguished triangle, then isomorphic to \eqref{distinguitriangle} by definition. The associated short exact sequence $0\rightarrow X^\bullet\xrightarrow{i_1^\bullet}Z(u)^\bullet\xrightarrow{p_{23}^\bullet}C(u)^\bullet\rightarrow 0$ (see middle row of \eqref{distinguidiag}) is clearly semisplit with canonical inclusion $w^n\coloneqq i_{23}^n:C(u)^n\rightarrow Z(u)^n$ fulfilling $p_{23}^n\circ w^n=\text{id}_{C(u)^n}$ for all $n\in\mathbb{Z}$. Since the original short exact sequence is isomorphic to it, it is semisplit as well.
	
	For the converse direction, suppose $0\rightarrow X^\bullet\xrightarrow{u^\bullet} Y^\bullet\xrightarrow{v^\bullet} Z^\bullet\rightarrow 0$ in $\mathsf{Kom^\#\!(A)}$ semisplits, thus (without loss of generality) that $Y^\bullet=X^\bullet\oplus Z^\bullet$, with canonical injections $u^\bullet=i_1^\bullet$, $w^\bullet=i_2^\bullet$ and projection $v^\bullet=p_2^\bullet$. Then forcedly $d_Y^n(x^n,z^n)=(d_X^n(x^n)-f^n(z^n),d_Z^n(z^n))$ for some family $\{f^n:Z^n\rightarrow X^{n+1}\}_{n\in\mathbb{Z}}$ which must necessarily form a morphism of complexes $f^\bullet:Z^\bullet\rightarrow X[1]^\bullet$ for $d_Y^\bullet$ to square to zero. With these preparations, let us consider the diagram
	\[
	\begin{tikzcd}
		X^\bullet\arrow[r, "i_1^\bullet"]\arrow[d, "\text{id}_X^\bullet"'] & X^\bullet\oplus Z^\bullet\arrow[r, "p_2^\bullet"]\arrow[d, "\text{id}_Y^\bullet"'] & Z^\bullet\arrow[r, "f^\bullet"]\arrow[d, "g^\bullet"] & X[1]^\bullet\arrow[d, "\text{id}_{X[1]}^\bullet"] \\
		X^\bullet\arrow[r, "i_1^\bullet"'] & X^\bullet\oplus Z^\bullet\arrow[r, "i_{23}^\bullet"'] & C(u)^\bullet\arrow[r, "p_1^\bullet"'] & X[1]^\bullet
	\end{tikzcd}\quad,
	\]
	where $g^\bullet\coloneqq(f^\bullet,0,\text{id}_Z^\bullet)$ is a morphism of complexes (easy check) which maps to $C(u)^\bullet=X[1]^\bullet\oplus X^\bullet\oplus Z^\bullet$, endowed with the differential $d_{C(u)}^n(x^{n+1},x^n,z^n)=(-d_X^{n+1}(x^{n+1}),i_1^{n+1}(x^{n+1})+d_X^n(x^n)-f^n(z^n),d_Z^n(z^n))$. Now, commutativity of the side squares is obvious, while that of the central one holds up to homotopy, $i_{23}^n-g^n\circ p_2^n=k^{n+1}\circ d_Y^n+d_{C(u)}^{n-1}\circ k^n$ for $k^\bullet:Y^\bullet\rightarrow C(u)^{\bullet-1},\, k^n(x^n,z^n)\coloneqq(x^n,0,0)$. Finally, $p_3^\bullet:C(u)^\bullet\rightarrow Z^\bullet$ inverts $g^\bullet$, because $p_3^\bullet\circ g^\bullet=\text{id}_Z^\bullet$ and $\text{id}_{C(u)}^n-g^n\circ p_3^n=\bar{k}^{n+1}\circ d_{C(u)}^n + d_{C(u)}^{n-1}\circ \bar{k}^n$ for $\bar{k}^n(x^{n+1},x^n,z^n)\coloneqq(x^n,0,0)$. We conclude that $(\text{id}_X^\bullet,\text{id}_Y^\bullet,g^\bullet)$ is an isomorphism of triangles in $\mathsf{H^\#\!(A)}$, so that the top row is a distinguished triangle.
\end{proof}

We are finally ready to conclude the proof of Theorem \ref{homotopycatistriangulated}.

\begin{proof}[Proof of \textup{(T3)}]
	Consider the initial braid \eqref{commbraid} formed by the red, green and orange distinguished triangles, which by Lemma \ref{semisplit} we know to yield semisplit short exact sequences. Consequently, we can safely assume that $Y^\bullet=X^\bullet\oplus Z'^\bullet$ and $Z^\bullet=Y^\bullet\oplus X'^\bullet=X^\bullet\oplus Z'^\bullet\oplus X'^\bullet$, with corresponding natural injections and projections, as well as semisplitting maps $f^\bullet:Z'^\bullet\rightarrow X[1]^\bullet$ and $(g^\bullet,h^\bullet):X'^\bullet\rightarrow Y[1]^\bullet$ entering the differentials of $Y^\bullet$ respectively $Z^\bullet$ (compare to the proof of lemma). So our braid updates to
	\[
	\begin{tikzcd}[column sep=0.2cm]
		X^\bullet \arrow[rr, orange, bend left = 45, "{i_1^\bullet}" black]\arrow[dr, red, "{i_1^\bullet}"' black] & \mkern-6mu\circlearrowleft & \mkern-24mu{X^\bullet\!\oplus\! {Z'^\bullet}\!\oplus\! {X'^\bullet}}\arrow[rr, green, bend left = 45, "{p_3^\bullet}" black]\arrow[dr, orange, end anchor={[xshift=-1.2ex, yshift=1.5ex]west}, "{w'^\bullet}" black] & & {X'^\bullet}\arrow[rr, cyan, bend left = 45, "{a''^\bullet}" black]\arrow[dr, green, end anchor={[xshift=1.3ex, yshift=1.5ex]west}, "{(g^\bullet,h^\bullet)}" black] & \mkern+12mu\circlearrowleft & {Z'[1]}^\bullet \\
		& X^\bullet\oplus Z'^\bullet\arrow[ur, green, start anchor={[xshift=-0.5ex,yshift=1.5ex]east}, "i_{12}^\bullet"' black]\arrow[dr, red, start anchor={[xshift=2ex]south}, end anchor={[xshift=-5.5ex, yshift=1.5ex]west}, "p_2^\bullet"' black] & & \mkern-12mu Y'^\bullet \arrow[dr, orange, start anchor={[yshift=0ex]south}, "w''^\bullet" black] & & {(X\oplus Z')[1]}^\bullet\arrow[ur, "{p_2[1]}^\bullet"'] \\
		& & \mkern-48mu Z'^\bullet\arrow[rr, red, bend right = 45, start anchor={[xshift=-2ex]south}, "f^\bullet"' black] & & {X[1]}^\bullet\arrow[ur, "{i_1[1]^\bullet}"']
	\end{tikzcd}.
	\]
	Now we are tempted to set $Y'^n=Z'^n\oplus X'^n$ with differential $d_{Y'}^n(z'^n,x'^n)=(d_{Z'}^n(z'^n)-h^n(x'^n),d_{X'}^n(x'^n))$. This does define a complex $(Y'^\bullet,d_{Y'}^\bullet)$ in $\mathsf{H^\#\!(A)}$, which moreover respects the semisplitting of the orange distinguished triangle. So the dashed blue arrows completing the diagram are the natural morphisms $a^\bullet\coloneqq i_1^\bullet:Z'^\bullet\rightarrow Y'^\bullet$ and $a'^\bullet\coloneqq p_2^\bullet: Y'^\bullet \rightarrow X'^\bullet$, while $w'^\bullet= p_{23}^\bullet$ and $w''^\bullet=(f^\bullet\oplus g^\bullet)$, so $w''^n(z'^n,x'^n)=f^n(z'^n) + g^n(x'^n)$. Lemma \ref{semisplit} applies to the semisplit short exact sequence $0\rightarrow Z'^\bullet\rightarrow Y'^\bullet\rightarrow X'^\bullet\rightarrow 0$ just constructed to confirm that $Z'^\bullet\xrightarrow{a^\bullet} Y'^\bullet\xrightarrow{a'^\bullet} X'^\bullet\xrightarrow{a''^\bullet} Z'[1]^\bullet$ is a distinguished triangle in $\mathsf{H^\#\!(A)}$ (in particular, $a''^\bullet$ is precisely the same morphism making the top right triangle commute).
	
	We must still verify that the completed diagram does commute. For the remaining two triangles this is immediate. For the left square, we easily see that $a^\bullet\circ p_2^\bullet=p_{23}^\bullet\circ i_{12}^\bullet$. The right one instead commutes only up to homotopy $k^\bullet: Y'^\bullet\rightarrow Y[1]^{\bullet-1}=Y^\bullet$ given by $k^n(z'^n,x'^n)\coloneqq(0,z'^n)$. Indeed we compute:
	\begin{align*}
		\big(k^{n+1}\!\circ\! d_{Y'}^n - d_Y^n\!\circ\! k^n\big)(z'^n,x'^n) \!&= k^{n+1}(d_{Z'}^n(z'^n)-h^n(x'^n),d_{X'}^n(x'^n))\!-\!d_Y^n(0,z'^n) \\
		&= (0,d_{Z'}^n(z'^n)-h^n(x'^n))-(-f^n(z'^n),d_{Z'}^n(z'^n))\\
		&= (f^n(z'^n)+g^n(x'^n),0) - (g^n(x'^n),h^n(x'^n)) \\
		&= \big(i_1[1]^n\circ(f^n\oplus g^n)-(g^n,h^n)\circ a'^n\big)(z'^n,x'^n)
	\end{align*}
	Therefore, axiom (T3) holds and, together with the previous proofs, we conclude that $\mathsf{H^\#\!(A)}$ has the structure of a triangulated category. 
\end{proof}

\subsection{Triangularity of derived categories}

We wish to take our discovery that homotopy categories of complexes are triangulated one step further. To pass this result to derived categories we must, however, account for the localization process. This imposes some further requirements on the localizing class of morphisms involved. We deal first with general triangulated categories and focus on derived categories only later on.

\begin{Def}
	Let $\mathsf{D}=(\mathsf{D},T)$ be a triangulated category and $S$ a localizing class of morphisms within it (see Definition \ref{locclass}). Then $S$ is \textbf{compatible with triangulation}\index{localizing class of morphisms!compatible with triangulation} if $\big(s\in S \Leftrightarrow T(s)\in S\big)$ and $\big(f,g\in S\text{ in }\eqref{TR3diag}\implies h\in S\big)$ hold.
\end{Def}

\begin{Thm}\label{derivedcatistriangulated}
	Let $\mathsf{D}=(\mathsf{D},T)$ be a triangulated category and $S$ a localizing class of morphisms in $\mathsf{D}$ which is compatible with its triangulation. Let $\mathsf{D}_S\coloneqq\mathsf{D}[S^{-1}]$ denote the associated derived category with localization functor $\mathsf{L}:\mathsf{D}\rightarrow\mathsf{D}_S$, as constructed in the proof of Proposition \ref{derivedexists} and described in Proposition \ref{derivedstructure}. Endow $\mathsf{D}_S$ with the induced translation functor $T_S:\mathsf{D}_S\rightarrow\mathsf{D}_S$, fulfilling $T=T_S$ on $\textup{obj}(\mathsf{D}_S)=\textup{obj}(\mathsf{D})$ and $T_S\circ\mathsf{L}=\mathsf{L}\circ T$.
	
	Then $\mathsf{D}_S$, together with $T_S$ and as distinguished triangles all those isomorphic to some $\mathsf{L}(\Delta)$ for $\Delta$ a distinguished triangle in $\mathsf{D}$, is a triangulated category. 
\end{Thm}

Due to its length, we chop again the proof into several parts. First recall that any $u\in\text{Hom}_{\mathsf{D}_S}(X,Y)$ is an equivalence class represented by a roof $X\xleftarrow{s} Z\xrightarrow{u'} Y$ in $\mathsf{D}$ with $s\in S$. Then by definition of $T_S$ we have that $T_S(u)\in\text{Hom}_{\mathsf{D}_S}(T(X),T(Y))$ is represented by the roof $T(X)\xleftarrow{T(s)} T(Z)\xrightarrow{T(u')} T(Y)$ in $\mathsf{D}$ with $T(s)\in S$ (since $S$ is assumed compatible), that is, $T_S([(s,u')])=[(T(s),T(u'))]$. By functoriality of $T$, this relation respects the defining equivalence relation \eqref{equivroofs} between roofs.

\begin{proof}[Proof of \textup{(T1)}]
	The triangle $\big(X\!\xrightarrow{[(\text{id}_X,\text{id}_X)]} \!X\!\xrightarrow{[(\text{id}_X,0)]} 0\xrightarrow{[(\text{id}_0,0)]} \!X[1]\!=\!T_S(X)\big)$, for any given $X\in\text{obj}(\mathsf{D}_S)$, is distinguished because it is isomorphic to\break $\mathsf{L}\big(X\xrightarrow{\text{id}_X} X\rightarrow 0\rightarrow X[1]\big)$, that is, the image of the distinguished triangle from (T1)a. for $\mathsf{D}$.
	
	Part b. is clear as well: if $\Delta_2$ is isomorphic to the distinguished triangle $\Delta_1\cong\mathsf{L}(\Delta)$ in $\mathsf{D}_S$, for $\Delta$ distinguished in $\mathsf{D}$, then $\Delta_2\cong\mathsf{L}(\Delta)$ is clearly a distinguished triangle as well.
	
	Part c. is a little less trivial. Consider $u=[(s,u')]\in\text{Hom}_{\mathsf{D}_S}(X,Y)$, represented by $X\xleftarrow{s} Z\xrightarrow{u'} Y$, and complete $u'\in\text{Hom}_\mathsf{D}(Z,Y)$ to a distinguished triangle $\Delta=\big(Z\xrightarrow{u'} Y\xrightarrow{v} U\xrightarrow{w} Z[1]\big)$ in $\mathsf{D}$. Then $\mathsf{L}(\Delta)$ fits into the top row of the diagram
	\[
	\begin{tikzcd}
		Z\arrow[r, "u'"]\arrow[d, "s"'] & Y\arrow[r, "v"]\arrow[d, "\text{id}_Y"'] & U\arrow[r, "w"]\arrow[d, "\text{id}_U"] & Z[1]\arrow[d, "{s[1]}"] \\
		X\arrow[r, "u"'] & Y\arrow[r, "v"'] & U\arrow[r, "{s[1]\circ w}"'] & X[1]
	\end{tikzcd}
	\]
	in $\mathsf{D}_S$, where each morphism except $u$ should be understood as being of the form $[(\text{id},\square)]$ (for example $u'\equiv[(\text{id}_Z,u')]\in\text{Hom}_{\mathsf{D}_S}(Z,Y)$). Using the composition law \eqref{roofcompo} for roofs, we easily deduce that this diagram commutes. Hence $(s,\text{id}_Y,\text{id}_U)$ is a morphism of triangles, and particularly an isomorphism, because $s$ is formally invertible in $\mathsf{D}_S$. The bottom row, then isomorphic to $\mathsf{L}(\Delta)$, is the desired completion of $u$ to a distinguished diagram in $\mathsf{D}_S$. 
\end{proof}

\begin{proof}[Proof of \textup{(T2)}]
	This is evident. Just expoit that $T_S\circ\mathsf{L}=\mathsf{L}\circ T$ and pass through the left/right shift of the associated distinguished triangles in $\mathsf{D}$. 
\end{proof}

Just like in the previous section, the axiom (TR3) plays a role into the proof of (T3), so we must check it beforehand.

\begin{proof}[Proof of \textup{(TR3)}]
	Since the descending arrows of \eqref{TR3diag} in $\mathsf{D}_S$ subtend each a representative roof, one should imagine a third row of objects and morphisms $(s,X'',\tilde{f})\rightarrow(t,Y'',\tilde{g})\rightarrow(r,Z'',\tilde{h})\rightarrow(s[1],X''[1],\tilde{f}[1])$ raising vertically from the diagram itself to form an elongated ``three-dimensional'' roof in $\mathsf{D}$ (see picture below, where $(f,g,h)$ in $\mathsf{D}_S$ are included for clarity). The dashed pair $(r,\tilde{h})$ representing $h\in\text{Hom}_{\mathsf{D}_S}(Z,Z')$ is what must be constructed. 
	
	Let us build the ``main beam'' $X''\xrightarrow{u''}Y''\xrightarrow{v''}Z''\xrightarrow{w''}X''[1]$ substaining our big roof. We first use condition b. of Definition \ref{locclass} on $X''\xrightarrow{u\circ s} Y\xleftarrow{t} Y''$ to construct two complementing arrows $X''\xleftarrow{\tilde{t}}\tilde{X}\xrightarrow{\tilde{u}}Y''$ in $\mathsf{D}$ with $\tilde{t}\in S$ and such that $(u\circ s)\circ\tilde{t}=t\circ\tilde{u}:\tilde{X}\rightarrow Y$ (cf. diagram \eqref{localsquare} left). The situation then looks like this:
	\[
	\begin{tikzcd}[column sep=0.15cm]
		& \tilde{X}\arrow[d, dashed, "\tilde{t}"']\arrow[drrrrr, dashed, "\tilde{u}"] & & & & & & & & & & & & & \\
		& X''\arrow[rrrrr, dashed, "u''"]\arrow[dr, "s"]\arrow[ddl, near start, "\tilde{f}"'] & & & & & Y''\arrow[dr, "t"] & & & & & Z''\arrow[dr, dashed, "r"] & & & & & \mkern-18mu X''[1]\arrow[dr, start anchor={[xshift=-2ex, yshift=-1.5ex]east}, end anchor={[xshift=-0ex, yshift=2ex]west}, "{s[1]}"] & \\
		& & X\arrow[rrrrr, near start, "u"]\arrow[dll, pos=0.5, "f"'] & & & & & Y\arrow[dll, pos=0.5, "g"']\arrow[rrrrr, near start, "v"] & & & & & Z\arrow[rrrrr, near start, end anchor={[xshift=-3.3ex]west}, "w"]\arrow[dll, dashed, pos=0.15, "h"] & & & & & \mkern-24mu X[1]\arrow[dll, pos=0.15, start anchor={[xshift=-2ex, yshift=-1.5ex]west}, end anchor={[xshift=-0.5ex, yshift=2ex]east}, "{f[1]}"]  \\
		X'\arrow[rrrrr, "u'"'] & & & & & Y'\arrow[from=uur, near start, crossing over, "\tilde{g}"']\arrow[rrrrr, "v'"'] & & & & & Z'\arrow[from=uur, near start, dashed, crossing over, "\tilde{h}"']\arrow[rrrrr, "w'"'] & & & & & X'[1]\arrow[from=uur, crossing over, near start, start anchor={[xshift=-4ex, yshift=-1.5ex]east}, "{\tilde{f}[1]}"']  & &
	\end{tikzcd}
	\]
	With $\tilde{X}$ sitting atop $X''$, we clearly see that the roof $X\xleftarrow{s\circ\tilde{t}}\tilde{X}\xrightarrow{\tilde{f}\circ\tilde{t}} X'$ represents $f:X\rightarrow X'$ as well (check with \eqref{equivroofs}), so that we can replace $(s,X'',\tilde{f})$ with it. The back square commutes by construction; for the front one, observe that
	\[
	u'\circ\tilde{f}\circ s^{-1}= u'\circ f=g\circ u=\tilde{g}\circ t^{-1}\circ u=\tilde{g}\circ\tilde{u}\circ (s\circ\tilde{t})^{-1}=\tilde{g}\circ\tilde{u}\circ\tilde{t}^{-1}\circ s^{-1}
	\]
	(the second equality follows by assumption), whence $u'\circ\tilde{f}\circ\tilde{t}=\tilde{g}\circ\tilde{u}:\tilde{X}\rightarrow Y'$, as desired (formally this holds only in $\mathsf{D}_S$, but then there exists a common $q\in S$ such that precomposition by it makes the equation valid in $\mathsf{D}$ as well). Therefore, we rename $\tilde{X}$ to $X''$ and set $u''\coloneqq\tilde{u}\in\text{Hom}_\mathsf{D}(X'',Y'')$. 
	
	Now, axiom (T1)c. for $\mathsf{D}$ allows us to complete $u''$ to a distinguished triangle $X''\xrightarrow{u''}Y''\xrightarrow{v''}Z''\xrightarrow{w''}X''[1]$ in $\mathsf{D}$. Then Lemma \ref{tr3} applied to both slopes of the resulting big roof yields dashed $r\in\text{Hom}_\mathsf{D}(Z'',Z)$ respectively $\tilde{h}\in\text{Hom}_\mathsf{D}(Z'',Z')$, with $r\in S$ because $S$ is compatible with triangulation. We conclude that $h\coloneqq[(r,\tilde{h})]\in\text{Hom}_{\mathsf{D}_S}(Z,Z')$ makes the triple $(f,g,h)$ a valid morphism of distinguished triangles in $\mathsf{D}_S$. 
\end{proof}

\begin{proof}[Proof of \textup{(T3)}]
	This time we will favour the octahedral diagram \eqref{octahedraldiag}, which makes the situation visually more clear. Namely, we will gradually construct the following figure:\footnote{Be appreciative... was hell to draw...}
	\[
	\begin{tikzcd}[column sep=0.5cm, row sep=0.3cm]
		& & \mkern-6mu Y\arrow[ddddll, red, start anchor={[xshift=-1ex, yshift=-1ex]west}, end anchor={[xshift=0.5ex]north}, "f'"' ]\arrow[ddrr, green, "g"] & & \\
		& & & & \\
		& \mkern+12mu X'\arrow[uur, green, "{g''[1]}"]\arrow[ddl, cyan, start anchor={[xshift=2.5ex, yshift=-1ex]west}, end anchor={[xshift=-0.6ex, yshift=1.5ex]east}, "{a''[1]}"']\arrow[from=ddd, cyan, near end, "r'"'] & & & Z\arrow[lll, green, near start, "g'"'] \\
		& & & & \\
		Z'\arrow[rrr, red, crossing over, "{f''[1]}"'] & & & X\arrow[uuuul, red, crossing over, pos=0.2, "f"]\arrow[uur, orange, "h"'] \\
		& X''\arrow[ddl, dashed, pos=0.2, "{\tilde{a}''[1]}"']\arrow[dr, dashed] & & & Z\arrow[uuu, orange, "\text{id}_Z"']\arrow[lll, dashed]\arrow[dddll, orange, pos=0.2, "\tilde{h}'"'] \\
		& & Y''\arrow[urr, dashed, pos=0.2, "q"]\arrow[uuuuuu, dashed, pos=0.6, "t"] & & \\
		Z''\arrow[uuu, cyan, "r"]\arrow[rrr, dashed, crossing over]\arrow[drr, cyan, dashed, "\tilde{a}"']\arrow[from=urr, dashed, crossing over] & & & \tilde{X}\arrow[uuu, orange, crossing over, pos=0.8, "s"']\arrow[uur, dashed, pos=0.2, "\tilde{h}"']\arrow[ul, dashed, pos=0.7, end anchor={[yshift=-1ex]east}, "p"] & \\
		& & Y'\arrow[uuul, dashed, cyan, pos=0.85, "\tilde{a}'"]\arrow[ur, orange, "{\tilde{h}''[1]}"'] & & 
	\end{tikzcd}
	\]
	The starting point is the upper pyramid in $\mathsf{D}_S$, formed by the distinguished triangles $(f,f',f'')$ and $(g,g',g'')$, with $h=g\circ f$ and $a''=f'\circ g''$; our goal is to construct a lower pyramid with two more distinguished triangles (orange and blue) and commuting lateral facets. As hinted, we will achieve this through an intermediate floor (dashed).
	
	Suppose first that $f$, $g$ are represented by the roofs $X\xleftarrow{s} U\xrightarrow{p} Y$ respectively $Y\xleftarrow{t}Y''\xrightarrow{q} Z$ for $s,t\in S$. According to diagram \eqref{roofcompo}, $h=g\circ f$ is represented by $X\xleftarrow{s\circ t'} \tilde{X}\xrightarrow{q\circ p'} Z$ for some capping $U\xleftarrow{t'}\tilde{X}\xrightarrow{p'} Y''$ such that $t'\in S$ and $p\circ t'=t\circ p'$. Then since also $(s\circ t', \tilde{X}, p\circ t')$ represents $f$, we may overwrite $U\leftarrow\tilde{X}$, $s\circ t'\leftarrow s$, $p'\leftarrow p$ and set $\tilde{h}\coloneqq q\circ p:\tilde{X}\rightarrow Z$, as reported in our octahedron.
	
	Now complete $p\in\text{Hom}_\mathsf{D}(\tilde{X},Y'')$ and $q\in\text{Hom}_\mathsf{D}(Y'',Z)$ to distinguished triangles $\tilde{X}\xrightarrow{p} Y''\rightarrow Z''\rightarrow \tilde{X}[1]$ respectively $Y''\xrightarrow{q} Z\rightarrow X''\rightarrow Y''[1]$ in $\mathsf{D}$, and trace some $\tilde{a}''[1]: X''\rightarrow Z''$ by imposing commutativity. This effectively gives an upper pyramid in $\mathsf{D}$ (dashed black). We claim it to be isomorphic in $\mathsf{D}_S$ to our original upper pyramid, possible if and only if the diagrams
	\[
	\begin{tikzcd}
		\tilde{X}\arrow[r, "p"]\arrow[d, "s"'] & Y''\arrow[r]\arrow[d, "t"'] & Z''\arrow[r]\arrow[d, dashed, "r"] & \tilde{X}[1]\arrow[d, "{s[1]}"] \\
		X\arrow[r, "f"'] & Y\arrow[r, "f'"'] & Z'\arrow[r, "f''"'] & X[1]
	\end{tikzcd}\qquad
	\begin{tikzcd}
		Y''\arrow[r, "q"]\arrow[d, "t"'] & Z\arrow[r]\arrow[d, "\text{id}_Z"'] & X''\arrow[r]\arrow[d, dashed, "r'"] & Y''[1]\arrow[d, "{t[1]}"] \\
		Y\arrow[r, "g"'] & Z\arrow[r, "g'"'] & X'\arrow[r, "g''"'] & Y[1]
	\end{tikzcd}
	\]
	are isomorphisms of distinguished triangles in $\mathsf{D}_S$ (as before, $s\equiv[(\text{id}_{\tilde{X}},s)]\in\text{Hom}_{\mathsf{D}_S}(\tilde{X},X)$ is implicitly understood). Focusing on the first one, whose left square commutes, we see that axiom (TR3) for $\mathsf{D}_S$ (shown above!) produces the missing arrow $r\in\text{Hom}_{\mathsf{D}_S}(Z'',Z')$. Since $S$ is compatible with triangulation, we further have $r\in S$, so that all vertical arrows are formally invertible in $\mathsf{D}_S$. This makes $(s,t,r)$ into an isomorphism. The argument for $(t,\text{id}_Z,r')$ is analogous.
	
	So we can simply focus on the black dashed upper pyramid. But this belonging to $\mathsf{D}$ immediately gives by axiom (T3) a lower counterpart as drawn above. Setting $a\coloneqq\tilde{a}\circ r^{-1}\in\text{Hom}_{\mathsf{D}_S}(Z',Y')$, $a'\coloneqq r'\circ\tilde{a}'\in\text{Hom}_{\mathsf{D}_S}(Y',X')$, $h'\coloneqq \tilde{h}'\circ\text{id}_Z\in\text{Hom}_{\mathsf{D}_S}(Z,Y')$ and $h''\coloneqq s[1]\circ \tilde{h}''\in\text{Hom}_{\mathsf{D}_S}(Y',X[1])$, we obtain a lower pyramid in $\mathsf{D}_S$ with the desired properties: $(h,h',h'')$, $(a,a',a'')$ form distinguished triangles in $\mathsf{D}_S$ (being these isomorphic to the image under localization of distinguished ones in $\mathsf{D}$), and the adjacent triangles commute by construction, $a'\circ h'=g'$ and $h''\circ a=f''$ (cf. last pair of diagrams).
	
	This also concludes the proof of Theorem \ref{derivedcatistriangulated}.    
\end{proof}

\begin{Cor}\label{D(A)istriangulated}
	The derived category $\mathsf{D^\#\!(A)}$ \textup(for $\#=\emptyset,+,-,\mathsf{b}$\textup) of an abelian category $\mathsf{A}$ is triangulated.
\end{Cor}

\begin{proof}
	By Theorem \ref{homotopycatistriangulated}, we already know that $\mathsf{H^\#\!(A)}$ together with the translation functor $T:\mathsf{H^\#\!(A)}\rightarrow\mathsf{H^\#\!(A)}$ and all distinguished triangles of complexes is a triangulated category. So according to Theorem \ref{derivedcatistriangulated}, it suffices to check that the localizing class $S\coloneqq\{\text{quasi-isomorphisms in }\mathsf{H^\#\!(A)}\}$ is compatible with its triangulated structure. Obviously, $s^\bullet\in S$ if and only if $T(s^\bullet)=s^{\bullet+1}\in S$. The condition $\big(f^\bullet,g^\bullet\in S\text{ in }\eqref{TR3diag}\implies h^\bullet\in S\big)$ is easily verified as well: just use Lemma \ref{tr3} to get some $h^\bullet\in\text{Hom}_\mathsf{H^\#\!(A)}(Z^\bullet,Z'^\bullet)$ and then apply the Five-Lemma to each quintet of vertical arrows $\big(H^n(f^\bullet),H^n(g^\bullet),H^n(h^\bullet),H^{n+1}(f^\bullet), H^{n+1}(g^\bullet)\big)$ connecting the long exact sequences in cohomology provided by Theorem \ref{longexactcohomtriangles}; each $H^n(h^\bullet):H^n(Z^\bullet)\rightarrow H^n(Z'^\bullet)$ is then an isomorphism, thus $h^\bullet\in S$.
	
	By Theorem \ref{derivedcatistriangulated}, $\mathsf{D^\#\!(A)}=\mathsf{H^\#\!(A)}[S^{-1}]$ then becomes a triangulated category when equipped with the induced translation functor $T_S:\mathsf{D^\#\!(A)}\rightarrow\mathsf{D^\#\!(A)}$ and as distinguished triangles all those isomorphic to the image under the localization functor $\mathsf{L}:\mathsf{H^\#\!(A)}\rightarrow\mathsf{D^\#\!(A)}$ of some distinguished triangle of complexes in $\mathsf{H^\#\!(A)}$.
\end{proof}

Having just shown that derived categories of complexes are triangulated, so that Definition \ref{cohomfuncontriangulated} for $\mathsf{D}=\mathsf{D^\#\!(A)}$ coincides with Definition \ref{cohomfunconabelian}, Proposition \ref{Homareexact} immediately implies the following corollary. 

\begin{Cor}\label{Homarecohomological}
	Let $\mathsf{A}$ be an abelian category, $X^\bullet\in\textup{obj}(\mathsf{D^\#\!(A)})$. By Proposition \ref{Homareexact}, the covariant functor $\textup{Hom}_\mathsf{D^\#\!(A)}(X^\bullet,\square):\mathsf{D^\#\!(A)}\rightarrow\mathsf{A}$ and the contravariant functor $\textup{Hom}_\mathsf{D^\#\!(A)}(\square,X^\bullet):\mathsf{D^\#\!(A)}\rightarrow\mathsf{A}$ are cohomological in the abelian sense.
	
	Moreover, by Theorem \ref{longexactcohomtriangles}, the cohomology functors $H^n:\mathsf{D^\#\!(A)}\rightarrow\mathsf{A}$ are cohomological in the triangulated sense as well.
\end{Cor}

Looking back, we observe that this result completes the proof of Theorem \ref{Extprop}a.\!\! . The triangularity of $\mathsf{D^\#\!(A)}$ also allows us to take the cone $C(u)^\bullet\in\text{obj}(\mathsf{D^\#\!(A)})$ completing any $u^\bullet\in\text{Hom}_\mathsf{D^\#\!(A)}(X^\bullet, Y^\bullet)$ (unique up to isomorphism by Definition \ref{conesintriangulated}) to be precisely the cone of complexes $C(u)^\bullet=X[1]^\bullet\oplus Y^\bullet$ from Definition \ref{coneofcomplexes}. Everything lines up neatly. Furthermore, we have:
	
\begin{Cor}
	Let $\mathsf{A}$ be an abelian category. Then any short exact sequence $0\rightarrow X^\bullet\xrightarrow{u^\bullet} Y^\bullet\xrightarrow{v^\bullet} Z^\bullet\rightarrow 0$ in $\mathsf{Kom^\#\!(A)}$ can be completed to a distinguished triangle $X^\bullet\xrightarrow{u^\bullet} Y^\bullet\xrightarrow{v^\bullet} Z^\bullet\xrightarrow{w^\bullet} X[1]^\bullet$ in $\mathsf{D^\#\!(A)}$ by an appropriate choice of $w^\bullet$, and any distinguished triangle in $\mathsf{D^\#\!(A)}$ is isomorphic to one of this form. 
\end{Cor}

\begin{proof}
	Notice that completing $u^\bullet$ to a distinguished triangle of complexes via axiom {(T1)c.} for $\mathsf{D^\#\!(A)}$ doesn't do the trick --- this would erase the data $(v^\bullet,Z^\bullet)$! However, by Proposition \ref{shortexactaretriangles} we know that $0\rightarrow X^\bullet\rightarrow Y^\bullet\rightarrow Z^\bullet\rightarrow 0$ is quasi-isomorphic to the middle row $0\rightarrow X^\bullet\xrightarrow{i_1^\bullet}Z(u)^\bullet\xrightarrow{p_{23}^\bullet}C(u)^\bullet\rightarrow 0$ of diagram \eqref{distinguidiag}, which clearly semisplits through $i_{23}^\bullet:C(u)^\bullet\rightarrow Z(u)^\bullet$. So Lemma \ref{semisplit} provides the desired $w^\bullet:Z^\bullet\rightarrow X[1]^\bullet$.
	
	Again by Lemma \ref{semisplit}, any distinguished $X^\bullet\xrightarrow{u^\bullet} Y^\bullet\xrightarrow{v^\bullet} Z^\bullet\xrightarrow{w^\bullet} X[1]^\bullet$ in $\mathsf{D^\#\!(A)}$ is ultimately isomorphic to a distinguished triangle stemming from a semisplit short exact sequence as highlighted above. 
\end{proof}

\newpage

\part{Category Theory of Sheaves}
\markboth{II\quad CATEGORY THEORY OF SHEAVES}{II\quad CATEGORY THEORY OF SHEAVES}
\thispagestyle{plain}

\vspace*{2cm}

The concept of sheaf started evolving at slow pace around the 1930's, stemming from the theory of cohomology, and was initially finalized to very specific settings by the various algebraic topologists of that epoch, among whom Leray, Cartan and Serre. In the mid 50's, also thanks to the contributions of Grothendieck and Zariski, its relevance for algebraic geometry became increasingly apparent, elevating sheaf theory to a free-standing mathematical discipline. It eventually provided algebraic geometers with powerful analytical tools such as coherent sheaf cohomology, of which we will have some taste.

On the other hand, schemes were pioneered by Grothendieck in his \textit{Éléments de Géométrie Algébrique} (\textit{EGA}) of 1960, laying the first brick of a theory which now efficiently unifies most of algebraic geometry and number theory, appealing to tools from both commutative algebra, topology and homological algebra. He defined affine schemes as being spectra of commutative rings equipped with a sheaf of rings (which in turn fomented the exploit of sheaf theory), then ``gluing'' together to form schemes with a wide range of properties.

In these two short paragraphs, we can already appreciate the many historical synergies that shaped algebraic geometry. Thanks to this, the B-side of homological mirror symmetry relies on a more traditional and firmer set of tools than the way more convoluted and comparatively young Fukaya categories (which appeared no earlier than the 90's), making it the better understood face of the mirror duality. Accordingly, we will see that the following theory is somewhat more digestible than that presented in the second half of \cite{[Imp21]}, though not without a few compromises, especially when it comes to schemes.      

This considerations aside, the goal of Part II is two-fold. On the one hand, we wish to learn about the algebraic geometry of sheaves, gradually restricting our attention to the framework of coherent sheaves on schemes, and hence on Calabi--Yau varieties; on the other, we work out the simplest known manifestations of homological mirror symmetry, which morally concludes our journey. To be more specific: 

\begin{itemize}[leftmargin=0.5cm]
\item In Chapter 5, we introduce the category of (pre)sheaves of abelian groups on a topological space and prove that it is abelian. Then we expand our vocabulary to include the essential working tools: pushforward, inverse image, skyscraper and local $\mathcal{H}om$- sheaves. In fact, we can study them in the broader context of sheaves of modules on a ringed space, where dual, tensor and pullback sheaves join our cast. Sometimes it is possible to express sheaves through copies of the structure sheaf equipping the underlying ringed space, in which case we speak of free and locally free sheaves. Unfortunately, these do not form abelian categories, so that one must enlarge them to capture (quasi-)coherent sheaves.
\newline But here we pause in order to refresh our memory on varieties and schemes. In particular, by choosing as fixed base the spectrum of our ground field $\mathbb{K}$, we see that schemes give rise to a category into which that of varieties over $\mathbb{K}$ embeds, so that we are encouraged to just look at schemes without loss of generality! After having listed all imaginable decorations they admit, we return to quasi-/coherent sheaves, now considered over (noetherian) schemes; in this particular setting, it is easier to argue that the arising categories are indeed abelian. Finally, we go again over all known types of sheaves and functors they induce, inspecting how well these behave with respect to coherence.     

\item In Chapter 6, we momentarily zoom back to sheaves of modules on topological spaces to see which right and left derived functors they produce, following the recipe of Part I. In this wider context, we introduce sheaf cohomology, which can be more conveniently computed through \v{C}ech cohomology; we see how natural this comparison is and on which spaces the two theories are actually isomorphic (luckily, on all those we care about). Then we get back to coherent sheaves on schemes and their bounded derived category $\mathsf{D^b}(X)$, and construct derived functors on it, bypassing all undesired obstructions by going through the better-behaved category of quasi-coherent sheaves.
\newline We also reserve a few pages to two famous investigational tools for smooth projective schemes, namely Serre duality and Fourier--Mukai theory. This set aside, we finally explain how to identify Calabi--Yau manifolds as ``analytified'' smooth projective varieties, thus reassuring ourselves that all nice properties and functors discussed that far are legal in the Calabi--Yau setting as well. In fact, we take the occasion to summarize the main features of $\mathsf{D^b}(X)$, now precisely the object appearing in Kontsevich's conjecture, which we finally restate the way it appears in \cite{[Imp21]}.     

\item In Chapter 7, we look at how the meticolous work of two theses' worth of theory drastically simplifies in low dimensions. We begin by describing Calabi--Yau 1-folds, which are elliptic curves, clarifying what shape (indecomposable) coherent sheaves thereon can possibly assume, and hence their B-side. Secondly, we explain their A-side, which sees them as just 2-tori and boils down to the Fukaya--Kontsevich category, the better-behaved evolution of Fukaya's construction. The latter is what enters the statement of homological mirror symmetry in dimension 1, whose proof is subsequently sketched. Finally, we say a few words for the 2-dimensional case of $K3$ surfaces.
\end{itemize}   

Our default reference for everything about algebraic geometry, from sheaves to schemes, is Hartshorne's masterful book \cite{[Har77]}, with some help by \cite{[Vak17]}. But the more homological content still referring to derived category theory is again borrowed from \cite{[GM03]}. When it comes to a more focused characterization of $\mathsf{D^b}(X)$ for noetherian schemes, and Fourier--Mukai transforms, we shift to \cite{[Huy06]}. The 1-dimensional case of elliptic curves follows instead the analyses \cite{[PZ00]} and \cite{[Kre00]} of Polishchuk, Zaslow and Kreussler. 

Nevertheless, we emphasize that this selection of literature is just the tip of a much bigger iceberg; some foundational works by other illustrious mathematicians like Grothendieck, Serre and Mukai (to just name a few) are credited only indirectly (if not at all), but can hopefully be back-traced from the sources used. Of course, the same can be said of Part I, and really of any far-reaching mathematical theory like Homological Mirror Symmetry.

\newpage
\pagestyle{fancy}

\section{The category of sheaves}
\thispagestyle{plain}

\subsection{The structure of abelian category}\label{ch5.1}

We begin by introducing the fundamental objects of investigation in the second part of this survey, sheaves, aiming at proving that they form an abelian category. These can be defined on any underlying category, but we will ultimately focus on sheaves of abelian groups and of modules. For a thorough introduction, the reader is referred to \cite[section II.1]{[Har77]}.

\begin{Def}									% SHEAF
Let $X$ be a topological space, $\mathsf{C}$ a category. A \textbf{presheaf $\mathcal{F}$ on $X$ with values in $\mathsf{C}$}\index{presheaf!on a topological space with values in $\mathsf{C}$} consists of:
\begin{itemize}[leftmargin=0.5cm]
	\item an object $\mathcal{F}(U)=\Gamma(U,\mathcal{F})\in\textup{obj}(\mathsf{C})$ for each open subset $U\subset X$, whose elements $s\in\mathcal{F}(U)$ we call \textbf{sections}\index{section} of $\mathcal{F}(U)$, and
	\item \textbf{restriction maps}\index{restriction map} $\rho^U_V\in\textup{Hom}_\mathsf{C}(\mathcal{F}(U),\mathcal{F}(V))$ for all pairs of open subsets $V\subset U\subset X$, whose action on sections we also write as $\rho^U_V(s)=s|_V$.
\end{itemize}
These data must fulfill: $\mathcal{F}(\emptyset)=0$, $\rho_U^U=\textup{id}_{\mathcal{F}(U)}$ for all open $U\subset X$, and $\rho^V_W\circ\rho^U_V=\rho^U_W$ for all triples of open subsets $W\subset V\subset U\subset X$.

The presheaf $\mathcal{F}$ is a \textbf{sheaf on $X$ with values in $\mathsf{C}$}\index{sheaf!!on a topological space with values in $\mathsf{C}$} if the sheaf property (or ``gluing'' property) holds: for any open subset $U\subset X$ with open cover $\{U_i\}_{i\in I}$ so that all sections $s_i\in\mathcal{F}(U_i)$ are compatible, $s_i|_{U_i\cap U_j} = s_j|_{U_i\cap U_j}$, there exists a unique $s\in\mathcal{F}(U)$ such that $s|_{U_i}=s_i$ for all $i\in I$. 
\end{Def}

\begin{Def}									% CATEGORY OF SHEAVES
Let $X$ be a topological space endowed with sheaves $\mathcal{F}$ and $\mathcal{G}$ with values in some category $\mathsf{C}$. A \textbf{morphism of sheaves}\index{morphism!of sheaves} $\chi:\mathcal{F}\rightarrow\mathcal{G}$ over $X$ is a collection of maps $\{\chi_U\in\textup{Hom}_\mathsf{C}(\mathcal{F}(U),\mathcal{G}(U))\mid U\subset X\text{ open}\}$ fulfilling\footnote{Here and in the following, $\rho$ is implicitly taken to denote the restriction maps of all sheaves, indiscriminately.}
\[
\rho^U_V\circ\chi_U=\chi_V\circ\rho^U_V\in\textup{Hom}_\mathsf{C}(\mathcal{F}(U),\mathcal{G}(V)) 
\]
for any pair of open subsets $V\subset U\subset X$. It is an isomorphism of sheaves if each $\chi_U$ is an isomorphism in $\mathsf{C}$.

Given another morphism of sheaves $\tilde{\chi}:\mathcal{G}\rightarrow\mathcal{H}$ over $X$, we simply define the composition $\tilde{\chi}\circ\chi:\mathcal{F}\rightarrow\mathcal{H}$ setwise by $(\tilde{\chi}\circ\chi)_U\coloneqq\tilde{\chi}_U\circ\chi_U\in\textup{Hom}_\mathsf{C}(\mathcal{F}(U),\mathcal{H}(U))$. These operations yield a well-defined category $\mathsf{Sh}(X;\mathsf{C})$, the \textbf{category of sheaves on $X$ with values in $\mathsf{C}$}\index{category!of sheaves on a topological space}. This is a full subcategory of $\mathsf{PSh}(X;\mathsf{C})$, the \textbf{category of presheaves on $X$ with values in $\mathsf{C}$}\index{category!of presheaves on a topological space}, whose morphisms are defined identically.
\end{Def}

Choosing for example $\mathsf{C}=\mathsf{Ab}$ yields \textbf{(pre)sheaves of abelian groups}\index{presheafofabelian@(pre)sheaf of abelian groups}, while:

\begin{Def}\label{ringedspace}
Endowing a topological space $X$ with any sheaf $\mathcal{S}_X$ with values in $\mathsf{C=Rings}$ (the category of rings) makes it a \textbf{ringed space}\index{ringed space} $X=(X,\mathcal{S}_X)$ with \textbf{structure sheaf}\index{sheaf!structure} $\mathcal{S}_X$. 

Given a further ringed space $(Y,\mathcal{S}_Y)$, a \textbf{morphism of ringed spaces}\index{morphism!of ringed spaces} $(f,f^\#):(X,\mathcal{S}_X)\rightarrow (Y,\mathcal{S}_Y)$ consists of a continuous map $f:X\rightarrow Y$ together with a morphism $f^\#:\mathcal{S}_Y\rightarrow f_*(\mathcal{S}_X)$ of sheaves of rings over $Y$ (where $f_*(\mathcal{S}_X)$ is the ``pushforward'' sheaf defined in \eqref{pushforwardsheaf} below).
\end{Def}

\begin{Def}										% STALK
Let $X$ be a topological space, $\mathcal{F}$ a presheaf on $X$ with values in $\mathsf{C}$. The \textbf{stalk}\index{stalk} (or \textit{fibre}) of $\mathcal{F}$ at some given $x\in X$ is 
\begin{equation}
	\mathcal{F}_x\coloneqq\varinjlim_{U\ni x}\mathcal{F}(U)= \{(U,s)\mid x\in U,\, U\subset X \text{ open}, \,s\in\mathcal{F}(U)\}/\!\sim\,,
\end{equation}
where $(U,s)\sim(U',s')$ if and only if there exists some open subset $V\subset U\cap U'$ such that $x\in V$ and $s|_V=s'|_V$. We call a class $s_x\coloneqq[(U,s)]\in\mathcal{F}_x$ a \textbf{germ}\index{germ} of $\mathcal{F}$ at $x$. Observe that stalks inherit the structure of the objects of $\mathsf{C}$.

Note also that any presheaf morphism $\chi:\mathcal{F}\rightarrow\mathcal{G}$ over $X$ induces morphisms on stalks $\chi_x:\mathcal{F}_x\rightarrow\mathcal{G}_x,\,s_x=[(U,s)]\mapsto\chi_x(s_x)\coloneqq[(U,\chi_U(s))]$ in $\mathsf{C}$ for each $x\in X$. In other words, the operation of taking the stalk at $x\in X$ yields a well-defined (exact) functor $\square_x:\mathsf{Sh}(X;\mathsf{C})\rightarrow\mathsf{C}$. 
\end{Def}

Stalks are relevant because they highlight the local nature of sheaves: for example, a morphism of sheaves is an isomorphism if and only if the induced morphisms on stalks are isomorphisms at every point of $X$ (see \cite[Proposition II.1.1]{[Har77]}). In studying them we are not limiting ourselves since any object $\mathcal{F}(U)$ turns out to be fully determined by the collection $\bigsqcup_{x\in U}\mathcal{F}_x$! Stalks also serve as a fundamental link between presheaves and sheaves:

\begin{Def}\label{asssheaf}
Let $X$ be a topological space, $\mathcal{F}$ a presheaf on $X$. For any open subset $U\subset X$ define
\begin{align}
	\mathcal{F}^+(U)\!\coloneqq\!\Big\{\varphi:U\!\rightarrow\!\bigsqcup_{x\in U}\mathcal{F}_x \Bigm|\, &\forall x\!\in\! U\!:\,\varphi_x\!\coloneqq\!\varphi(x)\!\in\!\mathcal{F}_x\text{ and }\exists U_x\!\subset\! U\text{ $x$-neighb. } \nonumber\\[-0.3cm]
	&\exists\, s\!\in\!\mathcal{F}(U_x)\text{ s.t. }\varphi_y\!=\!s_y\,\forall y\!\in\! U_x\Big\}\,.
\end{align}
This identifies a sheaf $\mathcal{F}^+$ on $X$, endowed with the obvious restriction maps. We call it the \textbf{sheaf associated to $\mathcal{F}$}\index{sheaf!associated to a presheaf} (or the \textit{sheafification} of $\mathcal{F}$). Then there is a natural morphism $\theta:\mathcal{F}\rightarrow\mathcal{F}^+$ such that each $\theta_U(s)\in\mathcal{F}^+(U)$ has as image the collection of germs of $s\in\mathcal{F}(U)$ in $U$. Of course, $\mathcal{F}^+_x\cong\mathcal{F}_x$ for all $x\in X$, and $(\mathcal{F}\text{ sheaf}\implies\mathcal{F}^+\cong\mathcal{F})$.

The following universal property holds: for any other sheaf $\mathcal{G}$ and morphism of presheaves $\chi:\mathcal{F}\rightarrow\mathcal{G}$, there exists a unique morphism of sheaves $\psi:\mathcal{F}^+\rightarrow\mathcal{G}$ such that $\psi\circ\theta=\chi$ (which also shows that the pair $(\mathcal{F}^+,\theta)$ is unique up to unique isomorphism).

We can use it to show that sheafification also assigns to a morphism of presheaves $\mathcal{F}\rightarrow\mathcal{G}$ a morphism of associated sheaves $\mathcal{F}^+\rightarrow\mathcal{G}^+$ (see the proof of Proposition \ref{biglemma}), and is consequently a functor $\varsigma:\mathsf{PSh}(X;\mathsf{C})\rightarrow\mathsf{Sh}(X;\mathsf{C})$, the \textbf{sheafification functor}\index{functor!sheafification}.
\end{Def}

Recall that our ultimate goal is to describe derived categories of some very specific types of sheaves. This requires us to go through Chapter 1 and check the necessary categorical prerequisites. We do this first for the categories just introduced.

\begin{Pro}\label{pre-/sheafcatadditive}
The categories $\mathsf{PSh}(X;\mathsf{Ab})$ and $\mathsf{Sh}(X;\mathsf{Ab})$ are additive.
\end{Pro}

\begin{proof}
We prove additivity for $\mathsf{Sh}(X;\mathsf{Ab})$ so that it will follow automatically for $\mathsf{PSh}(X;\mathsf{Ab})$ as well, by definition. Regarding axiom (A1) of Definition \ref{addcat}, for any two $\mathcal{F}_1,\mathcal{F}_2\in\text{obj}(\mathsf{Sh}(X;\mathsf{Ab}))$ the set $\text{Hom}_{\mathsf{Sh}(X;\mathsf{Ab})}(\mathcal{F}_1,\mathcal{F}_2)$ is an abelian group with structure induced by the single abelian groups $\text{Hom}_\mathsf{Ab}(\mathcal{F}_1(U),\mathcal{F}_2(U))$ for each $U\subset X$ open; specifically, addition of sheaf morphisms is defined setwise and biadditivity of the composition law follows similarly. 

The zero sheaf $0\in\text{obj}(\mathsf{Sh}(X;\mathsf{Ab}))$ of axiom (A2) is the expectable one: it sends each $U$ to the trivial abelian group $\{0\}$, and has restriction maps $\text{id}_{\{0\}}$, fulfilling the presheaf properties and the gluing property (trivial sections glue to trivial sections).

Finally, for axiom (A3) we define the direct sum of $\mathcal{F}_1,\mathcal{F}_2\in\text{obj}(\mathsf{Sh}(X;\mathsf{Ab}))$ to be the sheaf $\mathcal{F}_1 \oplus\mathcal{F}_2$ specified by $(\mathcal{F}_1 \oplus\mathcal{F}_2)(U)\coloneqq\mathcal{F}_1(U)\oplus\mathcal{F}_2(U)$ and $\rho^U_V\coloneqq(\rho_1)^V_U\oplus(\rho_2)^V_U$ (for $\rho_i$ the restriction map of $\mathcal{F}_i$; the sheaf property holding because true componentwise). Moreover, the morphisms of sheaves fitting into diagram \eqref{directprodsum} are the obvious inclusions and projections made up from setwise inclusions and projections of abelian groups.      
\end{proof}

In order to deal with axiom (A4), we must first understand what kernels and cokernels look like in $\mathsf{Sh}(X;\mathsf{Ab})$.

\begin{Def}
Let $\chi:\mathcal{F}\rightarrow\mathcal{G}$ be a morphism of presheaves on a topological space $X$ and $V\subset U\subset X$ any pair of open subsets. Then:
\begin{itemize}[leftmargin=0.5cm]
	\item the \textbf{kernel presheaf}\index{presheaf!kernel} of $\chi$ is the presheaf $\mathcal{K}er(\chi)\in\text{obj}(\mathsf{PSh}(X;\mathsf{Ab}))$ specified by 
	\begin{equation}
		\mathcal{K}er(\chi)(U)\coloneqq\ker(\chi_U:\mathcal{F}(U)\rightarrow\mathcal{G}(U))=(\chi_U)^{-1}(0)
	\end{equation}
	(indeed an abelian group) and $\rho^U_V:\mathcal{K}er(\chi)(U)\rightarrow\mathcal{K}er(\chi)(V),\,s\mapsto s|_{\mathcal{F}(V)}$;
	
	\item the \textbf{image presheaf}\index{presheaf!image} of $\chi$ is the presheaf $\mathcal{I}m(\chi)\in\text{obj}(\mathsf{PSh}(X;\mathsf{Ab}))$ specified by
	\begin{equation}
		\mathcal{I}m(\chi)(U)\coloneqq\text{im}(\chi_U:\mathcal{F}(U)\rightarrow\mathcal{G}(U))=\chi_U(\mathcal{F}(U))
	\end{equation}
	(an abelian group) and $\rho^U_V:\mathcal{I}m(\chi)(U)\rightarrow\mathcal{I}m(\chi)(V),\,s\mapsto s|_{\mathcal{G}(V)}$;
	
	\item the \textbf{cokernel presheaf}\index{presheaf!cokernel} of $\chi$ is the presheaf $\mathcal{C}oker(\chi)\in\text{obj}(\mathsf{PSh}(X;\mathsf{Ab}))$ specified by\footnote{One can also take quotients of presheaves $\mathcal{F}$ by subpresheaves $\mathcal{F}'$ --- any presheaf such that $\mathcal{F}'(U)<\mathcal{F}(U)$ is a subgroup for each open $U\subset X$, endowed with compatible restriction maps from $\mathcal{F}$ --- setwise defined in the obvious way. Then $\mathcal{C}oker(\chi)\cong\mathcal{G}/\mathcal{I}m(\chi)$.}
	\begin{equation}
		\mathcal{C}oker(\chi)(U)\coloneqq\text{coker}(\chi_U:\mathcal{F}(U)\rightarrow\mathcal{G}(U))=\mathcal{G}(U)/\text{im}(\chi_U)
	\end{equation}
	(abelian) and $\rho^U_V:\mathcal{C}oker(\chi)(U)\rightarrow\mathcal{C}oker(\chi)(V),\,s\mapsto s|_{\mathcal{G}(V)}$.
\end{itemize}
\end{Def}

\begin{Lem}\label{PShabelian}
The category $\mathsf{PSh}(X;\mathsf{Ab})$ is abelian.
\end{Lem}

\begin{proof}
To check axiom (A4) we must determine the canonical decomposition \eqref{candecompo} of any given presheaf morphism $\chi:\mathcal{F}\rightarrow\mathcal{G}$. The latter induces by definition group homomorphisms $\chi_U\in\text{Hom}_\mathsf{Ab}(\mathcal{F}(U),\mathcal{G}(U))$, which do admit a canonical decomposition of the form
\[
\begin{tikzcd}
	\ker(\chi_U)\arrow[r, hook, "k_U"] & \mathcal{F}(U)\arrow[r, "q_U"] & \mathcal{I}(U)\arrow[r, "j_U"] & \mathcal{G}(U)\arrow[r, two heads, "c_U"] & \text{coker}(\chi_U)\;,
\end{tikzcd}
\]
where $\mathcal{I}(U)=\text{coker}(k_U)=\ker(c_U)$ and $j_U\circ q_U=\chi_U$. In fact, these kernels and cokernels collectively describe the kernel $\mathcal{K}er(\chi)$ and cokernel $\mathcal{C}oker(\chi)$ presheaves, and the maps $k_U$, $q_U$, $j_U$ and $c_U$ upgrade to presheaf morphisms, compatible with restrictions, in particular satisfying $j\circ q=\chi$ respectively $\mathcal{C}oker(k)=\mathcal{I}=\mathcal{K}er(c)$. Therefore, $\mathsf{PSh}(X;\mathsf{Ab})$ is an abelian category.
\end{proof}

\begin{Rem}
Note that the kernel presheaf does fulfill the sheaf property. Indeed, suppose the sections $s_i\in\ker(\chi_{U_i})\subset\mathcal{F}(U_i)$ associated to the open cover $\{U_i\}_{i\in I}$ of $U$ agree on their overlaps, so that, when regarded as sections of $\mathcal{F}$, they glue to a unique $s\in\mathcal{F}(U)$. Then in $\mathcal{G}(U)$ holds $\chi_U(s)|_{U_i}=\chi_{U_i}(s|_{U_i})=\chi_{U_i}(s_i)=0$ (the first equality exploits that $\chi$ is a sheaf morphism, thus commuting with restrictions), so that they necessarily glue to the zero section. Hence, $\chi_U(s)=0$ by uniqueness, which means that $s\in\mathcal{K}er(\chi)(U)$, with $s|_{U_i}=s_i\in\mathcal{K}er(\chi)(U_i)$ for all $i\in I$ by construction. Therefore, we may speak of the \textbf{kernel sheaf}\index{sheaf!kernel} $\mathcal{K}er(\chi)\in\text{obj}(\mathsf{Sh}(X;\mathsf{Ab}))$. We remark that $\mathcal{K}er(\chi)_x\cong\ker(\chi_x)$ naturally for each $x\in X$.

On the other hand, it is possible to show that the image and the (quotient, thus the) cokernel presheaves are not sheaves! Counterexamples can be found in \cite[construction 13.12]{[Gat20]} and \cite[Proposition I.5.4.b]{[GM03]}. Instead, we should focus on their associated sheaves in $\text{obj}(\mathsf{Sh}(X;\mathsf{Ab}))$, the \textbf{image sheaf}\index{sheaf!image} $\mathcal{I}m(\chi)^+\cong\mathcal{F}/\mathcal{K}er(\chi)$ respectively the \textbf{cokernel sheaf}\index{sheaf!cokernel} $\mathcal{C}oker(\chi)^+\cong\mathcal{G}/\mathcal{I}m(\chi)^+$, as described in Definition \ref{asssheaf}. Despite this, we still have that $\mathcal{I}m(\chi)_x\cong\text{im}(\chi_x)$ and $\mathcal{C}oker(\chi)_x\cong\text{coker}(\chi_x)$ for each $x\in X$ (yet a hint at why working with stalks is much more advantageous!).
\end{Rem}

To check abelianity of the category of sheaves, we must then better understand its relation to the category of presheaves:

\begin{Pro}\label{biglemma}
The inclusion/forgetful functor $\iota:\mathsf{Sh}(X;\mathsf{Ab})\rightarrow\mathsf{PSh}(X;\mathsf{Ab})$ admits a left adjoint functor, that is, some $\varsigma:\mathsf{PSh}(X;\mathsf{Ab})\rightarrow\mathsf{Sh}(X;\mathsf{Ab})$ together with a natural isomorphism $\Psi:\textup{Hom}_{\mathsf{Sh}(X;\mathsf{Ab})}(\varsigma(\square),\blacksquare)\rightarrow\textup{Hom}_{\mathsf{PSh}(X;\mathsf{Ab})}(\square,\iota(\blacksquare))$ of bifunctors \textup(compare with Proposition \ref{adjointfunc}\textup). 
\end{Pro}  

\begin{proof}
Let us begin by constructing $\varsigma$. Given a presheaf $\mathcal{F}$, we are tempted to set $\varsigma(\mathcal{F})\coloneqq\mathcal{F}^+\in\text{obj}(\mathsf{Sh}(X;\mathsf{Ab}))$. This is correct (indeed $\varsigma$ is the sheafification functor from Definition \ref{asssheaf}), but in order to make the ensuing construction more transparent --- and as a wholesome exercise in what seen so far about sheaves --- we choose an alternative equivalent definition. For $U\subset X$ open, let
\[
\varsigma(\mathcal{F})(U)\coloneqq\Big\{[(U_i, s_i)_{i\in I}] \Bigm| \bigcup_{i\in I}U_i=U,\,s_i\in\mathcal{F}(U_i)\text{ s.t. }s_i|_{U_i\cap U_j}=s_j|_{U_i\cap U_j}\Big\}\,,
\]
where the relation defining the equivalence classes is
\[
(U_i, s_i)_{i\in I}\sim(U_j', s_j')_{j\in J} \Leftrightarrow \exists \{U_k''\}_{k\in K}: \forall i,j\,\exists U_k''\subset U_i\cap U_j'\text{ and } s_i|_{U_k''}=s_j'|_{U_k''}.
\]
Also define $\rho^U_V:\varsigma(\mathcal{F})(U)\rightarrow\varsigma(\mathcal{F})(V),\,\rho^U_V([(U_i, s_i)_{i\in I}])\coloneqq [(U_i\cap V, s_i|_{U_i\cap V})_{i\in I}]$ and $[(U_i, s_i)_{i\in I}] + [(U_j', s_j')_{j\in J}]\coloneqq [(U_i\cap U_j', s_i|_{U_i\cap U_j'} + s_j'|_{U_i\cap U_j'})_{(i,j)\in I\times J}]$, which give $\varsigma(\mathcal{F})(U)$ the structure of abelian group (with identity $[({U},0)]$), thus making $\varsigma(\mathcal{F})\in\text{obj}(\mathsf{PSh}(X;\mathsf{Ab}))$.

We want to show that it is in fact a sheaf. We do this by contradiction, noticing that a presheaf cannot be a sheaf if either there exists a non-zero section which restricts to zero subsections (they would glue to 0) or if there is a collection of compatible subsections which do not come from a section (again violating the gluing property). For the former scenario, suppose there exists some $0\neq s=[(U_i,s_i)_{i\in I}]\in\varsigma(\mathcal{F})(U)$ and $\{V_j\}_{j\in J}$ covering $U$ such that $s|_{V_j}=0$ for all $j$; then fixing $j$, by $\sim$ there exists some refinement $\{U_{j,k}'\}_{k\in K}$ of $\bigcup_i(U_i\cap V_j)=V_j$ such that $s_i|_{U_{j,k}'}=0|_{U_{j,k}'}=0$ for any $i,k$ with $U_{j,k}'\subset U_i\cap V_j$, implying that $s=[(U_{j,k}',0)_{k\in K}]=0$, a contradiction. In the second case, one can similarly deduce that any compatible sections $s_j\in\varsigma(\mathcal{F})(V_j)$, where $\bigcup_{j\in J}V_j=U$, can be obtained as $s|_{V_j}$ of some $s\in\varsigma(\mathcal{F})(U)$ (the details are in \cite[Proposition II.5.13]{[GM03]}). Consequently, $\varsigma(\mathcal{F})\in\text{obj}(\mathsf{Sh}(X;\mathsf{Ab}))$.
Moreover, given a presheaf morphism $\chi:\mathcal{F}\rightarrow\mathcal{G}$, let $\varsigma(\chi)$ be the sheaf morphism setwise defined as 
\[
\varsigma(\chi)_U:\varsigma(\mathcal{F})(U)\rightarrow\varsigma(\mathcal{G})(U),\,\varsigma(\chi)_U([(U_i, s_i)_{i\in I}])\coloneqq[(U_i, \chi_{U_i}(s_i))_{i\in I}]\,,
\]
indeed commuting with restrictions. This completes the construction of $\varsigma$.

Regarding $\Psi$, we define first the natural transformation $\varepsilon:\mathsf{Id}_{\mathsf{PSh}(X;\mathsf{Ab})}\rightarrow\iota\circ\varsigma$ yielding presheaf morphisms $\varepsilon(\mathcal{F}):\mathcal{F}\rightarrow(\iota\circ\varsigma)(\mathcal{F})$ given by\footnote{So here, in absence of indices, $[(U,s)]$ is understood as the class of pairs of one-element sets $\{U\}$ and $\{s\}$. The same notational cue applies below.} 
\[
\varepsilon(\mathcal{F})_U:\mathcal{F}(U)\rightarrow(\iota\circ\varsigma)(\mathcal{F})(U),\, s\mapsto[(U,s)]=\iota\big([(U,s)]\big)\,.
\]
Then for $\mathcal{F}\in\text{obj}(\mathsf{PSh}(X;\mathsf{Ab}))$ and $\mathcal{G}\in\text{obj}(\mathsf{Sh}(X;\mathsf{Ab}))$ we set
\begin{small}
	\[
	\Psi(\mathcal{F},\mathcal{G}):\textup{Hom}_{\mathsf{Sh}(X;\mathsf{Ab})}(\varsigma(\mathcal{F}),\mathcal{G})\rightarrow\textup{Hom}_{\mathsf{PSh}(X;\mathsf{Ab})}(\mathcal{F},\iota(\mathcal{G})),\,\Psi(\mathcal{F},\mathcal{G})(\chi)\coloneqq\iota(\chi)\circ\varepsilon(\mathcal{F})\,,
	\]
\end{small}
\!\!and aim to prove that it is a bijection. Suppose that $\iota(\chi)\circ\varepsilon(\mathcal{F})=0$, and $s=[(U_i,s_i)_{i\in I}]\in\varsigma(\mathcal{F})(U)$ any. Then $r\coloneqq\chi_U(s)\in\mathcal{G}(U)$ fulfills $r|_{U_i}=\chi_{U_i}(s|_{U_i})=\chi_{U_i}(s_i)=\chi_{U_i}([(U_i,s_i)])=\iota(\chi)_{U_i}([(U_i,s_i)])=(\iota(\chi)\circ\varepsilon(\mathcal{F}))_{U_i}(s_i)=0$ for each $i$, hence the sheaf property of $\mathcal{G}$ implies that necessarily $r=0$, possible if and only if $\chi=0$. This shows that $\Psi(\mathcal{F},\mathcal{G})$ is injective.

For surjectivity, let $\psi:\mathcal{F}\rightarrow\iota(\mathcal{G})$ be a presheaf morphism and $s=[(U_i,s_i)_{i\in I}]$ $\in\varsigma(\mathcal{F})(U)$ any, and $r_i\coloneqq \psi_{U_i}(s_i)\in\mathcal{G}(U_i)$. Then $r_i|_{U_i\cap U_j}=\psi_{U_i\cap U_j}(s_i|_{U_i\cap U_j})=\psi_{U_i\cap U_j}(s_j|_{U_i\cap U_j})=r_j|_{U_i\cap U_j}$ (the second equality following by definition of $\varsigma(\mathcal{F})(U)$), which in the sheaf $\mathcal{G}$ means that there is some $r\in\mathcal{G}(U)$ restricting to these $r_i$. We define $\chi:\varsigma(\mathcal{F})\rightarrow\mathcal{G}$ by $\chi_U(s)\coloneqq r$ (in fact independent of the actual representative $(U_i,s_i)_{i\in I}$). Then $\psi_{U_i}(s_i)=\chi_U(s)|_{U_i}=\iota(\chi)_{U_i}(s_i)=\iota(\chi)_{U_i}([(U_i,s_i)])=(\iota(\chi)\circ\varepsilon(\mathcal{F}))_{U_i}(s_i)\in\mathcal{G}(U_i)$, hence $\psi=\iota(\chi)\circ\varepsilon(\mathcal{F})=\Psi(\mathcal{F},\mathcal{G})(\chi)$, and indeed $\chi\in\textup{Hom}_{\mathsf{Sh}(X;\mathsf{Ab})}(\varsigma(\mathcal{F}),\mathcal{G})$ because $(\rho^U_V\circ\chi_U)(s)|_{U_i}=(\rho^{U_i}_{U_i\cap V}\circ\psi_{U_i})(s_i)=\psi_{U_i\cap V}(s_i|_{U_i\cap V})=\iota(\chi)_{U_i\cap V}([(U_i\cap V,s_i|_{U_i\cap V})])=(\chi_V\circ\rho^U_V)(s)|_{U_i}$. This proves surjectivity.
We leave the proof of naturality of $\Psi$ as an exercise.  
\end{proof}

This proposition allows us to finally regard the category of sheaves of abelian groups as an abelian category:

\begin{Thm}\label{Sh(X,Ab)isabelian}
The category $\mathsf{Sh}(X;\mathsf{Ab})$ is abelian.
\end{Thm}

\begin{proof}
We already know from Proposition \ref{pre-/sheafcatadditive} that $\mathsf{Sh}(X;\mathsf{Ab})$ is an additive category, so it only remains to check axiom (A4) of Definition \ref{abcat}. Consider a sheaf morphism $\chi:\mathcal{F}\rightarrow\mathcal{G}$ and its associated presheaf morphism $\iota(\chi):\iota(\mathcal{F})\rightarrow\iota(\mathcal{G})$, which, $\mathsf{PSh}(X;\mathsf{Ab})$ being abelian by Lemma \ref{PShabelian}, has canonical decomposition $\mathcal{K}er(\iota(\chi))\xrightarrow{k}\iota(\mathcal{F})\xrightarrow{q}\mathcal{I}\xrightarrow{j}\iota(\mathcal{G})\xrightarrow{c}\mathcal{C}oker(\iota(\chi))$ with $\mathcal{K}er(c)=\mathcal{I}=\mathcal{C}oker(k)$. Then we claim that
\begin{equation}\label{sheafcandecompo}
	\begin{tikzcd}
		\varsigma\big(\mathcal{K}er(\iota(\chi))\big)\arrow[r, "\varsigma(k)"] & \mathcal{F}\arrow[r, "\varsigma(q)"] & \varsigma(\mathcal{I})\arrow[r, "\varsigma(j)"] & \mathcal{G}\arrow[r, "\varsigma(c)"] & \varsigma\big(\mathcal{C}oker(\iota(\chi))\big)
	\end{tikzcd}
\end{equation}
is the canonical decomposition of $\chi$ in $\mathsf{Sh}(X;\mathsf{Ab})$, resulting from the application of the adjoint functor $\varsigma:\mathsf{PSh}(X;\mathsf{Ab})\rightarrow\mathsf{Sh}(X;\mathsf{Ab})$ of Proposition \ref{biglemma}, where we notice that the sheaves $\mathcal{F}=\iota(\mathcal{F})$ and $\mathcal{G}=\iota(\mathcal{G})$ have associated sheaves equal to themselves, that is, $\varsigma(\iota(\mathcal{F}))=\mathcal{F}$ and $\varsigma(\iota(\mathcal{G}))=\mathcal{G}$, thus $\varsigma(\iota(\chi))=\chi$ (cf. Definition \ref{asssheaf}). Clearly, functoriality of $\varsigma$ guarantees that $\varsigma(j)\circ\varsigma(q)=\varsigma(\iota(\chi))=\chi$. So we must only verify that $\varsigma\big(\mathcal{K}er(\iota(\chi))\big)$ and $\varsigma\big(\mathcal{C}oker(\iota(\chi))\big)$ are the kernel respectively cokernel of $\chi$. 

About the cokernel, consider any $\mathcal{Z}\in\text{obj}(\mathsf{Sh}(X;\mathsf{Ab}))$. By the characterizing property of Definition \ref{monepicoker}, the sequence
\[
\{0\}\!\rightarrow\!\text{Hom}_{\mathsf{PSh}}(\mathcal{C}oker(\iota(\chi)),\iota(\mathcal{Z}))\!\rightarrow\!\text{Hom}_{\mathsf{PSh}}(\iota(\mathcal{G}),\iota(\mathcal{Z}))\!\rightarrow\!\text{Hom}_{\mathsf{PSh}}(\iota(\mathcal{F}),\iota(\mathcal{Z}))
\]
is exact. Then, applying the bijections $\Psi(\square,\mathcal{Z})^{-1}:\textup{Hom}_{\mathsf{PSh}(X;\mathsf{Ab})}(\square,\iota(\mathcal{Z}))\longrightarrow\break \textup{Hom}_{\mathsf{Sh}(X;\mathsf{Ab})}(\varsigma(\square),\mathcal{Z})$ from Proposition \ref{biglemma}, we immediately get the necessarily exact sequence
\[
\{0\}\rightarrow\text{Hom}_{\mathsf{Sh}}(\varsigma\big(\mathcal{C}oker(\iota(\chi))\big),\mathcal{Z})\rightarrow\text{Hom}_{\mathsf{Sh}}(\mathcal{G},\mathcal{Z})\rightarrow\text{Hom}_{\mathsf{Sh}}(\mathcal{F},\mathcal{Z})\,,
\]
which in turn implies that $\text{coker}(\chi)\equiv\varsigma(c):\mathcal{G}\rightarrow\mathcal{C}oker(\chi)=\varsigma\big(\mathcal{C}oker(\iota(\chi))\big)$.

For the kernel, some more work is required. First, the fundamental identity $\chi\circ\varsigma(k)=\varsigma(\iota(\chi)\circ k)=0$ is satisfied. Then $\big(\varsigma\big(\mathcal{K}er(\iota(\chi))\big),\varsigma(k)\big)$ is the kernel of $\chi$ if and only if for any $\psi\in\text{Hom}_{\mathsf{Sh}(X;\mathsf{Ab})}(\mathcal{Z},\mathcal{F})$ with $\chi\circ\psi=0$ there exists a unique sheaf morphism $h:\mathcal{Z}\rightarrow \varsigma\big(\mathcal{K}er(\iota(\chi))\big)$ such that $\varsigma(k)\circ h=\psi:\mathcal{Z}\rightarrow\mathcal{F}$ (that is, the left triangle of diagram \eqref{kercokerdiag} commutes).

About uniqueness of $h$, for any $s=[(U_i,s_i)_{i\in I}]\in\varsigma\big(\mathcal{K}er(\iota(\chi))\big)(U)$ the condition $0\overset{!}{=}\varsigma(k)_U(s)=[(U_i,k_{U_i}(s_i))_{i\in I}]$ implies (by definition of $\sim$ in Proposition \ref{biglemma}) that there exists a refining cover $\{V_j\}_{j\in J}$ such that $k_{V_j}(s_i|_{V_j})=0$ for all $V_j\subset U_i$, hence $s_i|_{V_j}=0$ by injectivity of $k_{V_j}$. So $s=[(V_j,s_i|_{V_j})_{i,j}]=0$, meaning that $\varsigma(k)_U$ is injective, specifically an inclusion, so that $h_U=\psi_U$, which proves uniqueness of $h$ for the given input $\psi$.

Let us now construct $h$ explicitly. Pick $t\in\mathcal{Z}(U)$ and $r\coloneqq\psi_U(t)\in\mathcal{F}(U)$. Then $\chi_U(r)=0$ provides an open cover $\{U_i\}_{i\in I}$ such that $0=(\chi_U\circ\psi_U)(t)|_{U_i}=(\chi_{U_i}\circ\psi_{U_i})(t|_{U_i})$. By the universal property of the kernel $k_{U_i}:\mathcal{K}er(\iota(\chi))(U_i)\rightarrow\mathcal{F}(U_i)$ of the group homomorphism $\chi_{U_i}:\mathcal{F}(U_i)\rightarrow\mathcal{G}(U_i)$, we obtain a unique $\bar{h}_i\in\mathcal{K}er(\iota(\chi))(U_i)$ --- the image of $t|_{U_i}$ under the unique $\mathcal{Z}(U_i)\rightarrow\mathcal{K}er(\iota(\chi))(U_i)$ --- such that $k_{U_i}(\bar{h}_i)=\psi_{U_i}(t|_{U_i})$ (refer again to \eqref{kercokerdiag}). All these $\bar{h}_i$ are compatible and so identify a class $\bar{h}=[(U_i,\bar{h}_i)_{i\in I}]\in\varsigma\big(\mathcal{K}er(\iota(\chi))\big)(U)$. Then we define $h_U(t)\coloneqq\bar{h}$ (actually independent of the covering), which yields a sound sheaf morphism $h:\mathcal{Z}\rightarrow\varsigma\big(\mathcal{K}er(\iota(\chi))\big)$ such that
\begin{align*}
	\varsigma(k)_U\circ h_U(t)=\varsigma(k)_U([(U_i,\bar{h}_i)_{i\in I}])&=[(U_i,k_{U_i}(\bar{h}_i))_{i\in I}]=[(U_i,\psi_{U_i}(t|_{U_i}))_{i\in I}] \\
	&=\psi_U([(U_i,t|_{U_i})_{i\in I}])=\psi_U(t)\,,
\end{align*} 
as desired. Therefore, it holds $\ker(\chi)\equiv\varsigma(k):\mathcal{K}er(\chi)=\varsigma\big(\mathcal{K}er(\iota(\chi))\big)\rightarrow~\mathcal{F}$ by Definition \ref{monepicoker}.

Finally, wholly analogous arguments applied to $\varsigma(c)$ and $\varsigma(k)$ show that\footnote{Here we suppressed the natural isomorphism $\varepsilon:\mathsf{Id}_{\mathsf{PSh}(X;\mathsf{Ab})}\rightarrow\iota\circ\varsigma$ from the previous proposition.} $\mathcal{K}er(\varsigma(c))\!=\varsigma\big(\mathcal{K}er\big(\iota(\varsigma(c))\big)\big)\!=\varsigma(\mathcal{K}er(c))$ respectively $\mathcal{C}oker(\varsigma(k))\!=\varsigma(\mathcal{C}oker(k))$, whence $\mathcal{K}er(\varsigma(c))=\varsigma(\mathcal{I})=\mathcal{C}oker(\varsigma(k))$. We conclude that \eqref{sheafcandecompo} is indeed the canonical decomposition of $\chi:\mathcal{F}\rightarrow\mathcal{G}$, and consequently $\mathsf{Sh}(X;\mathsf{Ab})$ is abelian. 
\end{proof}

\begin{Rem}
In general, as mentioned by \cite[section 1.6]{[Wei94]} and \cite[Theorem 17.4.9]{[KS06]}, the category $\mathsf{Sh}(X;\mathsf{A})$ of sheaves on a topological space $X$ with values in any abelian category $\mathsf{A}$ is itself an abelian category, though not necessarily abelian as a subcategory of $\mathsf{PSh}(X;\mathsf{A})$ (because cokernels in the latter may differ; cf. Definition \ref{abeliansubcat}).
\end{Rem}

\subsection{Further properties of sheaves}\label{ch5.2}

Having just shown that the category of pre-/sheaves with values in an abelian category is abelian, a rich range of tools comes at our disposal. For example, we can adapt Lemma \ref{aboutmonepi} and talk about exact sequences.

\begin{Lem}\label{aboutmonepisheaves}
Let $\chi:\mathcal{F}\rightarrow\mathcal{G}$ be a morphism of sheaves with values in an abelian category $\mathsf{A}$. Then $\chi$ is:
\begin{itemize}[leftmargin=0.5cm]
	\item injective \textup(a monomorphism\textup) if and only if $\mathcal{K}er(\chi)=0$, possible if and only if $\chi_U\in\textup{Hom}_\mathsf{A}(\mathcal{F}(U),\mathcal{G}(U))$ has $\ker(\chi_U)=\{0\}$ for all open $U\subset X$, which is equivalent to each map of stalks $\chi_x:\mathcal{F}_x\rightarrow\mathcal{G}_x$ for $x\in X$ having $\ker(\chi_x)=0\in\textup{obj}(\mathsf{A})$;
	
	\item surjective \textup(an epimorphism\textup) if and only if $\mathcal{I}m(\chi)=\mathcal{G}$, possible if and only if $\chi_x\in\textup{Hom}_\mathsf{A}(\mathcal{F}_x,\mathcal{G}_x)$ has $\textup{im}(\chi_x)=\mathcal{G}_x\in\textup{obj}(\mathsf{A})$ for each $x\in X$ \textup(though not equivalent to having each $\chi_U$ surjective\textup!\textup).
\end{itemize}
\vspace*{-0.1cm}
Then a sequence $0\rightarrow\mathcal{F}\xrightarrow{\chi}\mathcal{G}\xrightarrow{\psi}\mathcal{H}\rightarrow 0$ in $\mathsf{Sh}(X;\mathsf{A})$ is \textup{short exact} if $\chi$ is injective, $\psi$ surjective and $\mathcal{K}er(\psi)=\mathcal{I}m(\chi)$. In this case, it will be canonically isomorphic to $0\rightarrow\mathcal{K}er(\psi)\hookrightarrow\mathcal{G}\twoheadrightarrow\mathcal{C}oker(\chi)\rightarrow 0$. More generally, a sequence $...\rightarrow\mathcal{F}^{n-1}\xrightarrow{\chi^{n-1}}\mathcal{F}^n\xrightarrow{\chi^n}\mathcal{F}^{n+1}\rightarrow...$ in $\mathsf{Sh}(X;\mathsf{A})$ is \textup{long exact} if for all $n\in\mathbb{N}$ holds $\mathcal{K}er(\chi^n)=\mathcal{I}m(\chi^{n-1})$. 

By above, this condition can be equivalently checked in $\mathsf{A}$ at the level of each stalk $x\in X$, as $\mathcal{K}er(\chi^n)_x\cong\ker(\chi_x^n)$ and $\mathcal{I}m(\chi^{n-1})_x\cong\textup{im}(\chi_x^{n-1})$. In particular, the stalk functor $\square_x:\mathsf{Sh}(X;\mathsf{A})\rightarrow\mathsf{A}$ is exact.
\end{Lem}  

We can also discuss exactness of additive functors involving sheaves. For example:

\begin{Lem}\label{sectfuncexact}
Let $X$ be a topological space, $U\subset X$ open. Then the functor $\Gamma(U,\square):\mathsf{Sh}(X;\mathsf{Ab})\rightarrow\mathsf{Ab}$ sending $\mathcal{F}\mapsto\mathcal{F}(U)\text{ and } \chi\mapsto\chi_U$ is left exact.
\end{Lem}

\begin{proof}
Firstly, we observe that $\Gamma(U,\square)$ is additive as a functor between additive categories (Definition \ref{additiveexactfunc}), because for any $\mathcal{F},\mathcal{G}\in\text{obj}(\mathsf{Sh}(X;\mathsf{Ab}))$ the map $\Gamma(U,\square)_{\mathcal{F},\mathcal{G}}:\text{Hom}_{\mathsf{Sh}(X;\mathsf{Ab})}(\mathcal{F},\mathcal{G})\rightarrow\text{Hom}_\mathsf{Ab}(\mathcal{F}(U),\mathcal{G}(U)),\,\chi\mapsto\chi_U$ is clearly a homomorphism of abelian groups.

Now, consider any short exact sequence $0\rightarrow\mathcal{F}\rightarrow\mathcal{F}'\rightarrow\mathcal{F}''\rightarrow 0$ in $\mathsf{Sh}(X;\mathsf{Ab})$. From Proposition \ref{biglemma} we recall that the inclusion (forgetful) functor $\iota:\mathsf{Sh}(X;\mathsf{Ab})\rightarrow\mathsf{PSh}(X;\mathsf{Ab})$ is the right adjoint to the sheafification functor $\varsigma:\mathsf{PSh}(X;\mathsf{Ab})\rightarrow\mathsf{Sh}(X;\mathsf{Ab})$. Then by Proposition \ref{adjointfunc}, $\iota$ is left exact (and $\varsigma$ right exact), meaning that $0\rightarrow\iota(\mathcal{F})\rightarrow\iota(\mathcal{F}')\rightarrow\iota(\mathcal{F}'')$ is exact in $\mathsf{PSh}(X;\mathsf{Ab})$, where by definition kernels and cokernels behave nicely setwise. Hence, $0\rightarrow\iota(\mathcal{F})(U)\rightarrow\iota(\mathcal{F}')(U)\rightarrow\iota(\mathcal{F}'')(U)$ is exact for any $U\subset X$. But of course $\iota(\mathcal{F})(U)=\mathcal{F}(U)$ and similarly for $\mathcal{F}',\mathcal{F}''$, confirming that $\Gamma(U,\square)$ is left exact.    
\end{proof}

For later purposes, it is useful to expand our vocabulary a little more.

\begin{Def}\label{pushforwardShAb}
Let $X,Y$ be topological spaces, $f:X\rightarrow Y$ a continuous map, and let $\mathsf{C}$ be any category. 
\begin{itemize}[leftmargin=0.5cm]
	\item Suppose $X$ is equipped with some $\mathcal{F}\in\text{obj}(\mathsf{Sh}(X;\mathsf{C}))$. Then for any open $V\subset Y$ the assignment 
	\begin{equation}\label{pushforwardsheaf}
		V\mapsto f_*\mathcal{F}(V)\coloneqq\mathcal{F}(f^{-1}(V))
	\end{equation}
	defines a sound sheaf $f_*\mathcal{F}$ on $Y$, the \textbf{pushforward sheaf along $f$}\index{sheaf!pushforward} (or \textbf{direct image sheaf}\index{sheaf!direct image}), equipped with restrictions induced by the restriction maps in $\mathcal{F}$ of the respective preimages. In particular, on stalks we have $(f_*\mathcal{F})_{f(x)}\cong\mathcal{F}_x$ for any $x\in X$.
	\newline Given some $\chi\in\text{Hom}_{\mathsf{Sh}(X;\mathsf{C})}(\mathcal{F}_1,\mathcal{F}_2)$, we may also define the morphism $f_*\chi\in\text{Hom}_{\mathsf{Sh}(Y;\mathsf{C})}(f_*\mathcal{F}_1,f_*\mathcal{F}_2)$ by $(f_*\chi)_V:s\in f_*\mathcal{F}_1(V)\mapsto\chi_{f^{-1}(V)}(s)\in f_*\mathcal{F}_2(V)$. This makes the pushforward into a functor $f_*:\mathsf{Sh}(X;\mathsf{C})\rightarrow\mathsf{Sh}(Y;\mathsf{C})$.
	
	\item Suppose $Y$ is equipped with some $\mathcal{G}\in\text{obj}(\mathsf{Sh}(Y;\mathsf{C}))$. Then for any open $U\subset X$ the assignment\footnote{The limit is necessary because $f(U)\subset Y$ is not an open subset in general!}
	\begin{equation}
		U\mapsto f^{-1}\mathcal{G}(U)\coloneqq\mkern-6mu\varinjlim_{V\supseteq f(U)}\mkern-6mu\mathcal{G}(V)
	\end{equation}
	defines a presheaf on $Y$, whose associated sheaf $f^{-1}\mathcal{G}$ we call the \textbf{inverse image sheaf}\index{sheaf!inverse image}, endowed with the restrictions induced by those of $\mathcal{G}$. In particular, on stalks we have $(f^{-1}\mathcal{G})_x\cong\mathcal{G}_{f(x)}$ for any $x\in X$.
	\newline Given some $\psi\in\text{Hom}_{\mathsf{Sh}(Y;\mathsf{C})}(\mathcal{G}_1,\mathcal{G}_2)$, we can apply the constructions of Proposition \ref{biglemma} to obtain some $f^{-1}\psi\in\text{Hom}_{\mathsf{Sh}(X;\mathsf{C})}(f^{-1}\mathcal{G}_1,f^{-1}\mathcal{G}_2)$, hence a functor $f^{-1}:\mathsf{Sh}(Y;\mathsf{C})\rightarrow\mathsf{Sh}(X;\mathsf{C})$. 
\end{itemize}
One can conveniently characterize the functor $f^{-1}$ as the left adjoint of $f_*$. Namely, the inverse image sheaf $f^{-1}\mathcal{G}$ is implicitly determined by the isomorphism of bifunctors $\text{Hom}_{\mathsf{Sh}(X;\mathsf{C})}(f^{-1}\mathcal{G},\mathcal{F})\cong\text{Hom}_{\mathsf{Sh}(Y;\mathsf{C})}(\mathcal{G},f_*\mathcal{F})$ (see \cite[Proposition II.6.17]{[GM03]} for a proof).
\end{Def}

\begin{Ex}\label{skyscrapersheaf}
Let $X$ be a topological space, $x\in X$ fixed, regarded as the one-point space $\{x\}$ included in $X$ via $\iota:\{x\}\hookrightarrow X$. Let $S\in\text{obj}(\mathsf{C})$ of a generic category $\mathsf{C}$, and endow $\{x\}$ with the sheaf $\mathcal{O}_x\in\text{obj}(\mathsf{Sh}(\{x\};\mathsf{C}))\equiv\text{obj}(\mathsf{C})$ such that $\mathcal{O}_x(\{x\})=S$ and $\mathcal{O}_x(\emptyset)=0$. Then the \textbf{skyscraper sheaf}\index{sheaf!skyscraper} $\iota_*\mathcal{O}_x\in\text{obj}(\mathsf{Sh}(X;\mathsf{C}))$ is given by
\vspace*{-0.2cm}
\begin{equation}
	U\subset X\text{ open}\longmapsto\iota_*\mathcal{O}_x(U)=\mathcal{O}_x(\iota^{-1}(U))= 
	\begin{cases}
		S & $if $ x\in U \\ 
		0 & $else $
	\end{cases}\,. 
	\vspace*{-0.4cm} 
\end{equation}
\hfill $\blacklozenge$
\end{Ex}

\begin{Lem}\label{inversepushforwardexact}
Let $f:X\rightarrow Y$ be a continuous map of topological spaces. Then the functor $f_*:\mathsf{Sh}(X;\mathsf{Ab})\rightarrow\mathsf{Sh}(Y;\mathsf{Ab})$ is left exact while the functor $f^{-1}:\mathsf{Sh}(Y;\mathsf{Ab})\rightarrow\mathsf{Sh}(X;\mathsf{Ab})$ is exact.
\end{Lem}

\begin{proof}
By above, the functors $f^{-1}, f_*$ are a pair of left respectively right adjoint functors between additive categories. Proposition \ref{adjointfunc} then implies that they are right respectively left exact. To see that $f^{-1}$ is also left exact, it suffices to work on stalks as suggested by Lemma \ref{aboutmonepisheaves}: $0\rightarrow\mathcal{G}\rightarrow\mathcal{G}'\rightarrow\mathcal{G}''\rightarrow 0$ short exact in $\mathsf{Sh}(Y;\mathsf{Ab})$ implies in particular that
\[
0\rightarrow\mathcal{G}_{f(x)}\cong(f^{-1}\mathcal{G})_x\rightarrow\mathcal{G'}_{f(x)}\cong(f^{-1}\mathcal{G'})_x\rightarrow\mathcal{G''}_{f(x)}\cong(f^{-1}\mathcal{G''})_x\rightarrow 0 
\]
is exact on the stalks at each $f(x)$, for all $x\in X$. Hence, $0\rightarrow f^{-1}\mathcal{G}\rightarrow f^{-1}\mathcal{G'}\rightarrow f^{-1}\mathcal{G''}\rightarrow 0$ is exact in $\mathsf{Sh}(X;\mathsf{Ab})$ as well. (Notice how the same procedure would have failed for $f_*$ due to $f$ not being surjective in general!)
\end{proof}

\begin{Def}\label{localHomShAb}
Let $X$ be a topological space, $\mathcal{F},\mathcal{G}\in\text{obj}(\mathsf{Sh}(X;\mathsf{Ab}))$. For any $U\subset X$ open, the assignment\footnote{Here $\mathcal{F}|_U,\mathcal{G}|_U\in\text{obj}(\mathsf{Sh}(U;\mathsf{Ab}))$ are the \textit{restricted sheaves}, whose abelian groups are defined only for open subsets $V\subset U$.}
\begin{equation}
	U\mapsto\text{Hom}_{\mathsf{Sh}(U;\mathsf{Ab})}(\mathcal{F}|_U,\mathcal{G}|_U)
\end{equation}
defines a sheaf $\mathcal{H}om(\mathcal{F},\mathcal{G})$ on $X$ with values in $\mathsf{Ab}$ (since each $\mathsf{Sh}(U;\mathsf{Ab})$ is abelian), called the \textbf{local $\mathcal{H}om$-sheaf}\index{sheaf!local $\mathcal{H}om$} (or \textit{sheaf of local morphisms}). 
\end{Def}

\subsection{Sheaves of modules}\label{ch5.3}

To continue our study, we need sheaves to have a richer structure.

\begin{Def}											% SHEAF OF MODULES
Let $(X,\mathcal{S})$ be a ringed space (see Definition \ref{ringedspace}). Then a \textbf{sheaf of $\mathcal{S}$-modules}\index{sheaf!of modules (or $\mathcal{S}$-module)} on $X$ (also abbreviated to \textbf{$\mathcal{S}$-module}) is a sheaf $\mathcal{F}$ on $X$ such that $\mathcal{F}(U)$ is an $\mathcal{S}(U)$-module\footnote{One can of course specify if we are dealing with left (as tacitly done here) or right modules. In the latter case, the arising category is denoted $\mathsf{Sh}(X;\mathsf{Mod}_\mathcal{S})$. If actually $\mathcal{S}\in\text{obj}(\mathsf{Sh}(X;\mathsf{CommRings}))$ (a sheaf of \textit{commutative} rings), we have $\mathsf{Sh}(X;\mathsf{Mod}_\mathcal{S})=\mathsf{Sh}(X;{}_\mathcal{S}\mathsf{Mod})$ as for standard modules.} for every open subset $U\subset X$, and the associated restrictions are $\mathcal{S}(U)$-module homomorphisms: 
\[
\rho^U_V(s_1+s_2)=\rho^U_V(s_1)+\rho^U_V(s_2)\quad\text{and}\quad \rho^U_V(\lambda\cdot s)=\rho^U_V(\lambda)\cdot\rho^U_V(s)\,,
\]
for all open $V\subset U\subset X$ and $s,s_1,s_2\in\mathcal{F}(U)$, $\lambda\in\mathcal{S}(U)$.

Given another sheaf $\mathcal{G}$ of $\mathcal{S}$-modules on $X$, $\chi:\mathcal{F}\rightarrow\mathcal{G}$ is a \textbf{morphism of $\mathcal{S}$-modules}\index{morphism!of sm@of $\mathcal{S}$-modules} if each ring morphism $\chi_U:\mathcal{F}(U)\rightarrow\mathcal{G}(U)$ is $\mathcal{S}(U)$-linear. With composition between any two such compatible morphisms being the usual one, we obtain $\mathsf{Sh}(X;{}_\mathcal{S}\mathsf{Mod})$, the \textbf{category of $\mathcal{S}$-modules on $X$}\index{category!of sm@of $\mathcal{S}$-modules}. We usually denote its Hom-groups by $\text{Hom}_\mathcal{S}(\mathcal{F},\mathcal{G})$.

These properties restrict to stalks: $\mathcal{F}_x$ is an $\mathcal{S}_x$-module and $\chi_x:\mathcal{F}_x\rightarrow\mathcal{G}_x$ is a morphism of $\mathcal{S}_x$-modules for each $x\in X$, thus resulting in the exact \textbf{stalk functor}\index{functor!stalk} $\square_x:\mathsf{Sh}(X;{}_\mathcal{S}\mathsf{Mod})\rightarrow{}_{\mathcal{S}_x}\!\mathsf{Mod}$.
\end{Def}

\begin{Rem}\label{sheafmodulesrem}
Obviously the structure sheaf is a module over itself, $\mathcal{S}\in\text{obj}(\mathsf{Sh}(X;{}_\mathcal{S}\mathsf{Mod}))$. We moreover observe that:
\renewcommand{\theenumi}{\alph{enumi}}
\begin{enumerate}[leftmargin=0.5cm]
	\item Direct sums/products, restrictions and quotients of $\mathcal{S}$-modules are themselves $\mathcal{S}$-modules, and the kernel, image and cokernel sheaves of morphisms of $\mathcal{S}$-modules naturally inherit the structure of $\mathcal{S}$-module; this tells us, jointly with an adaptation of Theorem \ref{Sh(X,Ab)isabelian}, that $\mathsf{Sh}(X;{}_\mathcal{S}\mathsf{Mod})$ \textit{is an abelian category}.
	
	\item Much like ordinary modules are a generalization of abelian groups --- which are just $\mathbb{Z}$-modules --- $\mathcal{S}$-modules, for $\mathcal{S}\in\text{obj}(\mathsf{Sh}(X;\mathsf{Rings}))$, generalize sheaves of abelian groups: indeed any $\mathcal{F}\in\text{obj}(\mathsf{Sh}(X;\mathsf{Ab}))$ can be regarded as a $\underline{\mathbb{Z}}$-module, where $\underline{\mathbb{Z}}$ denotes the constant sheaf\footnote{Let $A$ be a set endowed with the discrete topology. Then the \textit{constant sheaf} $\underline{A}\in\text{obj}(\mathsf{Sh}(X;\mathsf{Sets}))$ on $X$ is set by $\underline{A}(U)\coloneqq\{c:U\rightarrow A\mid c\text{ continuous}\}\cong A$ for any open $U\subset X$, equipped with the usual restriction maps.} over $\mathbb{Z}$. In other words, $\mathsf{Sh}(X;\mathsf{Ab})\cong\mathsf{Sh}(X;{}_\mathcal{\underline{\mathbb{Z}}}\mathsf{Mod})$.
	
	\item On a more specific note, we observe that for any $\mathcal{F}\in\textup{obj}(\mathsf{Sh}(X;{}_\mathcal{S}\mathsf{Mod}))$ there is an obvious isomorphism of $\mathcal{S}$-modules
	\begin{equation}\label{Homglobalsectiso}
		\text{Hom}_\mathcal{S}(\mathcal{S},\mathcal{F})\xrightarrow{\sim} \Gamma(X,\mathcal{F})=\mathcal{F}(X),\,\psi\mapsto \psi_X(1_{\mathcal{S}(X)})\,,
	\end{equation}
	where $1_{\mathcal{S}(X)}$ stands for the unit element of the module $\mathcal{S}(X)$. Similarly, $\text{Hom}_{\mathcal{S}|_U}(\mathcal{S}|_U,\mathcal{F}|_U)\cong \Gamma(U,\mathcal{F}|_U)$. This observation will prove repeatedly useful in the theory to come. 
\end{enumerate}
\end{Rem}

Now that we have introduced sheaves of modules, we can keep expanding our array of tools. In the following definition, we continue to denote the structure sheaf of $X$ by $\mathcal{S}$; the reader should be warned that most literature readily adopts the notation $\mathcal{O}_X$, which we save for the structure sheaf of regular functions on algebraic varieties (introduced in Definition \ref{sheafregfcts} below) and later on when we will talk about coherent sheaves!\footnote{One extra bit: for general topological spaces, we can assume that any sheaf of rings $\mathcal{S}$ is also a sheaf of $\mathbb{K}$-valued functions, that is, $\mathcal{S}(U)\subset\text{Hom}_\mathsf{Sets}(U,\mathbb{K})$ for each open $U\subset X$. (This makes sense of the composition with sections in $\mathcal{S}(U)$.) Upgrading to \textit{continuous} $\mathbb{K}$-valued functions in $\mathsf{Top}$, we obtain a valid counterpart to $\mathcal{O}_X$ on varieties.}

\begin{Def}\label{moresheafofmodules}
Let $(X,\mathcal{S})$ be a ringed space, $\mathcal{F},\mathcal{G}\in\text{obj}(\mathsf{Sh}(X;{}_\mathcal{S}\mathsf{Mod}))$.
\begin{itemize}[leftmargin=0.5cm]
	\item The local $\mathcal{H}om$-sheaf from Definition \ref{localHomShAb} upgrades to a sheaf of modules $\mathcal{H}om_{\mathcal{S}}(\mathcal{F},\mathcal{G})\in\text{obj}(\mathsf{Sh}(X;{}_{\mathcal{S}}\mathsf{Mod}))$ given by $U\mapsto\text{Hom}_{\mathcal{S}|_U}(\mathcal{F}|_U,\mathcal{G}|_U)$ for each $U\subset X$ open (the restricted sheaves clearly inherit the structure of $\mathcal{S}$-modules). In particular, by the standard arguments of Example \ref{exactfuncexample}, we have two left exact functors $\mathcal{H}om_{\mathcal{S}}(\mathcal{F},\square)$, $\mathcal{H}om_{\mathcal{S}}(\square,\mathcal{G}):\mathsf{Sh}(X;{}_{\mathcal{S}}\mathsf{Mod})\rightarrow\mathsf{Sh}(X;{}_{\mathcal{S}}\mathsf{Mod})$ (covariant respectively contravariant). The generalization of \eqref{Homglobalsectiso} then tells us that $\mathcal{H}om_{\mathcal{S}}(\mathcal{S},\square)\cong\mathsf{Id}_{\mathsf{Sh}(X;{}_{\mathcal{S}}\mathsf{Mod})}$. 
	
	\item The \textbf{dual sheaf}\index{sheaf!dual} to $\mathcal{F}$ is simply \begin{equation}
		\mathcal{F}^\vee\coloneqq \mathcal{H}om_{\mathcal{S}}(\mathcal{F},\mathcal{S})\in\text{obj}(\mathsf{Sh}(X;{}_{\mathcal{S}}\mathsf{Mod}))\,.
	\end{equation} 
	
	\item The presheaf given by the assignment
	\begin{equation}
		U\mapsto \mathcal{F}(U)\otimes_{\mathcal{S}(U)}\mathcal{G}(U)
	\end{equation}
	for any $U\subset X$ open can be sheafified to obtain the \textbf{tensor sheaf}\index{sheaf!tensor} $\mathcal{F}\otimes_{\mathcal{S}}\mathcal{G}\in\text{obj}(\mathsf{Sh}(X;{}_{\mathcal{S}}\mathsf{Mod}))$. Together with the definition of dual sheaf, it yields the canonical evaluation map of $\mathcal{S}$-modules $\mathcal{F}^\vee\otimes_{\mathcal{S}}\mathcal{F}\rightarrow\mathcal{S}$.
\end{itemize}
Consider now a morphism of ringed spaces $(f,f^\#):(X,\mathcal{S}_X)\rightarrow (Y,\mathcal{S}_Y)$ as in Definition \ref{ringedspace}.
\begin{itemize}[leftmargin=0.5cm]
	\item The pushforward sheaf $f_*\mathcal{F}$ set by \eqref{pushforwardsheaf} has natural structure of $\mathcal{S}_Y$-module, since $\mathcal{F}(f^{-1}(V))$ is a $\mathcal{S}_X(f^{-1}(V))$-module when endowed with the action $\lambda\cdot s\coloneqq f^\#_V(\lambda)\cdot s\in\mathcal{F}(f^{-1}(V))$ induced by $f^\#_V:\mathcal{S}_Y(V)\rightarrow\mathcal{S}_X(f^{-1}(V))$, for each open $V\subset Y$ (this is compatible with restrictions). Therefore, $f_*\mathcal{F}\in\text{obj}(\mathsf{Sh}(Y;{}_{\mathcal{S}_Y}\!\mathsf{Mod}))$ and we obtain the \textbf{pushforward functor}\index{functor!pushforward} \begin{equation}
		f_*:\mathsf{Sh}(X;{}_{\mathcal{S}_X}\!\mathsf{Mod})\rightarrow\mathsf{Sh}(Y;{}_{\mathcal{S}_Y}\!\mathsf{Mod})\,.
	\end{equation}
	
	\item In contrast, the inverse image sheaf $f^{-1}\mathcal{G}$ does not retain the structure of $\mathcal{S}_X$-module (it is merely a $f^{-1}\mathcal{S}_Y$-module). To obviate this, we operate a base change via the tensor product from above and let\footnote{In case we wish for a \textit{right} $\mathcal{S}_X$-module $f^*\mathcal{G}\in\text{obj}(\mathsf{Sh}(X;\mathsf{Mod}_{\mathcal{S}_X}))$, we must instead define $f^*\mathcal{G}\coloneqq f^{-1}\mathcal{G}\otimes_{f^{\!-\!1}\!\mathcal{S}_Y}\!\mathcal{S}_X$.} 
	\begin{equation}
		U\mapsto f^*\mathcal{G}(U)\coloneqq\mathcal{S}_X(U)\otimes_{f^{\!-\!1}\!\mathcal{S}_Y(U)}\!f^{-1}\mathcal{G}(U)
	\end{equation}
	for each open $U\!\subset\! X$. This yields a well-defined $f^*\mathcal{G}\!=\!\mathcal{S}_X\!\otimes_{f^{\!-\!1}\!\mathcal{S}_Y}\! f^{-1}\mathcal{G}$ $\in\text{obj}(\mathsf{Sh}(X;{}_{\mathcal{S}_X}\!\mathsf{Mod}))$, called the \textbf{pullback sheaf along $f$}\index{sheaf!pullback}, and the \textbf{pullback functor}\index{functor!pullback}
	\begin{equation}
		f^*:\mathsf{Sh}(Y;{}_{\mathcal{S}_Y}\!\mathsf{Mod})\rightarrow\mathsf{Sh}(X;{}_{\mathcal{S}_X}\!\mathsf{Mod})\,.
	\end{equation}
\end{itemize}
We observe that the pullback sheaf behaves as we would hope: the isomorphism of bifunctors
\begin{equation}\label{pullpushareadjoints}
	\text{Hom}_{\mathcal{S}_X}(f^*\mathcal{G},\mathcal{F})\cong\text{Hom}_{\mathcal{S}_Y}(\mathcal{G},f_*\mathcal{F})
\end{equation}
still holds. Also note that $f^*(\mathcal{S}_Y)\cong\mathcal{S}_X$ by construction.
\end{Def}

Lemma \ref{inversepushforwardexact} specializes to:

\begin{Cor}\label{pullbackpushforwardexact}
Let $f:(X,\mathcal{S}_X)\rightarrow (Y,\mathcal{S}_Y)$ be a morphism of ringed spaces. Then the functor $f_*:\mathsf{Sh}(X;{}_{\mathcal{S}_X}\!\mathsf{Mod})\rightarrow\mathsf{Sh}(Y;{}_{\mathcal{S}_Y}\!\mathsf{Mod})$ is left exact while the functor $f^*:\mathsf{Sh}(Y;{}_{\mathcal{S}_Y}\!\mathsf{Mod})\rightarrow\mathsf{Sh}(X;{}_{\mathcal{S}_X}\!\mathsf{Mod})$ is right exact.
\end{Cor}

Finally, we note that $\mathsf{Sh}(X;\mathsf{Ab})$ and $\mathsf{Sh}(X;{}_{\mathcal{S}}\mathsf{Mod})$ being abelian categories allows us to talk about Ext-abelian groups and \textbf{Ext-modules}\index{Ext-module},\footnote{Technically, $\text{Ext}_{\mathcal{S}}^k(\mathcal{F},\mathcal{G})$ is a module only in case $\mathcal{S}$ is a sheaf of commutative rings; otherwise, it is just an abelian group. Albeit misleading, this choice of language tries to distinguish them from the Ext-groups coming from $\mathsf{Sh}(X;\mathsf{Ab})$.} which we define again (see \eqref{Extgroup}) as
\begin{equation}\label{Extmodules}
\begin{aligned}
	\text{Ext}^k(\mathcal{F},\mathcal{G})&\coloneqq\text{Hom}_{\mathsf{D}(\mathsf{Sh}(X;\mathsf{Ab}))}(\mathcal{F}[0]^\bullet,\mathcal{G}[k]^\bullet)\in\text{obj}(\mathsf{Ab})\quad\text{and} \\
	\text{Ext}_{\mathcal{S}}^k(\mathcal{F},\mathcal{G})&\coloneqq\text{Hom}_{\mathsf{D}(\mathsf{Sh}(X;{}_{\mathcal{S}}\mathsf{Mod}))}(\mathcal{F}[0]^\bullet,\mathcal{G}[k]^\bullet)\in\text{obj}({}_{\mathcal{S}(\!X\!)}\!\mathsf{Mod})
\end{aligned}
\end{equation}
for $k\in\mathbb{Z}$ and $\mathcal{F},\mathcal{G}$ in $\mathsf{Sh}(X;\mathsf{Ab})$, respectively in $\mathsf{Sh}(X;{}_{\mathcal{S}}\mathsf{Mod})$. In particular, they obey Theorem \ref{Extprop}. We will return to them in Chapter 6, when dealing with classical derived functors.

\subsection{Free and locally free sheaves}\label{ch5.4}

It is desirable to express sheaves of modules through copies of the structure sheaf equipping the underlying ringed space. This allows us to talk about \textit{free} modules.

\begin{Def}
Let $(X,\mathcal{S})$ be a ringed space and $\mathcal{F}\in\text{obj}(\mathsf{Sh}(X;{}_\mathcal{S}\mathsf{Mod}))$ an $\mathcal{S}$-module. Then:
\begin{itemize}[leftmargin=0.5cm]
	\item $\mathcal{F}$ is \textbf{free}\index{sheaf!free} of \textbf{rank}\index{rank (of a sheaf)} $n\coloneqq |J|$ if it is of the form $\mathcal{F}\cong\bigoplus_{j\in J}\mathcal{S}\equiv\mathcal{S}^{\oplus n}$; then each $\mathcal{F}(U)\cong\mathcal{S}(U)^{\oplus n}$ is a free $\mathcal{S}(U)$-module of rank $n$.
	
	\item $\mathcal{F}$ is just \textbf{locally free}\index{sheaf!locally free} if there exists an open cover $\{U_i\}_{i\in I}$ of $X$ such that each restricted sheaf $\mathcal{F}|_{U_i}$ is a free $\mathcal{S}|_{U_i}$-module, 
	\[
	\mathcal{F}|_{U_i}\cong\mathcal{S}|_{U_i}^{\oplus n_i}
	\]
	(equivalently, if each $x\in X$ admits a neighbourhood $U_x\subset X$ such that $\mathcal{F}|_{U_x}$ is a free $\mathcal{S}|_{U_x}$-module), with $n_i$ being the rank of $\mathcal{F}$ at $U_i$, possibly infinite and the same for each $U_i$ in case $X$ is connected --- as we always implicitly assume to be the case.
\end{itemize}
Taking morphisms between (locally) free sheaves of $\mathcal{S}$-modules to be exactly the underlying morphisms of $\mathcal{S}$-modules, we obtain the categories $\mathsf{Sh}^\mathsf{lf}_n(X;{}_\mathcal{S}\mathsf{Mod})$ respectively $\mathsf{Sh}^\mathsf{f}_n(X;{}_\mathcal{S}\mathsf{Mod})$ (where $n$ stands for the rank, $<\infty$ unless declared otherwise). They are full subcategories of $\mathsf{Sh}(X;{}_\mathcal{S}\mathsf{Mod})$. 
\end{Def}

\begin{Rem}\label{locallyfreeRem}
Locally free sheaves of finite rank have desirable features when it comes to certain operations. For example, they are closed under tensor products, the local $\mathcal{H}om(\cdot,\cdot)$-sheaf and dualization. Given any $\mathcal{F}\in\text{obj}(\mathsf{Sh}^\mathsf{lf}_n(X;{}_\mathcal{S}\mathsf{Mod}))$ and $\mathcal{F}_1,\mathcal{F}_2\in\text{obj}(\mathsf{Sh}(X;{}_\mathcal{S}\mathsf{Mod}))$, we also have:
\begin{align*}
	(\mathcal{F}^\vee)^\vee&\cong\mathcal{F}, \\ \mathcal{H}om_{\mathcal{S}}(\mathcal{F},\mathcal{F}_1)&\cong\mathcal{F}^\vee\otimes_{\mathcal{S}}\mathcal{F}_1, \\
	\text{Hom}_{\mathcal{S}}(\mathcal{F}\otimes_{\mathcal{S}}\mathcal{F}_1,\mathcal{F}_2)&\cong\text{Hom}_{\mathcal{S}}(\mathcal{F}_1,\mathcal{H}om_{\mathcal{S}}(\mathcal{F},\mathcal{F}_2))\,.
\end{align*}

Moreover, given a morphism of ringed spaces $f:(X,\mathcal{S}_X)\rightarrow(Y,\mathcal{S}_Y)$ and some $\mathcal{F}\in\text{obj}(\mathsf{Sh}(X;{}_{\mathcal{S}_X}\!\mathsf{Mod}))$ and $\mathcal{G}\in\text{obj}(\mathsf{Sh}^\mathsf{lf}_n(Y;{}_{\mathcal{S}_Y}\mathsf{Mod}))$, the \textit{projection formula} holds: 
\[
f_*(\mathcal{F}\otimes_{\mathcal{S}_X}\!f^*\mathcal{G})\cong f_*\mathcal{F}\otimes_{\mathcal{S}_Y}\!\mathcal{G}\,. 
\]  
\end{Rem}

\begin{Def}\label{Picardgroup}
Let $(X,\mathcal{S})$ be a ringed space. An \textbf{invertible sheaf}\index{sheaf!invertible} on $X$ is any $\mathcal{F}\in\text{obj}(\mathsf{Sh}^\mathsf{lf}_1(X;{}_\mathcal{S}\mathsf{Mod}))$. Clearly, the tensor product of invertible sheaves yields an invertible sheaf, and for any given $\mathcal{F}\in\text{obj}(\mathsf{Sh}^\mathsf{lf}_1(X;{}_\mathcal{S}\mathsf{Mod}))$ there exists at least one $\mathcal{G}\in\text{obj}(\mathsf{Sh}^\mathsf{lf}_1(X;{}_\mathcal{S}\mathsf{Mod}))$ such that $\mathcal{F}\otimes_\mathcal{S}\mathcal{G}\cong\mathcal{S}$ (namely $\mathcal{G}=\mathcal{F}^\vee$). Writing in this case $\mathcal{F}\sim\mathcal{G}$, we obtain a well-defined equivalence relation $\sim$ on locally free sheaves of rank 1, and the quotient $\text{Pic}(X)\coloneqq \mathsf{Sh}^\mathsf{lf}_1(X;{}_\mathcal{S}\mathsf{Mod})/\sim$ is called \textbf{Picard group}\index{Picard group}, indeed a group with respect to $\otimes_\mathcal{S}$ and with inversion given by dualization. 
\end{Def}

We also state an important result which was alluded to in the Introduction:

\begin{Pro}\label{locfreeisvectorbundle}
Let $(X,\mathcal{S})$ be a ringed space. Then there exists a 1-to-1 correspondence between locally free sheaves of rank $n$ on $X$ and vector bundles of rank $n$ over $X$. \textup(For example, invertible sheaves correspond to line bundles.\textup)
\end{Pro}

\begin{proof}
Recall that a \textbf{vector bundle}\index{vector bundle} of rank $n$ over $X$ consists of the following data: a continuous surjection $\pi:E\rightarrow X$ such that for each $x\in X$ the fibre $\pi^{-1}(x)\subset E$ has the structure of an $n$-dimensional vector space and there exists an $x$-neighbourhood $U\subset X$ and a homeomorphism $\phi_{U}:\pi^{-1}(U)\rightarrow U\times\mathbb{K}^n$ with $\pi|_{\pi^{-1}(U)}=p_1\circ\phi_U:\pi^{-1}(U)\rightarrow U$ ($p_1$ being the first projection); moreover, $\phi_{U}|_{\pi^{-1}(y)}:\pi^{-1}(y)\rightarrow \{y\}\times\mathbb{K}^n\cong\mathbb{K}^n$ is a vector space isomorphism for all $y\in U$. Equivalently, there must exist an open cover $\{U_i\}_{i\in I}$ of $X$ (a \textit{trivializing cover}) and a homeomorphism $\phi_i\equiv\phi_{U_i}$ as above for each $U_i$. Then for each $U_i\cap U_j\neq\emptyset$ we obtain a transition function $T_{ij}:U_i\cap U_j\rightarrow\text{GL}(n,\mathbb{K})$ fulfilling $(\phi_i\circ \phi_j^{-1})(x,v)=(x,T_{ij}(x)[v])$ for all $(x,v)\in (U_i\cap U_j)\times\mathbb{K}^n$; these are subject to the cocycle condition $T_{jk}\circ T_{ij}=T_{ik}$ for every triple $U_i\cap U_j\cap U_k\neq \emptyset$.

Now, a fixed rank $n$ vector bundle $\pi:E\rightarrow X$ is naturally equipped with the sheaf $\mathcal{F}\coloneqq\Gamma(\square, E)$ of sections of $E$, on any open $U\subset X$ defined by $\mathcal{F}(U)\coloneqq\{s:U\rightarrow E\mid \pi\circ s=\text{id}_U\}$, clearly an $\mathcal{S}$-module; in fact, $\mathcal{F}\in\text{obj}(\mathsf{Sh}(X;{}_\mathcal{S}\mathsf{Mod}))$. Considering a trivializing cover $\{U_i\}_{i\in I}$ of $X$, any $s\in\mathcal{F}|_{U_i}(U\cap U_i)$ naturally identifies a unique $\phi_i\circ s:U\cap U_i\rightarrow (U\cap U_i)\times\mathbb{K}^n$, which can be interpreted as a section $\vec{s}_i=(s_i^1,...,s_i^n)$ in $\mathcal{S}|_{U_i}^{\oplus n}(U\cap U_i)$. Therefore, $\mathcal{F}|_{U_i}\cong\mathcal{S}|_{U_i}^{\oplus n}$, implying that $\mathcal{F}\in\text{obj}(\mathsf{Sh}_n^\mathsf{lf}(X;{}_\mathcal{S}\mathsf{Mod}))$.

Moreover, $T_{ij}(x)[(p_2\circ\vec{s}_i)(x)]=[(p_2\circ\vec{s}_j)(x)]\in\mathbb{K}^n$ for all $x\in U\cap U_i\cap U_j$ ($p_2$ denoting the second projection), so that the transition functions equipping the vector bundle precisely coincide with the natural functions $T_{ij}:U_i\cap U_j\rightarrow\text{GL}(n,\mathbb{K})$ attached to every locally free sheaf on $X$ with respect to the open cover $\{U_i\}_{i\in I}$. These themselves satisfy the cocycle condition. Therefore, the data of some $\mathcal{F}\in\text{obj}(\mathsf{Sh}_n^\mathsf{lf}(X;{}_\mathcal{S}\mathsf{Mod}))$ with transition functions $\{T_{ij}\}_{i,j\in I}$ is sufficient to uniquely reconstruct (up to isomorphism) a vector bundle of rank $n$ over $X$ --- by ``gluing the $U_i\times\mathbb{K}^n$ along $U_i\cap U_j$ using the $T_{ij}$''. This reverts the argument of the previous paragraph and thus proves the claimed bijection. 	  
\end{proof}

Despite all their virtues, though, locally free sheaves (and hence vector bundles) lack an essential property: they do not form an abelian category! This makes them unsuitable to our study. We must enlarge their category to include better behaved modules, specifically the so-called \textit{quasi-coherent} sheaves.

\begin{Def}\label{coherentonringed}
Let $(X,\mathcal{S})$ be a ringed space and $\mathcal{Q}\in\text{obj}(\mathsf{Sh}(X;{}_\mathcal{S}\mathsf{Mod}))$.
\begin{itemize}[leftmargin=0.5cm]
	\item $\mathcal{Q}$ is a \textbf{quasi-coherent sheaf}\index{sheaf!quasi-coherent} of $\mathcal{S}$-modules over $X$ if there exists an open cover $\{U_i\}_{i\in I}$ of $X$ and exact sequences of sheaves
	\begin{equation}\label{qcohpresentation}
		\mathcal{S}|_{U_i}^{\oplus n_i}\xrightarrow{\psi_i}\mathcal{S}|_{U_i}^{\oplus m_i}\xrightarrow{\chi_i}\mathcal{Q}|_{U_i}\rightarrow 0\,,
	\end{equation}
	with $n_i, m_i>0$ possibly infinite, for each $i\in I$. This amounts to require that $\mathcal{Q}$ is locally the cokernel of a morphism of free $\mathcal{S}$-modules: $\mathcal{Q}|_{U_i}=\mathcal{I}m(\chi_i)\cong \big(\mathcal{S}|_{U_i}^{\oplus m_i}\big)/\mathcal{K}er(\chi_i)=\mathcal{C}oker(\psi_i)$ (by exactness and a generalization of the Isomorphism Theorem for sheaves). One calls the exact sequence \eqref{qcohpresentation} the \textit{local presentation} of $\mathcal{Q}$.
	
	\item A quasi-coherent $\mathcal{Q}$ is moreover a \textbf{coherent sheaf}\index{sheaf!coherent} of $\mathcal{S}$-modules over $X$ if:
	\renewcommand{\labelitemii}{\textendash}
	\begin{itemize}[leftmargin=0.4cm]
		\item it is \textit{finitely-generated} (or \textit{of finite type}), that is, for each point $x\in X$ there exists some open $x$-neighbourhood $U_x\subset X$ and a surjective sheaf morphism $\chi(x):\mathcal{S}|_{U_x}^{\oplus n_x}\rightarrow \mathcal{Q}|_{U_x}$, for some (finite!) $n_x\in\mathbb{N}$;
		
		\item for all open subsets $U\subset X$, $n\in\mathbb{N}$ and morphisms $\chi:\mathcal{S}|_U^{\oplus n}\rightarrow \mathcal{Q}|_U$ of $\mathcal{S}$-modules, $\mathcal{K}er(\chi)\in\text{obj}(\mathsf{Sh}(X;{}_\mathcal{S}\mathsf{Mod}))$ is finitely-generated in the sense just explained.
	\end{itemize}	
	Note that if the surjective morphisms $\chi(x)$ in the first condition are actually isomorphisms of sheaves, then $\mathcal{Q}$ is locally free.
\end{itemize} 
Defining morphisms between quasi-/coherent sheaves of $\mathcal{S}$-modules to be exactly the underlying morphisms of $\mathcal{S}$-modules, we obtain the \textbf{category of quasi-coherent sheaves}\index{category!of quasi-coherent sheaves} $\mathsf{QCoh}(X;{}_\mathcal{S}\mathsf{Mod})$ respectively the \textbf{category of coherent sheaves}\index{category!of coherent sheaves} $\mathsf{Coh}(X;{}_\mathcal{S}\mathsf{Mod})$ of $\mathcal{S}$-modules over $X$. By definition, they are full subcategories of $\mathsf{Sh}(X;{}_\mathcal{S}\mathsf{Mod})$.
\end{Def}

\begin{Rem}
In the early literature, coherent sheaves were simply defined to be those quasi-coherent sheaves which are finitely-generated, which is useless when working away from noetherian schemes. On the latters, instead, coherent sheaves are synonym of finitely-presented $\mathcal{S}$-modules: $\mathcal{Q}\in\text{obj}(\mathsf{Sh}(X;{}_\mathcal{S}\mathsf{Mod}))$ is \textit{finitely-presented} if it fits into an exact sequence $\mathcal{S}^{\oplus n_i}\rightarrow\mathcal{S}^{\oplus n_j}\rightarrow \mathcal{Q}\rightarrow 0$ with $n_i, n_j<\infty$, implying it is finitely-generated.\footnote{Taking a step back to classical algebra, modules over a noetherian ring are finitely-generated iff they are finitely-presented iff they are coherent. Here, coherence of modules is defined in an analogous way, and in fact inspires the generalization to sheaves. That the category ${}_R\mathsf{CohMod}\subset{}_R\mathsf{Mod}$ of coherent $R$-modules is an abelian subcategory then serves as a motivating tease.}
\end{Rem}

Obviously, locally free sheaves are quasi-coherent, and even coherent when of finite rank on $X$ which is at least \textit{locally noetherian} as a topological space (that is, locally, descending chains of closed subsets become stationary). Therefore, we obtain the following chain of full inclusions of categories:\vspace*{0.1cm}
\begin{center}
\noindent\minibox[frame]{\begin{small}$
		\mathsf{Sh}^\mathsf{f}_n(X;{}_\mathcal{S}\mathsf{Mod})\!\subset\!\mathsf{Sh}^\mathsf{lf}_n(X;{}_\mathcal{S}\mathsf{Mod})\!\subset\!\mathsf{Coh}(X;{}_\mathcal{S}\mathsf{Mod})\!\subset\!\mathsf{QCoh}(X;{}_\mathcal{S}\mathsf{Mod})\!\subset\!\mathsf{Sh}(X;{}_\mathcal{S}\mathsf{Mod})$
\end{small}}
\end{center}\vspace*{0.3cm}

\noindent We know that the first two categories are not abelian, while the last one still is. Since it turns out that the failure of $\mathsf{Sh}^\mathsf{lf}_n(X;{}_\mathcal{S}\mathsf{Mod})$ to be abelian lies in it not containing cokernels, the very definition of quasi-coherent sheaf obviates this issue in a ``minimal manner''. Therefore, $\mathsf{QCoh}(X;{}_\mathcal{S}\mathsf{Mod})$ can be interpreted as the smallest abelian subcategory of $\mathsf{Sh}(X;{}_\mathcal{S}\mathsf{Mod})$ which is locally defined and contains locally free sheaves (of any rank) and their cokernels; in the locally noetherian and finite rank framework, this minimality role is stolen by coherent sheaves.     
\vspace*{0.3cm}

\noindent In order to more formally prove the claimed abelianity of $\mathsf{QCoh}(X;{}_\mathcal{S}\mathsf{Mod})$ and $\mathsf{Coh}(X;{}_\mathcal{S}\mathsf{Mod})$, however, we rather give an alternative definition of quasi-/coherent sheaf specialized to \textit{schemes} --- introduced in the next section --- which are the correct generalization of our ultimate objects of interest, Calabi--Yau manifolds (as we will argue in Section \ref{ch6.5}).

\subsection{Interlude: a refresher of varieties and schemes}\label{ch5.5}

Before proceeding with the theory of sheaves, it is worth revisiting some classical algebraic geometry, in particular the notion of schemes. Accordingly, our focus decisively shifts to \cite[chapters I-II]{[Har77]}, where a more comprehensive account on varieties and schemes can be found.

\begin{Rem}\label{aboutvarieties}
There are several equivalent definitions of variety adopted in literature. Here we briefly outline Harthorne's conventions, so to line up with the later treatment of schemes. We work over a fixed algebraically closed field $\mathbb{K}$ and take $n>0$.

Endow the affine space $\mathbb{A}^n\equiv\mathbb{K}^n$ with the Zariski topology $\mathcal{T}_\text{Zar}$, whose open sets are the complements of the algebraic sets --- those subsets of $\mathbb{A}^n$ which can be written as the vanishing locus of some subset of polynomials $S\subset \mathbb{K}[x_1,...,x_n]$. Then an \textbf{affine variety}\index{variety!affine} is any subset $X\subset\mathbb{A}^n$ closed with respect to $\mathcal{T}_\text{Zar}$ which also is irreducible (not the union of two closed, proper subsets). Open subsets thereof are called \textit{quasi-affine varieties}. To such $X$ we associate the \textit{affine coordinate ring} $A(X)\coloneqq \mathbb{K}[x_1,...,x_n]/I(X)$, with the \textit{ideal of $X$} given by $I(X)\coloneqq \{f\in\mathbb{K}[x_1,...,x_n]\mid f(x)=0\;\forall x\in X\}$.

Similarly, endow the projective space $\mathbb{P}^n\coloneqq(\mathbb{K}^{n+1}\setminus\{0\})/_\sim\subset\mathbb{A}^{n+1}$ (where $x\sim y$ if and only if $y=\lambda x\in\mathbb{K}^{n+1}$ for some $0\neq \lambda\in\mathbb{K}$) with the Zariski topology $\mathcal{T}_\text{Zar}$, whose open sets are again the complements of algebraic sets, now defined as the vanishing loci in $\mathbb{P}^n$ of subsets of \textit{homogeneous} polynomials in $\mathbb{K}[x_0,x_1,...,x_n]$ (now regarded as a graded ring). Then a \textbf{projective variety}\index{variety!projective} is any irreducible closed subset $Y\subset\mathbb{P}^n$, and open subsets thereof are called \textit{quasi-projective varieties}. To $Y$ we associate the \textit{homogeneous coordinate ring} $S(Y)\coloneqq \mathbb{K}[x_0,x_1,...,x_n]/I(Y)$ (a graded ring), where the ideal of $Y$ now reads $I(Y)\coloneqq\langle\{f\in\mathbb{K}[x_0,x_1,...,x_n]\mid f\text{ homogeneous},\, f(x)=0\,\forall x\in Y\}\rangle$, the ideal generated by homogeneous polynomials.  
\end{Rem} 

We recall the concept of regular function $X\rightarrow\mathbb{K}$, which generalizes that of continuous function (retrieved once the target is regarded as $(\mathbb{A}^1,\mathcal{T}_\text{Zar})$).

\begin{Def}\label{regularfunc}
Let $X\subset\mathbb{A}^n$ be a (quasi-)affine variety, $U\subset X$ open. Then a function $\varphi:U\rightarrow \mathbb{K}$ is called \textbf{regular on $U$}\index{regular function} if for any $x\in U$ there exist an open $x$-neighbourhood $U_x\subset U$ and polynomial functions $f_x, g_x:X\rightarrow \mathbb{K}$ in $A(X)$ with $g_x|_{U_x}\neq 0$ everywhere such that $\varphi|_{U_x}=\frac{f_x}{g_x}\big|_{U_x}$.

Similarly, let $Y\subset\mathbb{P}^n$ be a (quasi-)projective variety, $U\subset Y$ open. Then a function $\varphi:U\rightarrow \mathbb{K}$ is regular on $U$ if for any $x\in U$ there exist an open $x$-neighbourhood $U_x\subset U$ and homogeneous polynomial functions $f_x, g_x:Y\rightarrow \mathbb{K}$ in $S(Y)_d$ (thus of \textit{same degree} $d\geq 0$!) with $g_x|_{U_x}\neq 0$ everywhere such that $\varphi|_{U_x}=\frac{f_x}{g_x}\big|_{U_x}$. 
\end{Def}

\begin{Def}
Let $\mathbb{K}$ be an algebraically closed field. Then we call \textbf{variety}\index{variety} any among quasi-/affine and quasi-/projective varieties over $\mathbb{A}\equiv\mathbb{K}$.\footnote{And we call some given variety quasi-/affine or quasi-/projective if it is isomorphic to a quasi-/affine respectively quasi-/projective variety.} Given any two varieties $X,Y$, a \textbf{morphism of varieties}\index{morphism!of varieties} $\varphi:X\rightarrow Y$ is a continuous map such that for every open $V\subset Y$ and any regular function $f:V\rightarrow\mathbb{K}$ also $f\circ\varphi:\varphi^{-1}(V)\rightarrow\mathbb{K}$ is regular. The notions of composition of compatible morphisms and isomorphisms are defined in the obvious manner.

We thus obtain the \textbf{category of varieties}\index{category!of varieties} $\mathsf{Var}(\mathbb{K})$ over $\mathbb{K}$.
\end{Def}

In fact, any variety can be interpreted as a ringed space:

\begin{Def}\label{sheafregfcts}
Let $X$ be a variety. For any open $U\subset X$ let
\begin{equation}
	\mathcal{O}_X(U)\coloneqq\{\varphi:U\rightarrow \mathbb{K} \mid \varphi\text{ regular function on $U$}\}\in\text{obj}(\mathbb{K}\text{-}\mathsf{Alg})\,.
\end{equation}
This defines the \textbf{sheaf of regular functions}\index{sheaf!of regular functions} $\mathcal{O}_X\in\text{obj}(\mathsf{Sh}(X;\mathbb{K}\text{-}\mathsf{Alg}))$ on $X$, a sheaf of $\mathbb{K}$-algebras endowed with the obvious restriction morphisms.

Correspondingly, the stalk $\mathcal{O}_{X,x}\in\text{obj}(\mathbb{K}\text{-}\mathsf{Alg})$ is called the \textit{local ring of regular functions at $x\in X$} (indeed a local ring, because possessing a unique maximal ideal, that consisting of all regular functions vanishing at $x$).
\end{Def}

Because $\mathbb{K}$-algebras are in particular rings, $\mathcal{O}_X$ is a legitimate choice of structure sheaf on $X$. In fact, it is taken to be the default sheaf equipping any variety regarded as a ringed space.

\begin{Ex}\label{sheafholofcts}
By Definition \ref{sheafregfcts}, any variety $X$ over $\mathbb{K}=\mathbb{C}$ is naturally equipped with the structure sheaf $\mathcal{O}_X$ given by the \textbf{sheaf of holomorphic functions}\index{sheaf!of holomorphic functions}. These are the complex-valued declination of regular functions, and correspond to the notion of holomorphic functions from complex geometry.  \hfill $\blacklozenge$
\end{Ex}

We mention that if $X$ is an affine variety, then $\mathcal{O}_X(X)\cong A(X)$, while if $X$ is projective, then $\mathcal{O}_X(X)\cong\mathbb{K}$. This fact is exploited in the proof of the following statement: if $X$ is any variety and $Y$ is affine, then there is a natural bijection $\text{Hom}_{\mathsf{Var}(\mathbb{K})}(X,Y)\leftrightarrow\text{Hom}_{\mathbb{K}\text{-}\mathsf{Alg}}(A(Y),\mathcal{O}_X(X))$.

A relevant property of varieties is \textit{smoothness}, which coincides with the standard definition when working over $\mathbb{C}$ equipped with the usual topology.

\begin{Def}\label{smoothvar}
Let $X\subset\mathbb{A}^n$ be an affine variety of dimension\footnote{The dimension of a quasi-/affine variety is that of the underlying topological space, determined by the longest chain of irreducible closed subsets. The same is true for schemes.} $r\leq n$, with ideal $I(X)=\langle\{f_1,...,f_m\}\rangle$ generated by finitely many polynomials $f_1,...,f_m\in\mathbb{K}[x_1,...,x_n]$. Then $X$ is smooth if the rank of the matrix $\Big(\frac{\partial f_i}{\partial x_j}\Big)_{\!i,j}$ is $n-r$ at each $x\in X$.

More abstractly, any variety $X$ is \textbf{smooth}\index{variety!smooth (or nonsingular)} (or \textit{nonsingular}) if for each $x\in X$ the local ring $\mathcal{O}_{X,x}$ is a \textit{regular} local ring, meaning that $\text{dim}_k(\mathfrak{m}/\mathfrak{m}^2)=\text{dim}_\text{Krull}(\mathcal{O}_{X,x})$, where $\mathfrak{m}\subset\mathcal{O}_{X,x}$ is the unique maximal ideal and $k=\mathcal{O}_{X,x}/\mathfrak{m}$ is the residue field. 
\end{Def}
\vspace*{0.5cm}

\noindent We move our attention to schemes, outlining first their fundamental building blocks. 

\begin{Rem}\label{spectra}
Remember the algebraic notion of \textbf{spectrum}\index{spectrum} $\text{Spec}(R)$ of a ring $R$: it is the set of all prime ideals of $R$, made into a topological space by imposing the sets $V(I)\coloneqq\{\mathfrak{p}\in\text{Spec}(R)\mid \mathfrak{p}\supset I\}$, for $I\subset R$ any ideal, to be the closed subsets.
Then we can promote $\text{Spec}(R)$ to a ringed space by equipping it with the sheaf $\mathcal{O}_{\text{Spec}(R)}\in\text{obj}(\mathsf{Sh}(\text{Spec}(R);\mathsf{Rings}))$ given by
\begin{small}\begin{align*}
		\mathcal{O}_{\text{Spec}(R)}(U)\!\coloneqq\!\bigg\{s\!:U\!\rightarrow\!\bigsqcup_{\mathfrak{p}\in U}\! R_\mathfrak{p}\Bigm|\, &\forall\mathfrak{p}\!\in\! U\!: s(\mathfrak{p})\!\in\! R_\mathfrak{p}\text{ and }\exists U_\mathfrak{p}\!\subset\! U\text{ $\mathfrak{p}$-neighb. }\exists r_1,r_2\!\in\! R \\[-0.3cm]
		&\text{s.t. }\!\forall\mathfrak{q}\!\in\! U_\mathfrak{p}\text{ holds } r_2\!\notin\!\mathfrak{q} \text{ and }
		s(\mathfrak{q})\!=\! \frac{r_1}{r_2}\!\in\! R_\mathfrak{q}\bigg\}\,,
\end{align*}\end{small}
\!\!for each open $U\subset \text{Spec}(R)$ (where $R_\mathfrak{p}$ is the localization of $R$ at the prime ideal $\mathfrak{p}$), with obvious restriction maps. In particular, $\Gamma(\text{Spec}(R),\mathcal{O}_{\text{Spec}(R)})\cong R$ and the stalk at any $\mathfrak{p}\in \text{Spec}(R)$ is $\mathcal{O}_{\text{Spec}(R),\mathfrak{p}}\cong R_\mathfrak{p}$.
\end{Rem}

Now, to define schemes, we must first localize the concept of ringed space of Definition \ref{ringedspace}:

\begin{Def}\label{locallyringedspace}
A ringed space $(X,\mathcal{S}_X)$ is a \textbf{locally ringed space}\index{locally ringed space} if the stalk $\mathcal{S}_{X,x}\in\text{obj}(\mathsf{Rings})$ is a \textit{local} ring for each $x\in X$.

A \textbf{morphism of locally ringed spaces}\index{morphism!of locally ringed spaces} $(f,f^\#):(X,\mathcal{S}_X)\rightarrow (Y,\mathcal{S}_Y)$ is a morphism of the underlying ringed spaces such that the maps $f^\#_x:\mathcal{S}_{Y,f(x)}\rightarrow \mathcal{S}_{X,x}$ obtained by taking the direct limit over all $V\ni f(x)$ in $Y$ are \textit{local morphisms} of local rings.\footnote{That is, the preimage under $f^\#_x$ of the unique maximal ideal of $\mathcal{S}_{X,x}$ is the unique maximal ideal of $\mathcal{S}_{Y,f(x)}$.} It is moreover an isomorphism if and only if $f:X\rightarrow Y$ is a homeomorphism and $f^\#$ is an isomorphism of sheaves. 
\end{Def}

Since $(\text{Spec}(R),\mathcal{O}_{\text{Spec}(R)})$ is a locally ringed space for any ring $R$ (a consequence of $\mathcal{O}_{\text{Spec}(R),\mathfrak{p}}\cong R_\mathfrak{p}$), and since any ring morphism $r:R\rightarrow S$ naturally induces a morphism $(f,f^\#):(\text{Spec}(S),\mathcal{O}_{\text{Spec}(S)})\rightarrow (\text{Spec}(R),\mathcal{O}_{\text{Spec}(R)})$ of\break locally ringed spaces (cf. \cite[Proposition II.2.3]{[Har77]}), that of taking the spectrum turns out to be a functorial operation.

\begin{Def}
An \textbf{affine scheme}\index{scheme!affine} is a locally ringed space $(X,\mathcal{O}_X)$ which is isomorphic to some spectrum $(\text{Spec}(R),\mathcal{O}_{\text{Spec}(R)})$ (as locally ringed spaces). By abuse of notation, we will often denote it as the latter.

A \textbf{scheme}\index{scheme} is a locally ringed space $(X,\mathcal{O}_X)$ which admits an open cover by affine schemes; equivalently, for each $x\in X$ there exists an $x$-neighbourhood $U_x\subset X$ such that $(U_x,\mathcal{O}_X|_{U_x})$ is an affine scheme. Thus, the stalk at $x\in X$ is the local ring $\mathcal{O}_{X,x}\cong R_\mathfrak{p}$ for a suitable ring $R$ and prime ideal $\mathfrak{p}\in\text{Spec}(R)$.

A \textbf{morphism of schemes}\index{morphism of schemes} is simply a morphism of the underlying locally ringed spaces (and similarly for isomorphisms).
\end{Def} 

For example, $\text{Spec}(\mathbb{K})$ is an affine scheme whose underlying topological space is the one-point space $\{*\}$ and whose structure sheaf is $\mathbb{K}$ itself. 

\begin{Rem}\label{projspectra}
We can also work on graded rings. Let $S$ be one such. Then the \textbf{projective spectrum}\index{spectrum!projective} $\text{Proj}(S)$ is defined to be the set of all homogeneous prime ideals which do not contain the whole ideal $\bigoplus_{n>0}S_n\subset S$. Similarly to how done in Remark \ref{spectra}, $\text{Proj}(S)$ can be made into a topological space equipped with a suitable sheaf of rings $\mathcal{O}_{\text{Proj}(S)}$. In fact, $(\text{Proj}(S),\mathcal{O}_{\text{Proj}(S)})$ becomes a scheme (see \cite[Proposition II.2.5]{[Har77]} and the preceding paragraph for more details). Given any ring $R$, a prototypical example is
\[
P_R^n\coloneqq \text{Proj}(R[x_0,...,x_n])\cong P_\mathbb{Z}^n\times\text{Spec}(R)\equiv P_\mathbb{Z}^n\times_{\text{Spec}(\mathbb{Z})}\text{Spec}(R)\,,
\]
the \textit{projective $n$-space over $R$} (the product on the right-hand side is explained in Remark \ref{varschproductremark}). Choosing $R=\mathbb{K}$ to be an algebraically closed field, the subspace of closed points of $P_\mathbb{K}^n$ is naturally isomorphic to the projective space $\mathbb{P}^n$ over $\mathbb{K}$, the variety from Remark \ref{aboutvarieties}.
\end{Rem}	

As one would imagine, schemes form a category, but we are rather interested in the following richer notion.

\begin{Def}
Let $(S,\mathcal{O}_S)$ be some fixed scheme, which we call the \textbf{base scheme}\index{scheme!base}. A \textbf{scheme over $S$}\index{scheme!over a base} is a scheme $(X,\mathcal{O}_X)$ together with a morphism of schemes $\varphi_X:X\rightarrow S$. Given another scheme $(Y,\mathcal{O}_Y)$ over $S$, an \textbf{$S$-morphism}\index{morphism!over a base scheme} from $X$ to $Y$ is a scheme morphism $\chi:X\rightarrow Y$ such that $\varphi_X=\varphi_Y\circ\chi$ (also an $S$-isomorphism when $\chi$ is an isomorphism).

Defining the composition of $S$-morphisms so to respect this commutation law, we obtain $\mathsf{Sch}(S)$, the \textbf{category of schemes over $S$}\index{category!of schemes over a base} (and simply write $\mathsf{Sch}(R)\equiv\mathsf{Sch}(\text{Spec}(R))$ for any ring $R$ when no confusion may arise).

In particular, $\mathsf{Sch}(\mathbb{K})$ denotes the category of \textbf{schemes over a field}\index{scheme!over a field} $\mathbb{K}$ (any). We call the morphism $\varphi_X:X\rightarrow\text{Spec}(\mathbb{K})$ attached to any $X\in\text{obj}(\mathsf{Sch}(\mathbb{K}))$ its \textit{structure morphism}. 
\end{Def}

Here is a first important result, proved in \cite[Proposition II.2.6]{[Har77]}, which justifies why we can focus on schemes without loss of generality:

\begin{Pro}\label{XiPro}
Let $\mathbb{K}$ be an algebraically closed field. Then there exists a natural, fully faithful functor
\begin{equation}\label{Xi}
	\Xi:\mathsf{Var}(\mathbb{K})\rightarrow\mathsf{Sch}(\mathbb{K})\,.
\end{equation}
Moreover, for any variety $X\in\textup{obj}(\mathsf{Var}(\mathbb{K}))$ there is a homeomorphism $f$ from its underlying topological space to that formed by closed points of $\Xi(X)$, and its structure sheaf $\mathcal{O}_X$ \textup(the sheaf of regular functions on $X$\textup) is obtained by restriction of $\mathcal{O}_{\Xi(X)}$ through $f$.
\end{Pro}

We stress out that $\Xi$ is far from being an isomorphism of categories: varieties are not synonym of schemes! At the very best, we can restrict the target space to a suitable subcategory of $\mathsf{Sch}(\mathbb{K})$ to promote $\Xi$ to an equivalence. And we will do so, once we have an understanding of all the features decorating the image of any variety under $\Xi$. Anyway, full faithfulness already guarantees to us that when dealing with a \textit{morphism} between varieties we might as well regard it as a morphism of the associated schemes!

Here comes the next storm of definitions, hopefully providing the reader with a satisfactory vocabulary, which will be appealed to times and again when discussing future statements about derived categories. We refer to \cite[sections II.3--II.4]{[Har77]}.

\begin{Def}
A scheme $(X,\mathcal{O}_X)$ is:
\begin{itemize}[leftmargin=0.5cm]
	\item \textbf{integral}\index{scheme!integral} if $\mathcal{O}_X(U)$ is an integral domain for every open $U\subset X$, possible if and only if $X$ is reduced (each $\mathcal{O}_X(U)$ has no nilpotent elements) and irreducible (as a topological space);
	
	\item \textbf{locally noetherian}\index{scheme!locally noetherian} if it can be covered by affine schemes $\text{Spec}(R_i)$ associated to noetherian rings $R_i$ (rings where the ascending chain condition for ideals holds);
	
	\item \textbf{noetherian}\index{scheme!noetherian} if it is locally noetherian and compact, so that it can be covered by \textit{finitely} many affine schemes of noetherian rings as above (in turn implying that $X$ is a noetherian topological space); in particular, $\text{Spec}(R)$ is noetherian if and only if $R$ is.		
\end{itemize}
\noindent A morphism of schemes $f:X\rightarrow Y$ is:
\begin{itemize}[leftmargin=0.5cm]
	\item \textbf{compact}\index{morphism of schemes!compact} if $Y$ can be covered by affine schemes $\text{Spec}(S_i)\subset Y$ such that each $f^{-1}(\text{Spec}(S_i))\subset X$ is compact; 
	
	\item \textbf{locally of finite type}\index{morphism of schemes!locally of finite type} if $Y$ can be covered by affine schemes $\text{Spec}(S_i)\subset Y$ such that for each $i$ the subset $f^{-1}(\text{Spec}(S_i))\subset X$ admits an open cover by affine schemes $\text{Spec}(R_{ij})\subset X$ where each $R_{ij}$ is a finitely-generated $S_i$-algebra;
	
	\item \textbf{of finite type}\index{morphism of schemes!of finite type} if every such cover of each $f^{-1}(\text{Spec}(S_i))$ by affines $\text{Spec}(R_{ij})$ is finite, thus possible if and only if $f$ is locally of finite type and compact;
	
	\item \textbf{finite}\index{morphism of schemes!finite} if each $f^{-1}(\text{Spec}(S_i))=\text{Spec}(R_i)\subset X$ is readily affine, where $R_i$ is an $S_i$-algebra and a finitely-generated $S_i$-module (then $f$ is proper too, the converse being true only for affine varieties over $\mathbb{K}$);
	
	\item \textbf{separated}\index{morphism of schemes!separated} if the diagonal morphism $D:X\rightarrow X\times_Y X$ (making the analogue of diagram \eqref{varschproduct} below commute) is a closed immersion, so that we call $X$ \textbf{separated over $Y$}\index{scheme!separated over a base}, or simply \textit{separated} in case $Y=\text{Spec}(\mathbb{Z})$ (note: $f$ is always separated if both $X$ and $Y$ are affine or noetherian, like $\text{Spec}(\mathbb{K})$); 
	
	\item \textbf{proper}\index{morphism of schemes!proper} if it is separated, of finite type, and universally closed (that is, $f$ maps closed subsets of $X$ to closed ones in $Y$, and so does any morphism $X\times_Y Y'\rightarrow Y'$ induced by any base extension $Y'\rightarrow Y$), so that we call $X$ \textbf{proper over $Y$}\index{scheme!proper over a base}, or simply \textit{proper} in case $Y=\text{Spec}(\mathbb{Z})$.
\end{itemize}
We also assign the above attributes to $X$ itself in case $Y=\text{Spec}(\mathbb{K})$ and the structure morphism $f:X\rightarrow\text{Spec}(\mathbb{K})$ satisfies them (for example, we say that $X$ is separated or proper over $\mathbb{K}$).
\end{Def}

\begin{Rem}\label{varschproductremark}
Separatedness and properness of a morphism of schemes involve the notion of fibred products: the \textbf{fibred product}\index{fibred product} of $X,Y\in\text{obj}(\mathsf{Sch}(S))$ over some base scheme $S$ is the scheme $X\times_S Y\in\text{obj}(\mathsf{Sch}(S))$ together with the projection $S$-morphisms $p_X:X\times_S Y\rightarrow X$ and $p_Y:X\times_S Y\rightarrow Y$ fitting into the universal diagram:
\begin{equation}\label{varschproduct}
	\begin{tikzcd}
		Z\arrow[dr, dashed, "\psi"']\arrow[ddr, out=270, in=135, "\psi_X"']\arrow[drr, out=0, in=150, "\psi_Y"] & &\\
		& X\times_S Y\arrow[d, "p_X"]\arrow[r, "p_Y"'] & Y\arrow[d, "\varphi_Y"] \\
		& X\arrow[r, "\varphi_X"'] & S \\
	\end{tikzcd}\quad.\vspace*{-0.3cm}
\end{equation}
Spelled out explicitly, for each $Z\in\text{obj}(\mathsf{Sch}(S))$ and $S$-morphisms $\psi_X:Z\rightarrow X$ and $\psi_Y:Z\rightarrow Y$, there exists a further $S$-morphism $\psi:Z\rightarrow X\times_S Y$ (unique up to unique isomorphism) making the diagram commute. If $S=\text{Spec}(\mathbb{Z})$, we just speak of the product $X\times Y$. In particular, given another base scheme $S'$ with morphism $S'\rightarrow S$ (a \textbf{base extension}\index{base extension}), we can form a new fibred product $X\times_S S'\in\text{obj}(\mathsf{Sch}(S'))$.

As a bonus fact, we mention that \eqref{Xi} behaves nicely with respect to products: the product $X\times Y\in\text{obj}(\mathsf{Var(\mathbb{K})})$ of any two varieties, defined through an analogous universal property, is mapped to the fibred product $\Xi(X)\times_{\text{Spec}(\mathbb{K})} \Xi(Y)$. This should come as no surprise, since these operations indeed coincide with the respective categorical direct products. 
\end{Rem}

Noetherianity of the interested schemes is sufficient to conclude that any morphism between them is separated, but for properness something stronger is needed.

\begin{Def}
Given any scheme $(Y,\mathcal{O}_Y)$, the \textbf{projective $n$-space over $Y$}\index{projective space over a scheme} is $P_Y^n\coloneqq P_\mathbb{Z}^n\times Y\equiv P_\mathbb{Z}^n\times_{\text{Spec}(\mathbb{Z})} Y$. Then a scheme morphism $f:X\rightarrow Y$ is:
\begin{itemize}[leftmargin=0.5cm]
	\item \textbf{projective}\index{morphism of schemes!projective} if there exists some $n>0$ such that $f$ factors as $f=p\circ i$, where $i:X\rightarrow P_Y^n$ is a closed immersion and $p:P_Y^n\rightarrow Y$ the projection onto $Y$; 
	
	\item \textbf{quasi-projective}\index{morphism of schemes!quasi-projective} if there exists some scheme $X'$ such that $f=p\circ i$, where $i:X\rightarrow X'$ is an open immersion and $p:X'\rightarrow Y$ a projection.
\end{itemize}
In case $Y=\text{Spec}(\mathbb{K})$, we say that $X$ itself is a \textbf{quasi-/projective scheme}\index{scheme!quasi-/projective} (over $\mathbb{K}$).
\end{Def}

\begin{Pro}\label{whenproper}
The following properties about projective schemes and projective morphisms hold.
\begin{itemize}[leftmargin=0.5cm]
	\item Let $f:X\rightarrow Y$ be a morphism of noetherian schemes. Then if $f$ is projective, it is proper, while if $f$ is merely quasi-projective, it is separated and of finite type.
	
	\item \textup{(Chow's Lemma)} Let $S$ be a noetherian scheme and let $X\in\textup{obj}(\mathsf{Sch}(S))$ be proper. Then there exist a projective $X'\in\textup{obj}(\mathsf{Sch}(S))$ and an $S$-morphism $X'\rightarrow X$.
	
	\item Let $S$ be a noetherian scheme. Then $X\in\textup{obj}(\mathsf{Sch}(S))$ is projective if and only if it is proper and there exists a very ample sheaf on $X$ relative to $S$.\footnote{See \cite[sections II.5, II.7]{[Har77]} for the definition of \textit{very ample} and \textit{ample} invertible sheaf.} 
\end{itemize}
\end{Pro}

We are finally ready to specialize Proposition \ref{XiPro} as explained in \cite[Proposition II.4.10, Example II.3.2.1]{[Har77]}.

\begin{Thm}\label{imageunderXi}
Let $\mathbb{K}$ be an algebraically closed field. Then the image of \eqref{Xi} is the set of quasi-projective integral \textup(noetherian\textup) schemes over $\mathbb{K}$, so that
\begin{equation}
	\mathsf{Var}(\mathbb{K})\rightarrow\{X\!\in\!\textup{obj}(\mathsf{Sch}(\mathbb{K}))\mid X \text{ quasi-proj., integral}\}
\end{equation}
is an equivalence. Specifically,
\begin{equation}
	\mkern-12mu\Xi:\{X\!\in\!\textup{obj}(\mathsf{Var}(\mathbb{K})) \mid X\text{ proj.}\}\mapsto\{X\!\in\!\textup{obj}(\mathsf{Sch}(\mathbb{K})) \mid X\text{ proj., integral}\}
\end{equation}
\textup(namely, a projective variety $X$ is sent to the projective spectrum $\textup{Proj}(S(X))$ of its homogeneous coordinate ring, by \textup{\cite[Exercise II.2.14]{[Har77]}}\textup). 

In particular, by Proposition \ref{whenproper}, $\Xi(X)$ is an integral, noetherian, separated scheme of finite type over $\mathbb{K}$ for any $X\in\textup{obj}(\mathsf{Var}(\mathbb{K}))$.
\end{Thm}

This boosts the following definition, which is often taken in literature as the starting one of variety (possibly with some modifications, for example dropping noetherianity).

\begin{Def}\label{abstractvar}
An \textbf{abstract variety}\index{variety!abstract} is an integral, noetherian, separated scheme of finite type over an algebraically closed field $\mathbb{K}$.\footnote{To simplify the language, note that some literature calls varieties as quasi-projective varieties (which makes sense in light of Theorem \ref{imageunderXi}), while abstract varieties are just dubbed varieties.} If moreover proper over $\mathbb{K}$ (by Proposition \ref{whenproper}, the case when it is also projective), it is said to be \textbf{complete}\index{variety!complete}.
\end{Def}

We also explain when schemes over $\mathbb{K}$ are smooth, and when this is compatible with our definition of smoothness for varieties.

\begin{Def}\label{smoothsch}
A scheme $(X,\mathcal{O}_X)$ is \textbf{smooth over}\index{scheme!smooth over a field} $\mathbb{K}$ if and only if it is regular, meaning that all its local rings $\mathcal{O}_{X,x}$ ($\cong R_\mathfrak{p}$ for some ring $R$ and prime ideal $\mathfrak{p}\in\text{Spec}(R)$) are regular local rings (cf. Definition \ref{smoothvar}).

In particular, if $X$ is irreducible and separated over $\mathbb{K}$, then $X$ is smooth as a scheme if and only if it is smooth as a variety.
\end{Def}

Let us summarize with the help of a big diagram the chains of direct implications among the many attributes a scheme $X\in\text{obj}(\mathsf{Sch}(\mathbb{K}))$ can have (the grey arrows additionally require that $X$ be noetherian):
\begin{figure*}[htp]
\centering
\includegraphics[width=0.95\textwidth]{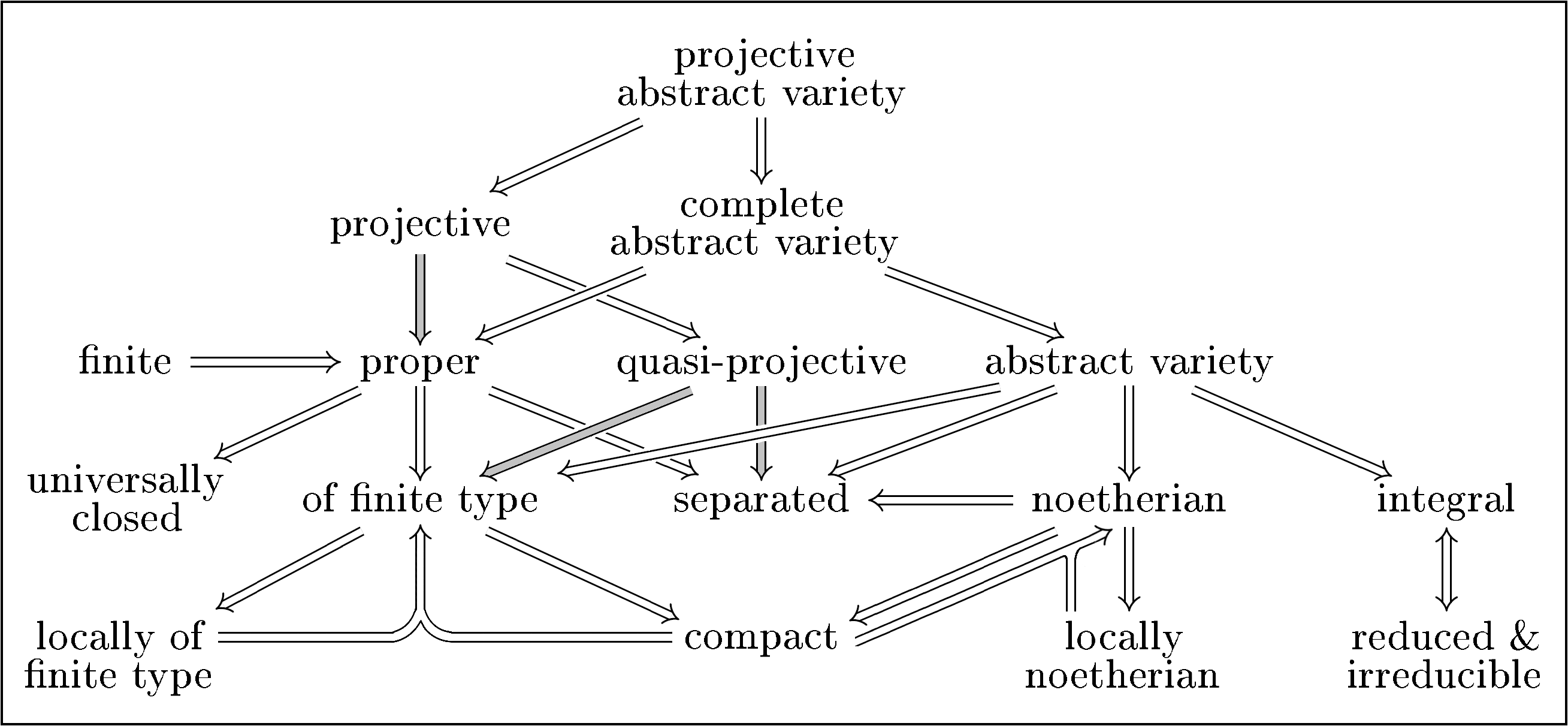}\quad
\end{figure*}\vspace*{-0.5cm}
\begin{equation}\label{bigdiagram}
%\begin{small}
%\begin{tikzpicture}[mybox/.style={draw, inner sep=1.5pt}]
%\node[mybox] (box){
%\begin{tikzcd}[column sep=1em, row sep=0.3cm]
%& & {\begin{aligned}\text{projective}\\[-0.2cm] \text{abstract variety}\end{aligned}}\arrow[dl, Rightarrow]\arrow[d, Rightarrow] & & \\
%& \text{projective}\arrow[d, Rightarrow, crossing over]\arrow[dr, Rightarrow] & {\begin{aligned}\text{complete}\\[-0.2cm] \!\!\text{abstract variety}\end{aligned}}\arrow[dr, Rightarrow]\arrow[dl, Rightarrow] &  & \\
%\text{finite}\arrow[r, Rightarrow] & \text{proper}\arrow[dl, Rightarrow]\arrow[d, crossing over, Rightarrow]\arrow[dr, crossing over, Rightarrow] & {\begin{aligned}\text{quasi-projective}\end{aligned}}\arrow[d, Rightarrow]\arrow[dl, Rightarrow] & {\begin{aligned}\text{abstract variety}\end{aligned}}\arrow[dl, Rightarrow]\arrow[d, Rightarrow]\arrow[dr, Rightarrow]  & \\
%{\begin{aligned}\text{universally}\\[-0.2cm] \!\!\text{closed}\end{aligned}} & {\begin{aligned}\text{of finite type}\end{aligned}}\arrow[dl, Rightarrow]\arrow[dr, Rightarrow]\arrow[urr, Leftarrow] & \text{separated} & \text{noetherian}\arrow[l, Rightarrow]\arrow[d, Rightarrow]\arrow[dl, Rightarrow] & \text{integral}\arrow[d, Leftrightarrow] \\
%{\begin{aligned}\text{locally of}\\[-0.2cm] \!\!\text{finite type}\end{aligned}} &  & \text{compact} & {\begin{aligned}\text{locally}\\[-0.2cm] \!\!\text{noetherian}\end{aligned}}  & {\begin{aligned}\text{reduced \&}\\[-0.2cm] \!\!\text{irreducible}\end{aligned}} 
%\end{tikzcd}
%};
%\end{tikzpicture}\end{small}
\end{equation}

Finally, for later purposes, we wish to establish a functor between modules and sheaves of modules on affine schemes. Thereto, we first need to refine the construction of the structure sheaf of spectra from Remark \ref{spectra}.

\begin{Def}\label{asssheafofmodules}
Let $R$ be a ring, $M$ a (left) $R$-module and $X=\text{Spec}(R)$ an affine scheme, with structure sheaf $\mathcal{O}_X$. Then the \textbf{sheaf $\widetilde{M}$ associated to $M$ on $X$}\index{sheaf!associated to a module} is determined by the assignment	
\begin{small}\begin{align*}
		\widetilde{M}(U)\!\coloneqq\!\bigg\{s\!:U\!\rightarrow\!\bigsqcup_{\mathfrak{p}\in U}\! M_\mathfrak{p}\Bigm|\, &\forall\mathfrak{p}\!\in\! U\!: s(\mathfrak{p})\!\in\! M_\mathfrak{p}\text{ and }\exists U_\mathfrak{p}\!\subset\! U\text{ $\mathfrak{p}$-neighb. }\exists m\!\in\! M\,\exists f\!\in\! R \\[-0.3cm]
		&\text{s.t. }\!\forall\mathfrak{q}\!\in\! U_\mathfrak{p}\text{ holds }f\!\notin\!\mathfrak{q} \text{ and }
		s(\mathfrak{q})\!=\! \frac{m}{f}\!\in\! M_\mathfrak{q}\bigg\}\,,
\end{align*}\end{small}
\!\!for each open $U\subset X$ ($M_\mathfrak{p}$ now indicating the localization of $M$ at the prime ideal $\mathfrak{p}$), with obvious restriction maps. Similarly to the structure sheaves of spectra, one can prove that indeed $\widetilde{M}\in\text{obj}(\mathsf{Sh}(X;{}_{\mathcal{O}_X}\!\mathsf{Mod}))$.
\end{Def}

\begin{Pro}\label{squaretildefunc}
Let $R$ be a ring, $M$ a left $R$-module and $X=\textup{Spec}(R)$ an affine scheme, with structure sheaf $\mathcal{O}_X$. Then $\Gamma(X,\widetilde{M})=M$, and the map $M\mapsto\widetilde{M}$ identifies an exact, fully faithful functor $\widetilde{\square}: {}_R\mathsf{Mod}\rightarrow \mathsf{Sh}(X;{}_{\mathcal{O}_X}\!\mathsf{Mod})$.
\end{Pro}

\begin{proof}
The equality $\Gamma(X,\widetilde{M})=M$ follows from an adaptation of the proof of $\Gamma(X,\mathcal{O}_X)\cong R$. Now, $\widetilde{\square}$ is exact because $(\widetilde{M})_\mathfrak{p}\cong M_\mathfrak{p}$ and the exactness of a sequence of sheaves can be checked at the level of stalks (by Lemma \ref{aboutmonepisheaves}). Finally, for any $M, N\in\text{obj}({}_R\mathsf{Mod})$ we have a natural map $(\widetilde{\square})_{M,N}:\text{Hom}_R(M,N)\rightarrow\text{Hom}_{\mathcal{O}_X}(\widetilde{M},\widetilde{N})$, of inverse $\Gamma(X,\square)$ by our initial arguments, proving its bijectivity and hence the full faithfulness of $\widetilde{\square}$.  
\end{proof}

In fact, one can further prove that $\widetilde{\square}$ and $\Gamma(X,\square)$ form an adjoint pair of functors: $\text{Hom}_R(M,\Gamma(X,\mathcal{F}))\cong\text{Hom}_{\mathcal{O}_X}(\widetilde{M},\mathcal{F})$ for any $M\in\text{obj}({}_R\mathsf{Mod})$ and $\mathcal{F}\in\text{obj}(\mathsf{Sh}(X;{}_{\mathcal{O}_X}\!\mathsf{Mod}))$.

\begin{Rem}\label{asssheafofmodforproj}
Given a graded ring $S$, one can define a sheaf $\widetilde{M}$ associated to any graded $S$-module $M$ on the projective spectrum $X=\text{Proj}(S)$ in a manner similar to how was done in Definition \ref{asssheafofmodules} on affine schemes. It turns out that $\widetilde{M}\in\text{obj}(\mathsf{QCoh}(X;{}_{\mathcal{O}_X}\!\mathsf{Mod}))$, and, in case $S$ is noetherian and $M$ is finitely-generated, that $\widetilde{M}\in\text{obj}(\mathsf{Coh}(X;{}_{\mathcal{O}_X}\!\mathsf{Mod}))$ (by \cite[Proposition II.5.11]{[Har77]}).

Then for any $n\in\mathbb{Z}$ we let $\mathcal{O}_X(n)\coloneqq\widetilde{S(n)}\in\text{obj}(\mathsf{QCoh}(X;{}_{\mathcal{O}_X}\!\mathsf{Mod}))$, where $S(n)\coloneqq\bigoplus_{k\in\mathbb{Z}}S_{n+k}$. In particular, we call $\mathcal{O}_X(1)$ the \textit{twisting sheaf of Serre} and $\mathcal{F}(n)\coloneqq\mathcal{F}\otimes_{\mathcal{O}_X}\mathcal{O}_X(n)$ the \textit{$n$-th twisted sheaf} of $\mathcal{F}\in\text{obj}(\mathsf{Sh}(X;{}_{\mathcal{O}_X}\!\mathsf{Mod}))$.

One can further define the \textit{graded $S$-module associated to} $\mathcal{F}$ by $\Gamma_*(\mathcal{F})\coloneqq\bigoplus_{n\in\mathbb{Z}}\Gamma(X,\mathcal{F}(n))$ and show that under mild conditions the sheaf it induces is isomorphic to $\mathcal{F}$ itself (cf. \cite[Proposition II.5.15]{[Har77]}).  
\end{Rem}

\subsection{Quasi-coherent and coherent sheaves}\label{ch5.6}

With a firmer grasp on schemes, we are ready to reprise the theory of sheaves from where we left it back in Section \ref{ch5.4}, providing first an alternative definition of quasi-/coherent sheaf.

\begin{Rem}\label{disclaimer}(\textit{Disclaimer})
The present section (and a generous amount of the remaining theory) heavily relies on schemes, relegating to literature the unavoidably more technical proofs to which they give rise, and often accepting statements ``on faith''. However, the reader may take solace in knowing/recalling that also the more analytical aspects brought up by Floer theory in \cite{[Imp21]} were subject to some major compromises, due to spacetime constraints and the author's ignorance; specularly, the machinery we now overlook is algebraical in nature --- yet another manifestation of mirror symmetry... Anyway, we will mostly follow \cite{[Har77]}, but a rich lore can also be found for example in \cite[chapter 13]{[Vak17]}. 
\end{Rem}

\begin{Def}\label{quasi-/coherentmodules}
Let $(X,\mathcal{O}_X)$ be a scheme. Then $\mathcal{F}\in\text{obj}(\mathsf{Sh}(X;{}_{\mathcal{O}_X}\!\mathsf{Mod}))$ is a \textbf{quasi-coherent $\mathcal{O}_X$-module}\index{ox@$\mathcal{O}_X$-module!quasi-coherent} if $X$ can be covered by a family of open affine subsets $U_i=\text{Spec}(R_i)\subset X$ such that for each $i$ there exists an $R_i$-module $M_i\in\text{obj}({}_{R_i}\!\mathsf{Mod})$ with $\mathcal{F}|_{U_i}\cong\widetilde{M}_i\in\text{obj}(\mathsf{Sh}(U_i;{}_{\mathcal{O}_{U_i}}\!\mathsf{Mod}))$, $\widetilde{M}_i$ being the sheaf associated to $M_i$ as per Definition \ref{asssheafofmodules}.

If each $M_i$ can be chosen to be finitely-generated (thus in ${}_{R_i}\!\mathsf{Mod}^\mathsf{fg}\subset {}_{R_i}\!\mathsf{Mod}$), $\mathcal{F}$ is moreover a \textbf{coherent $\mathcal{O}_X$-module}\index{ox@$\mathcal{O}_X$-module!coherent}. 

Morphisms of quasi-/coherent sheaves are those at the level of the underlying $\mathcal{O}_X$-modules.  
\end{Def}

We stress that the arising categories $\mathsf{QCoh}(X;{}_{\mathcal{O}_X}\!\mathsf{Mod})$ and $\mathsf{Coh}(X;{}_{\mathcal{O}_X}\!\mathsf{Mod})$ coincide with those on generic ringed spaces when Definition \ref{coherentonringed} is targeted to schemes, more precisely to noetherian schemes in the coherent case\footnote{Coherence is ill-behaved on arbitrary, non-noetherian schemes.} (see \cite[Exercise II.5.4]{[Har77]}). 

For example, the structure sheaf $\mathcal{O}_X$ of any scheme $(X,\mathcal{O}_X)$ is a coherent $\mathcal{O}_X$-module.

\begin{Pro}\label{squaretildefunctocoherent}
Let $R$ be a ring and $X=\textup{Spec}(R)$ an affine scheme with structure sheaf $\mathcal{O}_X$. Then the functor $\widetilde{\square}:M\mapsto\widetilde{M}$ from Proposition \ref{squaretildefunc} can be improved to give two equivalences
\begin{equation}\label{squaretildeequivalences}
	\mkern-18mu\widetilde{\square}:{}_R\mathsf{Mod}\rightarrow\mathsf{QCoh}(X;{}_{\mathcal{O}_X}\!\mathsf{Mod})\quad\text{resp.}\quad \widetilde{\square}:{}_R\mathsf{Mod}^\mathsf{fg}\rightarrow\mathsf{Coh}(X;{}_{\mathcal{O}_X}\!\mathsf{Mod})
\end{equation}
\textup(the latter if $R$ is noetherian\textup),\footnote{Apart from $\mathsf{Coh}(X;{}_{\mathcal{O}_X}\!\mathsf{Mod})$ being ill-behaved, if $R$ is non-noetherian we also have that ${}_R\mathsf{Mod}^\mathsf{fg}$ is \textit{not} an abelian category --- rather undesirable indeed!} both of quasi-inverse $\Gamma(X,\square):\mathcal{F}\mapsto\Gamma(X,\mathcal{F})$.
\end{Pro}  

\begin{proof}
The statement is a consequence of the more general fact (see \cite[Proposition II.5.4]{[Har77]}) that, for general schemes $(X,\mathcal{O}_X)$, $\mathcal{F}\in\text{obj}(\mathsf{Sh}(X;{}_{\mathcal{O}_X}\!\mathsf{Mod}))$ is quasi-coherent if and only if for each open $U=\text{Spec}(R_U)\subset X$ there exists some $M\in\text{obj}({}_{R_U}\!\mathsf{Mod})$ such that $\mathcal{F}|_U\cong\widetilde{M}\in\text{obj}(\mathsf{Sh}(U;{}_{\mathcal{O}_U}\!\mathsf{Mod}))$, and coherent on $X$ noetherian if and only if additionally $M\in\text{obj}({}_{R_U}\!\mathsf{Mod}^\mathsf{fg})$. (Actually, these are equivalent reformulations of Definition \ref{quasi-/coherentmodules}).

Choosing above $U=X$, thus $R_U=R$, we get the essential surjectivity of $\widetilde{\square}$, whereas full faithfulness at the level of morphisms of $\mathcal{O}_X$-modules is provided by Proposition \ref{squaretildefunc} (recall that a functor is an equivalence if and only if it is fully faithful and essentially surjective). The latter also confirms $\Gamma(X,\square)$ to be the associated quasi-inverse.  
\end{proof}

\begin{Pro}\label{cokernelscoherent}
Let $X$ be a scheme. The kernel, cokernel and image sheaves of any morphism of quasi-coherent sheaves are themselves quasi-coherent. 

If $X$ is also noetherian, the kernel, cokernel and image of any morphism of coherent sheaves are coherent as well.
\end{Pro}

\begin{proof}
This is a local verification, so we can take $X=\text{Spec}(R)$ to be affine. But then \eqref{squaretildeequivalences} allows us to regard any morphism of quasi-/coherent sheaves as a morphism of modules (respectively finitely-generated ones). Both ${}_R\mathsf{Mod}$ and ${}_R\mathsf{Mod}^\mathsf{fg}$ are abelian categories (the latter supposing $R$ is noetherian), so they contain kernels, cokernels and images of any of their morphisms. Application of the equivalences $\widetilde{\square}$ then proves the proposition.  
\end{proof}

As an immediate consequence, we come to the fundamental result of this section, which legitimizes the existence of the derived category of coherent sheaves.

\begin{Cor}\label{QCohisabelian}
The category $\mathsf{QCoh}(X;{}_{\mathcal{O}_X}\!\mathsf{Mod})$ is abelian for any scheme $(X,\mathcal{O}_X)$, while the category $\mathsf{Coh}(X;{}_{\mathcal{O}_X}\!\mathsf{Mod})$ is abelian for any noetherian scheme $(X,\mathcal{O}_X)$.
\end{Cor}

\begin{proof}
Since $\mathsf{QCoh}(X;{}_{\mathcal{O}_X}\!\mathsf{Mod})$ and $\mathsf{Coh}(X;{}_{\mathcal{O}_X}\!\mathsf{Mod})$ are full subcategories of the abelian category $\mathsf{Sh}(X;{}_{\mathcal{O}_X}\!\mathsf{Mod})$, it suffices to check closedness under finite sums and co-/kernels. The latter property is provided by Proposition \ref{cokernelscoherent}, while for the former we observe that given two quasi-/coherent sheaves $\mathcal{F}, \mathcal{G}$, with a shared open cover $\{U_i\}_{i\in I}$, it locally holds $(\mathcal{F}\oplus\mathcal{G})|_{U_i}=\mathcal{F}|_{U_i}\oplus\mathcal{G}|_{U_i}\cong \widetilde{M}_i\oplus\widetilde{N}_i\cong(\widetilde{M\oplus N})_i$ (for suitable modules), making $\mathcal{F}\oplus\mathcal{G}$ quasi-/coherent as well. Finally, the zero sheaf is obviously quasi-/coherent.
\end{proof}

\begin{Rem}\label{desirablecoherentprop}
Let us highlight which pleasant features are preserved under quasi-/coherence. Let $(X,\mathcal{O}_X)$ be a scheme.
\begin{itemize}[leftmargin=0.5cm]
	\item Given a quasi-coherent sheaf $\mathcal{F}$, the Hom-group and local $\mathcal{H}om$-sheaf yield
	\begin{equation}\label{locHomonqcoh}
		\begin{aligned}
			&\text{Hom}_{\mathcal{O}_X}\!(\mathcal{F},\square):\mathsf{QCoh}(X;{}_{\mathcal{O}_X}\!\mathsf{Mod})\rightarrow \mathsf{Ab}\,, \\
			&\mathcal{H}om_{\mathcal{O}_X}\!(\mathcal{F},\square):\mathsf{QCoh}(X;{}_{\mathcal{O}_X}\!\mathsf{Mod})\rightarrow \mathsf{QCoh}(X;{}_{\mathcal{O}_X}\!\mathsf{Mod})\,, 
		\end{aligned}
	\end{equation}
	both left exact. If moreover $\mathcal{F}$ is a coherent sheaf, then these functors specialize to the left exact
	\begin{equation}\label{locHomoncoh}
		\begin{aligned}
			&\text{Hom}_{\mathcal{O}_X}\!(\mathcal{F},\square):\mathsf{Coh}(X;{}_{\mathcal{O}_X}\!\mathsf{Mod})\rightarrow \mathsf{Vect}_\mathbb{K}\,, \\
			&\mathcal{H}om_{\mathcal{O}_X}\!(\mathcal{F},\square):\mathsf{Coh}(X;{}_{\mathcal{O}_X}\!\mathsf{Mod})\rightarrow \mathsf{Coh}(X;{}_{\mathcal{O}_X}\!\mathsf{Mod})
		\end{aligned}
	\end{equation}
	(where for Hom we further assumed that $X\in\text{obj}(\mathsf{Sch}(\mathbb{K}))$ is projective, and we denoted by $\mathsf{Vect}_\mathbb{K}$ the category of finite-dimensional vector spaces over $\mathbb{K}$).
	\newline Analogously, we have contravariant left exact functors $\text{Hom}_{\mathcal{O}_X}\!(\square,\mathcal{F})$ and $\mathcal{H}om_{\mathcal{O}_X}\!(\square,\mathcal{F})$. In particular, the dual sheaf $\mathcal{F}^\vee=\mathcal{H}om_{\mathcal{O}_X}\!(\mathcal{F},\mathcal{O}_X)$ of a coherent sheaf $\mathcal{F}$ is coherent as well.
	
	\item Also the tensoring operation behaves well with respect to quasi-/coherent sheaves, yielding the right exact functors \begin{equation}\label{tensoronqcoh}
		\begin{aligned}
			&\mathcal{F}\otimes_{\mathcal{O}_X}\!\square:\mathsf{QCoh}(X;{}_{\mathcal{O}_X}\!\mathsf{Mod})\rightarrow \mathsf{QCoh}(X;{}_{\mathcal{O}_X}\!\mathsf{Mod})\,, \\
			&\square\!\otimes_{\mathcal{O}_X}\mathcal{G}:\mathsf{QCoh}(X;\mathsf{Mod}_{\mathcal{O}_X})\rightarrow \mathsf{QCoh}(X;{}_{\mathcal{O}_X}\!\mathsf{Mod})\,,
		\end{aligned}
	\end{equation}
	for any $\mathcal{F}\in\text{obj}(\mathsf{QCoh}(X;\mathsf{Mod}_{\mathcal{O}_X}))$ and $\mathcal{G}\in\text{obj}(\mathsf{QCoh}(X;{}_{\mathcal{O}_X}\!\mathsf{Mod}))$, respectively the right exact 
	\begin{equation}\label{tensoroncoh}
		\begin{aligned}
			&\mathcal{F}\otimes_{\mathcal{O}_X}\!\square:\mathsf{Coh}(X;{}_{\mathcal{O}_X}\!\mathsf{Mod})\rightarrow \mathsf{Coh}(X;{}_{\mathcal{O}_X}\!\mathsf{Mod})\,, \\
			&\square\!\otimes_{\mathcal{O}_X}\mathcal{G}:\mathsf{Coh}(X;\mathsf{Mod}_{\mathcal{O}_X})\rightarrow \mathsf{Coh}(X;{}_{\mathcal{O}_X}\!\mathsf{Mod})\,,
		\end{aligned}
	\end{equation}
	when $\mathcal{F},\mathcal{G}$ are coherent right respectively left $\mathcal{O}_X$-modules.
	\newline Notice that on $X=\text{Spec}(R)$ affine holds $\Gamma(X,\mathcal{F}\otimes_{\mathcal{O}_X}\!\mathcal{G})\cong\Gamma(X,\mathcal{F})\otimes_R \Gamma(X,\mathcal{G})$. This can be used to define and show that tensor, symmetric and exterior powers of a quasi-coherent sheaf are quasi-coherent.
	
	\item Given a short exact sequence $0\rightarrow\mathcal{F}_1\rightarrow\mathcal{F}_2\rightarrow\mathcal{F}_3\rightarrow 0$ of quasi-coherent sheaves, with $\mathcal{F}_1$ and $\mathcal{F}_3$ also locally free, $\mathcal{F}_2$ must be locally free as well (on an affine subset the same holds for free sheaves).
	\newline Assume that $X$ is moreover noetherian, and let $\mathcal{F}$ be a coherent $\mathcal{O}_X$-module. Then $\mathcal{F}$ is locally free if and only if $\mathcal{F}_x$ is a free $\mathcal{O}_{X,x}$-module for each $x\in X$, and it is also an invertible sheaf if and only if there exists another coherent $\mathcal{G}$ such that $\mathcal{F}\otimes_{\mathcal{O}_X}\!\mathcal{G}\cong\mathcal{O}_X$. 
\end{itemize}
\end{Rem}

It is particularly useful to see how pullback and pushforward functors interact with quasi-/coherent sheaves. Keeping an eye on the many definitions of the last section, we have:

\begin{Pro}\label{pushpullcoherent}
Let $f:(X,\mathcal{O}_X)\rightarrow (Y,\mathcal{O}_Y)$ be a morphism of schemes \textup(thus of the underlying locally ringed spaces; cf. Definition \ref{locallyringedspace}\textup). Then:
\renewcommand{\theenumi}{\roman{enumi}}
\begin{enumerate}[leftmargin=0.6cm]
	\item The pullback functor specializes to $f^*\!\!: \mathsf{QCoh}(Y;{}_{\mathcal{O}_Y}\!\mathsf{Mod})\!\rightarrow\!\mathsf{QCoh}(X;{}_{\mathcal{O}_X}\!\mathsf{Mod})$ and, if $X$, $Y$ are both noetherian, to $f^*\!: \mathsf{Coh}(Y;{}_{\mathcal{O}_Y}\!\mathsf{Mod})\rightarrow \mathsf{Coh}(X;{}_{\mathcal{O}_X}\!\mathsf{Mod})$. They are both right exact.
	
	\item If either $X$ is noetherian or if $f$ is separated and compact, the pushforward functor specializes to the left exact $f_*\!: \mathsf{QCoh}(X;{}_{\mathcal{O}_X}\!\mathsf{Mod})\!\rightarrow \!\mathsf{QCoh}(Y;{}_{\mathcal{O}_Y}\!\mathsf{Mod})$. 
	\newline Moreover, if $f$ is a finite morphism of noetherian schemes or a projective morphism of schemes of finite type over a field $\mathbb{K}$ --- or more generally, if $f$ is a proper morphism between noetherian schemes --- then it induces the left exact $f_*: \mathsf{Coh}(X;{}_{\mathcal{O}_X}\!\mathsf{Mod})\rightarrow \mathsf{Coh}(Y;{}_{\mathcal{O}_Y}\!\mathsf{Mod})$.  
\end{enumerate}
\end{Pro}

\begin{proof}(\textit{Sketch})
Regarding the pullback, the matter is local and hence we can assume $X$ and $Y$ to be affine. Then one applies Proposition \ref{squaretildefunctocoherent} along with \cite[Proposition II.5.2]{[Har77]} to obtain point \textit{i}. Note that right exactness of $f^*=\mathcal{O}_X\otimes_{f^{\!-\!1}\!\mathcal{O}_Y}\! f^{-1}(\square)$ is inherited from that of $f^{-1}$ and the tensor product. 

For the pushforward, by assumption we can take $Y$ to be affine and cover $X$ with finitely many open affine subsets (since $f$ is compact). Then to prove that $f_*\mathcal{F}\in\text{obj}(\mathsf{QCoh}(Y;{}_{\mathcal{O}_Y}\!\mathsf{Mod}))$ for any $\mathcal{F}\in\text{obj}(\mathsf{QCoh}(X;{}_{\mathcal{O}_X}\!\mathsf{Mod}))$ one exploits that the extension $\mathcal{G}_2$ of any pair of quasi-coherent sheaves $\mathcal{G}_1,\mathcal{G}_3$ --- by definition fitting into a short exact sequence $0\rightarrow\mathcal{G}_1\rightarrow\mathcal{G}_2\rightarrow\mathcal{G}_3\rightarrow 0$ --- is itself quasi-coherent (see \cite[Proposition II.5.7]{[Har77]}). 

Restricting $f_*$ to coherent sheaves, we just mention how to prove the scenario where $f$ is projective: by locality, take $Y=\text{Spec}(A)$ for $A$ a finitely-generated $\mathbb{K}$-algebra (thus a ring), so that for any $\mathcal{F}\in\text{obj}(\mathsf{Coh}(X;{}_{\mathcal{O}_X}\!\mathsf{Mod}))$ the previous point already guarantees that $f_*\mathcal{F}$ is quasi-coherent; then by Proposition \ref{squaretildefunctocoherent} holds $f_*\mathcal{F}\cong(\widetilde{\square}\circ\Gamma(Y,\square))(f_*\mathcal{F})=\widetilde{\square}\big(\mathcal{F}(f^{-1}(Y))\big)=\Gamma(X,\mathcal{F})^\sim$, which is actually coherent because $\Gamma(X,\mathcal{F})$ is a finitely-generated $A$-module (by \cite[Theorem II.5.19]{[Har77]}). The more general proof when $f$ is proper is addressed in \cite[section III.8]{[Har77]}.  
\end{proof}

Observe that the right and left exactness of $f^*$ respectively $f_*$ can again be deduced from them being an adjoint pair of functors (see the isomorphism \eqref{pullpushareadjoints}), even in the quasi-/coherent world.

\begin{Cor}\label{globalsectforcoherent}
Let $X$ be a noetherian scheme. Then the global sections functor of Lemma \ref{sectfuncexact} \textup(when $U=X$\textup) specializes to the left exact functors:
\begin{equation}
	\begin{aligned}
		&\Gamma(X,\square): \mathsf{QCoh}(X;{}_{\mathcal{O}_X}\!\mathsf{Mod})\rightarrow\mathsf{Ab}\,, \\
		&\Gamma(X,\square):\mathsf{Coh}(X;{}_{\mathcal{O}_X}\!\mathsf{Mod})\rightarrow\mathsf{Ab}\,.
	\end{aligned}
\end{equation}
In particular, if $X\in\textup{obj}(\mathsf{Sch}(\mathbb{K}))$ is a noetherian scheme over $\mathbb{K}$, respectively a projective \textup(or just a proper\textup) one, we have the left exact functors:
\begin{equation}\label{globsectsfunconcoh}
	\begin{aligned}
		&\Gamma(X,\square):\mathsf{QCoh}(X;{}_{\mathcal{O}_X}\!\mathsf{Mod})\rightarrow\mathsf{Vect}_\mathbb{K}^{\leq\infty}\,, \\
		&\Gamma(X,\square):\mathsf{Coh}(X;{}_{\mathcal{O}_X}\!\mathsf{Mod})\rightarrow\mathsf{Vect}_\mathbb{K}
	\end{aligned}
\end{equation}
\textup(where $\mathsf{Vect}_\mathbb{K}^{\leq\infty}$ is the category of possibly infinite-dimensional $\mathbb{K}$-vector spaces\textup). 
\end{Cor}

\begin{proof}
This follows from Proposition \ref{pushpullcoherent} by simply choosing\footnote{Later in Remark \ref{sectfuncispushforward} we will see that the same holds for sheaves of modules, when choosing the target space to be the one-point space equipped with a suitable constant sheaf.} $Y=\text{Spec}(\mathbb{K})$ and $f:X\rightarrow \text{Spec}(\mathbb{K})$ equal to the structure morphism, which is proper over $\mathbb{K}$. Then Lemma \ref{sectfuncexact} guarantees left exactness. 

If $X$ is projective over $\mathbb{K}$, the last proof already alluded to $\Gamma(X,\mathcal{F})$ being a finitely-generated $\mathbb{K}$-module for any $\mathcal{F}\in\text{obj}(\mathsf{Coh}(X;{}_{\mathcal{O}_X}\!\mathsf{Mod}))$ (a fact more generally shown in the proof of Theorem \ref{sheafcohomforprojsch}), that is, a finite-dimensional $\mathbb{K}$-vector space. Relaxing to quasi-coherent sheaves, we are simply allowing vector spaces of infinite dimension too.
\end{proof}

\newpage
\pagestyle{fancy}

\section{The derived category of coherent sheaves}
\thispagestyle{plain}

\subsection{Derived functors on sheaves of modules}\label{ch6.1}

Let us take a step back to focus again on sheaves of modules; we will return to schemes and coherent sheaves only halfway through Section \ref{ch6.2}, once the ground work has been done.

According to Theorem \ref{Sh(X,Ab)isabelian}, Remark \ref{sheafmodulesrem}a. and our analysis from Chapter 2, we can legitimately talk about the derived category $\mathsf{D}(\mathsf{Sh}(X;\mathsf{A}))$ of sheaves on a topological space $X$ with values in any abelian category $\mathsf{A}$ respectively about the derived category $\mathsf{D}(\mathsf{Sh}(X;{}_\mathcal{S}\mathsf{Mod}))$ of $\mathcal{S}$-modules on a ringed space $(X,\mathcal{S})$. We already know a few functors relating their underlying categories of sheaves, so let us see if they induce derived functors.

\begin{Pro}\label{enoughinjectivemodules}
Let $(X,\mathcal{S})$ be a ringed space. Then $\mathsf{Sh}(X;{}_\mathcal{S}\mathsf{Mod})$ has enough injectives \textup(hence so does $\mathsf{Sh}(X;\mathsf{A})$ for any abelian category $\mathsf{A}$, by Remark \ref{sheafmodulesrem}b.\textup).
\end{Pro}

\begin{proof}
According to Definition \ref{injprojobj}, to show is that any $\mathcal{F}\in\text{obj}(\mathsf{Sh}(X;{}_\mathcal{S}\mathsf{Mod}))$ admits a right resolution by injective $\mathcal{S}$-modules, particularly that $\mathcal{F}$ injects into some $\mathcal{I}\in\text{obj}(\mathsf{Sh}(X;{}_\mathcal{S}\mathsf{Mod}))$ which is injective as an object. 

First notice that for each $x\in X$, the stalk $\mathcal{F}_x\in\text{obj}({}_{\mathcal{S}_x}\!\mathsf{Mod})$ injects\footnote{This is a known fact from standard algebra: given a ring $R$, every $R$-module is isomorphic to a submodule of an injective $R$-module.} into some other $\mathcal{S}_x$-module $I_x$ which is injective as an object of ${}_{\mathcal{S}_x}\!\mathsf{Mod}\equiv\mathsf{Sh}(\{x\};{}_\mathcal{S}\mathsf{Mod})$. Letting $\iota_x:\{x\}\hookrightarrow X$ denote the inclusion, so that $(\iota_x)_*I_x\in\text{obj}(\mathsf{Sh}(X;{}_\mathcal{S}\mathsf{Mod}))$ is the corresponding skyscraper sheaf (see Example \ref{skyscrapersheaf}), let us consider the sheaf $\mathcal{I}\coloneqq\prod_{x\in X}(\iota_x)_*I_x$ (an $\mathcal{S}$-module since $\mathsf{Sh}(X;{}_\mathcal{S}\mathsf{Mod})$ is abelian, and hence contains products of its objects!), with stalks $\mathcal{I}_x\cong I_x$.

Given any other $\mathcal{S}$-module $\mathcal{G}$, we have a canonical isomorphism $\text{Hom}_\mathcal{S}(\mathcal{G},\mathcal{I})\cong\prod_{x\in X}\text{Hom}_\mathcal{S}(\mathcal{G},(\iota_x)_*I_x)$. In turn, for each $x\in X$ it holds $\text{Hom}_\mathcal{S}(\mathcal{G},(\iota_x)_*I_x)\cong\text{Hom}_{\mathcal{S}_x}(\mathcal{G}_x,I_x)$, because any morphism $\chi$ of $\mathcal{S}$-modules on the left-hand side is setwise either the constant map $\chi_U\cong\chi_x$ to $I_x$ (if $x\in U$) or zero (else). Setting $\mathcal{G}=\mathcal{F}$, the resulting $\text{Hom}_\mathcal{S}(\mathcal{F},\mathcal{I})\cong\prod_{x\in X}\text{Hom}_{\mathcal{S}_x}(\mathcal{F}_x,I_x)$ tells us that the local inclusions $\mathcal{F}_x\hookrightarrow I_x$ induce an $\mathcal{S}$-module morphism $\mathcal{F}\rightarrow\mathcal{I}$, itself injective because injective on stalks.

Now, recalling the existence of exact stalk functors $\square_x: \mathsf{Sh}(X;{}_\mathcal{S}\mathsf{Mod})\rightarrow{}_{\mathcal{S}_x}\!\mathsf{Mod}$, we observe that by construction 
\[
\text{Hom}_\mathcal{S}(\square,\mathcal{I})=\prod_{x\in X}(\text{Hom}_{\mathcal{S}_x}(\square,I_x)\circ\square_x):\mathsf{Sh}(X;{}_\mathcal{S}\mathsf{Mod})^\text{opp}\rightarrow\mathsf{Ab}\,.
\] 
But each left exact functor $\text{Hom}_{\mathcal{S}_x}(\square,I_x):{}_{\mathcal{S}_x}\!\mathsf{Mod}^\text{opp}\rightarrow\mathsf{Ab}$ is actually exact, as suggested by Lemma \ref{Homforinjprojobj}! Consequently, so is its composition with $\square_x$ and after taking the direct product. Hence, $\text{Hom}_\mathcal{S}(\square,\mathcal{I})$ is exact. Finally, again Lemma \ref{Homforinjprojobj} implies that $\mathcal{I}$ is forcedly an injective sheaf of $\mathcal{S}$-modules, as desired.    
\end{proof}

Therefore, by Lemma \ref{enoughinjprojobj} and Definition \ref{derivedfunc}, we can consider right derived functors of left exact functors from $\mathsf{Sh}(X;\mathsf{A})$ or $\mathsf{Sh}(X;{}_\mathcal{S}\mathsf{Mod})$:

\begin{Ex}
Given a morphism of ringed spaces $f:(X,\mathcal{S}_X)\rightarrow (Y,\mathcal{S}_Y)$, as well as any $\mathcal{F}\in\text{obj}(\mathsf{Sh}(X;\mathsf{A}))$ and $\mathcal{G}\in\text{obj}(\mathsf{Sh}(X;{}_{\mathcal{S}_X}\!\mathsf{Mod}))$, the most prominent examples of right derived functors are:
\begin{itemize}[leftmargin=0.5cm]
	\item (From Example \ref{exactfuncexample})
	\begin{equation}
		\begin{aligned}
			&\mathsf{R}\text{Hom}_{\mathsf{Sh}(X;\mathsf{A})}(\mathcal{F},\square):\mathsf{D^+}(\mathsf{Sh}(X;\mathsf{A}))\rightarrow\mathsf{D^+(Ab)}\,, \\ 
			&\mathsf{R}\text{Hom}_{\mathcal{S}_X}(\mathcal{G},\square):\mathsf{D^+}(\mathsf{Sh}(X;{}_{\mathcal{S}_X}\!\mathsf{Mod}))\rightarrow\mathsf{D^+(Ab)}\,. \\
		\end{aligned}
	\end{equation}
	
	\item (From Definition \ref{moresheafofmodules})
	\begin{equation}\label{localExtfunc}\begin{footnotesize}
			\mkern-12mu\mathcal{E}xt_{\mathcal{S}_X}\!(\mathcal{G},\square)\!\coloneqq\!\mathsf{R}\mathcal{H}om_{\mathcal{S}_X}\!(\mathcal{G},\square)\!:\!\mathsf{D^+}(\mathsf{Sh}(X;{}_{\mathcal{S}_X}\!\mathsf{Mod}))\!\rightarrow\!\mathsf{D^+}(\mathsf{Sh}(X;{}_{\mathcal{S}_X}\!\mathsf{Mod})), \\
		\end{footnotesize}
	\end{equation}
	and similarly on the sheaf of abelian groups. (More on $\mathcal{E}xt$ in Definition \ref{localExt} below.) 
	
	\item (From Lemma \ref{inversepushforwardexact})
	\begin{equation}
		\begin{aligned}
			&\mathsf{R}f_*:\mathsf{D^+}(\mathsf{Sh}(X;\mathsf{A}))\rightarrow\mathsf{D^+}(\mathsf{Sh}(Y;\mathsf{A}))\,, \\ &\mathsf{R}f^{-1}:\mathsf{D^+}(\mathsf{Sh}(Y;\mathsf{A}))\rightarrow\mathsf{D^+}(\mathsf{Sh}(X;\mathsf{A}))\,.
		\end{aligned}
	\end{equation}
	
	\item (From Corollary \ref{pullbackpushforwardexact})
	\begin{equation}\label{Rflowerstar}
		\mathsf{R}f_*:\mathsf{D^+}(\mathsf{Sh}(X;{}_{\mathcal{S}_X}\!\mathsf{Mod}))\rightarrow\mathsf{D^+}(\mathsf{Sh}(Y;{}_{\mathcal{S}_Y}\!\mathsf{Mod}))
	\end{equation} 
	
	\item (From Lemma \ref{sectfuncexact}, for any open $U\subset X$) 
	\begin{equation}\label{RGammaU}
		\mathsf{R}\Gamma(U,\square):\mathsf{D^+}(\mathsf{Sh}(X;\mathsf{A}))\rightarrow\mathsf{D^+(A)}\,.
	\end{equation}	  
\end{itemize}\vspace*{-0.2cm}
\hfill $\blacklozenge$ 
\end{Ex}

\begin{Rem}\label{sectfuncispushforward}
Here are some relevant things to consider, starting with a piece of warning.
\begin{itemize}[leftmargin=0.5cm]
	\item (Important warning) We \textit{cannot} construct with the same approach the right derived functors of the left exact functors $\text{Hom}_{\mathsf{Sh}(X;\mathsf{A})}(\square,\mathcal{F})$, $\text{Hom}_{\mathcal{S}_X}(\square,\mathcal{G})$ and $\mathcal{H}om_{\mathcal{S}_X}\!(\square,\mathcal{G})$ because these are contravariant functors, meaning that we require the source categories to have enough projectives instead of injectives. However, this fails to hold, as we will mention in a while when wondering about the existence of left derived functors.
	
	\item Considering the classical derived functors (see Definition \ref{classicalderivedfunc}) of the right derived Hom-functors above, by Proposition \ref{ExtRHom} we obtain the alternative characterization of the sheaf-theoretic Ext-functors mentioned at \eqref{Extmodules} (often taken as their default definition): 
	\begin{equation}\label{ExtRHomformodules}
		\begin{aligned}
			&\text{Ext}_{\mathsf{Sh}(X;\mathsf{A})}^n(\mathcal{F},\square)\cong\mathsf{R}^n\text{Hom}_{\mathsf{Sh}(X;\mathsf{A})}(\mathcal{F},\square):\mathsf{Sh}(X;\mathsf{A})\rightarrow\mathsf{Ab}\,, \\ 
			&\text{Ext}_{\mathcal{S}_X}^n(\mathcal{G},\square)\cong\mathsf{R}^n\text{Hom}_{\mathcal{S}_X}(\mathcal{G},\square):\mathsf{Sh}(X;{}_{\mathcal{S}_X}\!\mathsf{Mod})\rightarrow\mathsf{Ab}\,, \\ 
		\end{aligned}  
	\end{equation}
	for each $n\in\mathbb{N}$ (thanks to Theorem \ref{Extprop}, we know that these are trivial functors in case $n<0$, while we retrieve the original Hom-functors for $n=0$).
	
	\item The functor \eqref{Rflowerstar} induces the classical $\mathsf{R}^nf_*:\mathsf{Sh}(X;{}_{\mathcal{S}_X}\!\mathsf{Mod})\rightarrow\mathsf{Sh}(Y;{}_{\mathcal{S}_Y}\!\mathsf{Mod})$ for each $n\in\mathbb{Z}$. One calls them \textit{higher direct image functors} and, for any $\mathcal{S}_X$-module $\mathcal{F}$, $\mathsf{R}^nf_*(\mathcal{F})=H^n(\mathsf{R}f_*(\mathcal{F}))\in\text{obj}(\mathsf{Sh}(Y;{}_{\mathcal{S}_Y}\!\mathsf{Mod}))$ is referred to as the \textbf{higher direct image}\index{higher direct image} of $\mathcal{F}$.
	
	\item Choosing specifically $(Y,\mathcal{S}_Y)=(\{y\},\underline{\mathbb{Z}})$, the resulting pushforward functor $f_*:\mathsf{Sh}(X;{}_{\mathcal{S}_X}\!\mathsf{Mod})\rightarrow\mathsf{Sh}(\{y\};{}_{\underline{\mathbb{Z}}}\mathsf{Mod})\cong\mathsf{Sh}(\{y\};\mathsf{Ab})\equiv\mathsf{Ab}$ is none other than the \textbf{global sections functor}\index{functor!global sections} $\Gamma(X,\square):\mathsf{Sh}(X;{}_{\mathcal{S}_X}\!\mathsf{Mod})\rightarrow\mathsf{Ab}$ from Lemma \ref{sectfuncexact}, specialized to $\mathcal{S}_X$-modules, thus left exact with associated right derived functor $\mathsf{R}\Gamma(X,\square)':\mathsf{D^+}(\mathsf{Sh}(X;{}_{\mathcal{S}_X}\!\mathsf{Mod}))\rightarrow \mathsf{D^+}(\mathsf{Ab})$.\footnote{The $'$ next to $\mathsf{R}\Gamma(X,\square)$ serves to distinguish it from the right derived functor $\mathsf{R}\Gamma(X,\square)$ on $\mathsf{D^+}(\mathsf{Sh}(X;\mathsf{A}))$ one obtains by choosing $U=X$ in \eqref{RGammaU}. It is a temporary notation, as motivated in Proposition \ref{sheafcohomologyisabelian} below.}
\end{itemize} 
\end{Rem}

We put particular emphasis on the right derived functors defined in \eqref{localExtfunc}:

\begin{Def}\label{localExt}
Let $(X,\mathcal{S})$ be a ringed space. For $n\in\mathbb{Z}$ and some given $\mathcal{F},\mathcal{G}\in\text{obj}(\mathsf{Sh}(X;{}_{\mathcal{S}}\mathsf{Mod}))$, the $n$-th classical derived functor of $\mathcal{H}om_{\mathcal{S}}(\mathcal{F},\square)$ defines the \textbf{local $\mathcal{E}xt^n$-sheaf}\index{sheaf!local $\mathcal{E}xt^n$} 
\begin{equation}
	\mathcal{E}xt_{\mathcal{S}}^n(\mathcal{F},\mathcal{G})\coloneqq H^n(\mathsf{R}\mathcal{H}om_{\mathcal{S}}(\mathcal{F},\mathcal{G}))\in\text{obj}(\mathsf{Sh}(X;{}_{\mathcal{S}}\mathsf{Mod}))\,,
\end{equation} 
which is also obtainable as the sheafification of the presheaf set by the assignment $U\subset X\text{ open}\mapsto \text{Ext}_{\mathcal{S}|_U}^n(\mathcal{F}|_U,\mathcal{G}|_U)$. 

As expectable from Corollary \ref{whenclassicalderivedtrivial}, $\mathcal{E}xt_{\mathcal{S}}^0(\mathcal{F},\mathcal{G})\cong \mathcal{H}om_{\mathcal{S}}(\mathcal{F},\mathcal{G})$ and\break $\mathcal{E}xt_{\mathcal{S}}^n(\mathcal{F},\mathcal{G})=0$ if $n<0$, whereas $\mathcal{E}xt_{\mathcal{S}}^0(\mathcal{S},\mathcal{G})\cong \mathcal{G} $ and $\mathcal{E}xt_{\mathcal{S}}^n(\mathcal{S},\mathcal{G})=0$ if $n\neq 0$ (for $\mathcal{H}om_{\mathcal{S}}(\mathcal{S},\square)$ is the identity functor; cf. Definition \ref{moresheafofmodules}).
\end{Def}

\begin{Rem}\label{localExtRem}
We state a couple properties regarding local $\mathcal{E}xt$-sheaves (without proof; the reader is advised to consult \cite[section III.6]{[Har77]}).
\begin{itemize}[leftmargin=0.5cm]
	\item Given $\mathcal{F},\mathcal{G}\in\text{obj}(\mathsf{Sh}(X;{}_{\mathcal{S}}\mathsf{Mod}))$, one can compute $\mathcal{E}xt_{\mathcal{S}}^n(\mathcal{F},\mathcal{G})$ by considering locally free projective resolutions of $\mathcal{F}$, that is, exact sequences $\mathcal{L}^\bullet\rightarrow\mathcal{F}\rightarrow 0$ where each $\mathcal{L}^i\in\text{obj}(\mathsf{Sh}^\mathsf{lf}_{r_i}(X;{}_{\mathcal{S}}\mathsf{Mod}))$ is locally free of finite rank $r_i<\infty$. Then $\mathcal{E}xt_{\mathcal{S}}^n(\mathcal{F},\mathcal{G})\cong H^n(\mathcal{H}om_\mathcal{S}(\mathcal{L}^\bullet,\mathcal{G}))$. (This is particularly useful when $X$ is a quasi-projective scheme over $\text{Spec}(R)$ for $R$ a noetherian ring, since any $\mathcal{F}\in\text{obj}(\mathsf{Coh}(X;{}_{\mathcal{S}}\mathsf{Mod}))$ has such a resolution.)
	
	\item Given $\mathcal{F},\mathcal{G}\in\text{obj}(\mathsf{Sh}(X;{}_{\mathcal{S}}\mathsf{Mod}))$ and some $\mathcal{L}\in\text{obj}(\mathsf{Sh}^\mathsf{lf}_r(X;{}_{\mathcal{S}}\mathsf{Mod}))$, it holds
	\begin{align*}
		&\text{Ext}^n(\mathcal{F}\otimes_\mathcal{S}\mathcal{L},\mathcal{G})\cong \text{Ext}^n(\mathcal{F},\mathcal{L}^\vee\otimes_\mathcal{S}\mathcal{G})\qquad\text{and} \\
		&\mathcal{E}xt_\mathcal{S}^n(\mathcal{F}\otimes_\mathcal{S}\mathcal{L},\mathcal{G})\cong \mathcal{E}xt_\mathcal{S}^n(\mathcal{F},\mathcal{L}^\vee\otimes_\mathcal{S}\mathcal{G})\cong \mathcal{E}xt_\mathcal{S}^n(\mathcal{F},\mathcal{G})\otimes_\mathcal{S}\mathcal{L}^\vee\,.
	\end{align*}     
\end{itemize}
\end{Rem}

\vspace*{0.5cm}
\noindent To obtain instead the left derived functors of the known right exact functors on sheaves, we would ideally need the projective counterpart to Proposition \ref{enoughinjectivemodules}. However, this fails to hold on general ringed spaces! (For example, one can use the projectivity diagram to easily prove that the extension by zero $\iota_*\underline{\mathbb{Z}}_x\in\textup{obj}(\mathsf{Sh}(X;{}_{\underline{\mathbb{Z}}}\mathsf{Mod}))$ of the constant sheaf $\underline{\mathbb{Z}}$ on any one-point space $\{x\}$ with a connected neighbourhood in $X$ is not the quotient of a projective sheaf.) 

Rather, we directly find adapted classes to the right exact functors under investigation. Example \ref{Torfunc} gives us a good idea for the tensoring functor (cf. \cite[Proposition III.8.5]{[GM03]}):

\begin{Lem}\label{flatsareadapted}
Let $(X,\mathcal{S})$ be a ringed space and $\mathcal{M}\in\textup{obj}(\mathsf{Sh}(X;\mathsf{Mod}_\mathcal{S}))$, yielding the functor $\mathcal{M}\otimes_\mathcal{S}\square:\mathsf{Sh}(X;{}_\mathcal{S}\mathsf{Mod})\rightarrow \mathsf{Sh}(X;\mathsf{Ab}),\,\mathcal{N}\mapsto\mathcal{M}\otimes_\mathcal{S}\mathcal{N}$, which is right exact \textup(just like for modules over a ring from Example \ref{exactfuncexample}\textup). Then the class of all flat sheaves is adapted to it, where any $\mathcal{N}\in\textup{obj}(\mathsf{Sh}(X;{}_\mathcal{S}\mathsf{Mod}))$ is \textup{flat} over $\mathcal{S}$ if $\square\otimes_\mathcal{S}\mathcal{N}$ is exact.\footnote{Observe then that $\mathcal{N}$ is flat if and only if each stalk $\mathcal{N}_x$ for $x\in X$ is a flat $\mathcal{S}_x$-module.}

The mirrored statement for the right exact functor $\square\otimes_\mathcal{S}\mathcal{N}:\mathsf{Sh}(X;\mathsf{Mod}_\mathcal{S})\rightarrow \mathsf{Sh}(X;\mathsf{Ab})$ holds as well, and we furthermore remark that in case $\mathcal{S}$ belongs to $\mathsf{Sh}(X;\mathsf{CommRings})$, both functors have target category $\mathsf{Sh}(X;\mathsf{Mod}_\mathcal{S})$. 
\end{Lem}

\begin{proof}
The most important verification here is that any $\mathcal{F}\in\textup{obj}(\mathsf{Sh}(X;{}_\mathcal{S}\mathsf{Mod}))$ is a quotient of flat sheaves. If, given any open subset $U\subset X$, we define $\mathcal{S}_U\in\text{obj}(\mathsf{Sh}(X;{}_\mathcal{S}\mathsf{Mod}))$ by $(\mathcal{S}_U)_x\coloneqq\mathcal{S}_x$ if $x\in U$ and $(\mathcal{S}_U)_x\coloneqq 0$ otherwise, then we obtain a flat left $\mathcal{S}$-module fulfilling $\text{Hom}_\mathcal{S}(\mathcal{S}_U,\mathcal{F})\cong \mathcal{F}(U)$ via $\psi\mapsto\psi_U(1_{\mathcal{S}(U)})$ (cf. the isomorphism \eqref{Homglobalsectiso}).

Now, consider a collection of open sets $\{U_i\}_{i\in I}$ with associated sections $a_i\in\mathcal{F}(U_i)$ such that for any $x\in X$ the stalk $\mathcal{F}_x$ is generated as an $\mathcal{S}_x$-module by those germs $(a_i)_x=[(U_i,a_i)]$ with $x\in U_i$. The by the previous isomorphism $\Gamma(\bigoplus_{i\in I}U_i,\mathcal{F})\cong$ $\text{Hom}_\mathcal{S}(\bigoplus_{i\in I}\mathcal{S}_{U_i},\mathcal{F})$ determined $\mathcal{S}$-module morphism $\phi:\bigoplus_{i\in I}\mathcal{S}_{U_i}\rightarrow\mathcal{F}$ is then clearly surjective, because it is surjective on its stalks. But any direct sum of flat sheaves is itself flat, whence it follows that $\mathcal{F}\cong (\bigoplus_{i\in I}\mathcal{S}_{U_i})/\ker(\phi)$, the desired quotient of flat sheaves.

According to Definition \ref{adaptedclass}, it remains to prove that $\mathcal{M}\otimes_\mathcal{S}\square$ maps acyclic complexes of flat sheaves to acyclic complexes of sheaves of abelian groups. This works just as in the scenario of modules over rings: one first shows that any short exact triple of left $\mathcal{S}_x$-modules with second and third object being flat also has first one flat (see \cite[Exercise III.7.5]{[GM03]}), and applies this result at the level of stalks for each $x\in X$.   
\end{proof}

Consequently, the following set of examples is now disclosed to us:

\begin{Ex}\label{derivedtensorfunc}
Let $(X,\mathcal{S})$ be a ringed space and let $\mathcal{M}\in\textup{obj}(\mathsf{Sh}(X;\mathsf{Mod}_\mathcal{S}))$ respectively $\mathcal{N}\in\textup{obj}(\mathsf{Sh}(X;{}_\mathcal{S}\mathsf{Mod}))$. Then by Lemma \ref{flatsareadapted} and Definition \ref{derivedfunc} there exist left derived functors
\begin{equation}
	\begin{aligned}
		&\mathcal{M}\overset{\mathsf{L}}{\otimes}_\mathcal{S}\square\coloneqq\mathsf{L}(\mathcal{M}\otimes_\mathcal{S}\square):\mathsf{D^-}(\mathsf{Sh}(X;{}_\mathcal{S}\mathsf{Mod}))\rightarrow\mathsf{D^-}(\mathsf{Sh}(X;\mathsf{Ab}))\,, \\ &\square\overset{\mathsf{L}}{\otimes}_\mathcal{S}\mathcal{N}\coloneqq\mathsf{L}(\square\otimes_\mathcal{S}\mathcal{N}):\mathsf{D^-}(\mathsf{Sh}(X;\mathsf{Mod}_\mathcal{S}))\rightarrow\mathsf{D^-}(\mathsf{Sh}(X;\mathsf{Ab}))\,,
	\end{aligned}
\end{equation}
whose classical counterparts from Definition \ref{classicalderivedfunc} yield cohomology sheaves denoted by $\text{Tor}_n^\mathcal{S}(\mathcal{M},\mathcal{N})\coloneqq H^{-n}(\mathcal{M}\overset{\mathsf{L}}{\otimes}_\mathcal{S}\mathcal{N})\in\text{obj}(\mathsf{Sh}(X;\mathsf{Ab}))$ (the same for both functors!), for $n\geq 0$, which is notationally consistent with Example \ref{Torfunc}.

Moreover, given a morphism of ringed spaces $f:(X,\mathcal{S}_X)\rightarrow (Y,\mathcal{S}_Y)$ with $\mathcal{S}_X,\mathcal{S}_Y\in\text{obj}(\mathsf{Sh}(X;\mathsf{CommRings}))$, from Corollary \ref{pullbackpushforwardexact} we know that the pullback functor is right exact, thus yielding the left derived functor
\begin{equation}
	\mathsf{L}f^*:\mathsf{D^-}(\mathsf{Sh}(Y;{}_{\mathcal{S}_Y}\!\mathsf{Mod}))\rightarrow\mathsf{D^-}(\mathsf{Sh}(X;{}_{\mathcal{S}_X}\!\mathsf{Mod}))
\end{equation}
with associated sheaves $\mathsf{L}^{-n}f^*(\mathcal{G})=H^{-n}(\mathcal{S}_X\overset{\mathsf{L}}{\otimes}_{f^{\!-\!1}\mathcal{S}_Y} f^{-1}\mathcal{G})$ for $n\geq 0$. \hfill $\blacklozenge$
\end{Ex}

\subsection{Sheaf cohomology and \v{C}ech cohomology}\label{ch6.2}

Applying Definition \ref{classicalderivedfunc} to the left exact global sections functor from Lemma \ref{sectfuncexact}, we arrive at the fundamental notion of \textit{sheaf cohomology}:

\begin{Def}\label{sheafcohomology}
Let $X$ be a topological space and $\Gamma(X,\square):\mathsf{Sh}(X;\mathsf{Ab})\rightarrow\mathsf{Ab}$ the global sections functor, yielding $\mathsf{R}\Gamma(X,\square):\mathsf{D^+}(\mathsf{Sh}(X;\mathsf{Ab}))\rightarrow \mathsf{D^+}(\mathsf{Ab})$. Given $n\in\mathbb{Z}$, we call the corresponding classical $n$-th right derived functor $H^n(X,\square)\coloneqq \mathsf{R}^n\Gamma(X,\square)= H^n(\mathsf{R}\Gamma(X,\square)):\mathsf{Sh}(X;\mathsf{Ab})\rightarrow\mathsf{Ab}$ the \textbf{$n$-th sheaf cohomology functor}\index{functor!sheaf cohomology}, and
\begin{equation}
	H^n(X,\mathcal{F})=H^n(\mathsf{R}\Gamma(X,\square))(\mathcal{F})\in\text{obj}(\mathsf{Ab})
\end{equation}
the \textbf{$n$-th sheaf cohomology group}\index{sheaf cohomology group} of $\mathcal{F}\in\text{obj}(\mathsf{Sh}(X;\mathsf{Ab}))$. In particular, by Corollary \ref{whenclassicalderivedtrivial}, we have $H^0(X,\mathcal{F})\cong\Gamma(X,\mathcal{F})=\mathcal{F}(X)$ and $H^n(X,\mathcal{F})\cong\{0\}$ if $n<0$.
\end{Def}

\begin{Pro}\label{sheafcohomologyisabelian}
Let $(X,\mathcal{S})$ be a ringed space, yielding the functor  $\mathsf{R}\Gamma(X,\square)':\mathsf{D^+}(\mathsf{Sh}(X;{}_{\mathcal{S}}\mathsf{Mod}))\rightarrow \mathsf{D^+}(\mathsf{Ab})$ constructed at the end of Remark \ref{sectfuncispushforward}. Denoting by $\Phi:\mathsf{Sh}(X;{}_{\mathcal{S}}\mathsf{Mod})\rightarrow\mathsf{Sh}(X;\mathsf{Ab})$ the \textup(clearly exact\textup) forgetful functor which interprets sheaves of $\mathcal{S}$-modules as mere sheaves of $\underline{\mathbb{Z}}$-modules \textup(thus of abelian groups\textup), it holds 
\[
\mathsf{R}\Gamma(X,\square)'\cong\mathsf{R}\Gamma(X,\square)\circ\mathsf{R}\Phi:\mathsf{D^+}(\mathsf{Sh}(X;{}_{\mathcal{S}}\mathsf{Mod}))\rightarrow\mathsf{D^+}(\mathsf{Ab})\,,
\]
where the right-hand $\mathsf{R}\Gamma(X,\square)$ is defined on $\mathsf{D^+}(\mathsf{Sh}(X;\mathsf{Ab}))$, as obtained in \eqref{RGammaU} for $U=X$.

In other words, when computing sheaf cohomology we can forget the additional structures on the topological space $X$ and its sheaves, and simply focus on the framework of Definition \ref{sheafcohomology}. \textup(Notationally, we can then always write $\mathsf{R}\Gamma(X,\square)$ unprimed.\textup)   
\end{Pro}

\begin{proof}
Clearly, $\Gamma(X,\square)'=\Gamma(X,\square)\circ\Phi:\mathsf{Sh}(X;{}_{\mathcal{S}}\mathsf{Mod})\rightarrow\mathsf{Ab}$, hence $\mathsf{R}\Gamma(X,\square)'$ $=\mathsf{R}(\Gamma(X,\square)\circ\Phi):\mathsf{D^+}(\mathsf{Sh}(X;{}_{\mathcal{S}}\mathsf{Mod}))\rightarrow\mathsf{D^+}(\mathsf{Ab})$. Our goal is to apply Proposition \ref{derivedfunccompo}, which produces a canonical isomorphism $\mathsf{R}(\Gamma(X,\square)\circ\Phi)\cong\mathsf{R}\Gamma(X,\square)\circ\mathsf{R}\Phi$. Thereto, we need two classes $\mathcal{R}_1\subset\text{obj}(\mathsf{Sh}(X;{}_{\mathcal{S}}\mathsf{Mod}))$, $\mathcal{R}_2\subset\text{obj}(\mathsf{Sh}(X;\mathsf{Ab}))$ adapted to $\Phi$ respectively to $\Gamma(X,\square)$ such that $\Phi(\mathcal{R}_1)\subset\mathcal{R}_2$.

A first valid candidate is $\mathcal{R}_1\coloneqq\{\mathcal{I}\in\text{obj}(\mathsf{Sh}(X;{}_{\mathcal{S}}\mathsf{Mod}))\mid \mathcal{I}\text{ is injective}\}$, by Proposition \ref{enoughinjectivemodules} and Lemma \ref{enoughinjprojobj}. Let $\mathcal{R}_2\coloneqq\{\mathcal{F}\in\text{obj}(\mathsf{Sh}(X;\mathsf{Ab}))\mid \mathcal{F}\text{ is flabby}\}$, where a sheaf $\mathcal{F}$ of abelian groups is ``flabby'' if the restriction morphism $\rho^X_U:\mathcal{F}(X)=\Gamma(X,\mathcal{F})\rightarrow\mathcal{F}(U)=\Gamma(U,\mathcal{F})$ is surjective for all $U\subset X$ open. Given a sheaf $\mathcal{F}$, the assignment $\mathcal{F}_\text{flabby}(U)\coloneqq\prod_{x\in U}\mathcal{F}_x$ then makes $\mathcal{F}_\text{flabby}$ clearly flabby. Assuming $\mathcal{F}\in\mathcal{R}_1$, $\Phi(\mathcal{F})_\text{flabby}$ contains by construction $\Phi(\mathcal{F})$ as a direct summand, thus necessarily flabby. This checks $\Phi(\mathcal{R}_1)\subset\mathcal{R}_2$.

We must still show that $\mathcal{R}_2$ is adapted to the left exact functor $\Gamma(X,\square)$. The first requirement of Definition \ref{adaptedclass} is obviously met: any $\mathcal{F}\in\text{obj}(\mathsf{Sh}(X;\mathsf{Ab}))$ is a subobject of its $\mathcal{F}_\text{flabby}\in\mathcal{R}_2$. Preservation of acyclicity involves more work. The first step is showing that any short exact sequence $0\rightarrow\mathcal{F}\xrightarrow{\chi}\mathcal{G}\xrightarrow{\psi}\mathcal{H}\rightarrow 0$ in $\mathsf{Sh}(X;\mathsf{Ab})$ with $\mathcal{F}\in\mathcal{R}_2$ induces an exact sequence of abelian groups $0\rightarrow\Gamma(X,\mathcal{F})\xrightarrow{\Gamma(X,\chi)}\Gamma(X,\mathcal{G})\xrightarrow{\Gamma(X,\psi)}\Gamma(X,\mathcal{H})\rightarrow 0$, so that by left exactness of the global sections functor one needs only check the surjectivity of $\Gamma(X,\psi)$. This is more tedious than difficult, so we rather refer the reader to \cite[Theorem III.8.3a)]{[GM03]}. A similar strategy is used to prove that if moreover $\mathcal{G}\in\mathcal{R}_2$, then also $\mathcal{H}\in\mathcal{R}_2$.

So suppose $\mathcal{F}^\bullet=(0\rightarrow\mathcal{F}^0\rightarrow\mathcal{F}^1\rightarrow...)\in\text{obj}(\mathsf{Kom^+}(\mathcal{R}_2))$ with differential $d^\bullet$ is acyclic. This means that $\mathcal{Z}^n\coloneqq\ker(d^n)=\text{im}(d^{n-1})$, so that each sequence $0\rightarrow\mathcal{Z}^n\hookrightarrow\mathcal{F}^n\xrightarrow{d^n}\mkern-20mu\rightarrow\mathcal{Z}^{n+1}\rightarrow 0$ is exact. Since $\mathcal{Z}^0=0\in\mathcal{R}_2$ and $\mathcal{F}^0\in\mathcal{R}_2$, also $\mathcal{Z}^1\in\mathcal{R}_2$ by the facts just mentioned. Thus, inductively, $\mathcal{Z}^n\in\mathcal{R}_2$ for all $n\in\mathbb{N}$, implying by above that $0\rightarrow\Gamma(X,\mathcal{Z}^n)\hookrightarrow\Gamma(X,\mathcal{F}^n)\xrightarrow{\Gamma(X,d^n)}\mkern-20mu\rightarrow\Gamma(X,\mathcal{Z}^{n+1})\rightarrow 0$ is always exact, hence $\text{im}(\Gamma(X,d^{n-1}))=\Gamma(X,\mathcal{Z}^n)=\ker(\Gamma(X,d^n))$. Therefore, $H^n(\Gamma(X,\mathcal{F}^\bullet))=\ker(\Gamma(X,d^n))/\text{im}(\Gamma(X,d^{n-1}))\cong\{0\}$ for all $n\in\mathbb{N}$, meaning that $\Gamma(X,\mathcal{F}^\bullet)$ is acyclic. So $\mathcal{R}_2$ is adapted to $\Gamma(X,\square)$.    
\end{proof}

\begin{Lem}\label{sheafcohomthroughext}
Let $(X,\mathcal{S})$ be a ringed space and $\mathcal{F}\in\textup{obj}(\mathsf{Sh}(X;{}_{\mathcal{S}}\mathsf{Mod}))$. Then
\[
H^n(X,\mathcal{F})\cong\textup{Ext}_{\mathcal{S}}^n(\mathcal{S},\mathcal{F})\quad\textup{\big(}\cong\textup{Ext}_{\mathsf{Sh}(X;\mathsf{Ab})}^n(\underline{\mathbb{Z}},\mathcal{F}),\;\text{by Proposition \ref{sheafcohomologyisabelian}}\textup{\big)}
\]
for all $n\in\mathbb{Z}$.
\end{Lem}

\begin{proof}
This is straightforward: $\text{Ext}_{\mathcal{S}}^n(\mathcal{S},\mathcal{F})\cong\mathsf{R}^n\text{Hom}_{\mathcal{S}}(\mathcal{S},\mathcal{F})$ by \eqref{ExtRHomformodules}, and $\mathsf{R}^n\text{Hom}_{\mathcal{S}}(\mathcal{S},\mathcal{F})\cong \mathsf{R}^n\Gamma(X,\mathcal{F})=H^n(X,\mathcal{F})$ by the isomorphism \eqref{Homglobalsectiso} and the very definition of sheaf cohomology. (As a consequence, observe that the known identities from Theorem \ref{Extprop} for $n\leq 0$ double-check here those of the trivial sheaf cohomology groups.)
\end{proof}

An immediate consequence of Lemma \ref{sheafcohomthroughext} and Theorem \ref{Extprop} is:

\begin{Cor}\label{longexactforsheafcohomology}
Let $(X,\mathcal{S})$ be a ringed space. Then any short exact sequence  $0\rightarrow\mathcal{F}_1\rightarrow\mathcal{F}_2\rightarrow\mathcal{F}_3\rightarrow 0$ in $\mathsf{Sh}(X;{}_{\mathcal{S}}\mathsf{Mod})$ induces a long exact sequence
\[
...\rightarrow H^n(X,\mathcal{F}_1)\rightarrow H^n(X,\mathcal{F}_2)\rightarrow H^n(X,\mathcal{F}_3)\rightarrow H^{n+1}(X,\mathcal{F}_1)\rightarrow...
\]
of sheaf cohomology groups. 
\end{Cor}

\vspace*{0.5cm}
\noindent Despite these incouraging results, sheaf cohomology is in practice quite difficult to compute. A useful tool to tackle it is \textit{\v{C}ech cohomology} (refer to \cite[section III.4]{[Har77]}).

\begin{Def}
Let $X$ be a topological space equipped with an open cover $\mathfrak{U}=(U_i)_{i\in I}$, where we choose a well-ordering for $I$. Given any $\mathcal{F}\in\text{obj}(\mathsf{Sh}(X;\mathsf{Ab}))$, we can define a complex of abelian groups $\check{C}^\bullet(\mathfrak{U},\mathcal{F})$ by setting
\begin{equation}\label{Cechcomplex}
	\check{C}^n(\mathfrak{U},\mathcal{F})\coloneqq\mkern-18mu\prod_{\begin{small}1\leq i_0<...<i_n\leq |I|\end{small}}\mkern-18mu\mathcal{F}(U_{i_0}\cap ...\cap U_{i_n})\in\text{obj}(\mathsf{Ab})\,,
\end{equation}
for $n\geq 0$, and $\check{C}^n(\mathfrak{U},\mathcal{F})\coloneqq 0$ for $n<0$. Thus, any $s\in \check{C}^n(\mathfrak{U},\mathcal{F})$ has the form $s=(s_{i_0...i_n})$ with $s_{i_0...i_n}\in\mathcal{F}(U_{i_0}\cap ...\cap U_{i_n})$, for each allowable $(n+1)$-tuple $1\leq i_0<...<i_n\leq |I|$ of indices in $I$. Note that we can also consider generic tuples (possibly unordered with repeated indices) upon imposing antisymmetry: $s_{i_0...i_n}=(-1)^\sigma s_{\sigma(i_0)...\sigma(i_n)}$ for $\sigma\in S_{n+1}$ such that $1\leq \sigma(i_0)<...<\sigma(i_n)\leq |I|$ (and $s_{i_0...i_n}=0$ upon index repetition). 

The differential $\check{d}^\bullet:\check{C}^\bullet(\mathfrak{U},\mathcal{F})\rightarrow \check{C}^{\bullet+1}(\mathfrak{U},\mathcal{F})$ is given by
\begin{equation}\label{Cechdifferential}
	(\check{d}^n s)_{i_0...i_{n+1}}\coloneqq \sum_{k=0}^{n+1}(-1)^k s_{i_0...\hat{i_k}...i_{n+1}}|_{U_{i_0}\cap...\cap U_{i_{n+1}}}\in \check{C}^{n+1}(\mathfrak{U},\mathcal{F})\,,
\end{equation}
for $s\in \check{C}^n(\mathfrak{U},\mathcal{F})$. (Recall that a hat on the indices signals omission, which explains why the final restriction to $U_{i_0}\cap...\cap U_{i_{n+1}}$ is necessary.) One can easily show that $\check{d}^{n+1}\circ \check{d}^n=0$, so that we effectively get a complex bounded from below $(\check{C}^\bullet(\mathfrak{U},\mathcal{F}),\check{d}^\bullet)\in\text{obj}(\mathsf{Kom^+(Ab)})$, the \textbf{\v{C}ech complex}\index{cech@\v{C}ech complex}.   
\end{Def}

\begin{Def}
Let $X$ be a topological space with an open cover $\mathfrak{U}=(U_i)_{i\in I}$. Then the \textbf{$n$-th \v{C}ech cohomology group}\index{cech@\v{C}ech cohomology group} of $\mathcal{F}\in\text{obj}(\mathsf{Sh}(X;\mathsf{Ab}))$ is defined as
\begin{equation}
	\check{H}^n(\mathfrak{U},\mathcal{F})\coloneqq H^n(\check{C}^\bullet(\mathfrak{U},\mathcal{F}))=\ker(\check{d}^n)/\text{im}(\check{d}^{n-1})\,.
\end{equation}
\end{Def}

\begin{Lem}\label{CechH0isglobalsect}
Let $X$ be a topological space with open cover $\mathfrak{U}=(U_i)_{i\in I}$, and $\mathcal{F}\in\textup{obj}(\mathsf{Sh}(X;\mathsf{Ab}))$. Then $\check{H}^0(\mathfrak{U},\mathcal{F})\cong\mathcal{F}(X)$.
\end{Lem}

\begin{proof}
By definition, $\check{H}^0(\mathfrak{U},\mathcal{F})=\ker(\check{d}^0)/\{0\}\cong\ker(\check{d}^0\!:\check{C}^0(\mathfrak{U},\mathcal{F})\rightarrow \check{C}^1(\mathfrak{U},\mathcal{F}))$. Any $s\in \check{C}^0(\mathfrak{U},\mathcal{F})=\prod_{i\in I}\mathcal{F}(U_i)$ consists of a family of sections $s_i\in\mathcal{F}(U_i)$, and $(\check{d}^0s)_{i,j}=(s_j-s_i)|_{U_i\cap U_j}$ for $i<j$, so that $\check{d}^0s\overset{!}{=}0$ implies that each couple of sections agree on their overlap. Then, by the sheaf property, there exists a unique $s\in\mathcal{F}(X)$ which restricts to such sections, hence $\ker(\check{d}^0)\subset \mathcal{F}(X)$, and the other inclusion follows by reversing the argument.
\end{proof}

Now we upgrade the \v{C}ech complex to a complex of sheaves.

\begin{Def}
Let $X$ be a topological space equipped with an open cover $\mathfrak{U}=(U_i)_{i\in I}$, $\mathcal{F}\in\text{obj}(\mathsf{Sh}(X;\mathsf{Ab}))$. The \textbf{\v{C}ech complex of sheaves}\index{cech@\v{C}ech complex!of sheaves} $\check{\mathcal{C}}^\bullet(\mathfrak{U},\mathcal{F})$ on $X$ is given by
\begin{equation}
	\begin{aligned}
		&\check{\mathcal{C}}^n(\mathfrak{U},\mathcal{F})\coloneqq\mkern-18mu\prod_{\begin{small}1\leq i_0<...<i_n\leq |I|\end{small}}\mkern-18mu (\iota_{i_0...i_n})_*(\mathcal{F}|_{U_{i_0}\cap...\cap U_{i_n}})\in\text{obj}(\mathsf{Sh}(X;\mathsf{Ab}))\,, \\
		&\check{\mathcal{C}}^n(\mathfrak{U},\mathcal{F})(V)= \mkern-18mu\prod_{\begin{small}1\leq i_0<...<i_n\leq |I|\end{small}}\mkern-18mu \mathcal{F}|_{U_{i_0}\cap...\cap U_{i_n}}(\iota_{i_0...i_n}^{-1}(V))=\check{C}^n(\mathfrak{U}\cap V,\mathcal{F}|_V)\,,
	\end{aligned}
\end{equation}
where $n\geq 0$, $\iota_{i_0...i_n}:U_{i_0}\cap...\cap U_{i_n}\hookrightarrow X$ are inclusions yielding pushforwards $(\iota_{i_0...i_n})_*:\mathsf{Sh}(U_{i_0}\cap...\cap U_{i_n};\mathsf{Ab})\rightarrow\mathsf{Sh}(X;\mathsf{Ab})$, and $\mathfrak{U}\cap V=(U_i\cap V)_{i\in I}$ is an open cover for the generic open subset $V\subset X$ .\footnote{Recall that the $(i_0,...,i_n)$-th summand $\mathcal{F}|_{U_{i_0}\cap...\cap U_{i_n}}(\iota_{i_0...i_n}^{-1}(V))$ is non-zero if and only if $V\subset U_{i_0}\cap...\cap U_{i_n}$.}

Restrictions are the obvious ones (constraining also the cover), whereas the differential $\check{d}^\bullet:\check{\mathcal{C}}^\bullet(\mathfrak{U},\mathcal{F})\rightarrow \check{\mathcal{C}}^{\bullet+1}(\mathfrak{U},\mathcal{F})$ is the collection of all maps $\check{d}_V^n:\check{\mathcal{C}}^n(\mathfrak{U},\mathcal{F})(V)\rightarrow \check{\mathcal{C}}^{n+1}(\mathfrak{U},\mathcal{F})(V)$ defined analogously to \eqref{Cechdifferential}. We thus obtain a genuine $\check{\mathcal{C}}^\bullet(\mathfrak{U},\mathcal{F})\in\text{obj}\big(\mathsf{Kom^+}(\mathsf{Sh}(X;\mathsf{Ab}))\big)$. In particular, $\check{\mathcal{C}}^n(\mathfrak{U},\mathcal{F})(X)=\check{C}^n(\mathfrak{U},\mathcal{F})$.  
\end{Def}

\begin{Lem}\label{Cechinjres}
Let $X$ be a topological space with open cover $\mathfrak{U}=(U_i)_{i\in I}$. For any $\mathcal{F}\in\textup{obj}(\mathsf{Sh}(X;\mathsf{Ab}))$ there exists a natural map $\varepsilon:\mathcal{F}\rightarrow \check{\mathcal{C}}^0(\mathfrak{U},\mathcal{F})$ such that the sequence of sheaves $0\rightarrow\mathcal{F}\xrightarrow{\varepsilon}\check{\mathcal{C}}^0(\mathfrak{U},\mathcal{F})\xrightarrow{\check{d}^0}\check{\mathcal{C}}^1(\mathfrak{U},\mathcal{F})\rightarrow...$ is exact. That is, $\mathcal{F}$ admits an injective resolution $\mathcal{F}\xrightarrow{\varepsilon}\check{\mathcal{C}}^\bullet(\mathfrak{U},\mathcal{F})$.
\end{Lem}

\begin{proof}
Since $\check{\mathcal{C}}^0(\mathfrak{U},\mathcal{F})=\prod_{i\in I} (\iota_i)_*(\mathcal{F}|_{U_i})$, let us define $\varepsilon$ to be the direct product over all maps $\mathcal{F}\rightarrow (\iota_i)_*(\mathcal{F}|_{U_i})$ for $i\in I$, which makes it natural. Then each (non-trivial) group homomorphism $\varepsilon_V:\mathcal{F}(V)\rightarrow \check{\mathcal{C}}^0(\mathfrak{U},\mathcal{F})(V)=\prod_{i\in I} \mathcal{F}|_V(U_i\cap V)$ is clearly injective by the sheaf property (which gives a \textit{unique} global section on $V$ to which any two equal direct products of compatible subsections must glue).

For $n\geq 1$, exactness of the \v{C}ech complex of sheaves is best shown stalkwise (which is an equivalent verification). Given $x\in X$, fix some index $j\in I$ such that $x\in U_j$, and define a homotopy $k^\bullet$ on $\check{\mathcal{C}}^\bullet(\mathfrak{U},\mathcal{F})_x$ made of group homomorphisms $k^n:\check{\mathcal{C}}^n(\mathfrak{U},\mathcal{F})_x\rightarrow \check{\mathcal{C}}^{n-1}(\mathfrak{U},\mathcal{F})_x,\, s_x=[(s,V)]\mapsto k^n(s_x)=[(k^ns,V)]$ with $(k^ns)_{i_0...i_{n-1}}\coloneqq s_{ji_0...i_{n-1}}$, sensible because the $x$-neighbourhood $V\subset X$ is chosen such that $V\subset U_j$ (and thus $V\cap U_{i_0}\cap...\cap U_{i_{n-1}}=V\cap U_j\cap U_{i_0}\cap...\cap U_{i_{n-1}}$). Then:
\begingroup
\allowdisplaybreaks
\begin{small}\begin{align*}
		(\check{d}&^{n-1}\circ k^n + k^{n+1}\circ \check{d}^n)(s)_{i_0...i_n} = \check{d}^{n-1}(k^ns)_{i_0...i_n} + k^{n+1}(\check{d}^ns)_{i_0...i_n} \\
		&= \sum_{k=0}^n(-1)^k(k^ns)_{i_0...\hat{i_k}...i_n}|_{U_{i_0}\cap...\cap U_{i_n}} + (\check{d}^ns)_{ji_0...i_n} \\
		&= \sum_{k=0}^n(-1)^k s_{ji_0...\hat{i_k}...i_n}|_{U_{i_0}\cap...\cap U_{i_n}} \!+\! \Big(\sum_{l=0}^{n+1}(-1)^l s_{i_0...\hat{i_l}...i_{n+1}}|_{U_{i_0}\cap...\cap U_{i_{n+1}}}\Big)_{\!j...i_n} \\
		&= \sum_{k=0}^n(-1)^k s_{ji_0...\hat{i_k}...i_n}|_{U_{i_0}\cap...\cap U_{i_n}} \!+\! s_{i_0...i_n} \!-\! \sum_{l=0}^n(-1)^l s_{ji_0...\hat{i_l}...i_n}|_{U_{i_0}\cap...\cap U_{i_n}} \\ 
		&= (\text{id}_{\check{\mathcal{C}}^n(\mathfrak{U},\mathcal{F})_x}-0)(s)_{i_0...i_n}\,, 
\end{align*}\end{small}
\endgroup
\!\!where we ultimately applied a shift of the index $l$ together with antisymmetry. Therefore, $\text{id}_{\check{\mathcal{C}}^n(\mathfrak{U},\mathcal{F})_x}$ is homotopic to the zero map, so that they are equal in cohomology, whence $H^n(\check{\mathcal{C}}^\bullet(\mathfrak{U},\mathcal{F})_x)=H^n(\text{id}_{\check{\mathcal{C}}^\bullet(\mathfrak{U},\mathcal{F})_x})(H^n(\check{\mathcal{C}}^\bullet(\mathfrak{U},\mathcal{F})_x))\cong\{0\}$. Finally, acyclicity implies the claimed exactness of $\check{\mathcal{C}}^\bullet(\mathfrak{U},\mathcal{F})$.  
\end{proof}

We come to the promised relationship with sheaf cohomology:

\begin{Pro}
Let $X$ be a topological space with open cover $\mathfrak{U}=(U_i)_{i\in I}$, $\mathcal{F}\in\textup{obj}(\mathsf{Sh}(X;\mathsf{Ab}))$. Then for each $n\geq 0$ there is a natural map
\begin{equation}\label{Cechsheafcohmap}
	h_n:\check{H}^n(\mathfrak{U},\mathcal{F})\rightarrow H^n(X,\mathcal{F})\,,
\end{equation}
functorial in $\mathcal{F}$.
\end{Pro}

\begin{proof}
A standard result in homological algebra (see for example \cite[Theorem III.1.3]{[GM03]} or \cite[Theorem IV.4.4]{[HS97]}) asserts that any morphism between the zeroth cohomology groups of any pair of an acyclic and an injective complex can be uniquely (up to homotopy) extended to a cochain map between the whole complexes.

In our case, for the given $\mathcal{F}\in\textup{obj}(\mathsf{Sh}(X;\mathsf{Ab}))$, consider some injective resolution $\mathcal{F}\rightarrow\mathcal{I}^\bullet$ in $\mathsf{Sh}(X;\mathsf{Ab})$, and the acyclic injective one $\mathcal{F}\xrightarrow{\varepsilon}\check{\mathcal{C}}^\bullet(\mathfrak{U},\mathcal{F})$ provided by Lemma \ref{Cechinjres}. Since Lemma \ref{CechH0isglobalsect} exhibits a canonical isomorphism $\check{H}^0(\mathfrak{U},\mathcal{F})\cong\Gamma(X,\mathcal{F})\cong H^0(X,\mathcal{F})\cong H^0(\Gamma(X,\mathcal{I}^\bullet))$ (where the last two isomorphisms follow from the definition of sheaf cohomology respectively from Remark \ref{classicalderivedonresolutions}), by the theorem mentioned above there exists a unique (hence natural) cochain map $\check{\mathcal{C}}^\bullet(\mathfrak{U},\mathcal{F})\rightarrow\mathcal{I}^\bullet$, which moreover can be extended on $\mathcal{F}$ with the identity. Then applying to it the functors $\Gamma(X,\square):\mathsf{Kom}^+(\mathsf{Sh}(X;\mathsf{Ab}))\rightarrow \mathsf{Kom^+(Ab)}$ and $H^n:\mathsf{Kom^+(Ab)}\rightarrow\mathsf{Ab}$ successively, we obtain $\check{\mathcal{C}}^\bullet(\mathfrak{U},\mathcal{F})(X)=\check{C}^\bullet(\mathfrak{U},\mathcal{F})\rightarrow\Gamma(X,\mathcal{I}^\bullet)$ and then the desired $h_n: \check{H}^n(\mathfrak{U},\mathcal{F})\rightarrow H^n(\Gamma(X,\mathcal{I}^\bullet))$ $\cong\mathsf{R}^n\Gamma(X,\square)(\mathcal{F})=H^n(X,\mathcal{F})$ (using once again \eqref{classderinjres}). Also, functoriality in $\mathcal{F}$ follows by construction. 
\end{proof}

\begin{Rem}
One can show (resorting to suitable spectral sequences; see \cite[Theorem III.8.3b]{[GM03]}) that in presence of an \textit{$\mathcal{F}$-acyclic} open cover $\mathfrak{U}=(U_i)_{i\in I}$ --- a cover such that $H^n(V,\mathcal{F}|_V)\cong\{0\}$ for every $n>0$ and for all possible finite intersections $V=U_{i_0}\cap...\cap U_{i_m}\subset X$ with $m\geq 0$ --- the canonical maps $h_n$ from the last proposition are actually isomorphisms!
\end{Rem}

Unfortunately, one generally has to deal with standard, non-acyclic open covers, losing the benefit just highlighted. However, things become more interesting when dealing with quasi-coherent sheaves on noetherian separated schemes, as we will shortly see. We first prove a preparatory lemma of general value.

\begin{Lem}\label{Cechonflasqueistrivial}
Let $X$ be a topological space with open cover $\mathfrak{U}=(U_i)_{i\in I}$. Let $\mathcal{F}\in\textup{obj}(\mathsf{Sh}(X;\mathsf{Ab}))$ be a \textbf{flasque sheaf}\index{sheaf!flasque}, that is, a sheaf with restriction maps $\rho_V^U:\mathcal{F}(U)\rightarrow\mathcal{F}(V)$ which are surjective for each pair $V\subset U\subset X$ of open subsets. Then $\check{H}^n(\mathfrak{U},\mathcal{F})\cong\{0\}$ for all $n>0$.
\end{Lem}

\begin{proof}(\textit{Sketch})
Take the injective resolution $\mathcal{F}\rightarrow \check{\mathcal{C}}^\bullet(\mathfrak{U},\mathcal{F})$ from Lemma \ref{Cechinjres}, and observe that each $\check{\mathcal{C}}^n(\mathfrak{U},\mathcal{F})=\prod_{1\leq i_0<...<i_n\leq |I|}(\iota_{i_0...i_n})_*(\mathcal{F}|_{U_{i_0}\cap...\cap U_{i_n}})\in\text{obj}(\mathsf{Sh}(X;\mathsf{Ab}))$ is itself flasque: every $\mathcal{F}|_{U_{i_0}\cap...\cap U_{i_n}}$ is flasque by definition, and so is its pushforward under continuous maps (such as $\iota_{i_0...i_n}$), and finally the product of flasque sheaves is flasque. Since by \cite[Remark III.2.5.1]{[Har77]} flasque sheaves are acyclic for $\Gamma(X,\square)$, our flasque resolution $\mathcal{F}\rightarrow \check{\mathcal{C}}^\bullet(\mathfrak{U},\mathcal{F})$ (which by \eqref{classderinjres} we can use to compute sheaf cohomology) gives trivial sheaf cohomology groups $H^n(X,\mathcal{F})$ for all $n>0$. Therefore, $\check{H}^n(\mathfrak{U},\mathcal{F})=H^n\big(\Gamma(X,\check{\mathcal{C}}^\bullet(\mathfrak{U},\mathcal{F}))\big)\cong H^n(X,\mathcal{F})\cong\{0\}$ for all $n>0$. 
\end{proof}

It is worth mentioning that injective $\mathcal{S}$-modules over any ringed space $(X,\mathcal{S})$ are flasque and hence, as exploited in the proof just concluded, induce trivial sheaf cohomology groups for each non-zero index.

\vspace*{0.3cm}
\noindent At last, we return to schemes and quasi-/coherent sheaves. (On this regard, we gently remind the reader of the disclaimer in Remark \ref{disclaimer}.)

\begin{Lem}\label{neq0sheafcohomonaffinesch}
Let $X=\textup{Spec}(R)$ be an affine scheme on a \textup(noetherian\textup) ring $R$ \textup(so that $X$ itself is noetherian\textup), and let $\mathcal{O}_X$ be its structure sheaf. Then $H^n(X,\mathcal{F})\cong \{0\}$ for all $\mathcal{F}\in\textup{obj}(\mathsf{QCoh}(X;{}_{\mathcal{O}_X}\!\mathsf{Mod}))$ and $n>0$. \textup(In particular, this is true for affine varieties.\textup)
\end{Lem}

\begin{proof}(\textit{Sketch})
Fix some $\mathcal{F}\in\text{obj}(\mathsf{QCoh}(X;{}_{\mathcal{O}_X}\!\mathsf{Mod}))$ and let $M\coloneqq\Gamma(X,\mathcal{F})\in\text{obj}({}_R\mathsf{Mod})$. Take an injective resolution $0\rightarrow M\rightarrow I^\bullet$ in ${}_R\mathsf{Mod}$ (always existing), which under the equivalence \eqref{squaretildeequivalences} yields an exact sequence $0\rightarrow \widetilde{M}\cong\mathcal{F}\rightarrow\widetilde{I}^\bullet$ of quasi-coherent sheaves on $X$. One can show that each $\widetilde{I}^\bullet$ is actually flasque (\cite[Proposition III.3.4]{[Har77]}). Thus, as argued in the last proof, the resulting flasque resolution gives trivial sheaf cohomology groups for each $n>0$, while $H^0(X,\mathcal{F})=\Gamma(X,\mathcal{F})=M$ by definition. 
\end{proof}

\begin{Cor}\label{embedinflasque}
Any $\mathcal{F}\in\textup{obj}(\mathsf{QCoh}(X;{}_{\mathcal{O}_X}\!\mathsf{Mod}))$ on a noetherian scheme $(X,\mathcal{O}_X)$ can be embedded in a flasque $\mathcal{G}\in\textup{obj}(\mathsf{QCoh}(X;{}_{\mathcal{O}_X}\!\mathsf{Mod}))$.
\end{Cor}

\begin{proof}
Cover $X$ by finitely many affines $U_i=\text{Spec}(R_i)$ and let $\widetilde{M}_i\coloneqq\mathcal{F}|_{U_i}$ $\in\textup{obj}(\mathsf{QCoh}(U_i;{}_{\mathcal{O}_{U_i}}\!\mathsf{Mod}))$, whose associated $M_i\in\text{obj}({}_{R_i}\mathsf{Mod})$ are assumed to be embedded in an injective $I_i\in\text{obj}({}_{R_i}\mathsf{Mod})$ for each $i$. Letting $\iota_i:U_i\hookrightarrow X$ denote the inclusion, set $\mathcal{G}\coloneqq\bigoplus_{i\in I}(\iota_i)_*(\widetilde{I}_i)\in\text{obj}(\mathsf{QCoh}(X;{}_{\mathcal{O}_X}\!\mathsf{Mod}))$, which is flasque because it is the direct sum of pushforwards of flasque sheaves $\widetilde{I}_i$ (again by \cite[Proposition III.3.4]{[Har77]}), and quasi-coherent because it is the direct sum of pushforwards (which preserve quasi-coherence; see Proposition \ref{pushpullcoherent}) of $\widetilde{I}_i\in\text{obj}(\mathsf{QCoh}(U_i;{}_{\mathcal{O}_{U_i}}\!\mathsf{Mod}))$.

But, by construction, $\mathcal{F}|_{U_i}$ injects into $\widetilde{I}_i$ for each $i$, which in turn induces $\mathcal{F}\hookrightarrow(\iota_i)_*(\widetilde{I}_i)$ and finally an injection $\mathcal{F}\hookrightarrow\mathcal{G}$. So $\mathcal{F}$ can be embedded in a flasque, quasi-coherent sheaf.
\end{proof}

\begin{Rem}\label{QCohhasenoughinjectives}
We have the following remarkable consequences:
\begin{itemize}[leftmargin=0.5cm]
	\item Lemma \ref{neq0sheafcohomonaffinesch} can be upgraded to: a noetherian scheme $X$ is affine if and only if $H^n(X,\mathcal{F})\cong \{0\}$ for all $\mathcal{F}\in\text{obj}(\mathsf{QCoh}(X;{}_{\mathcal{O}_X}\!\mathsf{Mod}))$ and $n>0$.
	
	\item The quasi-coherent sheaf $\mathcal{G}$ from the last proof is actually an injective object (in $\mathsf{Sh}(X;{}_{\mathcal{O}_X}\!\mathsf{Mod})$ as well), meaning that $\mathsf{QCoh}(X;{}_{\mathcal{O}_X}\!\mathsf{Mod})$ has enough injectives when $X$ is a noetherian (not necessary) scheme! We mitigate the enthusiasm by pointing out that it does \textit{not} have enough projectives, much like $\mathsf{Sh}(X;{}_{\mathcal{O}_X}\!\mathsf{Mod})$ doesn't.
	
	\item Moreover, on $(X,\mathcal{O}_X)$ noetherian any injective object of $\mathsf{QCoh}(X;{}_{\mathcal{O}_X}\!\mathsf{Mod})$ is flasque, which can be used to prove that sheaf cohomology can be computed also from $\mathsf{R}\Gamma(X,\square):\mathsf{D^+}(\mathsf{QCoh}(X;{}_{\mathcal{O}_X}\!\mathsf{Mod}))\rightarrow \mathsf{D^+}(\mathsf{Vect}_\mathbb{K}^{\leq\infty})$.  
\end{itemize}
\end{Rem}

\begin{Thm}\label{sheafcohomthroughCech}
Let $(X,\mathcal{O}_X)$ be a noetherian separated scheme equipped with an ordered affine open cover $\mathfrak{U}=(U_i)_{i\in I}$, and let $\mathcal{F}\in\textup{obj}(\mathsf{QCoh}(X;{}_{\mathcal{O}_X}\!\mathsf{Mod}))$. Then the natural maps $h_n$ in \eqref{Cechsheafcohmap} are group isomorphisms, that is, \v{C}ech cohomology computes sheaf cohomology:
\begin{equation}
	\check{H}^n(\mathfrak{U},\mathcal{F})\cong H^n(X,\mathcal{F})\,,
\end{equation}
for all $n\geq 0$.
\end{Thm}

\begin{proof}(\textit{Sketch})
For $n=0$ this is immediate from Lemma \ref{CechH0isglobalsect}: $H^0(X,\mathcal{F})\cong\Gamma(X,\square)(\mathcal{F})\cong\check{H}^0(\mathfrak{U},\mathcal{F})$. For $n>0$, begin by embedding $\mathcal{F}$ in a flasque $\mathcal{G}\in\textup{obj}(\mathsf{QCoh}(X;{}_{\mathcal{O}_X}\!\mathsf{Mod}))$ (possible by Corollary \ref{embedinflasque}), and complete to a short exact sequence $0\rightarrow\mathcal{F}\rightarrow\mathcal{G}\rightarrow\mathcal{H}\rightarrow 0$ through the quotient sheaf $\mathcal{H}$ of $\mathcal{O}_X$-modules given by the cokernel of this embedding. In particular, $\mathcal{H}\in\text{obj}(\mathsf{QCoh}(X;{}_{\mathcal{O}_X}\!\mathsf{Mod}))$ by Proposition \ref{cokernelscoherent}.

Now, observe that $U_{i_0}\cap...\cap U_{i_n}\subset X$ is an affine open subset (false if $X$ is not separated!). On it, \cite[Proposition II.5.6]{[Har77]} affirms that the short exact sequence above starting with $\mathcal{F}$ quasi-coherent induces a short exact sequence
\[
0\rightarrow\Gamma(U_{i_0}\cap...\cap U_{i_n},\mathcal{F})\rightarrow\Gamma(U_{i_0}\cap...\cap U_{i_n},\mathcal{G})\rightarrow\Gamma(U_{i_0}\cap...\cap U_{i_n},\mathcal{H})\rightarrow 0
\]
which, after taking products over all $1\leq i_0<...< i_n\leq|I|$ and any $n\geq 0$, in turn gives a short exact sequence $0\rightarrow\check{C}^\bullet(\mathfrak{U},\mathcal{F})\rightarrow\check{C}^\bullet(\mathfrak{U},\mathcal{G})\rightarrow\check{C}^\bullet(\mathfrak{U},\mathcal{H})\rightarrow 0$ in $\mathsf{Kom^+(Ab)}$. Application of the standard cohomology functor $H^n$ results\footnote{Warning: arguing that the short exact sequence $0\rightarrow\mathcal{F}\rightarrow\mathcal{G}\rightarrow\mathcal{H}\rightarrow 0$ of $\mathcal{O}_X$-modules directly induces a long exact one in \v{C}ech cohomology is wrong, because the functors $\check{H}^n(\mathfrak{U},\square)$ are not cohomological in general! (For example, if $\mathfrak{U}=\{X\}$, then $\check{H}^0(\mathfrak{U},\square)=\Gamma(X,\square)$ is just left exact!)} in the long exact sequence of \v{C}ech cohomology groups
\[
...\rightarrow\check{H}^n(\mathfrak{U},\mathcal{G})\cong\{0\}\rightarrow\check{H}^n(\mathfrak{U},\mathcal{H})\rightarrow \check{H}^{n+1}(\mathfrak{U},\mathcal{F})\rightarrow \check{H}^{n+1}(\mathfrak{U},\mathcal{G})\cong\{0\}\rightarrow...\,,
\]
where we used Lemma \ref{Cechonflasqueistrivial} to deduce the triviality of the groups of $\mathcal{G}$. Consequently, $\check{H}^n(\mathfrak{U},\mathcal{H})\cong \check{H}^{n+1}(\mathfrak{U},\mathcal{F})$ for all $n>0$.

On the other hand, Corollary \ref{longexactforsheafcohomology} tells us that $0\rightarrow\mathcal{F}\rightarrow\mathcal{G}\rightarrow\mathcal{H}\rightarrow 0$ readily provides the long exact sequence of sheaf cohomology groups
\[
...\rightarrow H^n(X,\mathcal{G})\cong\{0\}\rightarrow H^n(X,\mathcal{H})\rightarrow H^{n+1}(X,\mathcal{F})\rightarrow H^{n+1}(X,\mathcal{G})\cong\{0\}\rightarrow...\,,
\]
where we exploited that flasque sheaves make also sheaf cohomology groups trivial (by \cite[Proposition III.2.5]{[Har77]}). Thus, $H^n(X,\mathcal{H})\cong H^{n+1}(X,\mathcal{F})$ for all $n>0$.

Summing up, $\check{H}^1(\mathfrak{U},\mathcal{F})\cong \check{H}^0(\mathfrak{U},\mathcal{H})\cong H^0(X,\mathcal{H})\cong H^1(X,\mathcal{F})$ by the $n=0$ case for the quasi-coherent sheaf $\mathcal{H}$, and the statement follows by induction.
\end{proof}

\begin{Cor}\label{sheafcohomthroughCechRem}
Let $f:X\rightarrow Y$ be a morphism of noetherian separated schemes which is affine \textup(that is, there is some open cover of $Y$ by affines whose preimage is an open cover of $X$ by affines\textup), and let $\mathcal{F}\in\textup{obj}(\mathsf{QCoh}(X;{}_{\mathcal{O}_X}\!\mathsf{Mod}))$. Then for all $n\geq 0$ holds $H^n(X,\mathcal{F})\cong H^n(Y,f_*\mathcal{F})$.
\end{Cor}

Finally, we report the following very important result by Serre which addresses specifically the sheaf cohomology of projective schemes (see \cite[Theorem III.5.2]{[Har77]}; we already used it in the proof of Corollary \ref{globalsectforcoherent}) and generalizes the case of the projective space (discussed in \cite[Theorem III.5.1]{[Har77]}).

\begin{Thm}\label{sheafcohomforprojsch}
Let $(X,\mathcal{O}_X)$ be a projective scheme over a noetherian ring $R$, so that the invertible twisting sheaf $\mathcal{O}_X(1)$ is very ample \textup(the terminology is explained in Proposition \ref{whenproper} and Remark \ref{asssheafofmodforproj}\textup). Consider any $\mathcal{F}\in\textup{obj}(\mathsf{Coh}(X;{}_{\mathcal{O}_X}\!\mathsf{Mod}))$. Then $H^k(X,\mathcal{F})$ is a finitely-generated $R$-module for each $k\geq 0$ \textup(in particular, so is $\Gamma(X,\mathcal{F})=\mathcal{F}(X)$\textup). 

Moreover, there exists some $n_0\equiv n_0(\mathcal{F})\in\mathbb{Z}$ such that $H^k(X,\mathcal{F}(n))\cong\{0\}$ for each $n\geq n_0$ and $k>0$.
\end{Thm}

\subsection{Derived functors on coherent sheaves}\label{ch6.3}

Henceforth, we concentrate our efforts towards understanding the derived category of quasi-/coherent sheaves. We generally work on noetherian schemes, on which coherence is well behaved. Beware that in light of Theorem \ref{imageunderXi} the following analysis works on varieties as well. The coming material is borrowed from \cite[chapter 3]{[Huy06]}, more briefly reviewed in \cite[chapter 3]{[Hoc19]}. First of all, we shorten our terminology and notation.

\begin{Def}\label{dercatofsch}
Let $(X,\mathcal{O}_X)$ be a noetherian scheme, with associated category $\mathsf{Coh}(X;{}_{\mathcal{O}_X}\!\mathsf{Mod})$ of coherent $\mathcal{O}_X$-modules, abelian by Corollary \ref{QCohisabelian}. Then we call its bounded derived category 
\[
\mathsf{D^b}(X)\coloneqq \mathsf{D^b}(\mathsf{Coh}(X;{}_{\mathcal{O}_X}\!\mathsf{Mod}))
\]
the \textbf{derived category of $X$}\index{derived category!of a noetherian scheme}, for short. Of course,  Proposition \ref{derivedexists} constructs also the derived categories $\mathsf{D^\#\!}(X)\coloneqq\mathsf{D^\#\!}(\mathsf{Coh}(X;{}_{\mathcal{O}_X}\!\mathsf{Mod}))$ for $\#=\emptyset,+,-$.  

We say that two noetherian schemes $X,Y\in\text{obj}(\mathsf{Sch}(\mathbb{K}))$ over any field $\mathbb{K}$ are \textbf{derived equivalent}\index{derived equivalent (noetherian schemes)} if there exists a $\mathbb{K}$-linear\footnote{When working with schemes over a field $\mathbb{K}$, we implicitly assume their derived categories to be $\mathbb{K}$-linear (whose meaning was explained below Definition \ref{addcat}).} equivalence $\mathsf{D^b}(X)\rightarrow\mathsf{D^b}(Y)$ which is exact (recall that on derived categories exactness is always understood in the sense of Definition \ref{exactderivedfunc}).  
\end{Def}

Unfortunately, $\mathsf{Coh}(X;{}_{\mathcal{O}_X}\!\mathsf{Mod})$ doesn't have enough injectives nor projectives, making it apparently unsuitable when constructing derived functors... but not quite, for we can work in the larger abelian category $\mathsf{QCoh}(X;{}_{\mathcal{O}_X}\!\mathsf{Mod})$ of quasi-coherent $\mathcal{O}_X$-modules --- which at least has enough injectives (see Remark \ref{QCohhasenoughinjectives}) and suitable adapted classes --- and subsequently restrict to coherent sheaves! We will exploit this times and again.   

\begin{Rem}
As pointed out in Remark \ref{QCohhasenoughinjectives}, any quasi-coherent sheaf can be embedded into an injective $\mathcal{O}_X$-module. Furthermore, the full abelian subcategory $\mathsf{QCoh}(X;{}_{\mathcal{O}_X}\!\mathsf{Mod})\subset\mathsf{Sh}(X;{}_{\mathcal{O}_X}\!\mathsf{Mod})$ is thick (shown in \cite[Proposition II.5.7]{[Har77]}). Then we can apply Lemma \ref{abeliansubcatderivedequiv}, which allows us to interpret $\mathsf{D^\#\!}(\mathsf{QCoh}(X;{}_{\mathcal{O}_X}\!\mathsf{Mod}))$ as the full subcategory of $\mathsf{D^\#\!}(\mathsf{Sh}(X;{}_{\mathcal{O}_X}\!\mathsf{Mod}))$ which has quasi-coherent (standard) cohomology.  
\end{Rem}

On the same note of last remark, we have (cf. \cite[Proposition 3.5]{[Huy06]}):

\begin{Lem}\label{boundedCohinQCoh}
Let $X$ be a noetherian scheme. Then the natural inclusion functor
\[
\mathsf{I}_X:\mathsf{D^b}(X)\hookrightarrow\mathsf{D^b}(\mathsf{QCoh}(X;{}_{\mathcal{O}_X}\!\mathsf{Mod}))
\]
restricts to an equivalence between $\mathsf{D^b}(X)$ and the full \textup(triangulated\textup) subcategory of $\mathsf{D^b}(\mathsf{QCoh}(X;{}_{\mathcal{O}_X}\!\mathsf{Mod}))$ of bounded complexes with coherent \textup(standard\textup) cohomology.  
\end{Lem}

Now, we check the existence of all derived functors of Section \ref{ch6.1} at the level of coherent sheaves, one by one, also discussing what is the image of honest complexes of sheaves under them. Note that the assumptions on $X$ will be refined here and there (always referencing the vocabulary of Section \ref{ch5.5} summarized in diagram \eqref{bigdiagram}), though in the end they are always met by the schemes we ultimately favour, namely Calabi--Yau varieties.

\begin{Pro}\label{globalsectforcoherentbdd}
Let $X\in\textup{obj}(\mathsf{Sch}(\mathbb{K}))$ be a projective \textup(or just proper\textup) noetherian scheme. Then the \textup(exact\textup) right derived functor on bounded complexes of the left exact global sections functor $\Gamma(X,\square):\mathsf{Coh}(X;{}_{\mathcal{O}_X}\!\mathsf{Mod})\rightarrow\mathsf{Vect}_\mathbb{K}$ of \eqref{globsectsfunconcoh} exists\footnote{Observe that this is not granted because, again, $\mathsf{Coh}(X;{}_{\mathcal{O}_X}\!\mathsf{Mod})$ does not have enough injectives!} and can be constructed as the composition of exact functors.
\end{Pro}

\begin{proof}(\textit{Sketch})
By Corollary \ref{globalsectforcoherent}, and since $\mathsf{QCoh}(X;{}_{\mathcal{O}_X}\!\mathsf{Mod})$ has enough injectives, the left exact $\Gamma(X,\square):$ $\mathsf{QCoh}(X;{}_{\mathcal{O}_X}\!\mathsf{Mod})\rightarrow\mathsf{Vect}_\mathbb{K}^{\leq\infty}$ gives rise to the right derived functor $\mathsf{R}\Gamma(X,\square):\mathsf{D^+}(\mathsf{QCoh}(X;{}_{\mathcal{O}_X}\!\mathsf{Mod}))\rightarrow\mathsf{D^+}(\mathsf{Vect}_\mathbb{K}^{\leq\infty})$. 

Now, a theorem by Grothendieck (see \cite[Theorem III.2.7]{[Har77]}) tells us that the sheaf cohomology groups of any $\mathcal{F}\in\text{obj}(\mathsf{QCoh}(X;{}_{\mathcal{O}_X}\!\mathsf{Mod}))$ are $H^n(X,\mathcal{F})\cong\{0\}$ if $n>\text{dim}(X)$. Then Proposition \ref{derivedfuncrestrictstobdd} can be applied to show that $\mathsf{R}\Gamma(X,\square)$ descends to the derived categories of bounded complexes, $\mathsf{R}\Gamma(X,\square):\mathsf{D^b}(\mathsf{QCoh}(X;{}_{\mathcal{O}_X}\!\mathsf{Mod}))\rightarrow\mathsf{D^b}(\mathsf{Vect}_\mathbb{K}^{\leq\infty})$. We thus have the diagram
\begin{equation}\label{diagramforderivedfuncofbdd}
	\begin{tikzcd}
		\mathsf{D^+}(\mathsf{QCoh}(X;{}_{\mathcal{O}_X}\!\mathsf{Mod})) \arrow[rr, "{\mathsf{R}\Gamma(X,\square)}"] & & \mathsf{D^+}(\mathsf{Vect}_\mathbb{K}^{\leq\infty}) \\
		\mathsf{D^b}(\mathsf{QCoh}(X;{}_{\mathcal{O}_X}\!\mathsf{Mod})) \arrow[rr, "{\mathsf{R}\Gamma(X,\square)}"']\arrow[u, hook] & & \mathsf{D^b}(\mathsf{Vect}_\mathbb{K}^{\leq\infty})\arrow[u, hook] \\
		\mathsf{D^b}(X) \arrow[u, hook, "\mathsf{I}_X"]\arrow[rr, dashed] & & \mathsf{D^b}(\mathsf{Vect}_\mathbb{K})\arrow[u, hook]
	\end{tikzcd}\quad,
\end{equation}
where $\mathsf{I}_X$ denotes again the natural inclusion functor. But since, by Lemma \ref{boundedCohinQCoh}, the latter actually lands into the full subcategory of bounded complexes with coherent cohomology, which in turn maps under the middle $\mathsf{R}\Gamma(X,\square)$ to bounded complexes of finite-dimensional $\mathbb{K}$-vector spaces, we see that we can actually go around the lower square of the diagram, resulting in the composite dashed arrow drawn. As, moreover, the remaining sheaf cohomology groups of any coherent module are finite-dimensional (by Theorem \ref{sheafcohomforprojsch}), and $\mathsf{Vect}_\mathbb{K}\subset \mathsf{Vect}_\mathbb{K}^{\leq\infty}$ is clearly thick, a second use of Proposition \ref{derivedfuncrestrictstobdd} on this dashed arrow turns it into the desired functor $\mathsf{D^b}(X)\rightarrow \mathsf{D^b}(\mathsf{Vect}_\mathbb{K})$.   	     
\end{proof}

So, for any $\mathcal{F}^\bullet\in\text{obj}(\mathsf{D^b}(X))$ we may define $H^n(X,\mathcal{F}^\bullet)\coloneqq\mathsf{R}^n\Gamma(X,\mathcal{F}^\bullet)=$ $H^n(\mathsf{R}\Gamma(X,\mathcal{F}^\bullet))\in\text{obj}(\mathsf{Vect}_\mathbb{K})$, whence by the natural splitting of vector spaces we get:
\[
\mathsf{R}\Gamma(X,\mathcal{F}^\bullet)\cong\bigoplus_{n\in\mathbb{Z}}\mathsf{R}^n\Gamma(X,\mathcal{F}^\bullet)[-n]=\bigoplus_{n\in\mathbb{Z}}H^n(X,\mathcal{F}^\bullet)[-n]\in\text{obj}(\mathsf{D^+}(\mathsf{Vect}_\mathbb{K}))\,.
\]

\vspace*{0.3cm}
\noindent Let us move to the more general situation of pushforward functors.

\begin{Pro}
Let $f:X\rightarrow Y$ be a projective \textup(or just proper\textup) morphism of noetherian schemes. Then the right derived functor on bounded complexes of the left exact pushforward functor $f_*:\mathsf{Coh}(X;{}_{\mathcal{O}_X}\!\mathsf{Mod})\rightarrow\mathsf{Coh}(Y;{}_{\mathcal{O}_Y}\!\mathsf{Mod})$ of Proposition \ref{pushpullcoherent} exists and can be constructed as the composition of exact functors. \textup(Choosing $Y=\textup{Spec}(\mathbb{K})$ and $f$ equal to the structure morphism of $X\in\textup{obj}(\mathsf{Sch}(\mathbb{K}))$ yields Proposition \ref{globalsectforcoherentbdd}.\textup) 
\end{Pro}

\begin{proof}(\textit{Sketch})
We just mimick the proof for the case of the global sections functor. Hence, by Proposition \ref{pushpullcoherent}, on quasi-coherent sheaves we can construct the right derived functor $\mathsf{R}f_*:\mathsf{D^+}(\mathsf{QCoh}(X;{}_{\mathcal{O}_X}\!\mathsf{Mod}))\rightarrow\mathsf{D^+}(\mathsf{QCoh}(Y;{}_{\mathcal{O}_Y}\!\mathsf{Mod}))$.

Now, Grothendieck's theorem naturally generalizes to: $\mathsf{R}^nf_*(\mathcal{F})\cong\{0\}$ for any $\mathcal{F}\in\text{obj}(\mathsf{QCoh}(X;{}_{\mathcal{O}_X}\!\mathsf{Mod}))$ if $n>\text{dim}(X)$ ($f$ being a generic morphism of noetherian schemes suffices; this result may be deduced for example from \cite[Proposition III.8.1]{[Har77]}). A first call of Proposition \ref{derivedfuncrestrictstobdd} then specializes $\mathsf{R}f_*$ to $\mathsf{R}f_*:\mathsf{D^b}(\mathsf{QCoh}(X;{}_{\mathcal{O}_X}\!\mathsf{Mod}))\rightarrow\mathsf{D^b}(\mathsf{QCoh}(Y;{}_{\mathcal{O}_Y}\!\mathsf{Mod}))$.

But if we work with coherent sheaves $\mathcal{F}\in\text{obj}(\mathsf{Coh}(X;{}_{\mathcal{O}_X}\!\mathsf{Mod}))$ and $f$ as by assumption, then $\mathsf{R}^nf_*(\mathcal{F})$ is coherent as well for all $n\geq 0$ (see \cite[Theorem III.8.8]{[Har77]}). Combined with what said above, indeed we can construct $\mathsf{R}f_*:\mathsf{D^b}(X)\rightarrow\mathsf{D^b}(Y)$ by composition, so to obtain the dashed arrow in the generalized version of diagram \eqref{diagramforderivedfuncofbdd} (one employs again Lemma \ref{boundedCohinQCoh} and Proposition \ref{derivedfuncrestrictstobdd}).
\end{proof}

Then, for any $\mathcal{F}^\bullet\in\text{obj}(\mathsf{D^b}(X))$ we can define higher direct images
\begin{equation}
\mathsf{R}^nf_*(\mathcal{F}^\bullet)\coloneqq H^n(\mathsf{R}f_*(\mathcal{F}^\bullet))\in\text{obj}(\mathsf{Coh}(Y;{}_{\mathcal{O}_Y}\!\mathsf{Mod}))
\end{equation}
\big(which of course equals $H^n(X,\mathcal{F}^\bullet)$ in case $Y=\text{Spec}(\mathbb{K})$ and $f$ is the structure morphism of $X\in\text{obj}(\mathsf{Sch}(\mathbb{K}))$\big).

\begin{Rem}
The pushforward operation behaves well with respect to composition: for $f:X\rightarrow Y$ and $g:Y\rightarrow Z$ compatible morphisms of noetherian schemes, it holds $(g\circ f)_*=g_*\circ f_*:\mathsf{QCoh}(X;{}_{\mathcal{O}_X}\!\mathsf{Mod})\rightarrow\mathsf{QCoh}(Z;{}_{\mathcal{O}_Z}\!\mathsf{Mod})$. Then one can apply Proposition \ref{derivedfunccompo} to show that
\[
\mathsf{R}(g\circ f)_*\cong\mathsf{R}g_*\circ\mathsf{R}f_*:\mathsf{D^+}(\mathsf{QCoh}(X;{}_{\mathcal{O}_X}\!\mathsf{Mod}))\rightarrow\mathsf{D^+}(\mathsf{QCoh}(Z;{}_{\mathcal{O}_Z}\!\mathsf{Mod}))\,.
\]
(Exactly like in the proof of Proposition \ref{sheafcohomologyisabelian}, one chooses as adapted classes those of injective respectively flabby quasi-coherent sheaves.) It is most useful when $Z=\text{Spec}(\mathbb{K})$ and $g$ is the structure morphism of $Y$:
\begin{equation}\label{globalsectcompo}
	\mathsf{R}\Gamma(X,\square)\cong\mathsf{R}\Gamma(Y,\square)\circ\mathsf{R}f_*:\mathsf{D^+}(\mathsf{QCoh}(X;{}_{\mathcal{O}_X}\!\mathsf{Mod}))\rightarrow\mathsf{D^+}(\mathsf{Vect}_\mathbb{K}^{\leq\infty})\,.
\end{equation}
\end{Rem}

\vspace*{0.3cm}
\noindent Local $\mathcal{H}om$ is another source of interesting functors, according to Remark \ref{desirablecoherentprop}. First, observe that we can extend \eqref{locHomonqcoh} to the left exact functor
\begin{equation}\label{locHomoncomplexes}
\begin{aligned}
	&\mathcal{H}om_{\mathcal{O}_X}^\bullet(\mathcal{F}^\bullet,\square):\mathsf{H^+}(\mathsf{QCoh}(X;{}_{\mathcal{O}_X}\!\mathsf{Mod}))\rightarrow\mathsf{H^+}(\mathsf{QCoh}(X;{}_{\mathcal{O}_X}\!\mathsf{Mod}))\,, \\
	&\mathcal{H}om_{\mathcal{O}_X}^n(\mathcal{F}^\bullet,\mathcal{G}^\bullet)\coloneqq \prod_{k\in\mathbb{Z}}\mathcal{H}om_{\mathcal{O}_X}(\mathcal{F}^k,\mathcal{G}^{n+k})\in\text{obj}(\mathsf{QCoh}(X;{}_{\mathcal{O}_X}\!\mathsf{Mod})) \\
	&\text{and}\quad d^n\coloneqq d_\mathcal{G}^n -(-1)^n d_\mathcal{F}^n:\mathcal{H}om_{\mathcal{O}_X}^n(\mathcal{F}^\bullet,\mathcal{G}^\bullet)\rightarrow \mathcal{H}om_{\mathcal{O}_X}^{n+1}(\mathcal{F}^\bullet,\mathcal{G}^\bullet)\,,
\end{aligned}
\end{equation}
for any honest complex of quasi-coherent sheaves $\mathcal{F}^\bullet$ in $\mathsf{H^-}(\mathsf{QCoh}(X;{}_{\mathcal{O}_X}\!\mathsf{Mod}))$. 

\begin{Lem}
Let $X$ be a noetherian scheme which is regular, as per Definition \ref{smoothsch} \textup{\big(}this is the case of any smooth $X\in\textup{obj}(\mathsf{Sch}(\mathbb{K}))$\textup{\big)}. Then the right derived functor $\mathsf{R}\mathcal{H}om_{\mathcal{O}_X}(\mathcal{F},\square):\mathsf{D^b}(X)\rightarrow\mathsf{D^b}(X)$ exists for any $\mathcal{F}\in\textup{obj}(\mathsf{Coh}(X;{}_{\mathcal{O}_X}\!\mathsf{Mod}))$, and generalizes to the exact functor
\[
\mathsf{R}\mathcal{H}om_{\mathcal{O}_X}^\bullet(\square,\square):\mathsf{D^b}(X)^\textup{opp}\times\mathsf{D^b}(X)\rightarrow\mathsf{D^b}(X)\,.
\] 
\end{Lem}

\begin{proof}(\textit{Sketch})
Indeed, for any $\mathcal{F}\in\text{obj}(\mathsf{QCoh}(X;{}_{\mathcal{O}_X}\!\mathsf{Mod}))$ on $X$ a noetherian scheme we get the right derived functor
\[
\mathsf{R}\mathcal{H}om_{\mathcal{O}_X}(\mathcal{F},\square):\mathsf{D^+}(\mathsf{QCoh}(X;{}_{\mathcal{O}_X}\!\mathsf{Mod}))\rightarrow\mathsf{D^+}(\mathsf{QCoh}(X;{}_{\mathcal{O}_X}\!\mathsf{Mod}))\,.
\]
Supposing furthermore that $\mathcal{F}$ is coherent, the latter succeeds in restricting to the right derived $\mathsf{R}\mathcal{H}om_{\mathcal{O}_X}(\mathcal{F},\square):\mathsf{D^+}(X)\rightarrow\mathsf{D^+}(X)$, but to obtain an exact functor $\mathsf{D^b}(X)\rightarrow\mathsf{D^b}(X)$ we actually need $X$ to be a regular scheme.

Instead, starting with some $\mathcal{F}^\bullet$ in $\mathsf{H^-}(\mathsf{QCoh}(X;{}_{\mathcal{O}_X}\!\mathsf{Mod}))$, one finds that the class of complexes of injective sheaves is adapted to $\mathcal{H}om_{\mathcal{O}_X}^\bullet(\mathcal{F}^\bullet,\square)$, and thus constructs $\mathsf{R}\mathcal{H}om_{\mathcal{O}_X}^\bullet(\mathcal{F}^\bullet,\square):\mathsf{D^+}(\mathsf{QCoh}(X;{}_{\mathcal{O}_X}\!\mathsf{Mod}))\rightarrow\mathsf{D^+}(\mathsf{QCoh}(X;{}_{\mathcal{O}_X}\!\mathsf{Mod}))$, whence a bifunctor with left argument in $\mathsf{D^-}(\mathsf{QCoh}(X;{}_{\mathcal{O}_X}\!\mathsf{Mod}))^\text{opp}$. Finally, if $X$ is regular we obtain the desired right exact functor on bounded complexes of coherent sheaves.
We refer the reader to \cite{[Huy06]} for more details.
\end{proof}

For $\mathcal{F},\mathcal{G}\in\text{obj}(\mathsf{QCoh}(X;{}_{\mathcal{O}_X}\!\mathsf{Mod}))$, it then makes sense to define (in accordance with Definition \ref{localExt}) local $\mathcal{E}xt$-sheaves by:
\begin{equation}
\mathcal{E}xt_{\mathcal{O}_X}^n(\mathcal{F},\mathcal{G})\coloneqq\mathsf{R}^n\mathcal{H}om_{\mathcal{O}_X}(\mathcal{F},\mathcal{G})\in\text{obj}(\mathsf{QCoh}(X;{}_{\mathcal{O}_X}\!\mathsf{Mod}))\,,
\end{equation}
also coherent as soon as $\mathcal{F},\mathcal{G}$ both are. Analogously, supposing that $\mathcal{F}^\bullet$ belongs to $\mathsf{D^-}(\mathsf{QCoh}(X;{}_{\mathcal{O}_X}\!\mathsf{Mod}))$ and $\mathcal{G}^\bullet$ to $\mathsf{D^+}(\mathsf{QCoh}(X;{}_{\mathcal{O}_X}\!\mathsf{Mod}))$, we let
\begin{equation}
\mathcal{E}xt_{\mathcal{O}_X}^n(\mathcal{F}^\bullet,\mathcal{G}^\bullet)\coloneqq\mathsf{R}^n\mathcal{H}om_{\mathcal{O}_X}^\bullet(\mathcal{F}^\bullet,\mathcal{G}^\bullet)\in\text{obj}(\mathsf{QCoh}(X;{}_{\mathcal{O}_X}\!\mathsf{Mod}))\,,
\end{equation}
coherent as soon as $\mathcal{F}^\bullet,\mathcal{G}^\bullet\in\text{obj}(\mathsf{D^b}(X))$. 

\begin{Rem}\label{HomExtRem}
Here are some notable relations between the operators $\mathcal{H}om$, $\mathcal{E}xt$, Hom and Ext.
\begin{itemize}[leftmargin=0.5cm]
	\item Under the additional assumption that $X\in\text{obj}(\mathsf{Sch}(\mathbb{K}))$ is projective, arguing as above, from the inner-Hom left exact functor $\text{Hom}_{\mathcal{O}_X}^\bullet(\mathcal{F}^\bullet,\square)$ on $\mathsf{H^+}(\mathsf{QCoh}(X;{}_{\mathcal{O}_X}\!\mathsf{Mod}))$, which is defined identically to \eqref{locHomoncomplexes}, we can construct the exact bifunctor $\mathsf{R}\text{Hom}_{\mathcal{O}_X}^\bullet\!(\square,\square)\!:\!\mathsf{D^b}(X)^\textup{opp}\times\mathsf{D^b}(X)\rightarrow\mathsf{D^b}(\mathsf{Vect}_\mathbb{K})$ which generalizes the right derived functor $\mathsf{R}\text{Hom}_{\mathcal{O}_X}\!(\mathcal{F},\square)$ we would get from \eqref{locHomoncoh}. Then for any $\mathcal{F}^\bullet,\mathcal{G}^\bullet\in\text{obj}(\mathsf{D^b}(X))$ we may set 
	\begin{small}\begin{equation}\label{ExtRHomforcomplexes}
			\mkern-12mu\text{Ext}_{\mathcal{O}_X}^n(\mathcal{F}^\bullet,\mathcal{G}^\bullet)\!\coloneqq\!\mathsf{R}^n\text{Hom}_{\mathcal{O}_X}^\bullet(\mathcal{F}^\bullet,\mathcal{G}^\bullet)\!\cong\!\text{Hom}_{\mathsf{D^b}(X)}(\mathcal{F}^\bullet,\mathcal{G}^\bullet[n])\!\in\!\text{obj}(\mathsf{Vect}_\mathbb{K})\,,
	\end{equation}\end{small}
	\!\!in full accord with our standard definition \eqref{Extmodules} of Ext-modules. Note that $\text{Ext}_{\mathcal{O}_X}^n\!(\mathcal{F}^\bullet,\mathcal{G}^\bullet)$ is finite-dimensional by assumption on $X$ and Theorem \ref{sheafcohomforprojsch}.\footnote{So this is true of a projective variety $X$ as well; supposing it is smooth, one can show that $\text{Ext}_{\mathcal{O}_X}^n\!(\mathcal{F},\mathcal{G})\cong\{0\}$ for $n>\text{dim}(X)$ and $\mathcal{F},\mathcal{G}$ coherent (whence $\text{dim}_h(\mathsf{Coh}(X;{}_{\mathcal{O}_X}\!\mathsf{Mod}))=\text{dim}(X)$ for the homological dimension \eqref{homologicaldimension}), and that $\mathsf{R}\text{Hom}_{\mathcal{O}_X}^\bullet(\mathcal{F}^\bullet,\mathcal{G}^\bullet)\in\text{obj}(\mathsf{D^b(Ab)})$ for bounded complexes $\mathcal{F}^\bullet,\mathcal{G}^\bullet$ of coherent sheaves.} Moreover, for any $\mathcal{F}^\bullet\in\textup{obj}(\mathsf{D^-}(X))$ holds
	\begin{equation}\label{locHomcompo}
		\mathsf{R}\Gamma(X,\square)\circ\mathsf{R}\mathcal{H}om_{\mathcal{O}_X}\!(\mathcal{F}^\bullet,\square)\cong\mathsf{R}\textup{Hom}_{\mathcal{O}_X}\!(\mathcal{F}^\bullet,\square)\!:\mathsf{D^+}(X)\!\rightarrow~{\!\mathsf{D^+}(X)}\,.
	\end{equation} 
	
	\item Assume that $X$ is a noetherian scheme, $\mathcal{F}\in\text{obj}(\mathsf{Coh}(X;{}_{\mathcal{O}_X}\!\mathsf{Mod}))$ and $\mathcal{G}\in\text{obj}(\mathsf{Sh}(X;{}_{\mathcal{O}_X}\mathsf{Mod}))$. The local $\mathcal{E}xt$-sheaves then relate to the standard $\text{Ext}$-modules at the level of stalks as
	\[
	\mathcal{E}xt_{\mathcal{O}_X}^n(\mathcal{F},\mathcal{G})_x\cong\text{Ext}_{\mathcal{O}_{X,x}}^n(\mathcal{F}_x,\mathcal{G}_x)\in\text{obj}({}_{\mathcal{O}_{X,x}}\!\mathsf{Mod})\,,
	\]
	for all $x\in X$ and $n\geq 0$ (see \cite[Proposition III.6.8]{[Har77]}).
	\newline Suppose further that $X=\text{Proj}(S)$ is projective over some noetherian graded ring $S$ and $\mathcal{F},\mathcal{G}\in$ $\text{obj}(\mathsf{Coh}(X;{}_{\mathcal{O}_X}\!\mathsf{Mod}))$, then one can show that the Ext-modules ``eventually'' become the global sections modules of the $\mathcal{E}xt$-modules (cf. \cite[Proposition III.6.9]{[Har77]}).
	
	\item Given any $\mathcal{F}^\bullet$ in $\mathsf{D^-}(\mathsf{QCoh}(X;{}_{\mathcal{O}_X}\!\mathsf{Mod}))$, we define its dual complex to be 
	\begin{equation}
		(\mathcal{F}^\bullet)^\vee\coloneqq\mathsf{R}\mathcal{H}om(\mathcal{F}^\bullet,\mathcal{O}_X)\in\text{obj}\big(\mathsf{D^+}(\mathsf{QCoh}(X;{}_{\mathcal{O}_X}\!\mathsf{Mod}))\big)
	\end{equation}
	(assuming $\mathcal{O}_X$ to be concentrated in degree 0). If $X$ is a regular scheme, then $(\mathcal{F}^\bullet)^\vee$ actually belongs to $\mathsf{D^b}(X)$ for any $\mathcal{F}^\bullet\in\text{obj}(\mathsf{D^b}(X))$.
\end{itemize}
\end{Rem}

\vspace*{0.7cm}
\noindent We turn our focus to right exact functors on quasi-/coherent sheaves, beginning with the tensor product. As was the case for local $\mathcal{H}om$, we can extend $\mathcal{F}\otimes_{\mathcal{O}_X}\!\square$ of \eqref{tensoroncoh} to the right exact functor
\begin{equation}
\begin{aligned}
	&\mkern-18mu\mathcal{F}^\bullet\otimes_{\mathcal{O}_X}\!\square:\mathsf{H^-}(\mathsf{Coh}(X;{}_{\mathcal{O}_X}\!\mathsf{Mod}))\rightarrow\mathsf{H^-}(\mathsf{Coh}(X;{}_{\mathcal{O}_X}\!\mathsf{Mod}))\,, \\
	&\mkern-18mu(\mathcal{F}^\bullet\!\otimes_{\mathcal{O}_X}\!\mathcal{G}^\bullet)^n\!\coloneqq \mkern-12mu\bigoplus_{k+l=n\in\mathbb{Z}}\mkern-12mu\mathcal{F}^k\!\otimes_{\mathcal{O}_X}\!\mathcal{G}^l \;\;\text{and}\;\; d^n\!\coloneqq\!d_\mathcal{F}^n\otimes\text{id}_\mathcal{G}^n \!+\!(-1)^n \text{id}_\mathcal{F}^n\otimes d_\mathcal{G}^n\,,
\end{aligned}
\end{equation}
for any $\mathcal{F}^\bullet$ in $\mathsf{H^-}(\mathsf{Coh}(X;{}_{\mathcal{O}_X}\!\mathsf{Mod}))$, and similarly for the right exact $\square\!\otimes_{\mathcal{O}_X}\mathcal{G}^\bullet$.

\begin{Pro}\label{tensorbifunconbdd}
Let $X\in\textup{obj}(\mathsf{Sch}(\mathbb{K}))$ be smooth, projective. Given any $\mathcal{F}\in\textup{obj}(\mathsf{Coh}(X;\mathsf{Mod}_{\mathcal{O}_X}))$ and $\mathcal{G}\in\textup{obj}(\mathsf{Coh}(X;{}_{\mathcal{O}_X}\!\mathsf{Mod}))$, the left derived functors $\mathcal{F}\overset{\mathsf{L}}{\otimes}_{\mathcal{O}_X}\!\square\coloneqq\mathsf{L}(\mathcal{F}\otimes_{\mathcal{O}_X}\!\square):\mathsf{D^b}(X)\rightarrow\mathsf{D^b}(X)$ and $\square\overset{\mathsf{L}}{\otimes}_{\mathcal{O}_X}\mathcal{G}:$ $\mathsf{D^b}(X)\rightarrow\mathsf{D^b}(X)$ exist and generalize to the exact functor
\[
\square\overset{\mathsf{L}}{\otimes}_{\mathcal{O}_X}\!\square:\mathsf{D^b}(X)\times\mathsf{D^b}(X)\rightarrow\mathsf{D^b}(X)
\]
\textup(which is both commutative and associative in its arguments, like the standard tensor product\textup). 
\end{Pro}

\begin{proof}(\textit{Sketch})
Let $\mathcal{F}\in\textup{obj}(\mathsf{Coh}(X;\mathsf{Mod}_{\mathcal{O}_X}))$. The first bullet point of Remark \ref{localExtRem} suggests that any $\mathcal{E}\in\textup{obj}(\mathsf{Coh}(X;{}_{\mathcal{O}_X}\!\mathsf{Mod}))$ admits a projective resolution $\mathcal{L}^\bullet\rightarrow\mathcal{E}$ by locally free sheaves: in fact, any coherent sheaf on a projective scheme over a noetherian ring can be written as the quotient of a locally free sheaf of finite rank (shown in \cite[Corollary II.5.18]{[Har77]}). This plays into the class of locally free sheaves in $\mathsf{Coh}(X;{}_{\mathcal{O}_X}\!\mathsf{Mod})$ being adapted to $\mathcal{F}\otimes_{\mathcal{O}_X}\!\square$, proving that we can construct the left derived functor $\mathsf{L}(\mathcal{F}\otimes_{\mathcal{O}_X}\!\square):\mathsf{D^-}(X)\rightarrow\mathsf{D^-}(X)$.

But since $X$ is smooth of some finite dimension $\text{dim}(X)$, it turns out that any resolution of $\mathcal{E}$ as above has finite length as well, and using it to evaluate the classical derived functors (a legitimate move by \eqref{classderprojres}) thus yields $\text{Tor}_n^{\mathcal{O}_X}(\mathcal{F},\mathcal{G})=\mathsf{H}^{-n}(\mathcal{F}\overset{\mathsf{L}}{\otimes}_{\mathcal{O}_X}\mathcal{G})\cong\{0\}$ for all $-n<-\text{dim}(X)$ (remember our notation from Example \ref{derivedtensorfunc}). Then we can apply the usual Proposition \ref{derivedfuncrestrictstobdd} to restrict the tensoring to the exact functor $\mathcal{F}\,\overset{\mathsf{L}}{\otimes}_{\mathcal{O}_X}\!\square:\mathsf{D^b}(X)\rightarrow\mathsf{D^b}(X)$ desired. The argument for $\square\;\overset{\mathsf{L}}{\otimes}_{\mathcal{O}_X}\!\mathcal{G}$ is analogous.

More generally, if $\mathcal{F}^\bullet\in\text{obj}\big(\mathsf{H^-}(\mathsf{Coh}(X;\mathsf{Mod}_{\mathcal{O}_X}))\big)$, then one similarly shows that the class of complexes of locally free sheaves is adapted to $\mathcal{F}^\bullet\otimes_{\mathcal{O}_X}\!\square$, whence its left derived functor $\mathcal{F}^\bullet\overset{\mathsf{L}}{\otimes}_{\mathcal{O}_X}\!\square:\mathsf{D^-}(X)\rightarrow\mathsf{D^-}(X)$, and the same with respect to the right argument. The resulting bifunctor descends to the bounded derived category $\mathsf{D^b}(X)$ again due to $X$ being smooth.          
\end{proof}

In line with Example \ref{derivedtensorfunc}, we set
\begin{equation}
\text{Tor}_n^{\mathcal{O}_X}(\mathcal{F}^\bullet,\mathcal{G}^\bullet)\coloneqq H^{-n}(\mathcal{F}^\bullet\overset{\mathsf{L}}{\otimes}_{\mathcal{O}_X}\mathcal{G}^\bullet)\in\text{obj}(\mathsf{Coh}(X;{}_{\mathcal{O}_X}\!\mathsf{Mod}))\,,
\end{equation}
for any $\mathcal{F}^\bullet,\mathcal{G}^\bullet\in\text{obj}(\mathsf{D^b}(X))$.

\vspace*{0.3cm}
\noindent The remaining member of our cast is the pullback functor, which, by Proposition \ref{pushpullcoherent}, we know to be a little less demanding than the pushforward.

\begin{Cor}
Let $f:X\rightarrow Y$ be a morphism of smooth noetherian schemes. Then the left derived functor $\mathsf{L}f^*:\mathsf{D^b}(Y)\rightarrow\mathsf{D^b}(X)$ on bounded complexes of the right exact pullback functor $f^*:\mathsf{Coh}(Y;{}_{\mathcal{O}_Y}\!\mathsf{Mod})\rightarrow\mathsf{Coh}(X;{}_{\mathcal{O}_X}\!\mathsf{Mod})$ of Proposition \ref{pushpullcoherent} exists and can be constructed as the composition of exact functors.
\end{Cor}

\begin{proof}
This is an immediate consequence of the definition of $f^*$ and Proposition \ref{tensorbifunconbdd}, which can be shown to hold also in this more general setting (the requisite of smooth projective schemes over $\mathbb{K}$ was just for simplicity's sake). Indeed, the exact functor
\[
\mathsf{L}f^*\coloneqq \mathcal{O}_X\overset{\mathsf{L}}{\otimes}_{f^{-1}\mathcal{O}_Y}\!f^{-1}(\square):\mathsf{D^-}(Y)\rightarrow\mathsf{D^-}(X)
\]
descends to the respective bounded derived categories of complexes.
\end{proof} 

For any $\mathcal{G}^\bullet\in\text{obj}(\mathsf{D^b}(Y))$ we set
\begin{equation}
\mathsf{L}^nf^*(\mathcal{F}^\bullet)\coloneqq H^n(\mathsf{L}f^*(\mathcal{F}^\bullet))\in\text{obj}(\mathsf{Coh}(X;{}_{\mathcal{O}_X}\!\mathsf{Mod}))\,.
\end{equation}

Finally, we report a generous list of identities which are useful when manipulating bounded complexes of coherent sheaves. Their proof is left to those who survived this far. Suggestion: it suffices to reduce the analysis to locally free sheaves --- which form an adapted class for functors on coherent sheaves --- and thus recall Remark \ref{locallyfreeRem} and the other equalities concerning $\mathsf{Sh}^\mathsf{lf}_n(X;{}_{\mathcal{O}_X}\!\mathsf{Mod})$.

\begin{Pro}\label{lotofequalities}
Throughout, let $X,Y$ be noetherian schemes, $f:X\rightarrow Y$ a scheme morphism and $\mathcal{F}^\bullet, \mathcal{F}_i^\bullet\in\textup{obj}(\mathsf{D^b}(X))$, $\mathcal{G}^\bullet, \mathcal{G}_i^\bullet\in\textup{obj}(\mathsf{D^b}(Y))$ \textup(further assumptions are specified in brackets from case to case\textup). The following relations hold:\begin{small}
	\begingroup
	\allowdisplaybreaks
	\begin{align}
		&a)\mkern+6mu\mathsf{R}f_*(\mathcal{F}^\bullet)\overset{\mathsf{L}}{\otimes}_{\mathcal{O}_Y}\!\mathcal{G}^\bullet\cong\mathsf{R}f_*(\mathcal{F}^\bullet\overset{\mathsf{L}}{\otimes}_{\mathcal{O}_X}\!\mathsf{L}f^*(\mathcal{G}^\bullet)) &\mkern-50mu\big(f\text{ proper},\, X,Y\!\in\!\textup{obj}(\mathsf{Sch}(\mathbb{K}))\big) \nonumber \\
		&b)\mkern+6mu\mathsf{L}f^*(\mathcal{G}_1^\bullet)\overset{\mathsf{L}}{\otimes}_{\mathcal{O}_Y}\!\mathsf{L}f^*(\mathcal{G}_2^\bullet)\cong \mathsf{L}f^*(\mathcal{G}_1^\bullet\overset{\mathsf{L}}{\otimes}_{\mathcal{O}_Y}\!\mathcal{G}_2^\bullet) &\big(X,Y\text{ projective}\big) \nonumber \\
		&c)\mkern+6mu\textup{Hom}_{\mathsf{D^b}(X)}(\mathsf{L}f^*(\mathcal{G}^\bullet),\mathcal{F}^\bullet)\cong \textup{Hom}_{\mathsf{D^b}(Y)}(\mathcal{G}^\bullet,\mathsf{R}f_*(\mathcal{F}^\bullet)) &\big(f\text{ projective}\big) \nonumber \\
		&d)\mkern+6mu\mathsf{R}\mathcal{H}om(\mathcal{F}_1^\bullet,\mathcal{F}_2^\bullet\overset{\mathsf{L}}{\otimes}_{\mathcal{O}_X}\!\mathcal{F}_3^\bullet)\cong\mathsf{R}\mathcal{H}om(\mathcal{F}_1^\bullet,\mathcal{F}_2^\bullet)\overset{\mathsf{L}}{\otimes}_{\mathcal{O}_X}\!\mathcal{F}_3^\bullet  &\big(X\!\in\!\textup{obj}(\mathsf{Sch}(\mathbb{K}))\text{ smooth} \nonumber \\
		& \mkern+187mu\cong\mathsf{R}\mathcal{H}om(\mathsf{R}\mathcal{H}om(\mathcal{F}_2^\bullet,\mathcal{F}_1^\bullet),\mathcal{F}_3^\bullet) &\text{and projective}\big) \nonumber \\
		&e)\mkern+6mu\mathsf{R}\mathcal{H}om(\mathcal{F}_1^\bullet\overset{\mathsf{L}}{\otimes}_{\mathcal{O}_X}\!\mathcal{F}_2^\bullet,\mathcal{F}_3^\bullet)\cong\mathsf{R}\mathcal{H}om(\mathcal{F}_1^\bullet,\mathsf{R}\mathcal{H}om(\mathcal{F}_2^\bullet,\mathcal{F}_3^\bullet)) & (\text{same}) \nonumber \\
		&f)\mkern+6mu(\mathcal{F}_1^\bullet)^\vee\overset{\mathsf{L}}{\otimes}_{\mathcal{O}_X}\!\mathcal{F}_2^\bullet\cong \mathsf{R}\mathcal{H}om(\mathcal{F}_1^\bullet,\mathcal{F}_2^\bullet) & (\text{same}) \nonumber \\
		&g)\mkern+6mu\mathcal{F}^\bullet\cong((\mathcal{F}^\bullet)^{\vee})^\vee\!\!\cong \mathsf{R}\mathcal{H}om(\mathsf{R}\mathcal{H}om(\mathcal{F}^\bullet,\mathcal{O}_X),\mathcal{O}_X) & (\text{same}) \nonumber \\
		&h)\mkern+6mu\mathsf{L}f^*(\mathsf{R}\mathcal{H}om_{\mathcal{O}_Y}\!(\mathcal{G}_1^\bullet,\mathcal{G}_2^\bullet))\!\cong\! \mathsf{R}\mathcal{H}om_{\mathcal{O}_X}\!(\mathsf{L}f^*(\mathcal{G}_1^\bullet),\mathsf{L}f^*(\mathcal{G}_2^\bullet)) & \big(X,Y\text{ proj.},\,\mathcal{G}_1^\bullet\in\mathsf{D^{\!-}}\!(Y)\big) \nonumber \\
		&i)\mkern+6mup_{X*}(p_Y^*\mathcal{G}^\bullet)\cong\mathsf{R}\Gamma(Y,\mathcal{G}^\bullet)\otimes_{\mathcal{O}_X}\!\mathcal{O}_X & \big(X\!\xleftarrow{p_X}\!X\!\times \!Y\!\xrightarrow{p_Y}\!Y\text{ proj.}\big) \nonumber \\
		&j)\mkern+6mu\mathsf{R}\Gamma(X\!\times\! Y,p_X^*\mathcal{F}\overset{\mathsf{L}}{\otimes}_{\mathcal{O}_X}\!p_Y^*\mathcal{G}^\bullet)\cong\mathsf{R}\Gamma(X,\mathcal{F}^\bullet)\otimes\mathsf{R}\Gamma(Y,\mathcal{G}^\bullet) & (\text{same})  \nonumber
	\end{align}\endgroup\end{small}
\!\!\!To name a few: a\textup) is the \textbf{projection formula}\index{projection formula}, c\textup) is the \textbf{adjunction formula}\index{adjunction formula} \textup(indeed asserting that $\mathsf{L}f_*$ and $\mathsf{R}f^*$ are adjoint functors\textup), i\textup) is a special case of \textbf{flat base change}\index{flat base change} \textup(see \textup{\cite[equation (3.18)]{[Huy06]}}\textup) and j\textup) induces the \textbf{Künneth formula for sheaf cohomology}\index{Künneth formula for sheaf cohomology}.
\end{Pro}

\subsection{Serre Duality and Fourier--Mukai transforms}\label{ch6.4}

This section includes a few important applications of sheaf cohomology and derived functors. We begin with the famous duality result established by Serre, referring to \cite[section III.7]{[Har77]}.

\begin{Lem}
Let $(X,\mathcal{O}_X)$ be a projective scheme \textup(or just a proper one\textup) of dimension $n$ over any field $\mathbb{K}$. Then $X$ admits a unique \textbf{dualizing sheaf}\index{sheaf!dualizing} $\mathcal{K}_X\in\textup{obj}(\mathsf{Coh}(X;{}_{\mathcal{O}_X}\!\mathsf{Mod}))$ together with a morphism $t:H^n(X,\mathcal{K}_X)\rightarrow\mathbb{K}$ such that for all $\mathcal{F}\in\textup{obj}(\mathsf{Coh}(X;{}_{\mathcal{O}_X}\!\mathsf{Mod}))$ the natural pairing
\[
\textup{Hom}_{\mathcal{O}_X}(\mathcal{F},\mathcal{K}_X)\times H^n(X,\mathcal{F})\rightarrow H^n(X,\mathcal{K}_X)\xrightarrow{t}\mathbb{K}
\]
gives an isomorphism of vector spaces $\theta^0:\textup{Hom}_{\mathcal{O}_X}(\mathcal{F},\mathcal{K}_X)\xrightarrow{\sim} H^n(X,\mathcal{F})^*$ \textup(the right-hand side denotes the dual vector space of $H^n(X,\mathcal{F})$\textup).
\end{Lem}

\begin{Thm}[\textbf{Serre Duality Theorem}]\index{Serre Duality Theorem}\label{Serreduality}
Let $(X,\mathcal{O}_X)$ be a smooth projective scheme of dimension $n$ over $\mathbb{K}$ \textup(algebraically closed\textup), thus admitting a dualizing sheaf $\mathcal{K}_X\in\textup{obj}(\mathsf{Coh}(X;{}_{\mathcal{O}_X}\!\mathsf{Mod}))$. Then for all $\mathcal{F}\in\textup{obj}(\mathsf{Coh}(X;{}_{\mathcal{O}_X}\!\mathsf{Mod}))$ and $k\geq 0$ there exist natural functorial isomorphisms
\begin{equation}
	\theta^k:\textup{Ext}_{\mathcal{O}_X}^k(\mathcal{F},\mathcal{K}_X)\xrightarrow{\sim} H^{n-k}(X,\mathcal{F})^*
\end{equation}
\textup(where $\theta^0$ coincides with the isomorphism in the last lemma\textup).

On the other hand, $H^k(X,\mathcal{L})\cong H^{n-k}(X,\mathcal{L}^\vee\otimes_{\mathcal{O}_X}\! \mathcal{K}_X)^*$ for any finite rank $\mathcal{L}\in\textup{obj}(\mathsf{Sh}^\mathsf{lf}_r(X;{}_{\mathcal{O}_X}\!\mathsf{Mod}))$ and all $k\geq 0$.
\end{Thm}

For the more concrete example of the projective space $X=\mathbb{P}^n$ over $\mathbb{K}$, the reader is referred to \cite[Theorem III.7.1]{[Har77]}, which relies on the computation in \cite[section III.5]{[Har77]} of the sheaf cohomology of $\mathbb{P}^n$.

\begin{Rem}\label{cansheaf}
Suppose that $X$ is a smooth projective variety of dimension $n$ over the algebraically closed field $\mathbb{K}$. Then its dualizing sheaf is actually isomorphic to the \textbf{canonical line bundle}\index{canonical line bundle}\footnote{By definition, the canonical line bundle is the $n$-th exterior power of the \textit{sheaf of differentials} $\Omega_X$. The latter is constructed in \cite[section II.8]{[Har77]}; over smooth complex varieties, it is the sheaf of differential forms on $X$, as our intuition suggests.} (or \textbf{canonical sheaf}\index{sheaf!canonical})
\begin{equation}
	K_X\coloneqq \Lambda^n\Omega_X\cong\Omega_X^n\in\text{obj}(\mathsf{Sh}^\mathsf{lf}_1(X;{}_{\mathcal{O}_X}\mathsf{Mod}))\,,
\end{equation}
thus an invertible coherent sheaf. 

In fact, denoting by $\Omega_X^p$ the sheaf of differential $p$-forms on $X$, the duality isomorphisms of Theorem \ref{Serreduality} imply that
\begin{equation}\label{Serreisoonvariety}
	H^q(X,\Omega_X^p)\cong H^{n-q}(X,\Omega_X^{n-p})^*\,,
\end{equation}
for $0\leq p,q\leq n$ (just use that $\Omega_X^{n-p}\cong (\Omega^p)^\vee\otimes_{\mathcal{O}_X}\! K_X$). In particular, we obtain $H^n(X,K_X)\cong\mathbb{K}$, and the numbers $h^{p,q}(X)\coloneqq\text{dim}(H^q(X,\Omega_X^p))$ are invariants of the variety $X$.

At this point, the reader may have recognized \eqref{Serreisoonvariety} from \cite[section 7.2]{[Imp21]}: if $X$ is a complex manifold (a smooth projective variety over $\mathbb{C}$), these identities follow from a joint application of the Dolbeault Theorem $H_\text{Db}^{p,q}(X)\cong H^q(X,\Omega_X^{p,0})$ and the isomorphisms $H_\text{Db}^{p,q}(X)\cong H_\text{Db}^{n-p,n-q}(X)$ induced by the Hodge $\star$-operator; the $h^{p,q}(X)$'s are precisely the Hodge numbers of $X$.      
\end{Rem}

Focusing on varieties, Serre duality can also be rephrased in terms of \textit{Serre functors}:

\begin{Def}\label{Serrefunc}
Let $\mathsf{C}$ be a $\mathbb{K}$-linear category. A \textbf{Serre functor}\index{functor!Serre} is a $\mathbb{K}$-linear equivalence $\mathsf{S}:\mathsf{C}\rightarrow\mathsf{C}$ such that for any two $X,Y\in\text{obj}(\mathsf{C})$ there exists a functorial isomorphism of $\mathbb{K}$-vector spaces $\eta_{X,Y}:\text{Hom}_\mathsf{C}(X,Y)\xrightarrow{\sim}\text{Hom}_\mathsf{C}(Y,\mathsf{S}(X))^*$. 

Note that if $\mathsf{C}$ is also triangulated, then $\mathsf{S}$ is an exact functor.
\end{Def}

\begin{Cor}\label{SerredualityviaSerrefunc}	
Suppose $X=(X,\mathcal{O}_X)\in\textup{obj}(\mathsf{Sch}(\mathbb{K}))$ is smooth of dimension $n$ and projective. Then the \textup(exact\textup) functor
\begin{equation}
	\mathsf{S}_X\coloneqq(\square\!\otimes_{\mathcal{O}_X}\!K_X)[n]:\mathsf{D^\#}\!(X)\rightarrow\mathsf{D^\#}\!(X)\,,
\end{equation}
for $\#=\mathsf{b},+,-$ \textup(where $[n]$ stands for the $n$-fold shift in $\mathsf{D^\#}\!(X)$\textup) is a Serre functor on $\mathsf{D^\#}\!(X)$.
\end{Cor}

\begin{proof}
We wish to prove that for any two $\mathcal{F}^\bullet,\mathcal{G}^\bullet\in\text{obj}(\mathsf{D^b}(X))$ there exists an isomorphism of $\mathbb{K}$-vector spaces $\eta_{\mathcal{F}^\bullet\!,\mathcal{G}^\bullet}$ as in Definition \ref{Serrefunc}. This is however an immediate consequence of Theorem \ref{Serreduality}:
\begin{small}
	\begingroup
	\allowdisplaybreaks\begin{align*}
		\text{Hom}_{\mathsf{D^b}(X)}(\mathcal{G}^\bullet,\,&\mathsf{S}_X(\mathcal{F}^\bullet))^*= & \mkern-20mu{}\\ &=\text{Hom}_{\mathsf{D^b}(X)}(\mathcal{G}^\bullet,(\mathcal{F}^\bullet\otimes_{\mathcal{O}_X}\!K_X)[n])^* & (\text{by definition}) \\ 
		&\cong\mathsf{R}^n\text{Hom}_{\mathcal{O}_X}^\bullet(\mathcal{G}^\bullet,\mathcal{F}^\bullet\otimes_{\mathcal{O}_X}\!K_X)^* & (\text{by \eqref{ExtRHomforcomplexes}}) \\
		&\cong\mathsf{R}^n\text{Hom}_{\mathcal{O}_X}^\bullet(\mathsf{R}\mathcal{H}om_{\mathcal{O}_X}(\mathcal{F}^\bullet,\mathcal{G}^\bullet),K_X)^* & (\text{by Proposition \ref{lotofequalities}d)}) \\
		&=\text{Ext}_{\mathcal{O}_X}^n(\mathsf{R}\mathcal{H}om_{\mathcal{O}_X}(\mathcal{F}^\bullet,\mathcal{G}^\bullet),K_X)^* & (\text{by definition}) \\
		&\cong H^{n-n}(X,\mathsf{R}\mathcal{H}om_{\mathcal{O}_X}(\mathcal{F}^\bullet,\mathcal{G}^\bullet))^{**} & (\text{by Theorem \ref{Serreduality}}) \\
		&\cong \Gamma(X,\mathsf{R}\mathcal{H}om_{\mathcal{O}_X}(\mathcal{F}^\bullet,\mathcal{G}^\bullet))^{**} & (\text{by Definition \ref{sheafcohomology}}) \\
		&\cong \mathsf{R}^0\text{Hom}_{\mathcal{O}_X}^\bullet(\mathcal{F}^\bullet,\mathcal{G}^\bullet)^{**} & (\text{by \eqref{locHomcompo}}) \\
		&\cong \text{Hom}_{\mathsf{D^b}(X)}(\mathcal{F}^\bullet,\mathcal{G}^\bullet)^{**} & (\text{by \eqref{ExtRHomforcomplexes}}) \\
		&\cong\text{Hom}_{\mathsf{D^b}(X)}(\mathcal{F}^\bullet,\mathcal{G}^\bullet) & (\text{by finiteness of Hom-spaces})
	\end{align*}\endgroup\end{small}
\!\!\!Here we regarded $K_X$ as a complex concentrated in degree 0, and with our call of Proposition \ref{lotofequalities} we actually employed the Hom-counterpart to the identity $d)$ listed there, also applying $H^n$-cohomology on it. 
\end{proof}

With some slight adjustments, we see that imposing $\mathsf{S}_X$ to be a Serre functor unavoidably implies that $\textup{Ext}_{\mathcal{O}_X}^k(\mathcal{F},K_X)\cong H^{n-k}(X,\mathcal{F})^*$ for all $k\geq 0$ and coherent sheaves $\mathcal{F}$, which is the statement of Serre duality. Consequently, on $X\in\textup{obj}(\mathsf{Sch}(\mathbb{K}))$ smooth and projective, Corollary \ref{SerredualityviaSerrefunc} (restricted to coherent sheaves instead of complexes thereof) is equivalent to Theorem \ref{Serreduality}.

One can apply Serre duality to more easily show the following generation result, which gives an ``orthonormal basis'' to the derived category of smooth projective varieties:

\begin{Pro}\label{spanningclass}
Let $X$ be a smooth projective variety over $\mathbb{K}$ \textup(algebraically closed\textup). Then\footnote{Here $k(x)$ denotes the constant structure sheaf $\underline{\mathbb{K}}\in\text{obj}(\mathsf{D^b}(X))$ of the closed one-point space $\{x\}\subset X$.} $\{k(x)\in\textup{obj}(\mathsf{D^b}(X))\mid x\in X\}$ is an \textup{orthonormal spanning class} for the triangulated category $\mathsf{D^b}(X)$, meaning that the following conditions hold:
\begin{itemize}[leftmargin=0.5cm]
	\item \textup(orthonormality\textup) $\textup{Hom}_{\mathsf{D^b}(X)}(k(x),k(y)[n])\cong\{0\}$ $\forall x\neq y\in X$ $\forall n\in\mathbb{Z}$, and $\textup{Hom}_{\mathsf{D^b}(X)}(k(x),k(x))\cong\mathbb{K}$ $\forall x\in X$; 
	
	\item \textup(spanning\textup) $\textup{Hom}_{\mathsf{D^b}(X)}(k(x),\mathcal{F}^\bullet[n])\cong\{0\}$ $\forall x\in X$ $\forall n\in\mathbb{Z}$ $\implies$ $\mathcal{F}^\bullet\cong 0^\bullet$.
\end{itemize}
\textup(See \textup{\cite[Proposition 3.17]{[Huy06]}}.\textup)
\end{Pro}

We warn that this is \textit{not} the same concept as the classical version of \cite[Definition 3.22]{[Imp21]}: here, we don't mean that $\{k(x)\in\textup{obj}(\mathsf{D^b}(X))\mid x\in X\}$ forms\break the smallest full subcategory containing all bounded complexes of sheaves, translations and cones thereof, thus generating the whole $\mathsf{D^b}(X)$!

Finally, back to the realm of schemes, we mention the generalization of Serre duality addressed in \cite[section 3.4]{[Huy06]}.

\begin{Thm}[\textbf{Grothendieck--Verdier Duality}]\index{Grothendieck--Verdier Duality}\label{GVduality}
Let $X,Y\in\textup{obj}(\mathsf{Sch}(\mathbb{K}))$ be smooth and projective \textup(or just proper\textup), and let $f:X\rightarrow Y$ be a scheme morphism, of \textup{relative dimension} $\textup{dim}(f)\coloneqq\textup{dim}(X)-\textup{dim}(Y)$ and \textup{relative dualizing bundle} $K_f\coloneqq K_X\otimes_{\mathcal{O}_X}\!f^*K_Y^\vee\in\textup{obj}(\mathsf{Coh}(X;{}_{\mathcal{O}_X}\!\mathsf{Mod}))$ \textup(concentrated in degree 0\textup). Then for any $\mathcal{F}^\bullet\in\textup{obj}(\mathsf{D^b}(X))$ and $\mathcal{G}^\bullet\in\textup{obj}(\mathsf{D^b}(Y))$ there exists a functorial isomorphism
\begin{equation}
	\mathsf{R}f_*\big(\mathsf{R}\mathcal{H}om_{\mathcal{O}_X}\!\big(\mathcal{F}^\bullet,\mathsf{L}f^*(\mathcal{G}^\bullet)\otimes_{\mathcal{O}_X}\!K_f[\textup{dim}(f)]\big)\big)\cong\mathsf{R}\mathcal{H}om_{\mathcal{O}_Y}\!(\mathsf{R}f_*(\mathcal{F}^\bullet),\mathcal{G}^\bullet)\,.
\end{equation}
Note that applying $\Gamma(Y,\square)$ to both sides, using both \eqref{globalsectcompo} and \eqref{locHomcompo}, taking $H^0$-cohomology and finally applying \eqref{ExtRHomforcomplexes} yields
\[
\textup{Hom}_{\mathsf{D^b}(X)}(\mathcal{F}^\bullet,\mathsf{L}f^*(\mathcal{G}^\bullet)\otimes_{\mathcal{O}_X}\!K_f[\textup{dim}(f)])\cong \textup{Hom}_{\mathsf{D^b}(Y)}(\mathsf{R}f_*(\mathcal{F}^\bullet),\mathcal{G}^\bullet)\,.
\]
\end{Thm}

Upon choosing $Y=\text{Spec}(\mathbb{K})$ and $f$ equal the structure morphism of $X$ (hence $K_f=K_X$ and $\text{dim}(f)=\text{dim}(X)$), we get an alternative proof of Corollary \ref{SerredualityviaSerrefunc}. Doing the same at the level of coherent sheaves instead gives Theorem \ref{Serreduality}.

\vspace*{0.5cm}
\noindent We move to another important tool relating bounded derived categories of coherent sheaves. Our primary reference remains \cite{[Huy06]}.

\begin{Def}
Consider the product variety $X\times Y$ of two smooth projective $X,Y\in\text{obj}(\mathsf{Var}(\mathbb{K}))$, with projections $p_X:X\times Y\rightarrow X$ and $p_Y:X\times Y\rightarrow Y$. Let $\mathcal{P}^\bullet\in\text{obj}(\mathsf{D^b}(X\times Y))$ be any. Then its associated \textbf{Fourier--Mukai transform}\index{Fourier--Mukai transform} is the functor\footnote{More formally, we should write $\Phi_{\mathcal{P}^\bullet}(\square)=\mathsf{R}p_{Y*}(\mathsf{L}p_X^*(\square)\overset{\mathsf{L}}{\otimes}\mathcal{P}^\bullet)$, but $p_X$ is flat, so we don't ``need'' to derive it (see \cite[section III.9]{[Har77]} for more about flat morphisms). However, unless $\mathcal{P}^\bullet$ is a complex of locally free sheaves, we must left derive $\otimes\equiv\otimes_{\mathcal{O}_{X\times Y}}$.}
\begin{equation}
	\Phi_{\mathcal{P}^\bullet}\coloneqq\mathsf{R}p_{Y*}(p_X^*(\square)\overset{\mathsf{L}}{\otimes}\mathcal{P}^\bullet):\mathsf{D^b}(X)\rightarrow\mathsf{D^b}(Y)\,,
\end{equation}
exact because the composition of exact functors. We call $\mathcal{P}^\bullet$ the \textbf{Fourier--Mukai kernel}\index{Fourier--Mukai kernel} of $\Phi_{\mathcal{P}^\bullet}$ and, in case the latter is an equivalence, we refer to $X,Y$ as \textbf{Fourier--Mukai partners}\index{Fourier--Mukai partners}.
\end{Def}

Since $X\times Y\cong Y\times X$, the kernel $\mathcal{P}^\bullet$ also defines an exact functor $\mathsf{D^b}(Y)\rightarrow\mathsf{D^b}(X)$. We let the order of the two factors implicitly dictate which direction the arrow takes.

\begin{Ex}
Fourier--Mukai transforms crop up quite often, though silently. For example, the pushforward of a morphism of varieties $f:X\rightarrow Y$ is $\mathsf{R}f_*\cong\Phi_{\mathcal{O}_{\Gamma(f)}}:\mathsf{D^b}(X)\rightarrow\mathsf{D^b}(Y)$, the Fourier--Mukai transform induced by the structure sheaf $\mathcal{O}_{\Gamma(f)}\in\text{obj}(\mathsf{D^b}(X\times Y))$ (in degree 0) of the graph $\Gamma(f)\subset X\times Y$ of $f$. Following our convention, $\mathcal{O}_{\Gamma(f)}\in\text{obj}(\mathsf{D^b}(Y\times X))$ instead induces none other than the pullback $\mathsf{L}f^*\cong \Phi_{\mathcal{O}_{\Gamma(f)}}:\mathsf{D^b}(Y)\rightarrow\mathsf{D^b}(X)$. Many other instances are presented in \cite[Examples 5.4]{[Huy06]} and \cite[section 3.4]{[Hoc19]}.				\hfill $\blacklozenge$
\end{Ex}

\begin{Rem}
Fourier--Mukai transforms commute with Serre functors. Indeed, let $X,Y$ be smooth projective varieties and let $\Phi_{\mathcal{P}^\bullet}:\mathsf{D^b}(X)\rightarrow\mathsf{D^b}(Y)$ be the Fourier--Mukai transform of $\mathcal{P}^\bullet\in\text{obj}(\mathsf{D^b}(X\times Y))$. Define the bounded complexes $\mathcal{P}_\text{L}^\bullet\coloneqq((\mathcal{P}^\bullet)^\vee\otimes p_Y^*K_Y)[\text{dim}(Y)]$ and $\mathcal{P}_\text{R}^\bullet\coloneqq((\mathcal{P}^\bullet)^\vee\otimes p_X^*K_X)[\text{dim}(X)]$ in $\mathsf{D^b}(X\times Y)$ (in fact, the images of $(\mathcal{P}^\bullet)^\vee$ under projections and Serre functors). Then $\Phi_{\mathcal{P}_\text{L}^\bullet}\cong \Phi_{(\mathcal{P}^\bullet)^\vee}\circ\mathsf{S}_Y$ and $\Phi_{\mathcal{P}_\text{R}^\bullet}\cong \mathsf{S}_X\circ\Phi_{(\mathcal{P}^\bullet)^\vee}$ as functors $\mathsf{D^b}(Y)\rightarrow\mathsf{D^b}(X)$, moreover left respectively right adjoint to $\Phi_{\mathcal{P}^\bullet}$ (which is by itself promising, since any exact equivalence admits left and right adjoints).   
\end{Rem}

\begin{Lem}
Let $X,Y,Z$ be smooth projective varieties over $\mathbb{K}$, and let $\Phi_{\mathcal{P}_1^\bullet}:\mathsf{D^b}(X)\rightarrow\mathsf{D^b}(Y)$ and $\Phi_{\mathcal{P}_2^\bullet}:\mathsf{D^b}(Y)\rightarrow\mathsf{D^b}(Z)$ be the Fourier--Mukai transforms of $\mathcal{P}_1^\bullet\in\textup{obj}(\mathsf{D^b}(X\times Y))$ respectively $\mathcal{P}_2^\bullet\in\textup{obj}(\mathsf{D^b}(Y\times Z))$. 
Then $\Phi_{\mathcal{P}_2^\bullet}\circ \Phi_{\mathcal{P}_1^\bullet}$ is isomorphic to $\Phi_{\mathcal{P}_3^\bullet}:\mathsf{D^b}(X)\rightarrow\mathsf{D^b}(Z)$, the Fourier--Mukai transform with kernel
\[
\mathcal{P}_3^\bullet\coloneqq \mathsf{R}p_{X\!Z*}(p_{X\!Y}^*\mathcal{P}_1^\bullet\overset{\mathsf{L}}{\otimes} p_{Y\!Z}^*\mathcal{P}_2^\bullet)\in\textup{obj}(\mathsf{D^b}(X\times Z))\,.
\]
\textup(Note that there are other choices of $\mathcal{P}_3^\bullet$ if $\Phi_{\mathcal{P}_2^\bullet}\circ \Phi_{\mathcal{P}_1^\bullet}$ is not an equivalence.\textup) This gives the set of kernels a well-defined composition law, with two-sided identity $\mathcal{O}_{\Delta_X}\in\textup{obj}(\mathsf{D^b}(X\times X))$, the structure sheaf of the diagonal of $X$.
\end{Lem}

Orlov proved in \cite{[Orl03]} the following relation between fully faithful functors and Fourier--Mukai transforms:

\begin{Thm}\label{Orlovthm}
Let $X,Y$ be smooth projective varieties over $\mathbb{K}$, and let $\mathsf{F}:\mathsf{D^b}(X)\rightarrow\mathsf{D^b}(Y)$ be a fully faithful exact functor. If $\mathsf{F}$ has left and right adjoint functors, then there exists some $\mathcal{P}^\bullet\in\textup{obj}(\mathsf{D^b}(X\times Y))$ \textup(unique up to isomorphism\textup) such that $\mathsf{F}\cong\Phi_{\mathcal{P}^\bullet}$ as functors. We say that $\mathsf{F}$ is \textbf{of Fourier--Mukai type}\index{functor!of Fourier--Mukai type}.

This is always verified in case $\mathsf{F}$ is an equivalence, and then the underlying varieties fulfill $\textup{dim}(X)=\textup{dim}(Y)$. In particular, every autoequivalence in $\textup{Auteq}(\mathsf{D^b}(X))$ is induced by some kernel $\mathcal{P}^\bullet\in\textup{obj}(\mathsf{D^b}(X\times X))$, giving $\textup{Auteq}(\mathsf{D^b}(X))$ the structure of a group. 
\end{Thm}

Other remarkable results whose proofs draw from Fourier--Mukai theory are:

\begin{Pro}\label{whenprojvarareequiv}
Let $X,Y$ be smooth projective varieties over $\mathbb{K}$. Then:
\begin{itemize}[leftmargin=0.5cm]
	\item If there exists an equivalence $\mathsf{F}:\mathsf{Coh}(X;{}_{\mathcal{O}_X}\!\mathsf{Mod})\rightarrow \mathsf{Coh}(Y;{}_{\mathcal{O}_Y}\!\mathsf{Mod})$, then $X\cong Y$ as varieties.
	
	\item If there exists an exact equivalence $\mathsf{F}:\mathsf{D^b}(X)\!\rightarrow\! \mathsf{D^b}(Y)$ --- that is, if $X, Y$ are derived equivalent --- and if the \textup(anti-\textup)canonical line bundle of $X$ is ample, then $X\cong Y$ \textup(and the \textup(anti-\textup)canonical line bundle of $Y$ is ample as well\textup).\footnote{Excluding elliptic curves, this result --- due to Bondal and Orlov (in \cite{[BO97]}) --- classifies completely the derived categories of smooth projective varieties in dimension 1! The duo also proved, assuming again ampleness, that $\text{Auteq}(\mathsf{D^b}(X))\cong\mathbb{Z}\times(\text{Aut}(X)\ltimes\text{Pic}(X))$ --- where $\mathbb{Z}$ acts by translation [1], $\text{Aut}(X)$ by pullback and $\text{Pic}(X)$ by tensoring --- which settles de facto the complexity of $\mathsf{D^b}(X)$.}
\end{itemize}
\end{Pro}

It is worth asking when Fourier--Mukai transforms are fully faithful functors or even equivalences. Bondal, Orlov and Bridgeland (see \cite{[Bri98]}) answer as follows.

\begin{Pro}\label{whenFourierMukaiisequivalence}
Let $X,Y$ be smooth projective varieties over $\mathbb{K}$, with $n=\textup{dim}(X)$, and let $\Phi_{\mathcal{P}^\bullet}:\mathsf{D^b}(X)\rightarrow\mathsf{D^b}(Y)$ be the Fourier--Mukai transform of kernel $\mathcal{P}^\bullet\in\textup{obj}(\mathsf{D^b}(X\times Y))$. Then:
\begin{itemize}[leftmargin=0.5cm]
	\renewcommand{\labelitemi}{\textendash}
	\item $\Phi_{\mathcal{P}^\bullet}$ is fully faithful if and only if for any two closed $x_1,x_2\in X$ holds
	\[\begin{small}
		\textup{Hom}_{\mathsf{D^b}(Y)}\big(\Phi_{\mathcal{P}^\bullet}(k(x_1)),\Phi_{\mathcal{P}^\bullet}(k(x_2))[l]\big)\!\cong\!
		\begin{cases}
			\mathbb{K} & $if $\; x_1\!=\!x_2 \;\wedge\; l\!=\!0 \\ 
			\{0\} & $if $\; x_1\!\neq\! x_2 \;\vee\; l\!<\!0\;\vee\; l\!>\!n
		\end{cases}\,; 
	\end{small}\]
	
	\item if $\Phi_{\mathcal{P}^\bullet}$ is readily fully faithful, then it is also an equivalence if and only if $\textup{dim}(X)=\textup{dim}(Y)$ and $\mathcal{P}^\bullet\overset{\mathsf{L}}{\otimes} p_X^*K_X\cong \mathcal{P}^\bullet\overset{\mathsf{L}}{\otimes} p_Y^*K_Y$, or equivalently, if and only if $\Phi_{\mathcal{P}^\bullet}(k(x))\otimes_{\mathcal{O}_Y}\!K_Y\cong \Phi_{\mathcal{P}^\bullet}(k(x))$ for all closed $x\in X$.
	
	\item if $\textup{dim}(X)=\textup{dim}(Y)$ and the canonical line bundles are trivial, $K_X\cong\mathcal{O}_X$ respectively $K_Y\cong\mathcal{O}_Y$, then any fully faithful exact functor $\mathsf{D^b}(X)\rightarrow\mathsf{D^b}(Y)$ is automatically an equivalence.   
\end{itemize}
\end{Pro}

There are several other criteria describing the properties Fourier--Mukai transforms. The curious reader is referred to \cite[chapter 7]{[Huy06]}.

\subsection{The derived category of Calabi--Yau varieties}\label{ch6.5}

In order to define the bounded derived category of coherent sheaves on a Calabi--Yau manifold, we must first understand how to regard the latter as a variety and, ultimately, a scheme. From \cite[section 7.2]{[Imp21]} we know that:

\begin{Def}\label{Calabiyauman}
A \textbf{Calabi--Yau manifold}\index{Calabi--Yau manifold} $X^n=(X,\varpi)$ is a compact Kähler manifold $(X,J,\kappa^\mathbb{C})$ with Riemannian holonomy group $\text{Hol}(\kappa^\mathbb{C})=\text{SU}(n)$ and equipped with a unique (up to phase) holomorphic volume form $\varpi\in\Omega_X^{n,0}(X)\cong\mathcal{O}_X(X)$, called \textit{Calabi--Yau form}, fulfilling the equality $\kappa^n/n!=(-1)^{n(n-1)/2}(i/2)^n\varpi\wedge\overline{\varpi}$.
\end{Def}

A motivating though cautionary observation is that smooth (abstract) varieties over $\mathbb{C}$ can be interpreted as complex manifolds, but the converse is not generally true. Let us briefly see how this comes to be, and if there is any hope Calabi--Yau manifolds revert this implication (spoiler: they do). We refer to \cite[appendix B]{[Har77]} and \cite{[Vak17]}.

\begin{Def}
Let $\mathbb{D}^n\coloneqq\{z=(z_1,...,z_n)\in\mathbb{C}^n\mid |z_i|<1\;\forall 1\leq i\leq n\}\subset\mathbb{C}^n$ denote the complex unit disk, equipped with the structure sheaf of holomorphic functions. A \textbf{complex analytic space}\index{complex analytic space} is a locally ringed space $(\mathfrak{X},\mathcal{O}_\mathfrak{X})$ which admits an open cover by sets which are each isomorphic (as locally ringed spaces) to some \textbf{analytic variety}\index{variety!analytic} $Y\coloneqq V(\{f_1,...,f_k\})\subset\mathbb{D}^n$ --- the vanishing locus of a finite set of holomorphic functions $f_1,...,f_k:\mathbb{D}^n\rightarrow\mathbb{C}$, thus a closed subset (with respect to the standard topology $\mathcal{T}_{\mathbb{C}^n}$) made into a locally ringed space by endowing it with $\mathcal{O}_Y\coloneqq\mathcal{O}_{\mathbb{D}^n}/\langle\{f_1,...,f_k\}\rangle$.\footnote{Formally, a complex analytic space is then a locally ringed space over $\mathbb{C}$ which is locally isomorphic to some \textit{model space} $(Y,\mathcal{O}_Y)$ as just defined; the isomorphism in question is also called a \textit{local model for $X$}.} Letting morphisms between complex analytic spaces be the underlying morphisms of locally ringed spaces, we obtain the category of complex analytic spaces $\mathsf{Sp}_\mathbb{C}^\text{an}$. 

Suppose now that $(X,\mathcal{O}_X)$ is a scheme of finite type over $\mathbb{C}$. Then its \textbf{associated complex analytic space}\index{complex analytic space!associated} (or \textit{complex analytic variety}) $(X_\text{an},\mathcal{O}_{X_\text{an}})$ is obtained as follows: by definition, $X$ admits an open cover by affine schemes of the form $Y_i=\text{Spec}\big(\mathbb{C}[x_1,...,x_n]/\langle\{f_{i,1},...,f_{i,k}\}\rangle\big)$, whose polynomials, regardable as holomorphic functions $\mathbb{C}^n\rightarrow\mathbb{C}$, thus define complex analytic subspaces $(Y_i)_\text{an}\subset\mathbb{C}^n$; then $X_\text{an}$ is obtained by glueing the $(Y_i)_\text{an}$ consistently to how the $Y_i$ glue to form $X$, and similarly for its structure sheaf $\mathcal{O}_{X_\text{an}}$. This construction being functorial, we obtain what is called the \textbf{analytification functor}\index{functor!analytification}
\[
\square_\text{an}:\{X\in\text{obj}(\mathsf{Sch}(\mathbb{C}))\mid X\text{ of finite type}\}\rightarrow \mathsf{Sp}_\mathbb{C}^\text{an}\,.
\]

Additionally, recall that any coherent $\mathcal{O}_X$-module $\mathcal{F}\in\text{obj}(\mathsf{Coh}(X;{}_{\mathcal{O}_X}\!\mathsf{Mod}))$ is by Definition \ref{coherentonringed} locally of the form $\mathcal{F}|_U\cong\mathcal{C}oker(\psi)$ when restricted to any Zariski open $U\subset X$, where $\psi$ is a morphism of finitely-generated free sheaves fitting into some exact sequence $\mathcal{O}_X|_U^{\oplus n_i}\xrightarrow{\psi}\mathcal{O}_X|_U^{\oplus m_i}\xrightarrow{\chi}\mathcal{F}|_U\rightarrow 0$. Then $\mathcal{F}$ identifies a \textbf{coherent analytic sheaf}\index{sheaf!coherent analytic} $\mathcal{F}_\text{an}$, obtained by glueing the $\mathcal{F}_\text{an}|_{U_\text{an}}\cong\mathcal{C}oker(\psi_\text{an})$ coming from the induced $\psi_\text{an}:(\mathcal{O}_X)_\text{an}|_{U_\text{an}}^{\oplus n_i}\rightarrow(\mathcal{O}_X)_\text{an}|_{U_\text{an}}^{\oplus m_i}$ in the same way every $\mathcal{F}|_U$ glues to $\mathcal{F}$. Coherent analytic $\mathcal{O}_{X_\text{an}}$ modules form the category $\mathsf{Coh}(X_\text{an};{}_{\mathcal{O}_{X_\text{an}}}\!\!\!\mathsf{Mod})$.      
\end{Def}

The following correspondences are known for $X\in\text{obj}(\mathsf{Sch}(\mathbb{C}))$: $X$ is separated iff $X_\text{an}$ is Hausdorff, $X$ is reduced iff $X_\text{an}$ is reduced, $X$ is smooth iff $X_\text{an}$ is a complex manifold, and $X$ is proper iff $X_\text{an}$ is compact.

We point out that the implicit map $\varphi:X_\text{an}\rightarrow X$ of the underlying topological spaces sends $X_\text{an}$ bijectively to the closed points in $X$ and can be upgraded to a morphism of locally ringed spaces, so that $\mathcal{F}_\text{an}\cong\varphi^*\mathcal{F}$ for any $\mathcal{F}\in\text{obj}(\mathsf{Coh}(X;{}_{\mathcal{O}_X}\!\mathsf{Mod}))$, and moreover there are natural maps $H^n(X,\mathcal{F})\rightarrow H^n(X_\text{an},\mathcal{F}_\text{an})$ between sheaf cohomology groups for each $n\in\mathbb{Z}$.

Now, our question about the interchangeability of varieties (in the sense of Definition \ref{abstractvar}) with complex manifolds depends on how nice $\square_\text{an}$ actually is: by above, that smooth abstract varieties can be interpreted as complex manifolds is simply the assignment $X\mapsto X_\text{an}$. For the converse direction, we need $\square_\text{an}$ to be essentially surjective as a functor. Well, this fails to hold... but if we take $X$ to be projective (which implies of finite type over $\mathbb{C}$), analytification then makes our dreams come true, as proved by Serre in \cite{[Ser56]}.

\begin{Thm}
The specialized analytification functor
\begin{equation}\label{analytfunc}
	\mkern-12mu\square_\textup{an}\!:\!\{X\!\in\!\textup{obj}(\mathsf{Sch}(\mathbb{C}))\mid X\text{ projective}\}\rightarrow \{\mathfrak{X}\!\in\!\textup{obj}(\mathsf{Sp}_\mathbb{C}^\textup{an})\mid \mathfrak{X}\text{ compact}\}
\end{equation}
is essentially surjective\footnote{Specifically, if $\mathfrak{X}$ is a compact analytic subspace of the complex manifold $\mathbb{C}P^n$ (always possible), then there is a subscheme $X\subset\mathbb{C}P^n$ such that $X_\text{an}\cong\mathfrak{X}$.} and isomorphically injective on objects \textup(the latter meaning that: $X_\textup{an}\!\cong\! X_\textup{an}'\!\implies\! X\!\cong\! X'$\textup).

Moreover, the induced functor on coherent sheaves
\begin{equation}
	\square_\textup{an}:\mathsf{Coh}(X;{}_{\mathcal{O}_X}\!\mathsf{Mod})\rightarrow \mathsf{Coh}(X_\textup{an};{}_{\mathcal{O}_{X_\textup{an}}}\!\!\!\mathsf{Mod})
\end{equation}
is an equivalence, thus essentially surjective, and is isomorphically injective on objects, making the natural maps
\[
H^n(X,\mathcal{F})\rightarrow H^n(X_\textup{an},\mathcal{F}_\textup{an})
\]
into isomorphisms for all $n\in\mathbb{Z}$. 
\end{Thm}   

Now, does essential surjectivity still hold if we enlarge the target space of \eqref{analytfunc} to compact complex manifolds $\mathfrak{X}$? If this is the case, so that there exists some projective scheme $X$ with $X_\text{an}\cong\mathfrak{X}$ (in fact the unique such), then we call $\mathfrak{X}$ \textbf{projective algebraic}\index{complex analytic space!projective algebraic}. From \cite[section B.3]{[Har77]} we learn the following: 

\begin{Thm}
\textup{(Riemann)} Every compact complex manifold of dimension 1, said otherwise, every compact Riemann surface --- including Calabi--Yau 1-folds --- is projective algebraic.

\textup{(Chow--Kodaira)} Every compact complex manifold of dimension 2 with two algebraically independent meromorphic functions --- for example, Calabi--Yau 2-folds --- is projective algebraic.\footnote{In fact, it is then a \textit{Moishezon manifold}, a necessary condition according to \cite[Proposition B.3.3]{[Har77]}.}  
\end{Thm}

In dimension greater than 2, the situation becomes more complicated for general compact complex manifolds. This being said, according to \cite[Theorem 5.8]{[GHJ03]} we are in luck:

\begin{Thm}\label{3+DCYareprojectivealgebraic}
Any Calabi--Yau manifold of dimension $n\geq 3$ is isomorphic as a complex manifold to some submanifold of the complex projective space $\mathbb{C}P^m$ \textup(for $m>n$\textup) which is the analytification of a projective scheme over $\mathbb{C}$, making it projective algebraic.\footnote{More generally, any Kähler manifold which is Moishezon is projective algebraic.}\footnote{On the other hand, there are plenty of $K3$ surfaces --- instances of Calabi--Yau 2-folds --- which are not algebraic and hence cannot be embedded into the complex projective space. It should also be noted that the same issue may occur for complex tori of dimension $\geq 2$ (which advises further investigation of Theorem \ref{3+DCYareprojectivealgebraic}, here not perused; my thanks to Prof. Biran for bringing up this point).} 
\end{Thm}

We therefore see that $\square_\text{an}$ is essentially surjective when taking values in the category of compact complex analytic spaces which are Calabi--Yau manifolds. The projective schemes mapping to them coincide with what are called \textit{Calabi--Yau varieties}, a notion which can also be defined in the broader context of any algebraically closed field $\mathbb{K}$: 

\begin{Def}\label{Calabiyauvar}
An $n$-dimensional \textbf{Calabi--Yau variety}\index{Calabi--Yau variety} over an algebraically closed field $\mathbb{K}$ is a smooth proper variety $X=(X,\mathcal{O}_X)\in\text{obj}(\mathsf{Sch}(\mathbb{K}))$ whose canonical line bundle $K_X$ (see Remark \ref{cansheaf}) is trivial, $K_X\cong\mathcal{O}_X$ (thus coinciding with the sheaf of regular functions on $X$), and whose sheaf cohomology groups  $H^j(X,\mathcal{O}_X)\cong\{0\}$ are trivial for each $0< j<n$.

Over $\mathbb{K}=\mathbb{C}$, we need only replace ``proper'' variety by ``projective'' variety (which implies the former by Proposition \ref{whenproper}) and ``regular'' functions by ``holomorphic'' functions (see Example \ref{sheafholofcts}).
\end{Def}

\begin{Rem}
Similarly to the manifold case, addressed in \cite[Remark 7.7]{[Imp21]}, the mathematical literature often adopts possibly inequivalent definitions of Calabi--Yau variety: for example, the requirement about trivial sheaf cohomology groups is dropped (so that, when achieved, one talks about \textit{true} Calabi--Yau variety), and many authors just impose the first Chern class\footnote{Passing through \v{C}ech cohomology, one can show that $\text{Pic}(X)\cong H^1(X,\mathcal{O}_X^\times)$, where $\text{Pic}(X)$ is as in Definition \ref{Picardgroup} and $\mathcal{O}_X^\times$ is the sheaf of \textit{non-vanishing} regular functions on $X$. Furthermore, the short exact sequence of abelian sheaves $0\rightarrow\mathbb{Z}\rightarrow\mathcal{O}_X\xrightarrow{\text{exp}}\mathcal{O}_X^\times\rightarrow 0$ induces a long exact one with connecting homomorphism which, when $X$ is projective, is precisely $c_1:H^1(X,\mathcal{O}_X^\times)\cong\text{Pic}(X)\rightarrow H^2(X,\mathbb{Z})$. This provides a sheaf theoretic definition of the first Chern class $c_1$. This argument is highlighted in \cite[section 4.5.1]{[ABC+09]}.} $c_1(X)$ to vanish, which in turn implies triviality of the canonical line bundle. 

One can even more generally talk about Calabi--Yau varieties over fields which are not algebraically closed nor of characteristic 0, but we have no interest in pursuing this route here. 
\end{Rem}

Unraveling the many attributes of schemes discussed in Section \ref{ch5.5}, Definition \ref{Calabiyauvar} affirms that a Calabi--Yau variety $X\in\text{obj}(\mathsf{Sch}(\mathbb{K}))$ is a smooth \textit{complete} variety, always in the abstract sense of Definition \ref{abstractvar}; by diagram \eqref{bigdiagram} it therefore states that:   
\vspace*{0.3cm}

\minibox[frame]{A Calabi--Yau variety is a smooth integral noetherian separated scheme\\ of finite type over $\mathbb{K}$, with the additional requirements of triviality of\\ the canonical line bundle and most sheaf cohomology groups.}\vspace*{0.3cm}

Of course, $X$ identifies under the functor \eqref{Xi} some genuine variety $X\in\text{obj}(\mathsf{Var}(\mathbb{K}))$,\footnote{Then, by Definition \ref{smoothsch}, $X$ is also smooth as a variety, because it is integral, hence both reduced and irreducible.} also a ringed space equipped with the sheaf of regular functions.

Again, we stress out that authors are often in disharmony about the algebraical definition of Calabi--Yau varieties. The one we favoured, around which the theory of Part II is constructed, is particularly advantageous because it maximizes the load of information we can ever hope for, as will be reviewed shortly. 

Now that everything has been set, let us describe one final time the structure of $\mathsf{D^b}(X)$ as done back in Definition \ref{dercatdef}, this time assuming $X$ to be a Calabi--Yau manifold and hence, implicitly, the corresponding complex Calabi--Yau variety $X=(X,\mathcal{O}_X)\in\text{obj}(\mathsf{Sch}(\mathbb{C}))$.

\begin{Def}\label{dercatofCY}
Let $X$ be a Calabi--Yau manifold, with associated abelian category $\mathsf{Coh}(X;{}_{\mathcal{O}_X}\!\mathsf{Mod})$ whose homotopy category\footnote{Here we extend the notation adopted in Definition \ref{dercatofsch} to both the (bounded) homotopy category and the category of (bounded) complexes of $\mathsf{Coh}(X;{}_{\mathcal{O}_X}\!\mathsf{Mod})$.} $\mathsf{H^b}(X)\!\coloneqq\!\mathsf{H^b}(\mathsf{Coh}(X;{}_{\mathcal{O}_X}\!\mathsf{Mod}))$ has localizing class $S\coloneqq\{\text{quasi-isomorphisms in }\mathsf{H^b}(X)\}$. Then the bounded derived category $\mathsf{D^b}(X)=\mathsf{D^b}(\mathsf{Coh}(X;{}_{\mathcal{O}_X}\!\mathsf{Mod}))=\mathsf{H^b}(X)[S^{-1}]$ of $X$ has the following structure:
\begin{itemize}[leftmargin=0.5cm]
	\renewcommand{\labelitemi}{\textendash}
	\item objects in $\text{obj}(\mathsf{D^b}(X))$ are bounded cochain complexes $(\mathcal{F}^\bullet,d_\mathcal{F}^\bullet)$ where each $\mathcal{F}^n\in\text{obj}(\mathsf{Coh}(X;{}_{\mathcal{O}_X}\!\mathsf{Mod}))$ is a coherent $\mathcal{O}_X$-module;
	
	\item morphisms $\varphi\in\text{Hom}_{\mathsf{D^b}(X)}(\mathcal{F}^\bullet,\mathcal{G}^\bullet)$ are equivalence classes $[(\psi^\bullet,\chi^\bullet)]$ of roofs
	\begin{equation}
		\begin{tikzcd}
			& \mathcal{F}'^\bullet\arrow[dl, "\psi^\bullet"']\arrow[dr, "\chi^\bullet"] & \\
			\mathcal{F}^\bullet\arrow[rr, squiggly, "\varphi"'] & & \mathcal{G}^\bullet
		\end{tikzcd}\quad,
	\end{equation}
	where $\chi^\bullet\in\text{Hom}_{\mathsf{H^b}(X)}(\mathcal{F}'^\bullet,\mathcal{G}^\bullet)$ is a morphism of complexes of coherent sheaves and $\psi^\bullet\in S$. Two roofs are equivalent if they can be brought under a common roof as in diagram \eqref{equivroofs}, and the composition of morphisms occurs according to \eqref{roofcompo}.
\end{itemize}
The derived category is equipped with a localization functor $\mathsf{L}:\mathsf{H^b}(X)\rightarrow\mathsf{D^b}(X)$ which is the identity on objects and maps any $\chi^\bullet\in\text{Hom}_{\mathsf{H^b}(X)}(\mathcal{F}^\bullet,\mathcal{G}^\bullet)$ to the class $[(\text{id}_\mathcal{F}^\bullet,\chi^\bullet)]\in\text{Hom}_{\mathsf{D^b}(X)}(\mathcal{F}^\bullet,\mathcal{G}^\bullet)$. Furthermore, it fulfills that $\mathsf{L}(\psi^\bullet)$ is an isomorphism for any $\psi^\bullet\in S$, and if any other functor $\mathsf{F}:\mathsf{H^b}(X)\rightarrow\mathsf{D}$ acts this way, then it factors like $\mathsf{F}=\mathsf{G}\circ\mathsf{L}$ for a unique $\mathsf{G}:\mathsf{D^b}(X)\rightarrow\mathsf{D}$.
\end{Def}   

As a good excuse to gather the most important results encountered in this survey, here is a summary of the fundamental features we can infer about $\mathsf{D^b}(X)$:

\begin{itemize}[leftmargin=0.5cm]
\item By Corollary \ref{D(A)istriangulated}, $\mathsf{D^b}(X)$ is a triangulated category (see Definition \ref{triangcat}) when equipped with the translation functor $T=T[1]:\mathsf{D^b}(X)\rightarrow \mathsf{D^b}(X)$ of Definition \ref{translfunc} and with the distinguished triangles of complexes from Definition \ref{exacttriangles}.
\newline Moreover, the functors $\text{Hom}_{\mathsf{D^b}(X)}(\mathcal{F}^\bullet,\square)$, $\text{Hom}_{\mathsf{D^b}(X)}(\square,\mathcal{F}^\bullet):\mathsf{D^b}(X)\rightarrow\mathsf{Ab}$ for any $\mathcal{F}^\bullet\in\text{obj}(\mathsf{D^b}(X))$ are cohomological/triangulated (by Proposition \ref{Homareexact}), and so are the cohomology functors $H^n:\mathsf{D^b}(X)\rightarrow\mathsf{Coh}(X;{}_{\mathcal{O}_X}\!\mathsf{Mod})$ (by Theorem \ref{longexactcohomtriangles}), so that distinguished triangles in $\mathsf{D^b}(X)$ induce long exact Hom-/cohomology-sequences.

\item The fully faithful functor $\mathsf{J}':\mathsf{Coh}(X;{}_{\mathcal{O}_X}\!\mathsf{Mod})\hookrightarrow\mathsf{D^b}(X)$ from Proposition \ref{AinD(A)} allows us to interpret $\mathsf{Coh}(X;{}_{\mathcal{O}_X}\!\mathsf{Mod})$ as the full subcategory of $\mathsf{D^b}(X)$ of all $H^0$-complexes: a coherent $\mathcal{O}_X$-module $\mathcal{F}$ is identified with its 0-complex $\mathcal{F}[0]^\bullet$, and $\text{Hom}_{\mathcal{O}_X}(\mathcal{F},\mathcal{G})\cong\text{Hom}_{\mathsf{D^b}(X)}(\mathcal{F}[0]^\bullet,\mathcal{G}[0]^\bullet)$ for each pair of $\mathcal{O}_X$-modules $\mathcal{F}$, $\mathcal{G}$.
\newline Observe that $\mathsf{D^b}(X)$ is \textit{not} itself an abelian category, though by Lemma \ref{D(A)additive} it is additive (so that it can be studied, for example, through Theorem \ref{YonedaEmbThm}).

\item By Theorem \ref{Extprop} and Remark \ref{HomExtRem}, the coherent $\text{Ext}$-modules defined in \eqref{Extmodules} map any short exact sequence of coherent $\mathcal{O}_X$-modules to a bounded exact sequence starting with the $\text{Hom}_{\mathcal{O}_X}$-groups and trivial in degree $n>\text{dim}(X)$, since the homological dimension \eqref{homologicaldimension} of $\mathsf{Coh}(X;{}_{\mathcal{O}_X}\!\mathsf{Mod})$ is $\text{dim}(X)$ whenever $X$ is projective. 

\item For $X$ Calabi--Yau, the abelian category $\mathsf{QCoh}(X;{}_{\mathcal{O}_X}\!\mathsf{Mod})$ has enough injectives (see Remark \ref{QCohhasenoughinjectives}), thus allowing by Proposition \ref{globalsectforcoherentbdd} the existence of the right derived functor of $\Gamma(X,\square):\mathsf{QCoh}(X;{}_{\mathcal{O}_X}\!\mathsf{Mod})\rightarrow \mathsf{Vect}_\mathbb{K}$.
\newline In particular, sheaf cohomology is well defined on $X$, and can be computed through \v{C}ech cohomology (by Theorem \ref{sheafcohomthroughCech}, which applies since $X$ is noetherian and separated). Given an affine morphism $f:X\rightarrow Y$ of Calabi--Yau varieties, for all coherent $\mathcal{O}_X$-modules $\mathcal{F}$ and $n\geq 0$ holds $H^n(X,\mathcal{F})\cong H^n(Y,f_*\mathcal{F})$ (this is Corollary \ref{sheafcohomthroughCechRem}).

\item Always thanks to $\mathsf{QCoh}(X;{}_{\mathcal{O}_X}\!\mathsf{Mod})$ and Lemma \ref{boundedCohinQCoh}, we can carry out all the proofs of Section \ref{ch6.3} about existence of exact derived functors on $\mathsf{D^b}(X)$ (for $X$ meets all the assumptions we made there). Supposing $f:X\rightarrow Y$ is a morphism of Calabi--Yau varieties, we have $\mathsf{R}f_*:\mathsf{D^b}(X)\rightarrow\mathsf{D^b}(Y)$ (for $f$ projective or proper) and $\mathsf{L}f^*:\mathsf{D^b}(Y)\rightarrow\mathsf{D^b}(X)$. We can also consider\break $\mathsf{R}\mathcal{H}om_{\mathcal{O}_X}^\bullet(\square,\square):\mathsf{D^b}(X)^\text{opp}\times\mathsf{D^b}(X)\rightarrow\mathsf{D^b}(X)$ along with $\mathsf{R}\text{Hom}_{\mathcal{O}_X}^\bullet(\square,\square):$\break $\mathsf{D^b}(X)^\text{opp}\times\mathsf{D^b}(X)\rightarrow\mathsf{D^b}(\mathsf{Vect}_\mathbb{K})$ (whence local $\mathcal{E}xt$-sheaves, dual bounded complexes and $\text{Ext}$-modules), and $\square\overset{\mathsf{L}}{\otimes}_{\mathcal{O}_X}\!\square:\mathsf{D^b}(X)\times\mathsf{D^b}(X)\rightarrow\mathsf{D^b}(X)$ too. All these satisfy the identities described by Proposition \ref{lotofequalities}.  

\item By Theorem \ref{Serreduality}, Serre duality holds also on $X$ Calabi--Yau, so that we obtain isomorphisms $\text{Ext}_{\mathcal{O}_X}^k(\mathcal{F},\mathcal{O}_X)\cong H^{n-k}(X,\mathcal{F})^*$ for all $k\geq 0$, and the Grothendieck--Verdier duality of Theorem \ref{GVduality} generalizing it holds as well for any morphism $f:X\rightarrow Y$ of Calabi--Yau varieties. Moreover, Proposition \ref{spanningclass} provides a spanning class for $\mathsf{D^b}(X)$.
\newline Fourier--Mukai theory is applicable to Calabi--Yau varieties: Theorem \ref{Orlovthm} explains that any equivalence between the derived categories of two given Calabi--Yau varieties $X,Y$ is of Fourier--Mukai type (whence $\text{dim}(X)=\text{dim}(Y)$), while Proposition \ref{whenprojvarareequiv} asserts that if their categories of coherent sheaves are equivalent, then $X\cong Y$ (though note that the second bullet point there does not apply, for $K_X\cong\mathcal{O}_X$ is not ample!); finally, $\text{dim}(X)=\text{dim}(Y)$ implies that any fully faithful exact functor $\mathsf{D^b}(X)\rightarrow\mathsf{D^b}(Y)$ is an equivalence, by Proposition \ref{whenFourierMukaiisequivalence}. 

\item There are also a few facts which are specific to Calabi--Yau varieties, to be found in \cite{[Huy06]}: there, Conjecture 6.24 predicts that two \textit{birational} Calabi--Yau varieties are always derived equivalent (certified in Corollary 11.35 for 3-folds); $\mathsf{D^b}(X)$ admits no non-trivial \textit{semi-orthogonal decomposition} (cf. Exercise 8.8); any line bundle on $X$ is a \textit{spherical object} (see section 8.1), a notion which comes up quite often in mirror symmetry; and much more...     
\end{itemize}

At last, we are ready to look back at Kontsevich's conjecture with an even deeper understanding. We state it in the form of \cite[Conjecture 7.26]{[Imp21]}, with some redesign. The reader interested in the A-side of things is advised to consult \cite{[Imp21]}.

\begin{Con}[\textbf{Homological Mirror Symmetry Conjecture}]\index{Homological Mirror Symmetry Conjecture}\label{HMSconj}		% HMS
Given a Calabi--Yau manifold $X^n=(X,J,\kappa^\mathbb{C},\varpi)$, there exist a ``mirror'' Calabi--Yau manifold $\overline{X}=(\overline{X}, \overline{J},\overline{\kappa}^{\mathbb{C}},\overline{\varpi})$ of same dimension $n$ and equivalences of triangulated categories
\begin{equation}\label{HMS}
	\mathsf{D}^\pi(\widetilde{\mathscr{F}}(X))\cong \mathsf{D^b}(\overline{X}) \qquad\,\text{and}\qquad\, \mathsf{D}^\pi(\widetilde{\mathscr{F}}(\overline{X}))\cong \mathsf{D^b}(X)\,.
\end{equation}
Therefore, the following correspondences hold\textup:
\begin{itemize}[leftmargin=0.5cm]
	\renewcommand{\labelitemi}{\textendash}
	\item Split-closed twisted complexes built out of Lagrangian submanifolds $\mathcal{L}^\bullet\in\textup{obj}\big(\mathsf{D}^\pi(\widetilde{\mathscr{F}}(X))\big)$ correspond to bounded cochain complexes of coherent sheaves $\overline{\mathcal{E}}^\bullet\in\textup{obj}(\mathsf{D^b}(\overline{X}))$, and specularly $\overline{\mathcal{L}}^\bullet\in\textup{obj}\big(\mathsf{D}^\pi(\widetilde{\mathscr{F}}(\overline{X}))\big)$ to $\mathcal{E}^\bullet\in\textup{obj}\big(\mathsf{D^b}(X)\big)$.
	\item Floer cohomology \textup{Hom}-groups are isomorphic to coherent \textup{Ext}-groups \textup(as $\mathbb{C}$-vector spaces\textup) for each $n\geq 0$,  \begin{small}\begin{align*}
			& \textup{Hom}_{\mathsf{D}^\pi(\widetilde{\mathscr{F}}(X))}^n(\mathcal{L}_0^\bullet,\mathcal{L}_1^\bullet)\equiv HF^n(\mathcal{L}_0^\bullet,\mathcal{L}_1^\bullet)\cong\textup{Ext}_{\mathcal{O}_{\overline{X}}}^n(\overline{\mathcal{E}}_0^\bullet,\overline{\mathcal{E}}_1^\bullet)\equiv\textup{Hom}_{\mathsf{D^b}(\overline{X})}(\overline{\mathcal{E}}_0^\bullet,\overline{\mathcal{E}}_1^\bullet[n])\,, \\
			&\textup{Hom}_{\mathsf{D}^\pi(\widetilde{\mathscr{F}}(\overline{X}))}^n(\overline{\mathcal{L}}_0^\bullet,\overline{\mathcal{L}}_1^\bullet)\equiv HF^n(\overline{\mathcal{L}}_0^\bullet,\overline{\mathcal{L}}_1^\bullet)\cong\textup{Ext}_{\mathcal{O}_X}^n(\mathcal{E}_0^\bullet,\mathcal{E}_1^\bullet)\equiv\textup{Hom}_{\mathsf{D^b}(X)}(\mathcal{E}_0^\bullet,\mathcal{E}_1^\bullet[n])\,.
	\end{align*}\end{small}
	\!\!\!Remarkably, $H^n(\mathcal{L}^\bullet;\Lambda_\mathbb{C})\!\cong\!\textup{Ext}_{\mathcal{O}_{\overline{X}}}^n(\overline{\mathcal{E}}^\bullet,\overline{\mathcal{E}}^\bullet)$ and $H^n(\overline{\mathcal{L}}^\bullet;\Lambda_\mathbb{C})\!\cong\!\textup{Ext}_{\mathcal{O}_{X}}^n(\mathcal{E}^\bullet,\mathcal{E}^\bullet)$.
	\item Under these isomorphisms, the Floer products 
	\begin{align*}
		& \langle \mu^2\rangle: HF(\mathcal{L}_1^\bullet,\mathcal{L}_2^\bullet)\otimes HF(\mathcal{L}_0^\bullet,\mathcal{L}_1^\bullet)\rightarrow HF(\mathcal{L}_0^\bullet,\mathcal{L}_2^\bullet)\quad\text{and} \\
		& \langle \mu^2\rangle: HF(\overline{\mathcal{L}}_1^\bullet,\overline{\mathcal{L}}_2^\bullet) \otimes HF(\overline{\mathcal{L}}_0^\bullet,\overline{\mathcal{L}}_1^\bullet)\rightarrow HF(\overline{\mathcal{L}}_0^\bullet,\overline{\mathcal{L}}_2^\bullet)
	\end{align*}
	turn respectively into the Yoneda products 
	\begin{align*}
		& \diamond:\textup{Ext}(\overline{\mathcal{E}}_1^\bullet,\overline{\mathcal{E}}_2^\bullet)\otimes\textup{Ext}(\overline{\mathcal{E}}_0^\bullet,\overline{\mathcal{E}}_1^\bullet) \rightarrow\textup{Ext}(\overline{\mathcal{E}}_0^\bullet,\overline{\mathcal{E}}_2^\bullet)\qquad\text{and} \\
		& \diamond:\textup{Ext}(\mathcal{E}_1^\bullet,\mathcal{E}_2^\bullet)\otimes\textup{Ext}(\mathcal{E}_0^\bullet,\mathcal{E}_1^\bullet) \rightarrow\textup{Ext}(\mathcal{E}_0^\bullet,\mathcal{E}_2^\bullet)\,.
	\end{align*}  
\end{itemize}
From the physics viewpoint, we obtain a bijective correspondence between the moduli space of stable \textup{A}-branes $\mathcal{L}^\bullet$ on $X$ of the topologically twisted \textup{A}-model compactified by $(X,\kappa^\mathbb{C})$ and the moduli space of stable \textup{B}-branes $\overline{\mathcal{E}}^\bullet$ on $\overline{X}$ of the topologically twisted \textup{B}-model compactified by $(\overline{X},\overline{J})$, and specularly.
\end{Con}

\begin{Rem}
	The warnings of \cite[Remark 7.27]{[Imp21]} still apply. One delicate aspect, there taken into account, is that Fukaya categories are more generally (and correctly!) defined over the Novikov field $\Lambda_\mathbb{K}$ (see \cite[Definition 5.9]{[Imp21]} or Definition \ref{Seidelstuff} below), which prevents infinite sums from arising in the Floer composition maps. Accordingly, one should require the conjectured mirror Calabi--Yau manifolds to be algebraic varieties over $\Lambda_\mathbb{K}$ rather than the ground field $\mathbb{K}$! Luckily, in our favoured complex scenario $\mathbb{K}=\mathbb{C}$, $\Lambda_\mathbb{C}$ is algebraically closed, so that algebraic geometry over it is still quite ``standard''. In Chapter 7, when dealing with elliptic curves, we will (safely) ignore this issue, instead fully tackled by \cite{[AS09]}.\footnote{My thanks to Prof. Biran for prompting me to highlight this non-trivial inconsistency.}  
\end{Rem}

Note that Conjecture \ref{HMSconj} acts on the level of classical categories, by virtue of the split-closure construction ultimately taking cohomology on the A-side (whereas the bounded derived category is readily a standard category, as we very well know). Now that we have a firmer grasp of the B-side, we provide an additional, hidden interpretation of \eqref{HMS} as a duality between $A_\infty$-categories (we refer to \cite[subsection 8.2.1]{[ABC+09]}, adopting the notation of \cite[section 2.2]{[Imp21]}).

\begin{Rem}
Let $X$ be a smooth algebraic variety over an algebraically closed field $\mathbb{K}$. Then we can give its bounded derived category the structure of an $A_\infty$-category $\mathsf{D^b_\infty}(X)$ as follows. Let $\text{obj}(\mathsf{D^b_\infty}(X))\coloneqq\text{obj}(\mathsf{D^b}(X))$, and for any given $\mathcal{F}^\bullet,\mathcal{G}^\bullet\in\text{obj}(\mathsf{D^b_\infty}(X))$ define the $\mathbb{Z}$-graded $\mathbb{K}$-vector space 
\begin{align*}
	\text{hom}_{\mathsf{D^b_\infty}\!(X)}(\mathcal{F}^\bullet,\mathcal{G}^\bullet)\coloneqq\; &\text{Hom}_{\mathcal{O}_X}^\bullet(\mathcal{F}^\bullet,\mathcal{G}^\bullet)\qquad\text{for} \\ &\text{Hom}_{\mathcal{O}_X}^n(\mathcal{F}^\bullet,\mathcal{G}^\bullet)=\bigoplus_{k+l=n}\check{C}^k(\mathfrak{U},\mathcal{H}om_{\mathcal{O}_X}^l(\mathcal{F}^\bullet,\mathcal{G}^\bullet))\,,
\end{align*}
where the \v{C}ech complex was defined in \eqref{Cechcomplex} and $\mathcal{H}om$ in \eqref{locHomoncomplexes}, with differentials $\check{d}^\bullet$ (as in \eqref{Cechdifferential}) respectively $d^\bullet:\mathcal{H}om_{\mathcal{O}_X}^\bullet(\mathcal{F}^\bullet,\mathcal{G}^\bullet)\rightarrow \mathcal{H}om_{\mathcal{O}_X}^{\bullet+1}(\mathcal{F}^\bullet,\mathcal{G}^\bullet)$. 
\newline\big(The latter is more explicitly given by
\[
(d^n(\chi))_m=d_\mathcal{G}^{n+m}\circ \chi_m-(-1)^n\chi_{m+1}\circ d_\mathcal{F}^m\in\mathcal{H}om_{\mathcal{O}_X}(\mathcal{F}^m,\mathcal{G}^{n+m+1})
\]
for any $\chi=(\chi_m)_{m\in\mathbb{Z}}\in\mathcal{H}om_{\mathcal{O}_X}^n(\mathcal{F}^\bullet,\mathcal{G}^\bullet)$ with $\chi_m\in\mathcal{H}om_{\mathcal{O}_X}(\mathcal{F}^m,\mathcal{G}^{n+m})$.\big)

Therefore, any (homogeneous) $f\in\text{hom}_{\mathsf{D^b_\infty}\!(X)}^n(\mathcal{F}^\bullet,\mathcal{G}^\bullet)=\text{Hom}_{\mathcal{O}_X}^n(\mathcal{F}^\bullet,\mathcal{G}^\bullet)$ of degree $n$ is specified by a collection $(f_{n,i_0...i_k})$ (with $1\leq i_0<...<i_k\leq |I|$, always according to \eqref{Cechcomplex}), and it follows that
\begin{align*}
	&\mu_{\mathsf{D^b_\infty}\!(X)}^1:\text{hom}_{\mathsf{D^b_\infty}\!(X)}(\mathcal{F}^\bullet,\mathcal{G}^\bullet)\rightarrow\text{hom}_{\mathsf{D^b_\infty}\!(X)}(\mathcal{F}^\bullet,\mathcal{G}^\bullet)[1]\,, \\
	&\big(\mu_{\mathsf{D^b_\infty}\!(X)}^1(f)\big)_{n\!+\!1,i_0...i_{k+1}}\!\coloneqq\!(-1)^n(\check{d}^kf)_{n,i_0...i_{k+1}}\!+\!(-1)^{n+k} d^n(f)_{n\!+\!1,i_0...i_k}
\end{align*}
is a well-defined differential for $\text{hom}_{\mathsf{D^b_\infty}\!(X)}(\mathcal{F}^\bullet,\mathcal{G}^\bullet)$.

Moreover, supposing $g=(g_{m,j_0...j_{k'}})\in\text{hom}_{\mathsf{D^b_\infty}\!(X)}^{m}(\mathcal{G}^\bullet,\mathcal{E}^\bullet)$, we can define the second order composition map
\begin{align*}
	&\mu_{\mathsf{D^b_\infty}\!(X)}^2:\text{hom}_{\mathsf{D^b_\infty}\!(X)}(\mathcal{G}^\bullet,\mathcal{E}^\bullet)\otimes\text{hom}_{\mathsf{D^b_\infty}\!(X)}(\mathcal{F}^\bullet,\mathcal{G}^\bullet)\rightarrow\text{hom}_{\mathsf{D^b_\infty}\!(X)}(\mathcal{F}^\bullet,\mathcal{E}^\bullet)\,, \\
	&\big(\mu_{\mathsf{D^b_\infty}\!(X)}^2(g,f)\big)_{n,i_0...i_{k+k'}}\!\!=\!\big((-1)^{n\!+\!k(m\!-\!k')}g_{2n-k,i_0...i_{k'}}\!\circ\! f_{n,i_{k'}...i_{k+k'}}\big)|_{U_{i_0}\cap...\cap U_{i_{k+k'}}}.
\end{align*}

Finally, since composition of morphisms of sheaves is already associative, we can set $\mu_{\mathsf{D^b_\infty}\!(X)}^d\coloneqq 0$ for $d>2$. One can show that $\mu_{\mathsf{D^b_\infty}\!(X)}^1$ and $\mu_{\mathsf{D^b_\infty}\!(X)}^2$ do fulfill the $A_\infty$-associativity equations. Therefore, we have successfully constructed an $A_\infty$-structure for $\mathsf{D^b_\infty}(X)$.\footnote{Since higher order composition maps are trivial, $\mathsf{D^b_\infty}(X)$ is more than just an $A_\infty$-category: it is a \textit{DG-category} (differential graded category).}

The $\text{hom}_{\mathsf{D^b_\infty}\!(X)}$-spaces fulfill $H^n(\text{hom}_{\mathsf{D^b_\infty}\!(X)}(\mathcal{F}^\bullet,\mathcal{G}^\bullet))\!=\!\text{Ext}_{\mathcal{O}_X}^n(\mathcal{F}^\bullet,\mathcal{G}^\bullet)$ by \eqref{ExtRHomforcomplexes}. In fact, one can define inclusions $\text{Ext}_{\mathcal{O}_X}^\bullet(\mathcal{F}^\bullet,\mathcal{G}^\bullet)\hookrightarrow \text{hom}_{\mathsf{D^b_\infty}\!(X)}(\mathcal{F}^\bullet,\mathcal{G}^\bullet)$ which are quasi-isomorphisms with inverses $H^n$; in turn, this allows us to give the cohomological category $H(\mathsf{D^b_\infty}(X))$ (see \cite[Definition 2.6]{[Imp21]}) the structure of an $A_\infty$-category. But the crucial fact is that $H(\mathsf{D^b_\infty}(X))\cong\mathsf{D^b}(X)$ as ordinary categories, and so we obtain an $A_\infty$-structure for the standard bounded derived category (with non-trivial higher order compositions, beware!).

The moral is, we can also regard Conjecture \ref{HMSconj} as expressing a duality between (the zeroth cohomological categories of) $A_\infty$-categories!       
\end{Rem}

\newpage
\pagestyle{fancy}

\section{Homological mirror symmetry in low dimensions}
\thispagestyle{plain}

\subsection{The elliptic curve}\label{ch7.1}

Elliptic curves constitute the easiest manifestation of homological mirror symmetry. Before addressing it, in the present section we look at their fundamental features, including their derived category. Our primary references are \cite{[Kre00]} and \cite{[Por15]}.

\begin{Def}\label{ellipticcurve}
An \textbf{elliptic curve}\index{elliptic curve} is a Riemann surface (a 1-dimensional complex manifold) given by some quotient $E_\tau\coloneqq\mathbb{C}/\langle 1,\tau\rangle$, where $\langle 1,\tau\rangle\coloneqq\mathbb{Z}\oplus\tau\mathbb{Z}\subset\mathbb{C}\cong\mathbb{R}^2$ is the $\mathbb{Z}^2$-lattice identified by the \textbf{modular parameter}\index{modular parameter} $\tau\in\mathbb{C}$, assumed to satisfy $\mathfrak{Im}(\tau)>0$. The complex structure is the one inherited from $\mathbb{C}$.

Moreover, the quotient $\mathbb{C}/\langle 1,\tau\rangle$ preserves the natural metric and Kähler structure of $\mathbb{C}$, with induced Kähler form which is automatically closed being a top-dimensional differential form. This makes $E_\tau$ to a Kähler manifold, compact by definition. It can also be equipped with a holomorphic volume form (then a locally constant multiple of $dz=dx+idy$, by compactness) fulfilling the Calabi--Yau condition of Definition \ref{Calabiyauman}, so that $E_\tau$ is a 1-dimensional Calabi--Yau manifold. Observe though that its construction totally ignores its symplectic features (which we later highlight in Definition \ref{2torus}).

In fact, one can argue that any Calabi--Yau 1-fold must be of this form, that is, all Calabi-Yau manifolds in dimension 1 are elliptic curves.     
\end{Def}

From our considerations of Section \ref{ch6.5} regarding the interchangeability of Calabi--Yau manifolds and Calabi--Yau varieties, we are then interested in specific instances of smooth projective varieties in dimension 1, which we may also call \textbf{smooth projective curves}\index{smooth projective curve}.

Recall that one can fully classify smooth projective curves by resorting to Proposition \ref{whenprojvarareequiv}, with the sole exception of elliptic curves: at the end of \cite[chapter 5]{[Huy06]}, it is shown that two elliptic curves $E_{\tau_1}$, $E_{\tau_2}$ are derived equivalent (so there is an exact equivalence $\mathsf{D^b}(E_{\tau_1})\rightarrow\mathsf{D^b}(E_{\tau_2})$) if and only if $E_{\tau_1}\cong E_{\tau_2}$ as varieties. (This is in turn used to prove \cite[Corollary 5.46]{[Huy06]} about the relation between smooth complex projective curves and varieties.)   

However, we won't need to understand directly the bounded derived category of elliptic curves, thanks to what explained in the following remark.

\begin{Rem}\label{indecomposabletorsion}
Let us outline two useful notion:
\begin{itemize}[leftmargin=0.5cm]
	\item Let $(X,\mathcal{O}_X)$ be any ringed space. Then a coherent $\mathcal{F}\in\text{obj}(\mathsf{Coh}(X;{}_{\mathcal{O}_X}\!\mathsf{Mod}))$ is \textbf{indecomposable}\index{ox@$\mathcal{O}_X$-module!indecomposable} when it fulfills
	\[
	\mathcal{F}=\mathcal{F}_1\oplus\mathcal{F}_2\text{ for }\mathcal{F}_1,\mathcal{F}_2\in\text{obj}(\mathsf{Coh}(X;{}_{\mathcal{O}_X}\!\mathsf{Mod}))\implies\mathcal{F}_1\cong 0\;\vee\;\mathcal{F}_2\cong 0\,.
	\]
	Now, as mentioned in the recap about $\mathsf{D^b}(X)$ right after Definition \ref{dercatofCY}, smooth projective curves $X$ have homological dimension $\text{dim}_h(X)=1$. Then, according to Lemma \ref{whendimhleq1}, any bounded complex of coherent sheaves $\mathcal{F}^\bullet\in\text{obj}(\mathsf{D^b}(X))$ can be written as $\mathcal{F}^\bullet\cong\bigoplus_{n\in\mathbb{Z}} H^n(\mathcal{F}^\bullet)[-n]^\bullet\in\text{obj}(\mathsf{D^b}(X))$. 
	\newline This tells us that the full subcategory of $\mathsf{D^b}(X)$ of all bounded complexes which are finite direct sums of objects $\mathcal{F}[-n]^\bullet\in\text{obj}(\mathsf{D^b}(X))$ with every $\mathcal{F}\in\text{obj}(\mathsf{Coh}(X;{}_{\mathcal{O}_X}\!\mathsf{Mod}))$ indecomposable is equivalent to $\mathsf{D^b}(X)$. We can then legitimately restrict our attention to such subcategory.
	
	\item Let $(X,\mathcal{O}_X)$ be a scheme. Then a \textbf{torsion sheaf}\index{sheaf!torsion} $\mathcal{F}\in\text{obj}(\mathsf{Sh}(X;{}_{\mathcal{O}_X}\!\mathsf{Mod}))$ is an $\mathcal{O}_X$-module such that $\mathcal{F}(U)$ is a torsion abelian group for every $U\subset X$ open; equivalently, on $X$ a smooth projective curve, $\mathcal{F}$ is torsion if the stalk $\mathcal{F}_{(0)}$ at the \textit{generic point} $(0)\in\text{Spec}(\mathbb{C}[x])$ (the zero ideal) is zero.
	\newline In fact, it can be shown that any $\mathcal{F}\in\text{obj}(\mathsf{Coh}(X;{}_{\mathcal{O}_X}\!\mathsf{Mod}))$ decomposes as $\mathcal{F}\cong\mathcal{F}/\mathcal{F}_\text{tor}\oplus\mathcal{F}_\text{tor}$, where $\mathcal{F}_\text{tor}$ is its torsion part and $\mathcal{F}/\mathcal{F}_\text{tor}$ the remaining locally free part (also torsion-free). Assuming furthermore that $\mathcal{F}$ is indecomposable, it is then either an indecomposable locally free sheaf (thus a vector bundle, by Proposition \ref{locfreeisvectorbundle}) or an indecomposable torsion sheaf supported at one point (that is, with $\mathcal{F}_x\ncong 0$ at just one $x\in X$; we may call $\mathcal{F}$ a \textit{thickened skyscraper sheaf}).
\end{itemize}
\end{Rem}

Let us introduce some notation in order to specifically understand how indecomposable coherent sheaves look like on elliptic curves.

\begin{Pro}\label{linebundleonell}
Let $E_\tau$ be an elliptic curve. Then any line bundle on $E_\tau$ is of the form
\begin{equation}
	\mathcal{L}_\tau(\varphi)\coloneqq(\mathbb{C}\oplus\mathbb{C})/\langle1,\tau\rangle\in\textup{obj}(\mathsf{Sh}^\mathsf{lf}_1(E_\tau;{}_{\mathcal{O}_{E_\tau}}\!\mathsf{Mod}))\,,
\end{equation}
where $\varphi:\mathbb{C}\rightarrow\mathbb{C}^\times$ is a holomorphic function satisfying $\varphi(z+1)=\varphi(z)$ for all $z\in\mathbb{C}$ and the lattice $\langle 1,\tau\rangle$ acts by $1\cdot(z_1,z_2)=(z_1,z_2)$ respectively $\tau\cdot(z_1,z_2)=(z_1+\tau,\varphi(z_1)z_2)$.\footnote{It follows that any global section $f\in\mathcal{L}_\tau(\varphi)(E_\tau)$ is a holomorphic function $f:\mathbb{C}\rightarrow\mathbb{C}$ such that $f(z+\tau)=\varphi(z)f(z)$ and $f(z+1)=f(z)$ for all $z\in\mathbb{C}$.} In turn, any such $\varphi$ can be written as
\[
\varphi=t_x^*\varphi_0\cdot\varphi_0^{n-1}:\mathbb{C}\rightarrow\mathbb{C}^\times,\,z\mapsto\varphi_0(z+x)\cdot\varphi_0^{n-1}(z)
\]
for some $x=a\tau+b\in E_\tau$ and $n\in\mathbb{Z}$, where $\varphi_0:\mathbb{C}\rightarrow\mathbb{C}^\times,\,z\mapsto e^{-\pi i(\tau+2z)}$.

Furthermore, the space of global sections $H^0(\mathcal{L}_\tau(t_x^*\varphi_0^n))$ is $n$-dimensional with a possible basis given by $\big(t_x^*\theta[\frac{k}{n},0](n\tau,nz)\big)_{k=0}^{n-1}$, where $\theta$ denotes the \textbf{Jacobi theta-function}\index{Jacobi theta-function} given by
\begin{equation}
	\theta[a,b](\tau,z)\coloneqq\sum_{n\in\mathbb{Z}}\textup{exp}(\pi i(n+a)^2\tau + 2\pi i(n+a)(z+b))\,.
\end{equation}
\textup(These are related to the classical theta functions $\theta_\tau(z)\coloneqq\sum_{n\in\mathbb{Z}}e^{\pi i(n^2\tau+2nz)}$ as explained in \textup{\cite[section 4.1]{[Por15]}}, where some useful identities such as the product of theta functions are also provided.\textup)  
\end{Pro}

\begin{Def}\label{unipotentvectorbundle}
A vector bundle $F\rightarrow E_\tau$ is \textbf{unipotent}\index{vector bundle!unipotent} if there exists a filtration $0=F_0\subset F_1\subset...\subset F_r=F$ in $\mathsf{Sh}^\mathsf{lf}_1(E_\tau;{}_{\mathcal{O}_{E_\tau}}\!\mathsf{Mod})$ such that $F_{i+1}/F_i\cong\mathcal{O}_{E_\tau}$ for all $0\leq i< r$. 

Given any $V\in\text{obj}(\mathsf{Vect}_\mathbb{C})$ and $A\in\text{GL}(V)$, let us specifically write $F_\tau(V,A)\!\coloneqq(\mathbb{C}\oplus V)/\langle 1,\tau\rangle$ for those vector bundles determined by the actions $1\cdot(z_1,z_2)=(z_1,z_2)$ and $\tau\cdot(z_1,z_2)=(z_1+\tau,Az_2)$. If $A=e^N$ for some nilpotent $N\in\text{End}(V)$, then $F_\tau(V,e^N)$ is unipotent; if moreover $\text{dim}(\ker(N))=1$ (that is, $N$ is \textit{cyclic}), then $F_\tau(V,e^N)$ is also indecomposable.   
\end{Def}

The constructions we just gave appear to be somewhat dispersive, but surprisingly enough they are actually all we need to classify indecomposable vector bundles over elliptic curves, as shown by Atiyah in \cite{[Ati57]}.

\begin{Thm}\label{AtiyahThm}
Let $E_\tau$ be an elliptic curve. Any indecomposable vector bundle on $E_\tau$ is of the form $\pi_{r*}(\mathcal{L}_{r\tau}(\varphi)\otimes F_{r\tau}(V,e^N))$ for $N$ a cyclic nilpotent endomorphism. Here $\pi_r:E_{r\tau}\rightarrow E_\tau$ is the \textbf{isogeny}\index{isogeny} --- that is, the surjective homomorphism of algebraic groups with finite kernel --- induced by the inclusion of lattices $\langle 1,r\tau\rangle\subset\langle 1,\tau\rangle$ for some $r\in\mathbb{R}^+$.
\end{Thm}

By Remark \ref{indecomposabletorsion}, an indecomposable coherent sheaf $\mathcal{E}\in\text{obj}(\mathsf{Coh}(E_\tau;{}_{\mathcal{O}_{E_\tau}}\!\mathsf{Mod}))$ which is not an indecomposable vector bundle must necessarily be an indecomposable torsion sheaf supported at one point $x\in E_\tau$, fully determined by the space of global sections $V\coloneqq\mathcal{E}(E_\tau)\in\text{obj}(\mathsf{Vect}_\mathbb{C})$, which is also a $\mathcal{O}_{E_\tau,x}$-module whose maximal ideal is generated by $f(z)=z-x$, and by some choice of cyclic nilpotent $N\in\text{End}(V)$; we then write $\mathcal{E}$ as
\begin{equation}\label{elliptictorsionsheaf}
\mathcal{S}(x,V,N)\coloneqq(\mathcal{O}_{E_\tau,x}\otimes V)/\langle f(z)-N/2\pi i\rangle\in\text{obj}(\mathsf{Coh}(E_\tau;{}_{\mathcal{O}_{E_\tau}}\!\mathsf{Mod}))\,.
\end{equation}

\begin{Cor}\label{indecompocoherentonell}
Let $E_\tau$ be an elliptic curve. Any indecomposable coherent sheaf $\mathcal{E}\in\textup{obj}(\mathsf{Coh}(E_\tau;{}_{\mathcal{O}_{E_\tau}}\!\mathsf{Mod}))$ is either an indecomposable vector bundle of the form $\pi_{r*}(\mathcal{L}_{r\tau}(\varphi)\otimes F_{r\tau}(V,e^N))$ like in Theorem \ref{AtiyahThm} or an indecomposable torsion sheaf of the form $\mathcal{S}(x,V,N)$ as in \eqref{elliptictorsionsheaf}, where in both cases $N\in\textup{End}(V)$ is cyclic and nilpotent.
\end{Cor}  

Before moving to the A-side, we provide a convenient lemma which is used in the proof of the homological mirror symmetry for elliptic curves.

\begin{Lem}\label{ellSerreduality}
Let $E_\tau$ be an elliptic curve and $\mathcal{E}_1,\mathcal{E}_2\in\textup{obj}(\mathsf{Coh}(E_\tau;{}_{\mathcal{O}_{E_\tau}}\!\mathsf{Mod}))$. Then there exist functorial isomorphisms
\begin{equation}
	\begin{aligned}
		&\textup{Hom}_{\mathsf{D^b}(E_\tau)}(\mathcal{E}_1[0]^\bullet,\mathcal{E}_2[0]^\bullet)=\textup{Ext}_{\mathcal{O}_{E_\tau}}^0(\mathcal{E}_1,\mathcal{E}_2)\cong\textup{Hom}_{\mathcal{O}_{E_\tau}}(\mathcal{E}_1,\mathcal{E}_2)\,, \\
		&\textup{Hom}_{\mathsf{D^b}(E_\tau)}(\mathcal{E}_1[0]^\bullet,\mathcal{E}_2[1]^\bullet)=\textup{Ext}_{\mathcal{O}_{E_\tau}}^1(\mathcal{E}_1,\mathcal{E}_2)\cong\textup{Hom}_{\mathcal{O}_{E_\tau}}(\mathcal{E}_2,\mathcal{E}_1)^*\,,
	\end{aligned} 
\end{equation}
while $\textup{Hom}_{\mathsf{D^b}(E_\tau)}(\mathcal{E}_1[0]^\bullet,\mathcal{E}_2[k]^\bullet)=\textup{Ext}_{\mathcal{O}_{E_\tau}}^k(\mathcal{E}_1,\mathcal{E}_2)\cong\{0\}$ for any other $k\in\mathbb{Z}$.

It follows that $\diamond:\textup{Ext}_{\mathcal{O}_{E_\tau}}^1(\mathcal{E}_2,\mathcal{E}_3)\times\textup{Ext}_{\mathcal{O}_{E_\tau}}^0(\mathcal{E}_1,\mathcal{E}_2)\rightarrow \textup{Ext}_{\mathcal{O}_{E_\tau}}^1(\mathcal{E}_1,\mathcal{E}_3)$ and $\diamond:\textup{Ext}_{\mathcal{O}_{E_\tau}}^0(\mathcal{E}_2,\mathcal{E}_3)\times\textup{Ext}_{\mathcal{O}_{E_\tau}}^1(\mathcal{E}_1,\mathcal{E}_2)\rightarrow \textup{Ext}_{\mathcal{O}_{E_\tau}}^1(\mathcal{E}_1,\mathcal{E}_3)\cong\textup{Hom}_{\mathcal{O}_{E_\tau}}(\mathcal{E}_3,\mathcal{E}_1)^*$ can be interpreted as the maps $(F,f)\mapsto F(f\circ\square)$ respectively $(f,F)\mapsto F(\square\circ f)$, where $f$ are suitable \textup{Hom}-morphisms while $F$ are functionals in the dual \textup{Hom}-spaces.  
\end{Lem}

\begin{proof}(\textit{Sketch})
The cases $k\neq 1$ follow as usual from the definition of Ext-modules and from $\text{dim}(E_\tau)=1$ (so that higher Ext's vanish). For $k=1$, one refines the proof of Serre duality (Theorem \ref{Serreduality} or Corollary \ref{SerredualityviaSerrefunc}) as done in \cite[Lemma 2.7]{[Kre00]}. 
\end{proof}

\subsection{The Fukaya category of elliptic curves}\label{ch7.2}

Now, we construct the A-side of elliptic curves and see that they verify a simplified version of homological mirror symmetry, first discussed by Polishchuk and Zaslow in \cite{[PZ00]}, and more rigorously proven by Kreussler in \cite{[Kre00]}. First, let us be more explicit about the symplectic structure of elliptic curves, by describing what will be their mirror manifolds predicted by Conjecture \ref{HMSconj}.

\begin{Def}\label{2torus}
Let $\mathbb{T}^2=\mathbb{R}^2/\mathbb{Z}^2$ be the \textbf{2-torus}\index{2-torus}, a Kähler manifold equipped with Kähler form $\kappa\coloneqq Adx\wedge dy\in\Omega^2(\mathbb{T}^2)$, where $A>0$ denotes its volume (or Euclidean area). A choice of ``B-field'' in $H^2_\text{dR}(\mathbb{T}^2;\mathbb{R})/H^2_\text{dR}(\mathbb{T}^2;\mathbb{Z})$, represented by $Bdx\wedge dy\in\Omega^2(\mathbb{T}^2)$ for some $B\in\mathbb{R}$, then defines the \textbf{complexified Kähler parameter}\index{complexified Kähler parameter} $\rho\coloneqq B+iA\in\mathbb{C}$ with $\mathfrak{Im}(\rho)>0$, which provides a complexified Kähler form $\rho dx\wedge dy\in\Omega^{1,1}(\mathbb{T}^2)$ for $\mathbb{T}^2$ and allows us to intepret the 2-torus as an elliptic curve of modular parameter $\rho$, thus in particular as a 1-dimensional Calabi--Yau manifold.

We denote $\mathbb{T}^2$ equipped with such $\rho$ by $E^\rho$. Observe that its construction totally ignores its complex structure as an elliptic curve.
\end{Def}

From \cite[Definition 6.13]{[Imp21]} we know that the general construction of Fukaya categories characterizing the A-side is highly non-trivial: objects are (compact, oriented, spin, graded) Lagrangian submanifolds of vanishing Maslov class, fulfilling the ``no-bubbling'' condition and equipped with a choice of graded lift and spin structure; morphisms belong to graded free modules constituting the Floer cochain complex, depending on a choice of perturbation and Floer data for each tuple of objects; finally, there are multilinear higher order composition maps, obtained by counting pseudo-holomorphic disks bounded by the Lagrangian submanifolds, which fulfill the $A_\infty$-associativity equations. This yields an $A_\infty$-category, of which one then needs to consider an ``enriched'' version (see \cite[Definition 7.25]{[Imp21]}) for the purposes of Conjecture \ref{HMSconj}, which speaks of the split-closed derived category $\mathscr{F}^0(E^\rho)\coloneqq\mathsf{D}^\pi(\widetilde{\mathscr{F}}(E^\rho))=H^0\big(\Pi(Tw\widetilde{\mathscr{F}}(E^\rho))\big)$. Quite a daunting task!

Luckily, in the 1-dimensional case of elliptic curves several simplifications can be made, making our analysis a lot easier. For example, in \cite[Proposition 6.31]{[Imp21]} it is sketched how $\mathscr{F}(\mathbb{T}^2)$ can be split-generated by a pair of longitude and meridian loops. However, the most convenient strategy is to ignore the $A_\infty$-structure of the enriched Fukaya category of $E^\rho$ altogether, and rather define the latter \textit{directly}. In doing so we closely follow the ``rectified'' version of \cite[section 3]{[Kre00]}.

\begin{Def}\label{F0Erho}
Let $E^\rho$ be the 2-torus of Definition \ref{2torus}. Then the objects of $\mathscr{F}^0(E^\rho)$ are triples $L=(\Lambda,\alpha, M)\in\text{obj}(\mathscr{F}^0(E^\rho))$ where:
\begin{itemize}[leftmargin=0.5cm]
	\renewcommand{\labelitemi}{\textendash}
	\item $\Lambda\subset E^\rho$ is a closed submanifold given by the quotienting of an affine line $(x(t),y(t))\subset\mathbb{R}^2$ (for $t\in\mathbb{R}$) of rational slope, so that we obtain a \textbf{simple loop}\index{simple loop} in $E^\rho$ (that is, a non-self-intersecting closed curve; see Figure \ref{2toruscompo} below or \cite[Example 6.30]{[Imp21]} for a reference). Equipped with a choice of orientation, such $\Lambda$ is specifically a special Lagrangian submanifold of $E^\rho$ (which are all the geodesics of $E^\rho$, and conversely).\footnote{A brief overview of special Lagrangian submanifolds and the related concepts of minimal and calibrated submanifolds can be found in \cite[appendix B]{[Por15]}.}
	
	\item $\alpha\in\mathbb{R}$ is such that there exists some $t_0\in\mathbb{R}$ satisfying $x(0)+iy(0)+e^{i\pi\alpha}=x(t_0)+iy(t_0)$; this determines the \textbf{grading}\index{grading (of a loop)} of $\Lambda=(x(t),y(t))$ (which is used to define the Maslov index).
	
	\item $M$ is an abbreviation for the \textit{local system}\footnote{Equivalently, a local system is a locally constant sheaf of complex vector spaces, or a complex vector bundle with a flat connection (the variant adopted by \cite[Definition 7.23]{[Imp21]}).} on $\Lambda$ determined by the \textbf{monodromy operator}\index{monodromy operator} $M\in\text{GL}(V)$ with eigenvalues of unit modulus: $V$ is the representation space of $r:\pi_1(\Lambda)=\pi_1(\mathbb{S}^1)\cong\mathbb{Z}\rightarrow\text{GL}(V)$ and $M$ is the image under $r$ of the $\mathbb{Z}$-generator $z_M$ determined by the orientation on $\Lambda$ given by $e^{i\pi\beta}$, where $\beta\in(-\frac{1}{2},\frac{1}{2})$ is the unique real number such that $\beta-\alpha\in\mathbb{Z}$.  
\end{itemize}
Furthermore, for any $L=(\Lambda,\alpha, M)\in\text{obj}(\mathscr{F}^0(E^\rho))$ and $x\in \Lambda$ we choose a fixed interval $\{\tilde{x}+t\lambda\mid 0\leq t\leq 1\}\subset\tilde{\Lambda}$, where $\tilde{\Lambda}$ is the universal cover of $\Lambda$ (thus the corresponding unquotiented affine line in $\mathbb{R}^2$), $\lambda\in\mathbb{R}^2$ is the unique vector parallel to $\tilde{\Lambda}$ representing a lift of $z_M$ (upon imposing a positive orientation) and $\tilde{x}\in\tilde{\Lambda}$ is the suitable point in this interval which maps to $x$, in turn used to replace $V\cong M_x$.  
\end{Def}

\begin{Def}
The category $\mathscr{F}^0(E^\rho)$ is equipped with the translation functor $T^1:\mathscr{F}^0(E^\rho)\rightarrow \mathscr{F}^0(E^\rho),\,(\Lambda,\alpha, M)\mapsto (\Lambda,\alpha, M)[1]\coloneqq (\Lambda,\alpha+1, M)$.

Consider now two $L_k=(\Lambda_k,\alpha_k, M_k)\in\text{obj}(\mathscr{F}^0(E^\rho))$ for $k=0,1$, and let $\alpha\coloneqq\alpha_1-\alpha_0\in\mathbb{R}$. For their Hom-spaces we must distinguish two cases:
\begin{itemize}[leftmargin=0.5cm]
	\renewcommand{\labelitemi}{\textendash}
	\item if $\Lambda_0\neq\Lambda_1$, then 
	\begin{equation}\label{ellmorphfordifferentlambdas}
		\begin{aligned}
			\mkern-12mu\text{Hom}_{\mathscr{F}^0(E^\rho)}(L_0,L_1)\coloneqq
			\begin{cases}
				\{0\} & $if $ \alpha\!\notin\! [0,1)\,, \\ 
				\mkern-33mu\displaystyle{\bigoplus_{\mkern+40mu x\in\Lambda_0\cap\Lambda_1}}\mkern-24mu\text{Hom}((M_0)_x,(M_1)_x) & $else$\,.
			\end{cases} 
		\end{aligned}
	\end{equation}
	
	\item if $\Lambda_0=\Lambda_1$ (hence $\alpha\in\mathbb{Z}$), then
	\begin{small}\begin{align}\label{ellmorphforsamelambdas}
			\mkern-6mu\text{Hom}_{\mathscr{F}^0(E^\rho)}(L_0,L_1)&\coloneqq
			{\begin{cases}
					\{0\} & $if $ \alpha\neq 0,1\,, \\
					H^0(\Lambda_0,\mathcal{H}om(M_0,M_1)) & $if $ \alpha=0\,, \\
					H^1(\Lambda_0,\mathcal{H}om(M_0,M_1)) & $if $ \alpha=1
			\end{cases}} \\
			&\,\cong\begin{cases}
				\{0\} & $if $ \alpha\neq 0,1\,, \\
				\{f\!\in\!\text{Hom}(V_0,V_1)\mid M_1\!\circ\! f\!=\!f\!\circ\! M_0\} & $if $ \alpha=0\,, \\
				\text{Hom}(V_0,V_1)/M_1\!\circ\!\text{Hom}(V_0,V_1)\!\circ\! M_0^{-1} & $if $ \alpha=1\,.
			\end{cases}\nonumber
	\end{align}\end{small}
\end{itemize}
In particular, for any $L_0,L_1\in\text{obj}(\mathscr{F}^0(E^\rho))$ the \textbf{``symplectic'' Serre duality}\index{symplectic@``symplectic'' Serre duality} holds:
\begin{equation}\label{sympSerreduality}
	\text{Hom}_{\mathscr{F}^0(E^\rho)}(L_0,L_1[1])\cong\text{Hom}_{\mathscr{F}^0(E^\rho)}(L_1,L_0)^*\,,
\end{equation}
canonically and functorially. 
\end{Def}

Now, it only remains to define the composition law for morphisms. Since in \cite{[PZ00]} only the case of three distinct Lagrangian submanifolds is addressed, we keep following the thorough analysis of \cite{[Kre00]}.

\begin{Def}
Let $L_k=(\Lambda_k,\alpha_k,M_k)\in\text{obj}(\mathscr{F}^0(E^\rho))$ for $k=0,1,2$ and $u\in\text{Hom}_{\mathscr{F}^0(E^\rho)}(L_0,L_1)$, $v\in\text{Hom}_{\mathscr{F}^0(E^\rho)}(L_1,L_2)$ (both non-zero), so that by definition we must have $\alpha_0\leq\alpha_1\leq \alpha_0+1$ and $\alpha_1\leq\alpha_2\leq \alpha_1+1$. To define $v\circ u\in\text{Hom}_{\mathscr{F}^0(E^\rho)}(L_0,L_2)$, and hence the composition law
\[
\circ:\text{Hom}_{\mathscr{F}^0(E^\rho)}(L_1,L_2)\otimes\text{Hom}_{\mathscr{F}^0(E^\rho)}(L_0,L_1)\rightarrow \text{Hom}_{\mathscr{F}^0(E^\rho)}(L_0,L_2)\,,
\]
we distinguish between the following cases.
\begin{itemize}[leftmargin=0.5cm]
	\renewcommand{\labelitemi}{\textendash}
	\item If $\alpha_0<\alpha_1<\alpha_2 < \alpha_0+1$, so that the $L_k$ are all distinct, then given $u_{x_0}\in\text{Hom}_{\mathscr{F}^0(E^\rho)}((M_0)_{x_0},(M_1)_{x_0})\subset\text{Hom}_{\mathscr{F}^0(E^\rho)}(L_0,L_1)$ for $x_0\in\Lambda_0\cap\Lambda_1$ and $v_{x_1}\in\text{Hom}_{\mathscr{F}^0(E^\rho)}((M_1)_{x_1},(M_2)_{x_1})\subset\text{Hom}_{\mathscr{F}^0(E^\rho)}(L_1,L_2)$ for $x_1\in\Lambda_1\cap\Lambda_2$, we define $v\circ u$ fibrewise by
	\begin{equation}
		(v\circ u)_{x_2}\coloneqq\mkern-18mu\sum_{[(\phi;z_0,z_1,z_2)]}\mkern-18mu e^{2\pi i\!\int\!\phi^*(\rho dx\wedge dy)}\mathbb{P}(M_2)\circ v_{x_1}\circ\mathbb{P}(M_1)\circ u_{x_0}\circ\mathbb{P}(M_0)
	\end{equation}
	in $\text{Hom}_{\mathscr{F}^0(E^\rho)}((M_0)_{x_2},(M_2)_{x_2})$, for any $x_2\in \Lambda_0\cap\Lambda_2$. Here $\mathbb{P}(M_k):(M_k)_{x_{k-1}}$ $\rightarrow (M_k)_{x_k}$ denotes the parallel transport, and we sum\footnote{This procedure is analogous to the counting of pseudoholomorphic triangles defining the Floer product, as illustrated in \cite[section 6.1]{[Imp21]}.} over all equivalence classes of holomorphic maps $\phi:\mathbb{D}^2\rightarrow E^\rho$ from the unit disk with marked distinct boundary points $z_0,z_1,z_2\in\mathbb{D}^2$ such that $\phi(z_k)=x_k$ and that the boundary arcs $z_k\!\curvearrowright\! z_{k+1}$ are mapped to $\Lambda_k$ (thus $z_2\!\curvearrowright\! z_0$ to $\Lambda_2$), considering two configurations as equivalent when related by an automorphism of $\mathbb{D}^2$. The situation is depicted in Figure \ref{2toruscompo} below.
	
	\item If $\alpha_2>\alpha_0+1$, then $v\circ u\coloneqq 0$.
	
	\item If $\alpha_0=\alpha_1<\alpha_2$ (thus $\Lambda_0=\Lambda_1\neq\Lambda_2$), then $v\circ u$ is induced by $v_x\otimes u\mapsto v_x\circ u_x$ upon taking the direct sum over all $x\in\Lambda_0\cap\Lambda_2$, as suggested by \eqref{ellmorphfordifferentlambdas}. The situation where $\alpha_0+1=\alpha_1=\alpha_2$ and $\Lambda_0\neq\Lambda_1=\Lambda_2$ works similarly.
	
	\item If $\alpha_0=\alpha_1=\alpha_2$, so $\Lambda_0=\Lambda_1=\Lambda_2$, we can just compose the linear homomorphisms $f_k$ coming from the second case of \eqref{ellmorphforsamelambdas}.
	
	\item Finally, if $\alpha_0+1=\alpha_2$, one can use the symplectic Serre duality to return to one of the above cases.    
\end{itemize}
Note that the first scenario is actually the only one depending on $\rho$.	
\end{Def}

\begin{figure}[htp]
\centering
\includegraphics[width=0.6\textwidth]{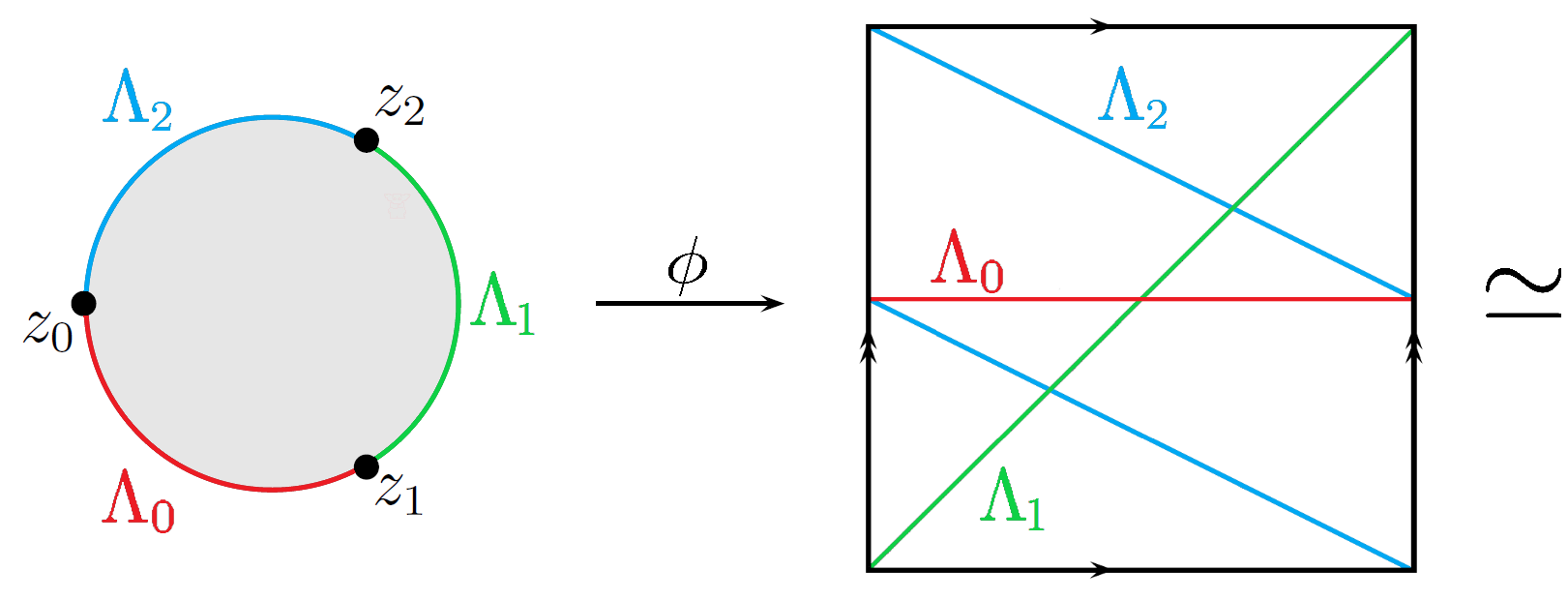}\quad
\includegraphics[width=0.35\textwidth]{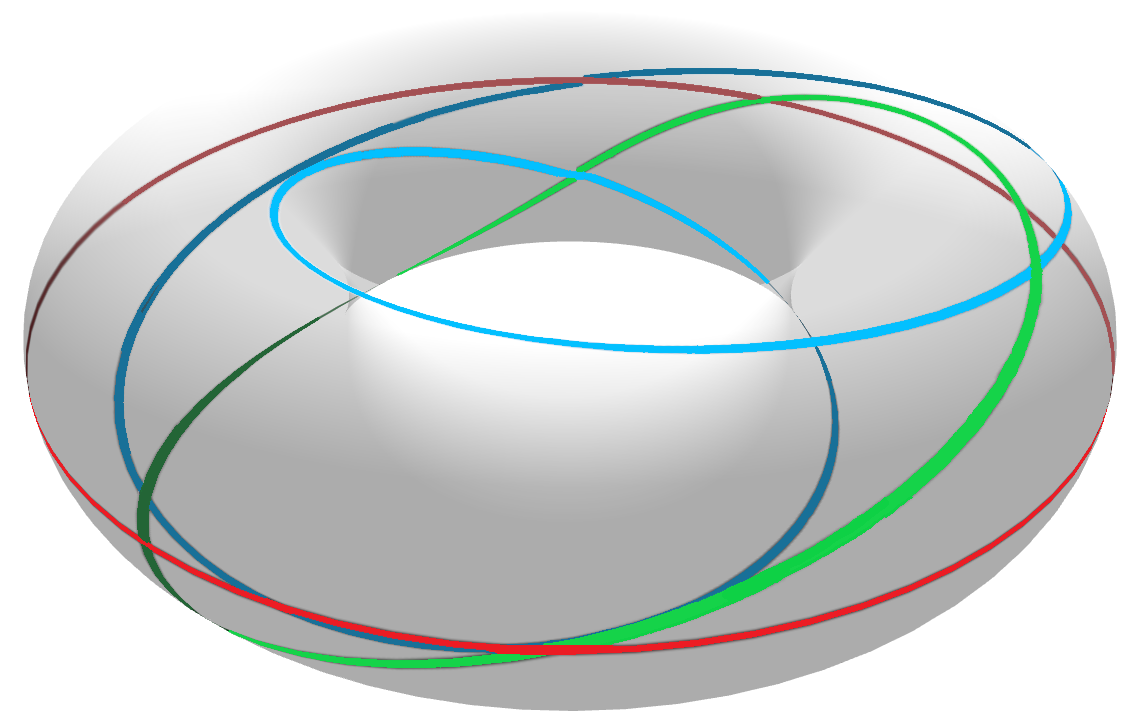}	
\caption{A possible map $\phi:\mathbb{D}^2\rightarrow E^\rho$ entering the definition of composition law for the case of distinct Lagrangian submanifolds. In particular, depicted are the affine lines $(t,1/2)$, $(t,t)$ and $(t,-t/2+1/2)$ in $\mathbb{R}^2$ modulo $\mathbb{Z}^2$, yielding $\Lambda_0$, $\Lambda_1$ respectively $\Lambda_2$.}
\label{2toruscompo}	
\end{figure}
\vspace*{0.3cm}

\noindent We now encounter the main issue originally overlooked in \cite{[PZ00]}: the category $\mathscr{F}^0(E^\rho)$ is a preadditive category (only fulfilling axiom (A1) of Definition \ref{addcat}), but it is not additive like all derived categories (cf. Lemma \ref{D(A)additive}) since it does not contain direct sums, making it impossible to construct any sort of equivalence as conjectured by homological mirror symmetry. However, we observe the following general fact.

\begin{Lem}\label{frompreaddtoaddcat}
Let $\mathsf{C}$ be a preadditive category. Let $\underline{\mathsf{C}}$ be the category having the following structure:
\begin{itemize}[leftmargin=0.5cm]
	\renewcommand{\labelitemi}{\textendash}
	
	\item $\textup{obj}(\underline{\mathsf{C}})=\coprod_{k\geq 0}\prod_{l=1}^k\textup{obj}(\mathsf{C})$, so that objects of $\underline{\mathsf{C}}$ are ordered tuples of objects of $\mathsf{C}$ of any length. 
	
	\item Given $\underline{X}=(X_1,...,X_k)$ and $\underline{Y}=(Y_1,...,Y_l)$ in $\textup{obj}(\underline{\mathsf{C}})$ for $k,l>0$, we let $\textup{Hom}_{\underline{\mathsf{C}}}(\underline{X},\underline{Y})\coloneqq\prod_{\begin{tiny}1\leq i\leq k, 1\leq j\leq l\end{tiny}}\textup{Hom}_\mathsf{C}(X_i,Y_j)$. We then define the zero object $\underline{0}\in\textup{obj}(\underline{\mathsf{C}})$ by requiring it to be both initial and terminal, that is, $\textup{Hom}_{\underline{\mathsf{C}}}(\underline{0},\underline{Y})\coloneqq\{0\}$ and $\textup{Hom}_{\underline{\mathsf{C}}}(\underline{X},\underline{0})\coloneqq\{0\}$ \textup(if $\mathsf{C}$ admits a zero object $0_\mathsf{C}$, then any tuple $(0_\mathsf{C},...,0_\mathsf{C})$ is isomorphic to $\underline{0}$\textup).
	
	\item The composition law for morphisms is given by the standard multiplication of matrices. 
\end{itemize}
Then $\underline{\mathsf{C}}$ is an additive category and the additive functor $\mathsf{C}\rightarrow\underline{\mathsf{C}}$ acting on objects as $X\mapsto \underline{X}\coloneqq(X)$ \textup(the 1-tuple\textup) is a universal functor \textup(see \textup{\cite[section III.1]{[Har77]}} for the definition of \textup{universal}\textup). 
\end{Lem} 

This motivates us to study the additive category induced by $\mathscr{F}^0(E^\rho)$ in this way, rather than $\mathscr{F}^0(E^\rho)$ itself. Also, we start to denote the modular parameter $\rho$ by $\tau$, to enforce the coming statement of homological mirror symmetry.

\begin{Def}\label{FukayaKontsevichcat}
Let $E^\tau$ be a 2-torus, with split-closed derived category $\mathscr{F}^0(E^\tau)$ $=\mathsf{D}^\pi(\widetilde{\mathscr{F}}(E^\tau))=H^0\big(\Pi(Tw\widetilde{\mathscr{F}}(E^\tau))\big)$. Then we call 
\[
\mathscr{F\!K}^0(E^\tau)\coloneqq\underline{\mathscr{F}^0(E^\tau)}
\]
the \textbf{Fukaya--Kontsevich category}\index{Fukaya--Kontsevich category}. By Lemma \ref{frompreaddtoaddcat} it is an additive category, and the additive functor $\mathscr{F}^0(E^\tau)\rightarrow\mathscr{F\!K}^0(E^\tau),\,L\mapsto \underline{L}\coloneqq(L)$ allows us to keep denoting its objects and morphisms without the underscore.

For any $r\in\mathbb{N}$, let $p_r:E^{r\tau}\rightarrow E^\tau,\,(x,y)\mapsto(rx,y)$. Then we define the additive functors:
\begin{itemize}[leftmargin=0.5cm]
	\item $p_{r*}:\mathscr{F\!K}^0(E^{r\tau})\rightarrow \mathscr{F\!K}^0(E^\tau)$, given by $p_{r*}(\Lambda,\alpha,M)\coloneqq (p_r(\Lambda),\alpha',p_{r*}M)$, where: $\alpha'\in(k-\frac{1}{2},k+\frac{1}{2}]$ is the unique admissible value fulfilling the grading condition of Definition \ref{F0Erho} and $k\in\mathbb{Z}$ is determined by $\alpha$, which must belong to this same interval; $p_{r*}M\in\text{GL}(V^{\oplus d})$ is given by $(p_{r*}M)(v_1,...,v_d)\coloneqq(v_2,...,v_d,M(v_1))$ (supposing $p_r$ on $\Lambda$ is of degree $d$).
	\newline On morphisms in any $\text{Hom}_{\mathscr{F\!K}^0(E^{r\tau})}(L_0,L_1)$, we must again discern between the cases $p_r(\Lambda_0)\neq p_r(\Lambda_1)$, $\Lambda_0=\Lambda_1$ and $\Lambda_0\neq\Lambda_1\,\wedge\,p_r(\Lambda_0)= p_r(\Lambda_1)$, and check the compatibility.        
	
	\item $p_r^*:\mathscr{F\!K}^0(E^\tau)\rightarrow\mathscr{F\!K}^0(E^{r\tau})$ with $p_r^*(\Lambda,\alpha,M)\!\coloneqq\!\bigoplus_{k=1}^n (\Lambda^{(k)},\alpha',(p_r^{(k)})^*M)$, where: we assume $p_r^{-1}(\Lambda)$ has $n$-many connected components $\Lambda^{(k)}$ and every restriction $p_r^{(k)}:\Lambda^{(k)}\rightarrow\Lambda$ has degree $d\coloneqq r/n$; $\alpha'\in(k-\frac{1}{2},k+\frac{1}{2}]$ shares again the same interval of $\alpha$; $(p_r^{(k)})^*M=M^d\in\text{GL}(V)$.\footnote{In the interpretation of local systems as locally constant sheaves of complex vector spaces, $p_{r*}M$ stands for the monodromy of the pushforward of the local system yielding $M$, while $(p_r^{(k)})^*M$ for the pullback.} This is also where the necessity of closure under direct sums becomes apparent, since $p_r^*$ doesn't map to $\mathscr{F}^0(E^\tau)$ if $n>1$!
	\newline For morphisms, we need only investigate when the Lagrangian submanifolds are equal and when not.  
\end{itemize}
\end{Def}

\subsection{Homological mirror symmetry in dimension 1}\label{ch7.3}

We finally come to the statement of homological mirror symmetry for elliptic curves, as properly formulated in \cite[Theorem 4.1]{[Kre00]}. It is worth observing that the mirror pair consisting of two elliptic curves lines up with what classical mirror symmetry predicts: an analysis of their Hodge diamonds suggests that elliptic curves are self-dual (see \cite[Example 7.10]{[Imp21]}).  

\begin{Thm}\label{HMSindim1}
Let $\tau\in\mathbb{C}$ be such that $\mathfrak{Im}(\tau)>0$. Then there exists an equivalence
\begin{equation}
	\mathsf{E}_\tau:\mathsf{D^b}(E_\tau)\rightarrow\mathscr{F\!K}^0(E^\tau)
\end{equation}
which is compatible with the respective shift functors.
\end{Thm}

\begin{proof}(\textit{Sketch})
We first define $\mathsf{E}_\tau$ on objects. By Remark \ref{indecomposabletorsion}, it suffices to focus on indecomposable coherent sheaves $\mathcal{E}\in\textup{obj}(\mathsf{Coh}(E_\tau;{}_{\mathcal{O}_{E_\tau}}\!\mathsf{Mod}))$, which by Corollary \ref{indecompocoherentonell} are either indecomposable vector bundles or indecomposable torsion sheaves supported at one point. We proceed in order.
\begin{itemize}[leftmargin=0.5cm]
	\item Let $\mathcal{E}$ be first an indecomposable vector bundle of the form $\mathcal{E}=\mathcal{L}_\tau(\varphi)\otimes F_\tau(V,e^N)\in\text{obj}(\mathsf{D^b}(E_\tau))$, for $\varphi=t_x^*\varphi_0\cdot\varphi_0^{n-1}$ as by Proposition \ref{linebundleonell} and $x=a\tau+b\in E_\tau$, $n\in\mathbb{Z}$ and the indecomposable $F_\tau(V,e^N)$ as by Definition \ref{unipotentvectorbundle}. Then we set 
	\[
	\mathsf{E}_\tau(\mathcal{L}_\tau(\varphi)\otimes F_\tau(V,e^N))\coloneqq(\Lambda,\alpha,M)\in\text{obj}(\mathscr{F\!K}^0(E^\tau))\,,
	\]
	where $\Lambda\coloneqq(a+t,(n-1)a+nt)/\mathbb{Z}^2$, $\alpha\in(-\frac{1}{2},\frac{1}{2}]$ unique with $e^{i\pi\alpha}=\frac{1+in}{\sqrt{1+n^2}}$, and $M\coloneqq e^{-2\pi i b\,\text{id}_V + N}$.
	
	\item Suppose now $\pi_r:E_{r\tau}\rightarrow E_\tau$ is an isogeny and $\mathcal{E}=\pi_{r*}(\mathcal{L}_{r\tau}(\varphi)\otimes F_{r\tau}(V,e^N))\in\text{obj}(\mathsf{D^b}(E_{\tau}))$, which by Theorem \ref{AtiyahThm} is the general form taken by any indecomposable vector bundle on $E_\tau$. Then, for $p_r:E^{r\tau}\rightarrow E^\tau$ as in Definition \ref{FukayaKontsevichcat}, we set
	\[
	\mathsf{E}_\tau(\mathcal{E})\coloneqq p_{r*}\mathsf{E}_{r\tau}(\mathcal{E})=(\Lambda,\alpha,M)\in\text{obj}(\mathscr{F\!K}^0(E^\tau))\,,
	\]
	where $\Lambda\coloneqq\big(a+t,(\frac{n}{r}-1)a+\frac{n}{r}t\big)/\mathbb{Z}^2$, $\alpha\in(-\frac{1}{2},\frac{1}{2}]$ with $e^{i\pi\alpha}=\frac{1+in/r}{\sqrt{1+n^2/r^2}}$, and again $M\coloneqq e^{-2\pi i b\,\text{id}_V + N}$.
	
	\item If $\mathcal{E}=\mathcal{S}(x,V,N)\in\text{obj}(\mathsf{D^b}(E_{\tau}))$ is a torsion sheaf supported at $x=a\tau+b\in E_\tau$, then we simply set
	\[
	\mathsf{E}_\tau(\mathcal{S}(a\tau+b,V,N))\coloneqq((-a,t)/\mathbb{Z}^2,1/2,e^{2\pi i b\,\text{id}_V+N})
	\]
	(so that $\Lambda$ is a meridian loop).   
\end{itemize}
To define $\mathsf{E}_\tau$ on morphisms, it suffices to establish the maps
\[
(\mathsf{E}_\tau)_{\mathcal{E}_1,\mathcal{E}_2}:\text{Hom}_{\mathsf{D^b}(E_\tau)}(\mathcal{E}_1,\mathcal{E}_2[n])\rightarrow\text{Hom}_{\mathscr{F\!K}^0(E^\tau)}(\mathsf{E}_\tau(\mathcal{E}_1),\mathsf{E}_\tau(\mathcal{E}_2)[n])
\]
for $n=0,1$, with $\mathcal{E}_1,\mathcal{E}_2\in\text{obj}(\mathsf{D^b}(E_\tau))$ any. But by Lemma \ref{ellSerreduality} and the symplectic Serre duality \eqref{sympSerreduality}, the case $n=0$ is actually enough. Again, we go by steps.
\begin{itemize}[leftmargin=0.5cm]
	\item We deal first with indecomposable vector bundles of the form $\mathcal{E}_i=\mathcal{L}_\tau(\varphi_i)\otimes F_\tau(V_i,e^{N_i})\in\text{obj}(\mathsf{D^b}(E_\tau))$, for $\varphi_i=t_{x_i}^*\varphi_0\cdot\varphi_0^{n_i-1}$ and $x_i=a_i\tau+b_i\in E_\tau$ ($i=0,1$), such that $\mathsf{E}_\tau(\mathcal{E}_i)=(\Lambda_i,\alpha_i,M_i)$ with $\Lambda_0\neq\Lambda_1$. By \cite[Proposition 2]{[PZ00]}, there are isomorphisms $H^0(\mathcal{L}_\tau(\varphi))\otimes V\cong H^0(\mathcal{L}_\tau(\varphi)\otimes F_\tau(V,e^{N}))$ for any such coherent sheaf; applied to $\varphi=\varphi_1\varphi_0^{-1}$ and $V=\text{Hom}(V_0,V_1)$, we get $\text{Hom}_{\mathsf{D^b}(E_\tau)}(\mathcal{E}_1,\mathcal{E}_2)\cong H^0(\mathcal{L}_\tau(\varphi_1\varphi_0^{-1}))\otimes\text{Hom}(V_0,V_1)$ (assuming without loss of generality that $n_0<n_1$). 
	\newline By linearity of $\mathsf{E}_\tau$, it is then enough to specify it on elements of the form $\theta\otimes f\in H^0(\mathcal{L}_\tau(\varphi_1\varphi_0^{-1}))\otimes\text{Hom}(V_0,V_1)$, since theta functions build a basis for $H^0(\mathcal{L}_\tau(\varphi_1\varphi_0^{-1}))$; according to Proposition \ref{linebundleonell}, we specifically need an assignment for $t_{x_{01}}^*\theta[k/(n_1\!-\!n_0),0]((n_1\!-\!n_0)\tau,(n_1\!-\!n_0)z)$, where $x_{01}=(x_1\!-\!x_0)/(n_1\!-\!n_0)$ and $0\leq k<n_1\!-\!n_0$. This is
	\begin{small}\[
		t_{x_{01}}^*\theta\Big[\frac{k}{n_1\!-\!n_0},0\Big]\big((n_1\!-\!n_0)\tau,(n_1\!-\!n_0)z\big)\mapsto e_k\!\coloneqq\!\Big(\frac{a_1\!-\!a_0\!+\!k}{n_1\!-\!n_0},\frac{n_0a_1\!-\!n_1a_0\!+\!n_0k}{n_1\!-\!n_0}\Big)\,,
		\]\end{small}
	\!\!so that these $e_k$'s are precisely the intersection points of $\Lambda_0$ with $\Lambda_1$. The case of coinciding Lagrangian submanifolds, $\Lambda_0=\Lambda_1$, is addressed in \cite{[Kre00]}.
	
	\item Now we extend the definition on morphisms of $\mathsf{E}_\tau$ to arbitrary vector bundles $\pi_{r_i*}(\mathcal{E}_i)=\pi_{r_i*}(\mathcal{L}_{r_i\tau}(\varphi_i)\otimes F_{r_i\tau}(V_i,e^{N_i}))\in\text{obj}(\mathsf{D^b}(E_{\tau}))$ for $r_0,r_1>0$ possibly equal. Actually, we don't; those interested can read about this passage in the proof of \cite[Theorem 4.1]{[Kre00]}.
	
	\item If $\mathcal{E}_0=\mathcal{L}_\tau(\varphi_0)\otimes F_\tau(V_0,e^{N_0})$ is an indecomposable vector bundle and $\mathcal{E}_1=\mathcal{S}(x_1,V_1,N_1)\in\text{obj}(\mathsf{D^b}(E_\tau))$ an indecomposable torsion sheaf supported at $x_1=a_1\tau+b_1\in E_\tau$, then the corresponding Lagrangian submanifolds intersect at just one point, leading to canonical isomorphisms $\text{Hom}_{\mathsf{D^b}(E_\tau)}(\mathcal{E}_0,\mathcal{E}_1)\cong\text{Hom}(V_0,V_1)\cong V_0^*\otimes V_1$ and $\text{Hom}_{\mathscr{F\!K}^0(E^\tau)}\cong\text{Hom}(V_0,V_1)$. Then 
	\begin{small}
		\begin{align*}
			&(\mathsf{E}_\tau)_{\mathcal{E}_0,\mathcal{E}_1}:V_0^*\otimes V_1\rightarrow V_0^*\otimes V_1\,, \\
			&(f_0,v_1)\mapsto e^{-\pi i(\tau(na_1^2+2a_0a_1)+2(a_1b_0+a_0b_1+na_1b_1))}e^{-(a_0+na_1)f_0\otimes N_1(v_1)+a_1N_0^*(f_0)\otimes v_1}.
		\end{align*}
	\end{small}
	\!\!Moreover, on a more general $\pi_{r_0*}(\mathcal{E}_0)=\pi_{r_0*}(\mathcal{L}_{r_0\tau}(\varphi_0)\otimes F_{r_0\tau}(V_0,e^{N_0}))$, to define the dashed arrow $(\mathsf{E}_\tau)_{\pi_{r_0*}(\mathcal{E}_0),\mathcal{E}_1}$ in the following diagram one takes the detour:
	\begin{small}
		\[
		\begin{tikzcd}
			\text{Hom}_{\mathsf{D^b}(E_\tau)}(\pi_{r_0*}(\mathcal{E}_0),\mathcal{E}_1)\arrow[d, "\sim"']\arrow[r, dashed] &  \text{Hom}_{\mathscr{F\!K}^0(E^\tau)}\big(\mathsf{E}_\tau(\pi_{r_0*}(\mathcal{E}_0)),\mathsf{E}_\tau(\mathcal{E}_1)\big) \\ \text{Hom}_{\mathsf{D^b}(E_\tau)}(\mathcal{E}_0,\pi_{r_0}^*(\mathcal{E}_1))\arrow[d, "\mathsf{E}_{r\tau}"'] & \text{Hom}_{\mathscr{F\!K}^0(E^\tau)}\big(p_{r_0*}(\mathsf{E}_\tau(\mathcal{E}_0)),\mathsf{E}_\tau(\mathcal{E}_1)\big)\arrow[u, "\sim"'] \\ \text{Hom}_{\mathscr{F\!K}^0(E^\tau)}\big(\mathsf{E}_\tau(\mathcal{E}_0),\mathsf{E}_\tau(\pi_{r_0}^*(\mathcal{E}_1))\big)\arrow[r, "\sim"'] & \text{Hom}_{\mathscr{F\!K}^0(E^\tau)}\big(\mathsf{E}_\tau(\mathcal{E}_0),p_{r_0}^*(\mathsf{E}_\tau(\mathcal{E}_1))\big)\,,\arrow[u, "\sim"']		
		\end{tikzcd}
		\]
	\end{small}
	\!\!\!where we used \cite[Lemma 3.4]{[Kre00]} to turn pushforwards of source objects into pullbacks of target objects with respect to both $\pi_{r_0}$ and $p_{r_0}$, also exploiting that these morphisms turn into each other when commuted with $\mathsf{E}_\tau$; the map $\mathsf{E}_{r\tau}$ is known from the previous paragraph. We point out that the construction of $\mathscr{F\!K}^0(E^\tau)$ is here essential since $\pi_{r_0}^*(\mathcal{E}_1)$ is not indecomposable in general!
	
	\item The last case occurs when both coherent sheaves are indecomposable torsion sheaves, say $\mathcal{E}_i=\mathcal{S}(x_i,V_i,N_i)\in\text{obj}(\mathsf{D^b}(E_\tau))$, supported at $x_i=a_i\tau+b_i\in E_\tau$, so that by above $\mathsf{E}_\tau(\mathcal{E}_i)=((-a_i,t)/\mathbb{Z}^2,1/2,e^{2\pi i b_i\,\text{id}_{V_i}+N_i})$. So: if $x_0\neq x_1$, then $\text{Hom}_{\mathsf{D^b}(E_\tau)}(\mathcal{E}_0,\mathcal{E}_1)$, thus $(\mathsf{E}_\tau)_{\mathcal{E}_0,\mathcal{E}_1}$, is trivial; $a_0\neq a_1$ implies $\Lambda_0\cap\Lambda_1=\emptyset$ and $\text{Hom}_{\mathscr{F\!K}^0(E^\tau)}(\mathsf{E}_\tau(\mathcal{E}_0),\mathsf{E}_\tau(\mathcal{E}_1))\cong\{0\}$; $a_0=a_1$ and $b_0\neq b_1$ imply $\Lambda_0=\Lambda_1$ and $H^n(\Lambda_0,\mathcal{H}om(M_0,M_1))\cong\{0\}$, since $M_0$ and $M_1$ have no common eigenvalues; if $x_0=x_1$, then $\text{Hom}_{\mathsf{D^b}(E_\tau)}(\mathcal{E}_0,\mathcal{E}_1)=\text{Hom}_{\mathcal{O}_{E_\tau,x_0}}((V_0,N_0),(V_1,N_1))=\{f\in\text{Hom}(V_0,V_1)\mid f\circ N_0=N_1\circ F,\text{ iff }f\circ M_0=M_1\circ f\}=H^0(\Lambda_0,\mathcal{H}om(M_0,M_1))$.     
\end{itemize}
It only remains to check that $\mathsf{E}_\tau$ is well behaved with respect to the composition of morphisms, which will confirm it to be a functor. Hence, given any $\mathcal{E}_i\in\text{obj}(\mathsf{D^b}(E_\tau))$ for $i=0,1,2$, we must prove that the diagram
\[
\begin{small}
	\begin{tikzcd}
		\text{Hom}(\mathcal{E}_1[k],\mathcal{E}_2[l])\otimes\text{Hom}(\mathcal{E}_0,\mathcal{E}_1[k])\arrow[r, "\circ"]\arrow[d, "\mathsf{E}_\tau\otimes\mathsf{E}_\tau"'] & \text{Hom}(\mathcal{E}_0,\mathcal{E}_2[l])\arrow[d, "\mathsf{E}_\tau"] \\
		\text{Hom}(\mathsf{E}_\tau(\mathcal{E}_1)[k],\mathsf{E}_\tau(\mathcal{E}_2)[l])\otimes\text{Hom}(\mathsf{E}_\tau(\mathcal{E}_0),\mathsf{E}_\tau(\mathcal{E}_1)[k])\arrow[r, "\circ"'] & \text{Hom}(\mathsf{E}_\tau(\mathcal{E}_0),\mathsf{E}_\tau(\mathcal{E}_2)[l])
	\end{tikzcd}
\end{small}
\]
commutes, where $0\leq k\leq l\leq 1$ and $\circ$ denotes the appropriate composition law. But again by Serre duality, we can just restrict to the case $k=l=0$. We remain (even more) vague here:
\begin{itemize}[leftmargin=0.5cm]
	\item In case we are dealing with a triplet of indecomposable vector bundles in $\mathsf{D^b}(E_\tau)$ whose image are three distinct Lagrangian submanifolds, compatibility with compositions revolves around the manipulation of theta functions, as shown in \cite[section 5.2]{[PZ00]}. If only two among them are equal, one reduces their Hom-space to that of linear maps of vector spaces commuting with the monodromy operators, and the other Hom-space is treated with \cite[Proposition 2]{[PZ00]}. If instead all Lagrangian submanifolds are equal, one can take the defining $\varphi$ maps to fulfill $\varphi_0=\varphi_1\varphi_2$ and proceed.
	
	\item For general indecomposable vector bundles given by pushforward by isogenies, the analysis is carried out in \cite[section 5.4]{[PZ00]}.
	
	\item Finally, if at least one of the indecomposable coherent sheaves is torsion, the reader is referred to \cite[section 5.5]{[PZ00]}. 
\end{itemize}
This whole construction makes the functor $\mathsf{E}_\tau$ additive and fully faithful already. What is missing is just essential surjectivity, which will then upgrade $\mathsf{E}_\tau$ to the claimed equivalence. So let $(\Lambda,\alpha,M)\in\text{obj}(\mathscr{F\!K}^0(E^\tau))$ be any, without loss of generality indecomposable, in the sense that $M$ is indecomposable (not the direct sum of non-trivial operators), say then of the form $M=e^{-2\pi i b+N}\in\text{GL}(V)$ for $N\in\text{End}(V)$ cyclic and $b\in\mathbb{R}$. We can further safely assume that $\alpha\in(-1/2,1/2]$: if $\alpha=1/2$, then $(\Lambda,\alpha,M)\cong\mathsf{E}_\tau(\mathcal{S}(-a\tau-b,V,N))$ for a suitable $-1<a\leq 0$; otherwise, letting $(r,n)\in\mathbb{N}^2$ be relatively prime such that $r+in\in\mathbb{R}e^{i\pi\alpha}$ and assuming $a$ is such that $0\leq\frac{ra}{n}<\frac{1}{n}$ is the smallest $x$-intercept of a line representing $\Lambda$, it is $(\Lambda,\alpha,M)\cong\mathsf{E}_\tau\big(\pi_{r*}(\mathcal{L}_{r\tau}(t_{ar\tau+b}^*\varphi_0\cdot\varphi_0^{n-1})\otimes F_{r\tau}(V, e^N))\big)$.

Therefore, essential surjectivity is proven, and so is homological mirror symmetry in dimension 1.  
\end{proof}

We did it! (Although we skimmed over quite a lot of details.) The demanding reader is once again referred to the original work of Polishchuk and Zaslow --- where a proof for the simplest non-trivial version of Theorem \ref{HMSindim1} is discussed (see \cite[section 4]{[PZ00]}) --- and the complementing paper by Kreussler. In particular, the composition law in terms of the $A_\infty$-structure of both sides is addressed in \cite{[Pol01]}. We point again to \cite{[Por15]} for a complete overview of the subject, and to \cite{[Rei07]} for an in-depth analysis of the functoriality of $\mathsf{E}_\tau$. 

We also remind us again of \cite[Theorem 6.4]{[AS09]}, where the ground field is taken to be the Novikov field $\Lambda_\mathbb{R}$ and the mirror partner of the 2-torus is specifically shown to be the \textit{Tate elliptic curve}.

\subsection{Homological mirror symmetry in dimension 2}\label{ch7.4}

If the homological mirror symmetry of elliptic curves already looks highly non-trivial, in dimension 2 the situation becomes almost impenetrable, given our current background. Therefore, we will be sketchy to say the least.

First of all, let us identify what are the 2-dimensional Calabi--Yau manifolds. In terms of varieties, we expect to deal with \textbf{smooth projective surfaces}\index{smooth projective surface}. We refer to \cite[sections 5.3, 21.1]{[GHJ03]} and \cite[section 10.1]{[Huy06]} for partial accounts.

\begin{Def}\label{CY2folds}
Calabi--Yau manifolds of dimension 2 over $\mathbb{C}$ are either:
\begin{itemize}[leftmargin=0.5cm]
	\renewcommand{\labelitemi}{\textendash}
	\item \textbf{$K3$ surfaces}\index{k3@$K3$ surface}, the only ones which are compact and simply connected as Kähler manifolds (so, strictly speaking, the only admissible ones according to Definition \ref{Calabiyauman}); by Definition \ref{Calabiyauvar} they are those compact complex surfaces $X$ with trivial canonical line bundle and $H^1(X,\mathcal{O}_X)\cong\{0\}$. Some are algebraic, in that they are smooth quartic hypersurfaces in $\mathbb{P}^3$ over $\mathbb{C}$, such as the complex algebraic variety given by the vanishing locus $V(x_0^4+x_1^4+x_2^4+x_3^4)\subset \mathbb{P}^3$ (\textit{Fermat quartic}); some other are suitable quotients of abelian surfaces (like \textit{Kummer surfaces}), or \textit{elliptic fibrations} or \textit{complete intersections}.\footnote{Although not every $K3$ surface is algebraic, those which are still form a subset which is dense in the moduli space of all $K3$ surfaces.}
	
	\item \textbf{Abelian surfaces}\index{abelian surface} --- that is, 2-dimensional \textit{abelian varieties}, which are projective algebraic varieties with a structure of algebraic group --- which are not simply-connected, such as \textit{complex} 2-dimensional tori, or equivalently real 4-tori $\mathbb{T}^4$, among which we find products of pairs of elliptic curves. 
	
	\item Literature sometimes counts \textit{Enriques surfaces} and \textit{hyperelliptic surfaces} as Calabi--Yau 2-folds.  
\end{itemize}
\end{Def}

Since Conjecture \ref{HMSconj} is specifically tailored for Calabi--Yau manifolds in the ``strict sense'' of Definition \ref{Calabiyauman}, we mainly focus on the objects of the first bullet point.

\begin{Rem}
Here is a couple of relevant facts about $K3$ surfaces and friends:
\begin{itemize}[leftmargin=0.5cm]
	\item On the B-side, in contrast to the case of elliptic curves, the bounded derived category cannot discern non-isomorphic $K3$ or abelian surfaces; in other words, distinct pairs can be derived equivalent. Luckily, by \cite[Proposition 12.1]{[Huy06]}, these are the sole exceptions in dimension 2: any smooth projective variety $Y$ whose bounded derived category is equivalent to $\mathsf{D^b}(X)$, with $X$ a smooth projective surface which is neither $K3$, nor abelian, nor elliptic, is $Y\cong X$ (as varieties). To prove this, one needs knowledge about the Enriques classification for algebraic surfaces.
	
	\item Any smooth projective variety derived equivalent to a $K3$ surface is itself a $K3$ surface (by \cite[Corollary 10.2]{[Huy06]}). In fact, all $K3$ surfaces are diffeomorphic. The global Torelli theorem (\cite[Theorem 10.4]{[Huy06]}) affirms moreover that any two $K3$ surfaces are isomorphic if and only if there exists a Hodge isometry between their second integer cohomology groups.
\end{itemize}
\end{Rem}

Seidel proved in \cite{[Sei13]} a version of homological mirror symmetry for quartic $K3$ surfaces. We try to explain his result, following the already succint summary of \cite[section 8.5]{[ABC+09]}. 

\begin{Def}\label{Seidelstuff}
Let $\mathbb{C}[[q]]$ denote the ring of formal power series over $\mathbb{C}$ in the formal variable $q$ (whose elements are identified by tuples indexed by $\mathbb{N}_0$ which eventually vanish). Let
\begin{small}\[
	\mathcal{X}\coloneqq\big\{\big([x_0:x_1:x_2:x_3],q\big)\in\mathbb{P}^3\times\text{Spec}(\mathbb{C}[[q]])\bigm| q\cdot(x_0^4+x_1^4+x_2^4+x_3^4)+x_0x_1x_2x_3=0\big\}
	\]\end{small}
\!\!be a family of Fermat $K3$ surfaces, onto which the group
\begin{small}\[
	\Gamma_{16}\!\coloneqq\!\bigg\{(a_0,a_1,a_2,a_3)\!\in\!\mathbb{C}^4\biggm| a_0^4\!=\!a_1^4\!=\!a_2^4\!=\!a_3^4\!=\!\prod_{i=0}^3a_i\!=\!1\bigg\}\Big/\Big\{(a,a,a,a)\!\in\!\mathbb{C}^4\mid a^4\!=\!1\Big\}
	\]\end{small}
\!\!acts diagonally as $(a_0,a_1,a_2,a_3)\cdot(x_0,x_1,x_2,x_3)\coloneqq(a_0x_0,a_1x_1,a_2x_2,a_3x_3)$, thus yielding the quotient $\widehat{\mathcal{X}}\coloneqq\mathcal{X}/\Gamma_{16}$.

The generic fibre $\widehat{\mathcal{X}}_q$ of the canonical map $\widehat{\mathcal{X}}\rightarrow\text{Spec}(\mathbb{C}[[q]])$ is then a $K3$ surface over the field $\mathbb{C}((q))\coloneqq\text{Frac}(\mathbb{C}[[q]])$ (the fraction field of $\mathbb{C}[[q]]$, or field of formal Laurent series), so that we can consider its bounded derived category $\mathsf{D^b}(\widehat{\mathcal{X}}_q)$ and, more specifically, $\mathsf{D^b}(\widehat{\mathcal{X}}_q)\otimes\Lambda_\mathbb{C}$, obtained by tensoring all the $\text{Hom}_{\mathsf{D^b}(\widehat{\mathcal{X}}_q)}$-spaces with the Novikov field
\begin{small}\[
	\Lambda_\mathbb{C}\coloneqq\bigg\{\sum_{i\in\mathbb{Z}}a_iq^{\lambda_i}\biggm| a_i\in\mathbb{C}\text{ s.t. }a_i=0\text{ for }i\ll 0,\,\lambda_i\in\mathbb{R} \text{ s.t. }\lambda_i\xrightarrow{i\to\infty}\infty \bigg\}\,.
	\]\end{small}
\end{Def}

\begin{Thm}\textup{(\cite[Theorem 1.3]{[Sei13]})}\label{HMSforK3conj}
Let $X\subset\mathbb{P}^3$ be a smooth quartic surface over $\mathbb{C}$, regarded as a Kähler manifold by equipping it with the Kähler form obtained by restriction of the Fubini--Study metric of $\mathbb{P}^3$. Then there exist an automorphism of formal power series $\psi\in\textup{Aut}(\mathbb{C}[[q]])$ and an equivalence
\begin{equation}\label{HMSforK3}
	\mathsf{D^\pi}(\mathscr{F}(X))\rightarrow\psi^*(\mathsf{D^b}(\widehat{\mathcal{X}}_q)\otimes\Lambda_\mathbb{C})
\end{equation}
as triangulated categories. \textup(Actually, both derived categories arise as non-trivial deformations of the same category.\textup)
\end{Thm}

\begin{proof}(\textit{Very vague sketch})
The main strategy is to work with two sets which split-generate (in the sense of \cite[Definition 4.15]{[Imp21]}) the categories in \eqref{HMSforK3} respectively.

On the B-side, Beilinson's theorem \cite[Theorem 4.72]{[ABC+09]} states that $\mathsf{D^b}(\mathbb{P}^3)$ is generated by the set $\{\mathcal{O}_{\mathbb{P}^3}(-3),\mathcal{O}_{\mathbb{P}^3}(-2),\mathcal{O}_{\mathbb{P}^3}(-1),\mathcal{O}_{\mathbb{P}^3}\}$ of Serre-twisted locally free sheaves, so that $\Gamma_{16}$ has the effect of producing $4\cdot 16=64$ split-generators for $\mathsf{D^b}(\widehat{\mathcal{X}}_q)$.

On the A-side, $\mathcal{X}$ has 4 singular fibres over $\text{Spec}(\mathbb{C}[[q]])$ with 16 singular points each, inducing a total of 64 Lagrangian spheres, which can be shown to split-generate $\mathsf{D^\pi}(\mathscr{F}(X))$.

We mention that it is so far unknown (or at least, back in \cite{[Sei13]}) whether the map $\psi$ in the statement coincides with the standard mirror map... 
\end{proof}

\begin{Rem}
Theorem \ref{HMSforK3conj} could be confronted with the prediction of classical mirror symmetry asserting that the mirror dual of a $K3$ surface is again a $K3$ surface (see \cite[Example 7.10]{[Imp21]}).
\end{Rem}

Homological mirror symmetry has been studied also for other ``Calabi--Yau 2-folds'' listed in Definition \ref{CY2folds}. For example, Abouzaid and Smith deal with the 4-torus in \cite{[AS09]}, Fukaya discusses abelian surfaces in \cite{[Fuk00]}, and Kontsevich and Soibelman concern themselves with torus fibrations in \cite{[KS01]}. 

Finally, homological mirror symmetry in dimension 3 remains (even more) obscure, much like the classification of Calabi--Yau 3-folds, still an open problem. This is of course the most fascinating prospect for string theorists, for Calabi--Yau 3-folds encode the microscopic, compactified extra dimensions predicted by superstring theory. The best intuition is due to Strominger, Yau and Zaslow in \cite{[SYZ96]}, where mirror symmetry is studied through the physics lens of T-duality. (Consult for example the end of \cite[section 7.3]{[Imp21]} for a qualitative description of the SYZ Conjecture, and a nice picture). This said, higher dimensions do not always pose an obstacle, as the work \cite{[She14]} of Sheridan on homological mirror symmetry of Calabi--Yau hypersurfaces in the complex projective space of \textit{any} dimension shows!   

Research goes on, but we better stop here.

\newpage
\thispagestyle{plain}

\clearpage
\section*{Outlook}
\addcontentsline{toc}{section}{Outlook}

In the concluding paragraphs of \cite{[Imp21]}, I praised the beauty of Homological Mirror Symmetry, capable of bringing under a common roof seemingly detached areas of pure mathematics and theoretical physics. It is but one of many examples that illustrate how our advancing knowledge of the mathematical landscape is revealing a more subtly interwoven reality than once believed.

The present work reinforces this view by showing how more ``classical'' theories like homological algebra and algebraic geometry, not strictly finalized to Kontsevich's conjecture like perhaps a fair amount of the A-side, are capable of interacting in unexpected ways. Yet the historical record teaches us that many ideas are developed on a out-of-necessity basis, and hence are intrinsically entangled with the disciplines they are birthed from. It should then be welcomed with surprise that a solid and free-standing architecture such as the one discussed here fits into an exotic beast like Homological Mirror Symmetry. This is both an encouraging sign that we are on the right path, and a further motivation to keep going down it, which might reveal to us the power to solve insanely hard problems by simply moving to the ``mirror universe''.

On a more romantic and personal note, a student of mathematics or physics cannot help but be fascinated by the unifying vision promised by theories such as Homological Mirror Symmetry, tickling that primordial sense of wonder which is essential fuel for the research effort. In this sense, I hope that whoever reads these pages will let himself be inspired by the ``symphony'' they present.

But before I get too poetic or philosophical, I stop here and salute those who made it through my complete program on homological mirror symmetry, rewarding them with the following cartoon about its predictive beauty.
\vspace*{0.5cm}

\begin{figure}[htp]
\centering
\includegraphics[width=1\textwidth]{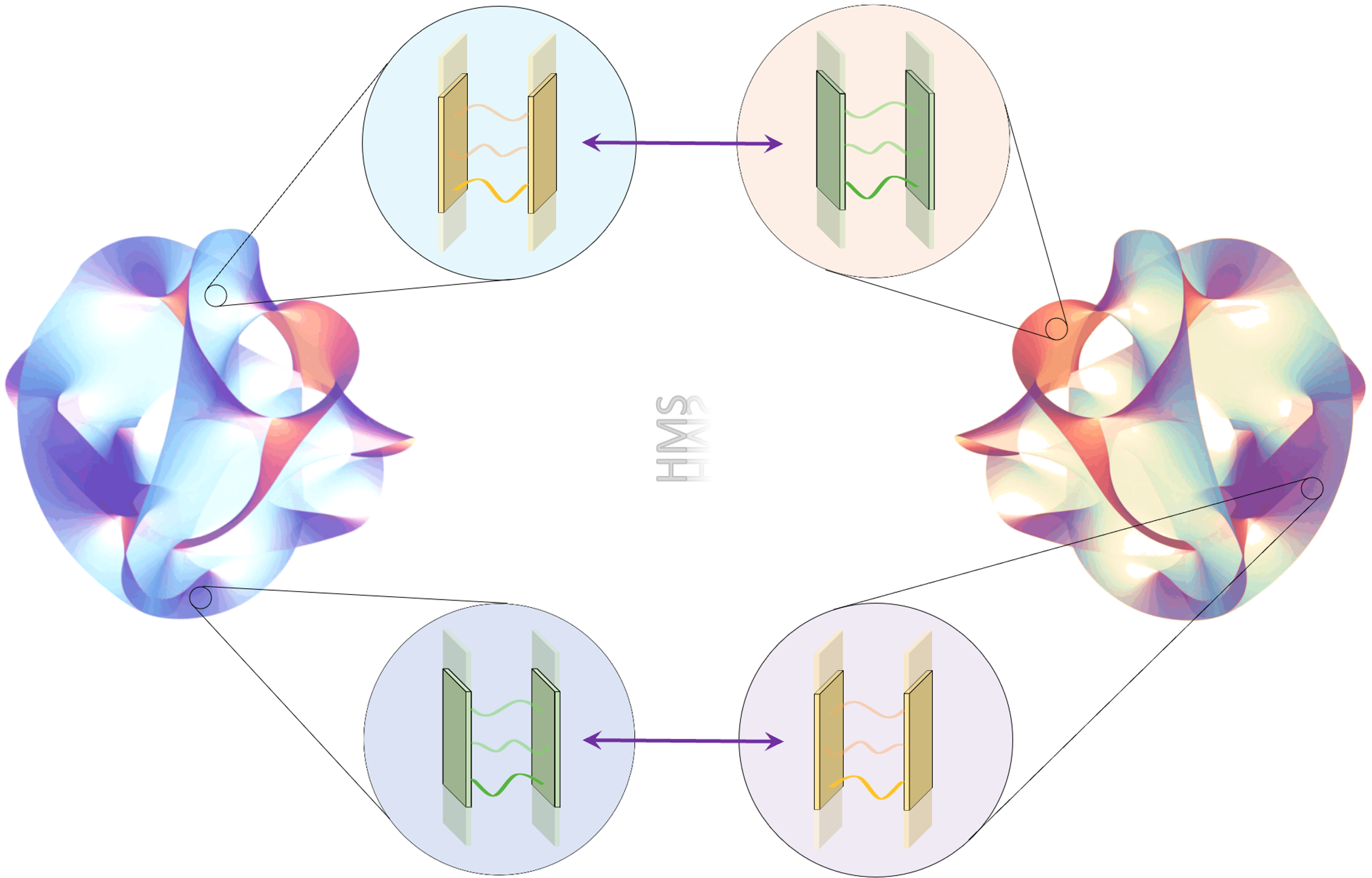}
\caption{A self-explanatory cartoon of the Homological Mirror Symmetry Conjecture. [Source: \begin{footnotesize}\texttt{https://plus.maths.org/content/hidden-dimensions}\end{footnotesize}]}
\end{figure}

\newpage
\thispagestyle{plain}

\clearpage
\section*{}

\newpage
\thispagestyle{plain}

\printindex
	
\end{document}